\documentclass[12pt,a4paper,twoside]{scrartcl}

\usepackage[utf8]{inputenc}

\usepackage[T1]{fontenc}
\usepackage{lmodern}
\usepackage[german,english]{babel}

\usepackage[tmargin=22mm,bmargin=22mm,lmargin=20mm,rmargin=20mm,
    marginparsep=0pt,marginparwidth=20mm]{geometry}

\usepackage{latexsym,amsmath,amssymb,mathtools,textcomp}
\usepackage{booktabs}
\usepackage{microtype}
\usepackage{morefloats}
\usepackage{dpfloat}
\usepackage{afterpage}
\usepackage[normalem]{ulem}

\usepackage{eucal}
\let\mathcal=\CMcal
\def\matheu{\EuScript}

\usepackage{xcolor}
\usepackage{hyperref}

\hypersetup{
  pdftitle={On the Structure of the Graph of Unique Symmetric Base Exchanges of Bispanning Graphs},
  pdfauthor={Timo Bingmann},
  pdfsubject={unique symmetric base exchange, tree pair graph, bispanning graph},
  colorlinks=true,
  pdfborder={0 0 0},
  bookmarksopen=true,
  bookmarksopenlevel=1,
  bookmarksnumbered=true,
  linkcolor=blue!60!black,
  citecolor=blue!60!black,
  urlcolor=blue!60!black,
  filecolor=green!60!black,
  pdfpagemode=UseNone,
  unicode=true,
}

\usepackage{float}
\usepackage[font=normalsize, labelformat=simple]{subcaption}

\usepackage[amsmath,thmmarks,framed,hyperref]{ntheorem}
\usepackage{framed}

\usepackage{makeidx}
\makeindex

\makeatletter
\let\realemph=\emph

\def\@emphi{\@ifnextchar[{\@emphiA}{\@emphiB}}
\def\@emphiA[#1]{\index{#1}\@emphi}
\def\@emphiB#1{\realemph{#1}}
\def\@emphiC<#1!#2>{\index{#1!#2}\index{#2!#1}\@emphi}
\def\@emphiN#1{\realemph{#1}}

\def\emph{\@ifnextchar[{\@emphi}{\@ifnextchar<{\@emphiC}{\@emphiN}}}
\makeatother

\usepackage[nomain]{glossaries}
\newglossary[syl]{symbols}{syi}{syg}{Symbols}
\newglossary[exgl]{exgraphs}{exgi}{exgg}{Exchange Graphs}

\def\symbol#1#2#3{%
  \newglossaryentry{symb:#1}{%
    name={#2},description={#3},sort={#1},type=symbols%
  }\glsadd{symb:#1}%
}

\def\exchangegraphsymbol#1#2#3{%
  \newglossaryentry{exgr:#1}{%
    name={#2},description={#3},sort={#1},type=exgraphs%
  }\glsadd{exgr:#1}%
}

\newglossarystyle{mylist}{%
  {\begin{description}[labelindent=0ex, labelwidth=12ex, itemsep=0ex, parsep=0pt plus .3 pt, leftmargin=18ex]}
  {\end{description}}%
  \def\glossarysection[##1]##2{\textbf{##2}\par\nopagebreak}%
  \renewcommand*{\glsgroupheading}[1]{}%
  \renewcommand*{\glsnamefont}{\unboldmath\let\mathbb=\mathbbnormal}%
}

\makeglossaries
\glossarystyle{mylist}

\def\Rtext#1{\def\mybreakright{#1}}
\renewtheoremstyle{break}%
  {\item[\rlap{\vbox{\hbox to \textwidth{\bfseries ##1\ ##2\hskip 0pt plus 1fil}\hbox{\strut}}}]}%
  {\def\mybreakright{}\let\R=\Rtext%
   \item[\rlap{\vbox{\hbox to \textwidth{\bfseries ##1\ ##2\ (##3\unskip)\hskip 0pt plus 1fil\textmd{\mybreakright}}\hbox{\strut}}}]}

\renewtheoremstyle{nonumberplain}%
  {\item[\itshape ##1\quad]}%
  {\item[\itshape ##1\ (##3)\quad]}

\theoremstyle{break}
\theoremheaderfont{\bfseries}
\theorembodyfont{\normalfont}

\def\theoremframecommand{\colorbox{black!7}}
\theoremframepreskip{0cm}
\theoremframepostskip{0cm}
\theoreminframepreskip{0cm}
\theoreminframepostskip{0cm}

\newshadedtheorem{theorem}{Theorem}[section]
\newshadedtheorem{lemma}[theorem]{Lemma}
\newshadedtheorem{corollary}[theorem]{Corollary}
\newshadedtheorem{definition}[theorem]{Definition}
\newshadedtheorem{remark}[theorem]{Remark}
\newshadedtheorem{example}[theorem]{Example}
\newshadedtheorem{conjecture}[theorem]{Conjecture}

\theoremstyle{nonumberplain}
\theoremheaderfont{\em}
\theoremsymbol{\ensuremath{\;\square}}
\theorempreskip{0pt}
\newtheorem{proof}{Proof}

\numberwithin{equation}{section}

\usepackage{graphicx}

\usepackage{array,tabularx,multirow}

\usepackage{enumitem}

\setlist[enumerate]{topsep=0pt, labelsep=1ex}
\setlist[description]{font=\normalfont,topsep=0pt}

\setlist[enumerate,1]{label=(\roman*), ref={\thetheorem\,(\roman*)}}

\usepackage{bbm}

\makeatletter
\DeclareRobustCommand*{\bfseries}{%
  \not@math@alphabet\bfseries\mathbf
  \fontseries\bfdefault\selectfont
  \boldmath\let\mathbbnormal=\mathbb\let\mathbb=\mathbbm
}
\makeatother

\usepackage{csquotes}
\usepackage[backend=bibtex,style=alphabetic,language=english,maxbibnames=99]{biblatex}
\usepackage{xpatch}
\xpatchbibdriver{article}
{\usebibmacro{note+pages}}
{\usebibmacro{note+pages}%
  \newunit\newblock
  \usebibmacro{publisher+location+date}}{}{}

\usepackage{tikz}
\usetikzlibrary{calc, intersections, matrix,positioning, decorations.markings, shapes.geometric, backgrounds}

\usetikzlibrary{external}
\tikzset{
  external/prefix={cache/},
  external/optimize=true,
}
\newbox\tikzCollectBox
\tikzset{
  phantom text/.style={
    execute at begin node={\setbox\tikzCollectBox=\hbox\bgroup},
    execute at end node={\egroup\phantom{\copy\tikzCollectBox}},
  },
  no text/.style={
    execute at begin node={\setbox\tikzCollectBox=\hbox\bgroup},
    execute at end node={\egroup},
  },
  tangent/.style 2 args={
    decoration={
      markings,% switch on markings
      mark=
      at position #1
      with
      {
        \coordinate (#2 point) at (0pt,0pt);
        \coordinate (#2 unit vector) at (1,0pt);
        \coordinate (#2 orthogonal unit vector) at (0pt,1);
      }
    },
    postaction=decorate
  },
  partial ellipse/.style args={#1:#2:#3}{
    insert path={+ (#1:#3) arc (#1:#2:#3)}
  },
}

\tikzstyle{dash dot dot}=[dash pattern=on 3pt off 2pt on \the\pgflinewidth off 2pt on \the\pgflinewidth off 2pt]

\colorlet{dred}{red}
\colorlet{dblue}{blue}

\tikzset{
  graphbase/.style={
    X/.style={
      draw, black, semithick
    },
    R/.style={
      draw, dred, thick,
      dash pattern=on 10pt off 1pt, dash phase=5pt,
      line join=round, line cap=butt,
    },
    B/.style={
      draw, dblue, thick,
      dash pattern={},
      line join=round, line cap=butt,
    },
    Rthick/.style={
      dred!25, line width=8pt, line join=round, line cap=round,
    },
    Bthick/.style={
      dblue!25, line width=8pt, line join=round, line cap=round,
    },
    Xsemithick/.style={
      black!25, line width=6pt, line join=round, line cap=round,
    },
    alabel/.style={inner sep=2pt},
    vdot/.style={draw,shape=circle,fill=white,inner sep=2.5pt,text depth=0pt},
    vlabel/.style={auto,inner sep=1pt},
    ndot/.style={inner sep=0,outer sep=3pt},
  },
  graphfinal/.style={
    graphbase,
    odot/.style={no text,draw=black,solid,shape=circle,fill=black,inner sep=0,minimum size=1.5mm},
    odotsmall/.style={no text,draw=black,solid,shape=circle,fill=black,inner sep=1pt},
    elabel/.style={no text},
    UElabel/.style={sloped, font=\footnotesize},
  },
  graphnumber/.style={
    graphbase,
    odot/.style={draw,solid,shape=circle,inner sep=1pt,minimum size=1.5mm},
    elabel/.style={shape=circle,fill=white,inner sep=1pt},
  },
  taugraph/.style={
    vdot/.style={
      draw,shape=circle,fill=white,inner sep=2.5pt,text depth=0pt
    },
    BigNode/.style={
      draw, fill=black!5,
      shape=circle, transform shape,
      no text,
    },
    X/.style={
      black, thick,
    },
    R/.style={
      dred, thick,
      dash pattern=on 10pt off 1pt, dash phase=5pt,
    },
    B/.style={
      dblue, thick,
      dash pattern={},
    },
    BR/.style={
      dblue, thick, dash pattern=on 5pt off 7pt,
      postaction={
        draw, dred, thick, dash pattern=on 5pt off 7pt, dash phase=6pt,
      },
      BigLabel/.append style={text=blue!50!red},
    },
    BigLabel/.style={
      sloped, font=\boldmath\footnotesize, inner sep=1pt, fill=white,
      allow upside down=true,
    },
    SmallNode/.style={
      inner sep=1pt,
      font=\tiny\bfseries,
      draw, shape=circle,
      no text, fill=black,
    },
    SmallLabel/.style={
      font=\scriptsize\bfseries, inner sep=1pt, outer sep=0pt, auto,
    },
  },
}

\def\CompleteFourWheelCoordinates#1{
  \begin{scope}[
    shift=(#1.center),
    scale=0.7]

    \coordinate (0) at (0,0);
    \coordinate (1) at (0,1);
    \coordinate (2) at (-1,0);
    \coordinate (3) at (0,-1);
    \coordinate (4) at (1,0);
  \end{scope}
}

\def\CompleteFourWheel#1#2{
  \begin{scope}[
    shift=(#1.center),
    scale=0.7,
    SmallLabel0/.style={SmallLabel, yshift=-0.5mm},
    SmallLabel1/.style={SmallLabel, swap, inner sep=0.5pt},
    SmallLabel2/.style={SmallLabel, xshift=0.5mm},
    SmallLabel3/.style={SmallLabel, swap, inner sep=0.5pt},
    SmallLabel4/.style={SmallLabel, yshift=0.5mm},
    SmallLabel5/.style={SmallLabel, swap, inner sep=0.5pt},
    SmallLabel6/.style={SmallLabel, xshift=-0.5mm},
    SmallLabel7/.style={SmallLabel, swap, inner sep=0.5pt},
    ]

    \node (0) [at={(0,0)}, SmallNode] {0};
    \node (1) [at={(0,1)}, SmallNode] {1};
    \node (2) [at={(-1,0)}, SmallNode] {2};
    \node (3) [at={(0,-1)}, SmallNode] {3};
    \node (4) [at={(1,0)}, SmallNode] {4};

    \def\sA{{0,1,0,2,0,3,0,4}}
    \def\tA{{1,2,2,3,3,4,4,1}}

    \foreach \x [count=\i from 0] in {#2} {
      \pgfmathtruncatemacro{\s}{\sA[\i]}
      \pgfmathtruncatemacro{\t}{\tA[\i]}
      \draw[\x] (\s) -- node[SmallLabel\i] {\i} (\t);
    }
  \end{scope}
}

\usepackage{fancyhdr}
\usepackage[english]{isodate}

\fancypagestyle{plain}{
  \setlength{\headheight}{20pt}
  \setlength{\footskip}{32pt}
  \fancyhead{}
  \fancyfoot{}
  \fancyfoot[LE,RO]{\normalsize\thepage}

}

\fancypagestyle{normal}{
  \setlength{\headheight}{20pt}
  \setlength{\footskip}{32pt}
  \fancyhead{}
  \fancyhead[LE]{\normalsize\textsc{\nouppercase{\leftmark}}}
  \fancyhead[RO]{\normalsize\textsc{\nouppercase{\rightmark}}}
  \fancyfoot{}
  \fancyfoot[LE,RO]{\normalsize\thepage}
}

\usepackage[ruled,vlined,linesnumbered,norelsize]{algorithm2e}
\DontPrintSemicolon
\def\NlSty#1{\textnormal{\fontsize{8}{10}\selectfont{}#1}}
\SetKwSty{texttt}
\SetFuncSty{textsf}
\SetCommentSty{textit}
\SetDataSty{textsf}
\def\listalgorithmcfname{List of Algorithms}
\def\algorithmautorefname{Algorithm}
\let\chapter=\section % fix a problem with algorithm2e

\SetKwProg{Function}{Function}{}{}

\makeatletter
\newcommand{\lSetKwArray}[2]{%
  \expandafter\def\csname @#1\endcsname##1{\DataSty{#2[}\ArgSty{##1}\DataSty{]}}%
  \expandafter\def\csname#1\endcsname{%
    \@ifnextchar\bgroup{\csname @#1\endcsname}{\DataSty{#2}\xspace}}%
}
\makeatother

\mathchardef\ordinarycolon\mathcode`\:
\mathcode`\:=\string"8000
\begingroup \catcode`\:=\active
\gdef:{\mathrel{\mathop\ordinarycolon}}
\endgroup

\makeatletter
\def\xnoarrowfill@#1{%
  \setboxz@h{$#1-\m@th$}\ht\z@\z@%
  $#1\m@th\vphantom{\rightarrow}\copy\z@\mkern-6mu\cleaders%
  \hbox{$#1\mkern-2mu\copy\z@\mkern-2mu$}\hfill%
  \mkern-6mu\box\z@$}
\newcommand{\xnoarrow}[2][]{%
  \ext@arrow 3095\xnoarrowfill@{#1}{#2}}
\makeatother

\def\arr#1{[\,#1\,]}
\def\contr{\mathbin{/}}
\def\dotcup{\mathbin{\dot\cup}}

\def\xchgarrow#1#2{\underset{\smash{#2}}{\xrightarrow{\smash{#1}}}}
\def\UExchgarrow#1#2{\underset{\smash{\text{UE}\;#2}}{\xrightarrow{\smash{#1}}}}
\def\UEbixchgarrow#1#2{\underset{\smash{\text{UE}\;#2}}{\xnoarrow{\smash{#1}}}}

\begin{document}

%%%%%%%%%%%%%%%%%%%%%%%%%%%%%%%%%%%%%%%%%%%%%%%%%%%%%%%%%%%%%%%%%%%%%%%%%%%%%%%%

\pagestyle{empty} % no page numbers

\begin{titlepage}

  \begin{center}\large

    \vspace*{5mm}

    \centerline{\includegraphics[height=15mm]{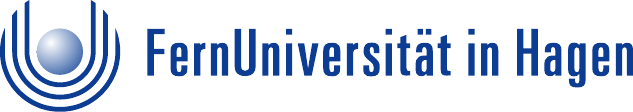}}

    \vfill

    Diploma Thesis
    \vspace*{2cm}

    {\bfseries\huge On the Structure of the Graph of \\ Unique Symmetric Base Exchanges of \\ Bispanning Graphs \par}

    \vfill

    Timo Bingmann

    \vspace*{15mm}

    Karlsruhe, January 14th, 2016

    \vspace*{30mm}

    \begin{tabular}{rl}
      Advisor: & Prof.\ Dr.\ Winfried Hochstättler \\
    \end{tabular}

    \vspace*{10mm}

    Chair for Discrete Mathematics and Optimization \\
    Department of Mathematics and Computer Science \\
    FernUniversität in Hagen

    \vspace*{12mm}
  \end{center}

\end{titlepage}

%%%%%%%%%%%%%%%%%%%%%%%%%%%%%%%%%%%%%%%%%%%%%%%%%%%%%%%%%%%%%%%%%%%%%%

\pagestyle{plain}

\vspace*{0pt}\vfill

\hrule\medskip

Hiermit versichere ich, dass ich diese Arbeit selbständig verfasst und keine anderen, als die angegebenen Quellen und Hilfsmittel benutzt und die wörtlich oder inhaltlich übernommenen Stellen als solche kenntlich gemacht habe.

\bigskip

\noindent
Karlsruhe, den 14. Januar 2016

% Unterschrift (handgeschrieben)

\vspace*{5cm}

\clearpage

%%%%%%%%%%%%%%%%%%%%%%%%%%%%%%%%%%%%%%%%%%%%%%%%%%%%%%%%%%%%%%%%%%%%%%%%%%%%%%%%

\vspace*{0pt}\vfill

\selectlanguage{english}
%\begin{abstract}
\centerline{\bfseries Abstract}

Bispanning graphs are undirected graphs with an edge set that can be decomposed into two disjoint spanning trees.  The operation of symmetrically swapping two edges between the trees, such that the result is a different pair of disjoint spanning trees, is called an edge exchange or a symmetric base exchange.  The graph of symmetric base exchanges of a bispanning graph contains a vertex for every valid pair of disjoint spanning trees, and edges between them to represent all possible edge exchanges. We are interested in a restriction of these graphs to only unique symmetric base exchanges, which are edge exchanges wherein selecting one edge leaves only one choice for selecting the other. In this thesis, we discuss the structure of the graph of unique symmetric edge exchanges, and the open question whether these are connected for all bispanning graphs.

This abstract problem can be nicely rephrased into a coloring game with two players: Alice and Bob are given a bispanning graph colored with two disjoint spanning trees, and Alice gets to flip the color of any edge. This creates a cycle in one color and a cut in the other, and Bob must then flip a different edge to repair the constraint that both colors represent disjoint spanning trees. Alice's objective is to invert the color of all edges in the graph, and Bob's to prevent this. We are interested in whether Alice can find a sequence of unique edge exchanges for any bispanning graph, since these leave Bob no choice in which edge to select, hence allowing Alice to win with certainty.

In this thesis, we first define and discuss the properties of bispanning graphs in depth. Intuitively, these are locally dense enough to allow the two disjoint spanning trees to reach all vertices, but sparse enough such that disjoint edge sets do not contain cycles. The whole class of bispanning graphs can be inductively constructed using only two operations, which makes the class tractable for inductive proofs.

We then describe in detail directed, undirected, and simplified versions of edge exchange graphs, first with unrestricted edge exchanges, and then with the restriction to unique symmetric base exchanges. These exchange graphs are related to a set of conjectures put forth by White in 1980 on base exchanges in matroids, and also to conjectures on cyclic base orderings of matroids. To date, these conjectures have not been proven in full generality, despite overwhelming computational evidence.

As steps towards showing the conjecture that the graph of unique symmetric base exchanges is connected for all bispanning graphs, we prove a composition method to construct the unique exchange graph of any bispanning graph from the exchange graphs of smaller bispanning graphs. Furthermore, using a computer program developed alongside this thesis, we are able to enumerate and make statements about all small bispanning graphs and their exchanges graphs.

Our composition method classifies bispanning graphs by whether they contain a non-trivial bispanning subgraph, and by vertex- and edge-connectivity. For bispanning graphs containing a non-trivial bispanning subgraph, we prove that the unique exchange graph is the Cartesian graph product of two smaller exchange graphs. For bispanning graphs with vertex-connectivity two, we show that the bispanning graph is the $2$-clique sum of two smaller bispanning graphs, and that the unique exchange graph can be built by joining their exchange graphs and forwarding edges at the join seam. And for all remaining bispanning graphs, we prove a composition method at a vertex of degree three, wherein the unique exchange graph is constructed from the exchange graphs of three reduced bispanning graphs.

We conclude this thesis with ideas and evidence for future approaches to proving the connectivity of the unique exchange graphs and show the most difficult bispanning graphs instances.
%\end{abstract}

\vfill\vfill\vfill
\clearpage

\vspace*{0pt}\vfill

\selectlanguage{german}
%\begin{abstract}
\centerline{\bfseries Zusammenfassung}

Bispannende Graphen sind ungerichtete Graphen, deren Kantenmenge sich in zwei disjunkte aufspannende Bäume zerlegen lässt. Man nennt einen symmetrischen Tausch von zwei Kanten zwischen den beiden Bäumen einen zulässigen Kantentausch oder einen symmetrischen Basenwechsel, wenn das Ergebnis ein anderes Paar disjunkter aufspannender Bäume ist. Der Graph der symmetrischen Basenwechsel eines bispannenden Graphen enthält einen Knoten für jedes gültige Paar disjunkter aufspannender Bäume und eine Kante für jeden zulässigen Kantentausch. Wir interessieren uns für die Einschränkung dieses Graphen auf zwingende symmetrische Basenwechsel, bei denen durch die Wahl eines der Tauschkanten die andere eindeutig bestimmt wird. Die vorliegende Arbeit befasst sich mit der Struktur des Graphen der zwingenden symmetrischen Basenwechsel und mit der offenen Frage, ob dieser für alle bispannenden Graphen verbunden ist.

Dieses abstrakte Problem lässt sich anschaulich als Färbungsspiel auf einem Graphen mit zwei Spielern darstellen: Alice und Bob ist ein bispannender Graph gegeben, in dem zwei aufspannende Bäume durch zwei verschiedene Kantenfarben gekennzeichnet sind. Alice darf die Farbe einer Kante tauschen. Hierdurch entsteht ein Kreis in einer Farbe und ein Schnitt in der anderen. Bob muss nun durch Umfärben einer anderen Kanten die Bedingungen wiederherstellen, dass beide Farben zwei disjunkte aufspannende Bäume darstellen. Alices Ziel ist die Farben aller Kanten im Graphen zu tauschen, Bobs dies zu verhindern. Wir interessieren uns dafür, ob Alice eine Folge von zwingenden symmetrischen Basenwechseln finden kann, denn diese zwingen Bob eine bestimmte Kante zu wählen und erlauben es somit Alice mit Sicherheit zu gewinnen.

In dieser Arbeit definieren und diskutieren wir zuerst die Eigenschaften von bispannenden Graphen. Intuitiv sind diese lokal dicht genug, um zwei disjunkte aufspannende Bäume zu zulassen, aber licht genug, dass disjunkte Kantenmengen keine Kreise enthalten. Die Klasse der bispannenden Graphen lässt sich mit nur zwei Operationen induktiv konstruieren, was sie für induktive Beweise greifbar macht.

Dann beschreiben wir im Detail gerichtete, ungerichtete und einfache Varianten von Basenwechselgraphen, zuerst ohne Einschränkung und dann auf zwingende symmetrische Basenwechsel beschränkt. Diese Basenwechselgraphen stehen in Beziehung zu Vermutungen von White aus dem Jahre 1980 und zu weiteren Vermutungen zur zyklischen Basenanordnung von Matroiden. Diese Vermutungen wurden bis heute noch nicht in voller Allgemeinheit bewiesen, trotz einer überwältigenden Anzahl mit Computer verifizierten Beispielen.

Als Schritte um die Vermutung zu zeigen, dass alle Basenwechselgraphen trotz Einschränkung auf zwingende Basenwechsel verbunden sind, beweisen wir eine Methode, um den zwingenden Basenwechselgraphen jedes bispannenden Graphen aus den Basenwechselgraphen kleinerer bispannender Graphen zusammenzusetzen. Für bispannende Graphen, die einen nicht-trivialen bispannenden Teilgraphen enthalten, zeigen wir, dass der zwingende Basenwechselgraph das Kartesische Graphprodukt von zwei kleineren Basenwechselgraphen ist. Für bispannende Graphen mit Knotenzusammenhang zwei können wir beweisen, dass dieser sich als die $2$-Clique-Summe von zwei kleineren bispannenden Graphen darstellen lässt, und dass der zwingende Basenwechselgraph sich durch Zusammenfügen der Basenwechselgraphen dieser beiden konstruieren lässt. Für die übrigen bispannenden Graphen zeigen wir eine Reduktion an einem Knoten mit Grad drei und eine Methode, den zwingende Basenwechselgraphen aus den Basenwechselgraphen von drei reduzierten bispannenden Graphen zu erzeugen.

Als Abschluss der Arbeit diskutieren wir Ideen und Hinweise für zukünftige Ansätze die Verbundenheit des zwingenden Basenwechselgraphen zu beweisen, und verweisen auf die schwierigsten Instanzen bispannender Graphen.

%\end{abstract}
\selectlanguage{english}

\vfill\vfill\vfill
\clearpage

%%%%%%%%%%%%%%%%%%%%%%%%%%%%%%%%%%%%%%%%%%%%%%%%%%%%%%%%%%%%%%%%%%%%%%%%%%%%%%%%

\vspace*{0pt}\vfill

\section*{Acknowledgments and Thanks}

I would like to thank my parents and friends for emotionally supporting me in the long time it took to make this thesis. Special thanks goes to Prof.\ Hochstättler for enabling me to write it with (overly) long periods of interruptions. Additional special thanks go to my Taekwondo sports friends for persistent weekly inquiries about the state of the thesis, and also to my work colleagues for recognizing bispanning graphs as interesting graphs and the deeper problems as worthwhile and demanding.

\vfill\vfill\vfill
\clearpage

%%%%%%%%%%%%%%%%%%%%%%%%%%%%%%%%%%%%%%%%%%%%%%%%%%%%%%%%%%%%%%%%%%%%%%%%%%%%%%%%

\pagestyle{normal}
% mark sections in fancy headers on left and subsections on right
\renewcommand\sectionmark[1]{\markboth{\thesection\quad\MakeUppercase{#1}}{\thesection\quad\MakeUppercase{#1}}}
\renewcommand\subsectionmark[1]{\markright{\thesubsection\quad\MakeUppercase{#1}}}

% Table of Contents
\def\contentsname{Table of Contents}
\pdfbookmark[1]{\contentsname}{summary}
\tableofcontents

%%%%%%%%%%%%%%%%%%%%%%%%%%%%%%%%%%%%%%%%%%%%%%%%%%%%%%%%%%%%%%%%%%%%%%%%%%%%%%%%
\clearpage

\pdfbookmark[1]{\listfigurename}{lof}
\listoffigures

\pdfbookmark[1]{\listtablename}{lot}
\listoftables

{
\pdfbookmark[1]{\listalgorithmcfname}{loa}
\parskip=0pt
\listofalgorithms
}

%\clearpage
{\def\R#1{}
\pdfbookmark[1]{List of Theorems and Definitions}{lotd}
\section*{List of Theorems and Definitions}
\parskip=0pt

\makeatletter
\def\thm@@thmline@noname#1#2#3#4#5{%
    \ifx\\#5\\%
        \@dottedtocline{-2}{0em}{4em}%
            {\protect\numberline{#2}#3}%
            {#4}%
    \else
        \ifHy@linktocpage\relax\relax
            \@dottedtocline{-2}{0em}{4em}%
                {\protect\numberline{#2}#3}%
                {\hyper@linkstart{link}{#5}{#4}\hyper@linkend}
        \else
            \@dottedtocline{-2}{0em}{4em}%
                {\hyper@linkstart{link}{#5}{\protect\numberline{#2}#3}%
                  \hyper@linkend}%
                {#4}%
        \fi
    \fi}%
\def\thm@@thmline@name#1#2#3#4#5{%
    \ifx\\#5\\%
        \@dottedtocline{-2}{0em}{4em}%
            {#1 \protect\numberline{#2}#3}%
            {#4}
    \else
        \ifHy@linktocpage\relax\relax
            \@dottedtocline{-2}{0em}{4em}%
                {#1 \protect\numberline{#2}#3}%
                {\hyper@linkstart{link}{#5}{#4}\hyper@linkend}%
        \else
            \@dottedtocline{-2}{0em}{6ex}%
                {\hyper@linkstart{link}{#5}%
                  {\hbox to 10.5ex{#1\hfill}\protect\numberline{#2}#3}\hyper@linkend}%
                {#4}%
        \fi
    \fi}
\makeatother

\theoremlisttype{allname}
\listtheorems{theorem,definition,lemma,corollary,remark}
}

% add space to list of theorems between sections
\let\Section\section
\def\section{\addtotheoremfile{\protect\medskip}\Section}

\clearpage

%%%%%%%%%%%%%%%%%%%%%%%%%%%%%%%%%%%%%%%%%%%%%%%%%%%%%%%%%%%%%%%%%%%%%%%%%%%%%%%%

\pagestyle{normal}
% mark sections in fancy headers on left and subsections on right
\renewcommand\sectionmark[1]{\markboth{\thesection\quad\MakeUppercase{#1}}{\thesection\quad\MakeUppercase{#1}}}
\renewcommand\subsectionmark[1]{\markright{\thesubsection\quad\MakeUppercase{#1}}}

% ------------------------------------------------------------------------------

\section{Introduction: Unique Base Exchanges as a Coloring Game}

In this thesis we consider whether the restriction to unique or ``forced'' symmetric base exchanges still allows a complete serial exchange of any pair of disjoint spanning trees in bispanning graphs.  The underlying abstract problem is best introduced using a coloring game on a graph.

Let there be two players: Alice and Bob, who play the following game on a special type of graph. The board they play on is a graph whose edge set admits decomposition into exactly two edge-disjoint spanning trees. Such graphs are called \emph{bispanning}, and an example is shown in figure~\ref{fig:game step}. Alice's tree \textcolor{blue}{$T_A$ is colored blue} and Bob's tree \textcolor{red}{$T_B$ is colored red}.

The two players move in turns and Alice gets to start: she selects one edge in the graph and switches the edge's color. Due to the properties of the two spanning trees, switching the color of any edge creates a cycle in one tree and a cut in the other. These violate the constraint that $T_A$ and $T_B$ are spanning trees, and Bob now has to fix this by selecting one edge \emph{different from Alice's} and flipping its color. After the graph is fixed, Alice continues with her next move.

For example, Alice decides to switch the red edge $e_1$ in figure~\ref{fig:game step} to blue. This creates the cycle and cut marked in blue. In this case, Bob has \emph{no choice} but to switch the color of $f_1$ to restore the spanning trees. The reader can verify that no edge other than $e_1$ suffices.

The objective of Alice is to start with $T_A$ and slowly turn her tree into $T_B$, inverting the colors of all edges in the graph. Bob's goal is to prevent Alice from doing so. We invite the reader to play Alice's role using the Java applet at \url{http://panthema.net/2016/uegame/}. The game raises the obvious question: \emph{is there a strategy with which Alice can win on any graph}?

This question is remarkably difficult, and we will consider only strategies where Alice exclusively uses moves which leave Bob \emph{no choice} in the edge he can color. In matroids, these \emph{forced moves} are called unique symmetric base exchanges. Using a computer program, we verified that \emph{Alice can win} for all bispanning graphs with up to 20 vertices. Even more surprising: Alice and Bob never have to color an edge \emph{more than once} on any bispanning graph with up to 12 vertices.

\begin{figure}[b!]\centering

%V10:i0x878y594/i1x481y870/i2x537y1134/i3x1194y676/i4x899y1074/i5x178y706/i6x640y601/i7x764y902/i8x926y362/i9x544y360/;E18:i0t1h4c1/i1t2h4c1/i2t8h9c2/i3t0h8c1/i4t9h1c1/i5t5h9c1/i6t7h2c2/i7t3h4c2/i8t5h6c1/i9t6h0c2/i10t2h5c2/i11t0h7c1/i12t3h0c1/i13t8h3c2/i14t3h7c2/i15t1h5c2/i16t6h1c2/i17t6h7c1/;
%V10:i0x490y225/i1x374y462/i2x150y257/i3x639y109/i4x510y398/i5x230y471/i6x334y254/i7x374y54/i8x707y406/i9x411y608/;E18:i0t1h4c1/i1t2h4c1/i2t8h9c2/i3t0h8c1/i4t9h1c1/i5t5h9c1/i6t7h2c2/i7t3h4c2/i8t5h6c1/i9t6h0c2/i10t2h5c2/i11t0h7c1/i12t3h0c1/i13t8h3c2/i14t3h7c2/i15t1h5c2/i16t6h1c2/i17t6h7c1/;

  \def\exampleGraph{
    \coordinate (n0) at (0,0);
    \coordinate (n1) at (-0.8,-0.9);
    \coordinate (n2) at (-2.0,-0.1);
    \coordinate (n3) at (0.9,0.8);
    \coordinate (n4) at (0.2,-0.7);
    \coordinate (n5) at (-1.6,-1.1);
    \coordinate (n6) at (-0.9,0.1);
    \coordinate (n7) at (-0.6,0.9);
    \coordinate (n8) at (1.3,-1.0);
    \coordinate (n9) at (-0.9,-1.5);

    \node (Alice) at (-1.8,0.7) {\includegraphics[height=1cm]{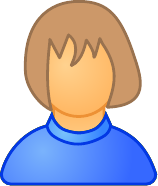}};
    \node (Bob) at (1.6,-0.6) {\includegraphics[height=1cm]{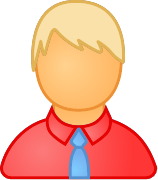}};
  }


  \caption{Example of one round of turns in Alice and Bob's game.}\label{fig:game step}
\end{figure}

A proof that Alice can win for all bispanning graphs will remain open in this thesis; despite overwhelming computation evidence, we could not find a complete proof.  Instead, we consider the structure of the ``super-graphs'' in which edges are moves in the game: each edge resembles a forced exchange done by Alice and Bob. If these \emph{(base) exchange graphs} $\tau_3(G)$ (see section~\ref{sec:exchange graph}) are connected for all bispanning graphs then Alice can win using forced moves.  We show three \emph{composition} methods to construct exchange graphs from known exchange graphs of smaller bispanning graphs. For any bispanning graph, at least one of the three compositions applies: for bispanning graphs containing a \emph{non-trivial bispanning subgraph}, we prove that the exchange graph is the Cartesian graph product of two smaller exchange graphs. For bispanning graphs with \emph{vertex-connectivity two}, we prove that the exchange graph can be build by joining the two smaller exchanges graphs and forwarding edges. And for all remaining atomic bispanning graphs, we prove a composition method at \emph{a vertex of degree three}, wherein the exchange graph is constructed from the exchange graphs of three reduced bispanning graphs.

Adjoint with this thesis, we wrote a computer program which simulates Alice and Bob's game and can construct complete exchange graphs. Using it, we verified all presented theorems about bispanning and exchange graphs. Furthermore, by using a graph enumeration library by McKay and Piperno \cite{McKay201494}, we were able to enumerate all small bispanning graphs and present empirical evidence of various hypothesis in this thesis.

While we can present a composition method for exchange graphs, we could not prove their connectivity using them. This remains unsolved, because conclusions over multiple levels of such compositions are unclear. In the closing section, we discuss four possible future approaches to solving this problem, and then a number of the most ``difficult'' bispanning graphs.

% ------------------------------------------------------------------------------

\subsection{Alternative Rules and the Shannon/Lehman Switching Game}

Many difficult abstract problems have been recast as games on graphs~\cite{lason2015coloring}, and sometimes these lead to deeper insights in the field of graph or matroid theory.

Possibly the most famous is the Shannon switching game, also called Lehman's switching game~\cite{lehman1964solution}. In this game, two players ``Cut'' and ``Short'' play on an arbitrary graph with two designated vertices $A$ and $B$. The graph is initially uncolored, and the players move in turns. On Cut's turn, he can delete any uncolored edge, and on Short's turn he can color any remaining edge. If Short can color a path from $A$ to $B$ before Cut disconnects them, then Short wins, otherwise Cut wins. Lehman applied matroid theory to the problem and was the first to arrive at a proper solution. He classified graphs together with designated vertices as either a ``cut game'' if Cut always wins, a ``short game'' if Short always wins, or a ``neutral game'' if the player with the first move wins. A main result is that a graph is a short game, if and only if there is a subgraph which is bispanning and contains both $A$ and $B$. Hence, a bispanning graph is a short game for any pair of vertices. Furthermore, a graph $G$ is a cut game if and only if $G$ augmented with an edge $\{A,B\}$ is not a short game. The remaining combination, when $G$ is not a short game, and $G$ plus $\{A,B\}$ is a short game, characterizes all neutral games.

Comparing the switching game to our bispanning graph game, the decisive difference is that the bispanning constraint must be valid after each round. This greatly restricts the possible moves and makes it more difficult to solve. We can consider what happens to Alice and Bob's game when the rules are slightly changed. For example, what happens if Alice and Bob decide to \emph{cooperate} in inverting the graph? In this scenario, Alice no longer needs forced moves. The resulting game is then rather easy, as many more edge pairs yield valid exchange moves, and the exchange graph is connected~\cite{farber1985edge}.

Another alternative is to \emph{restrict} which edges Alice can choose to force Bob's moves. In our game, Alice may color any edge in the graph. What happens if Alice may only color edges in Bob's tree (or equivalently: only in her tree)? Andres, Hochstättler, and Merkel~\cite{merkel2009basentauschspiel,andres2014base} showed for this restricted game, that an induced complete graph $K_4$ forfeits Alice's goal of inverting the graph. On the other hand, Alice wins in other special graph classes like wheel graphs $W_n$. Even more peculiarly, in some graphs it depends on the initial pairs of trees whether Alice can win. It remains unclear what the precise favorable property of the graph must be for Alice.

% ------------------------------------------------------------------------------

\subsection{From Coloring Games to Problems on Matroids}

Our bispanning graph game is related to one of the many abstract problems proposed by White in 1980 on graphs and matroids~\cite{white1980unique}. These are stated in the language of symmetric base exchanges on matroids (see section~\ref{sec:matroids} for more on matroids).

White distinguishes unique base exchanges, of which the ``forced'' moves in Alice and Bob's game are examples, from ordinary symmetric exchanges, which he calls transitive exchanges and are not necessarily unique. Furthermore, he describes subset exchanges and single-element serial exchange sequences. The main objective of White's paper is to characterize the matroid classes, within which unique or transitive base exchanges are sufficient to transform any sequence of bases into any other. We will review his conjectures and results in section~\ref{sec:matroids}. The paper states two main conjectures: that unique base exchanges suffice for \emph{all regular} matroids, and that transitive base exchanges suffice for \emph{all} matroids.

The second conjecture about transitive base exchanges is regarded as important for algebraic geometry, and was quickly reformulated in algebraic terms~\cite{blasiak2008toric}: that the toric ideal of matroid bases is generated by quadric binomials. Blasiak showed in 2008 that this is true for all graphic matroids~\cite{blasiak2008toric}. Bonin extended this in 2013 to all sparse paving matroids~\cite{bonin2013basis}. Thereafter, Laso{\'n} and Micha{\l}ek proved the conjecture for strongly base orderable matroids~\cite{lason2014toric}, and up to saturation, i.e., the saturation of both ideals are equal.

Much less is known about unique exchanges. Even though considerable computational evidence supports the conjuncture for regular matroids, a proof even for graphical matroids is still open. Andres, Hochstättler, and Merkel show that restriction to ``one-sided'' unique exchanges on complementary base pairs yields a class excluding some common graphical matroids~\cite{merkel2009basentauschspiel,andres2014base}. The only general result on unique exchanges in regular matroids is by McGuinness, who uses Seymour's decomposition theorem~\cite{seymour1980decomposition} to show that for every base pair at least one element of the base yields a unique exchange~\cite{mcguinness2014base}. But it is unclear if a sequence of these can be added up to exchange any pair of bases.

% ------------------------------------------------------------------------------

\subsection{Other Problems on Bispanning Graphs}

Adding weights to graphs naturally suggests other problems related to spanning trees. Most well-known is the problem of finding spanning trees of minimum weight sum (MSTs)~\cite{boruuvka1926jistem,jarnik1930jistem,kruskal1956shortest,prim1957shortest}, which is now a standard chapter in undergraduate algorithm studies. Looking beyond, instead of finding just a minimum (or maximum) spanning tree, one can order the weight sum of all possible spanning trees, and then ask for a $k$-th smallest (or largest) spanning tree ($k$-MST).

If we consider the whole base exchange graph, where each vertex represents a particular spanning tree, then MSTs are a special subset of the vertices, or more generally, each vertex falls into a particular subset containing all $k$-th smallest trees. Kano proves that any pair of weighted spanning trees is connected in this graph by a path which uses only edge swaps that increase the tree's weight sum~\cite{kano1987maximum}. Using this theorem, he shows special cases for four general conjectures about the distances between spanning trees of different weights. Mayr and Plaxton settled one of these conjectures: any particular $k$-th smallest spanning tree can be obtained from a minimum spanning tree with at most $k-1$ edge swaps~\cite{mayr1992spanning}.

Baumgart considers an open conjecture in Mayr and Plaxton's paper, which if proven true, would imply three more conjectures by Kano~\cite{baumgart2009ranking, baumgart2010partitioning}. The conjecture states that if $S \dotcup T$ are disjoint spanning trees of a bispanning graph $G = (V,E)$ such that the weight sum of $S$ is less than that of $T$, and such that $T$ is the only spanning tree of its weight, then there are at least $|V| - 1$ spanning trees with pairwise different weight sums strictly smaller than $T$. Baumgart proves this conjecture for the cases when $S$ is the only spanning tree of its weight, and for the case when $G$ has no minor isomorphic to the complete graph $K_4$. In this thesis we refer to and reuse some of his decomposition ideas in a different manner.

For weighted bispanning graphs another problem arises naturally: to partition the edge set into two disjoint spanning trees such that one has minimal and the other maximal weight sum. Jochim surveys this problem and describes four algorithms based on the more general matroid intersection theorem~\cite{jochim2014algorithmen}. Since the four algorithms have a time complexity of at least $O(|E|^4)$, she describes attempts to adapt existing partitioning algorithms for bispanning graphs to the minimum-maximum spanning tree partitioning problem. This problem appears to be much more difficult than expected, specially considering that finding only a minimal or maximal spanning tree is computationally easy.

% ------------------------------------------------------------------------------

\subsection{Overview of the Thesis}

In section~\ref{sec:graph} we review basic graph theory and matroid theory to build a foundation for the remaining thesis. Since, we require graphs with parallel edges, the definitions are somewhat more complex than in an introductory graph theory textbook. Central for edge exchanges and exchange graphs are the definition of fundamental cycle and cut of an edge. We also prove many theorems on trees in graphs in detail, since these are required for conclusions about bispanning graphs.

Section~\ref{sec:bispanning} first introduces bispanning graphs and block matroids. We address basic theorems about them, the two characterization methods by Nash-Williams and Tutte, and then present an inductive construction method for all bispanning graphs based on the two operation \emph{double-attach} and \emph{edge-split-attach}. The last subsection then reviews an algorithm by Roskind and Tarjan to construct two disjoint spanning trees in any given graph, as we used this algorithm in our computer program.

In section~\ref{sec:exchange graph} we define unique symmetric edge exchanges in bispanning graphs, and describe various types of exchange graphs: directed, undirected and simplified versions. We then show the full exchange graphs of all bispanning graphs with three or four vertices as three examples, prove some straight-forward theorems for exchange graphs that follow from their definition, and compare our view of exchange graphs with the conjectures given by White~\cite{white1980unique}.  Finally, we show a different perspective on paths through an exchange graph by discussing cyclic base ordering in bispanning graphs and how to construct them for unrestricted symmetric edge exchanges.

Section~\ref{sec:decomposing} then presents our compositions for exchange graphs. We first classify bispanning graphs using their vertex- and edge-connectivity and whether they contain a non-trivial bispanning subgraph. For bispanning graphs of the second type, we show in subsection~\ref{sec:decomposition composite} that their exchange graphs are isomorphic to the Cartesian graph product of two smaller bispanning graphs. For bispanning graphs with vertex-connectivity two, we show in subsection~\ref{sec:decompose 2vconn} that they can be gained from two smaller bispanning graphs using a $2$-clique sum operation, and that the exchange graph of the smaller graphs can be joined into the one of the original graph. The only remaining class of bispanning graphs have vertex-connectivity three and contain no bispanning subgraph. In all other authors' works \cite{mayr1992spanning,baumgart2009ranking} these are also the most challenging instances. In subsection~\ref{sec:reduce vdeg3} we are able to prove a method on how to construct the original graph's exchange graph from the exchange graphs of three smaller bispanning graphs. While the last step allows an inductive composition for any bispanning graph, it is not clear how to use it to prove the connectivity of all exchange graphs. We present our ideas on this open topic in subsection~\ref{sec:connectivity ideas}.

%%%%%%%%%%%%%%%%%%%%%%%%%%%%%%%%%%%%%%%%%%%%%%%%%%%%%%%%%%%%%%%%%%%%%%%%%%%%%%%%
\clearpage

\section{Basic Definitions and Theorems on Graphs and Matroids}\label{sec:graph}

In this section, we review basic definitions for graphs and matroids to fix notation and terms, and prove some fundamental theorems that will be used in the remainder of the thesis. The nomenclature is based on multiple textbooks~\cite{harary1969graph,halin1989graphentheorie,west2001introduction,diestel2010graph}, and adapted for \emph{multigraphs} where necessary. Readers familiar with standard graph theory are welcome to scan or skip this section, but may want to stop and regard the less common theorems~\ref{thm:fundamental cycle}--\ref{thm:duality cycle cut}, and the section on matroids. The index at the back of the thesis (page~\pageref{sec:index}) provides an accessible cross-reference in case the reader needs to lookup unknown terms.

% ------------------------------------------------------------------------------

\subsection{Preliminaries and Notation}

In this thesis we assume familiarity with basic set theory, and review notation here only shortly to provide a foundation for the following sections.

\begin{definition}[basic set notation and operations]
A \emph[set]{set} is a collection of elements. Elements of a set can be listed within $\{\ldots\}$. $X \cup Y$ denotes the \emph<set!union>{union}, $X \cap Y$ the \emph<set!intersection>{intersection}, and $X \setminus Y$ the \emph<set!difference>{difference} between two sets $X$ and $Y$. Furthermore, $X \dotcup Y$ denotes the \emph<set!disjoint union>{disjoint union} implying $X \cap Y = \emptyset$, and $X \triangle Y := (X \cup Y) \setminus (X \cap Y) = (X \setminus Y) \cup (Y \setminus X)$ the \emph<set!symmetric difference>{symmetric difference}.

For brevity, we let $X + a := X \cup \{a\}$, and $X - a := X \setminus \{a\}$. Multiple operations without parenthesis $X + a - b$ are meant to be read from left to right as $(X + a) - b$, and are unambiguous. These ``arithmetic'' operators always denote operations on \emph{single elements}, never on sets of elements.

The \emph{identity function} on a set $X$ is $\operatorname{id}_X : X \rightarrow X$. The \emph{inverse function} of a bijection $\varphi : X \rightarrow Y$, is $\varphi^{-1} : Y \rightarrow X$.

Lines, primes and other accents on symbols carry \emph{no special meaning}, e.g., $\overline{X}$ is not necessarily the complement or closure of $X$, $X'$ is not a derivate set, and $\hat{X}$ is not a conjunctive set.
\end{definition}

\begin{definition}[partitions of a set]
A \emph[partition]{partition} is a set of disjoint subsets, e.g., $P = \{ \{v_1,v_2,v_3\} \} \cup \{ \{v\} \mid v \in V \setminus \{v_1,v_2,v_3\} \}$. Elements of a partition are also called \emph<member!partition>{members}, and the partitions $\{ V \}$ and $\{ \{v\} \mid v \in V \}$ are called \emph<trivial!partition>{trivial} partitions of $V$.

Given a set $V$, we write $V = V_1 \dotcup \cdots \dotcup V_k$ to define a partition $\{ V_1, \ldots, V_k \}$ of $V$ into $k$ subsets of $V$, where all subsets are pairwise disjoint and an arbitrary order is implied on them.
\end{definition}

To make notation more concise, we remove subscripts on symbols when they are clear from the context, to the extent that this improves comprehensiveness.

% ------------------------------------------------------------------------------

\subsection{Basic Graph Definitions}

We begin with the definition of an undirected graph, which we will plainly refer to as a \emph{graph}.

\begin{definition}[graph, vertex, edge, incidence, and ends]
  An \emph<graph!undirected>[graph]{(undirected) graph} $G = (V,E,\delta)$ consists of a set of \emph[vertex]{vertices} $V$, a set of \emph[edge]{edges} $E$, and an \emph<incidence!function>{incidence function} $\delta : E \rightarrow \binom{V}{2}$, where $\binom{V}{2}$ is the set of all subsets of $V$ containing exactly two elements.  For an edge $e \in E$ with $\delta(e) = \{ v_1,v_2 \}$, the vertices $v_1$, $v_2$ are called the \emph[edge!ends]{ends} of $e$.
\end{definition}

Most authors would call the preceding definition an \emph{undirected multigraph without loops}, because it allows multiple edges between a pair of vertices. This is the reason for the complication with the incidence function $\delta$. However, in the context of this thesis we \emph{require} parallel edges and refer to the common ``graph'' structure without parallel edges, as a \emph{simple undirected graph}. The primary focus of this thesis are games on a specific class of undirected graphs, however, during their investigation we also require \emph{directed graphs}, which are discussed separately in section~\ref{sec:directed graphs}.

The next two definitions concern parallel edges, and lay down the basic incidence and adjacency structure of vertices and edges in graphs.

\begin{definition}[parallel edge and simple graph]\label{def:simple graph}
  \begin{enumerate}
  \item Two edges $e_1,e_2 \in E$ in a graph $G = (V,E,\delta)$ are called \emph<parallel!edge>{parallel}, if they have the same ends, so if $\delta(e_1) = \delta(e_2)$.

  \item An undirected graph is called \emph<simple!graph>{simple}, if it contains no parallel edges.

  \item In a simple graph $G = (V,E,\delta)$ we can identify edges $e \in E$ with their ends $\delta(e) \in \binom{V}{2}$ since they are unique. Thus we can specify a simple undirected graph $G$ with just $(V,E)$, where $V$ are the vertices and $E \subseteq \binom{V}{2}$ the edges, and tacitly assume $\delta = \operatorname{id}_E$.

  \end{enumerate}
\end{definition}

\begin{definition}[incidence, adjacency, and vertex degree $\deg_G(v)$]
  If $G = (V,E,\delta)$ is an undirected graph, then
  \begin{enumerate}
  \item a vertex $v \in V$ is called \emph<incident!vertex>{incident} to an edge $e \in E$ if $v \in \delta(e)$, likewise

  \item an edge $e \in E$ is called \emph<incident!edge>{incident} to a vertex $v \in V$ if $v \in \delta(e)$,

  \item two vertices $v_1,v_2 \in V$ are called \emph<adjacent!vertex>{adjacent} if they are incident to a common edge $e \in E$, thus if $e \in E$ exists with $\delta(e) = \{ v_1, v_2 \}$, and

  \item two different edges $e_1 \neq e_2 \in E$ are called \emph<adjacent!edge>{adjacent} if they are incident to a common vertex $v$, thus if $v \in V$ exists with $v \in \delta(e_1)$ and $v \in \delta(e_2)$.

  \item The \emph<degree!vertex>{degree} $\deg_G(v)$\symbol{degree}{$\deg_G(v)$}{degree of vertex $v$} (or \emph<valency!vertex>{valency}) of a vertex $v \in V$ is the number of edges incident to $v$ in $G$, thus $\deg_G(v) := | \{ e \in E \mid v \in \delta(e) \} |$.

  \item A vertex of degree 0 is called \emph<isolated!vertex>{isolated}, while a vertex of degree 1 is called a \emph<pendent!vertex>{pendent} vertex.

  \end{enumerate}
\end{definition}

Even though the degree of a vertex is a very old and fundamental graph theoretic concept, it plays an important role in the remainder of this thesis in conjunction with the following theorem, which is also called the ``First Theorem of Graph Theory'' or the ``Handshake Lemma''.

\begin{theorem}[sum of all vertex degrees \R{\cites[\S\,16]{euler1741solutio}{biggs1976graph}}]\label{thm:handshake}
  If $G = (V,E,\delta)$ is an undirected graph then
$$\sum_{v \in V} \deg_G(v) = 2 \cdot |E| \,.$$
\end{theorem}
\begin{proof}
  Every edge $e \in E$ is incident to exactly two vertices $v \in V$, thus summing over all vertices counts each edge twice.
\end{proof}

To be able to compare graphs, we need to define when two graph structures are isomorphic (equal in structure) and identify substructures in graphs.

\begin{definition}[isomorphisms of undirected graphs]\label{def:isomorphism undirected}
  \begin{enumerate}
  \item An \emph<isomorphism!graph>{isomorphism} from a graph $G_1=(V_1,E_1,\delta_1)$ to a graph $G_2=(V_2,E_2,\delta_2)$ consists of a bijection $\varphi_v : V_1 \rightarrow V_2$ on the vertex sets and a bijection $\varphi_e : E_1 \rightarrow E_2$ on the edge sets, such that the incidence of vertices and edges remains the same, namely $\{ \varphi_v(x) \mid x \in \delta_1(e) \} = \{ x \mid x \in \delta_2(\varphi_e(e)) \}$ for all edges $e \in E_1$.

  \item Two graphs $G_1$ and $G_2$ are called \emph<isomorphic!graph>{isomorphic}, written $G_1 \cong G_2$, if an isomorphism from $G_1$ to $G_2$ exists.
    \symbol{G1=G2}{$G_1 \cong G_2$}{isomorphic graphs}

  \end{enumerate}
\end{definition}

\begin{definition}[subgraph, induced and spanning subgraph]
  \begin{enumerate}
  \item A graph $G' = (V',E',\delta')$ is called a \emph[subgraph]{subgraph} of $G = (V,E,\delta)$, denoted by $G' \subseteq G$, if $V' \subseteq V$, $E' \subseteq E$, $\delta' = \delta|_{E'}$ and $\delta'(e') \subseteq V'$ for all $e' \in E'$.

  \item Given a graph $G = (V,E,\delta)$ and a vertex subset $V' \subseteq V$, then the subgraph $G[V'] := (V',E',\delta|_{E'})$ with $E' = \{ e \in E \mid \delta(e) \subseteq V' \}$ is called the subgraph of $G$ \emph<vertex-induced!subgraph>{vertex-induced} by $V'$. It contains all vertices of $V'$ and all edges with both ends in $V'$.\symbol{G.V}{$G[V']$}{vertex-induced subgraph}

  \item Given a graph $G = (V,E,\delta)$ and an edge subset $E' \subseteq E$, then the subgraph $G[E'] := (V',E',\delta|_{E'})$ with $V' = \bigcup_{e \in E'} \delta'(e)$ is called the subgraph of $G$ \emph<edge-induced!subgraph>{edge-induced} by $E'$. It contains all edges of $E'$ and all vertices at their ends.\symbol{G.E}{$G[E']$}{edge-induced subgraph}

  \item A subgraph $G' \subseteq G$ is called \emph<spanning!subgraph>{spanning} if $G'$ contains all vertices of $G$.

  \end{enumerate}
\end{definition}

Of the many examples of undirected small graphs, we highlight only the following two classes (see figure~\ref{fig:example graphs}).

\begin{figure}\centering
  \tikzset{every picture/.style={scale=1.5, graphfinal}}
  \begin{tikzpicture}

    \begin{scope}[scale=1.2,rotate=-30,yshift=-3mm]

      \node (n1) [odot] at (0,0) {1};
      \node (n2) [odot] at (0*120:1) {2};
      \node (n3) [odot] at (1*120:1) {3};
      \node (n4) [odot] at (2*120:1) {4};

      \draw[X] (n1) -- (n2);
      \draw[X] (n1) -- (n3);
      \draw[X] (n1) -- (n4);
      \draw[X] (n2) -- (n3);
      \draw[X] (n2) -- (n4);
      \draw[X] (n3) -- (n4);

    \end{scope}

    \begin{scope}[xshift=35mm,scale=0.95]

      \node (n1) [odot] at (0,0) {1};
      \node (n2) [odot] at (1,-1) {2};
      \node (n3) [odot] at (1,1) {3};
      \node (n4) [odot] at (-1,1) {4};
      \node (n5) [odot] at (-1,-1) {5};

      \draw[X] (n1) -- (n2);
      \draw[X] (n1) -- (n3);
      \draw[X] (n1) -- (n4);
      \draw[X] (n1) -- (n5);
      \draw[X] (n2) -- (n3);
      \draw[X] (n3) -- (n4);
      \draw[X] (n4) -- (n5);
      \draw[X] (n5) -- (n2);

    \end{scope}

    \begin{scope}[xshift=70mm,scale=1.1,rotate=18,yshift=-0.7mm]

      \node (n1) [odot] at (0,0) {1};
      \node (n2) [odot] at (0*72:1) {2};
      \node (n3) [odot] at (1*72:1) {3};
      \node (n4) [odot] at (2*72:1) {4};
      \node (n5) [odot] at (3*72:1) {5};
      \node (n6) [odot] at (4*72:1) {6};

      \draw[X] (n1) -- (n2);
      \draw[X] (n1) -- (n3);
      \draw[X] (n1) -- (n4);
      \draw[X] (n1) -- (n5);
      \draw[X] (n1) -- (n6);
      \draw[X] (n2) -- (n3);
      \draw[X] (n3) -- (n4);
      \draw[X] (n4) -- (n5);
      \draw[X] (n5) -- (n6);
      \draw[X] (n6) -- (n2);

    \end{scope}

  \end{tikzpicture}
  \caption{The complete graphs $K_4 = W_4$, and the wheel graphs $W_5$ and $W_6$.}\label{fig:example graphs}
\end{figure}
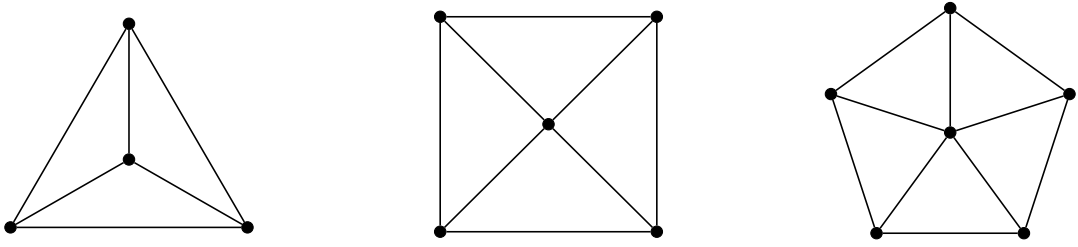

\begin{definition}[complete graph $K_n$, $k$-clique and wheel graph $W_n$]
  \begin{enumerate}
  \item For $n \in \mathbb{N}_1$ the undirected simple graph $(V,E)$ with $V = \{ v_1,\ldots,v_n \}$ and $E = \binom{V}{2}$ is called the \emph<complete!graph>{complete graph} $K_n$\symbol{Kn}{$K_n$}{complete graph with $n$ vertices} with $n$ vertices and $\frac{n(n-1)}{2}$ edges.

  \item A subgraph $G' \subseteq G$ of an undirected graph $G$ is called a \emph{$k$-clique}\index{clique@$k$-clique} of $G$ if $G' \cong K_k$.

  \item For $n \in \mathbb{N}_1$ the undirected simple graph $(V,E)$ with $V = \{ v_1,\ldots,v_n \}$ and $E = \{ \{v_1,v_i\} \mid i = 2,\ldots,n \} \cup \{ \{v_i, v_{i+1}\} \mid i = 2,\ldots,n-1 \} \cup \{ \{ v_n, v_2 \} \}$ is called the \emph[graph!wheel][wheel graph]{wheel graph} $W_n$\symbol{Wn}{$W_n$}{wheel graph with $n$ vertices} with $n$ vertices and $2n - 2$ edges.

  \end{enumerate}
\end{definition}

New graphs can be constructed from old graphs in many ways, and in the following chapters we construct new graphs and show isomorphism to old ones. For easier exhibition we define three basic operations: addition, deletion, and contraction.

\begin{definition}[vertex and edge deletion and addition]
  Given a graph $G = (V,E,\delta)$ and
  \begin{enumerate}
  \item a vertex set $V' \subseteq V$, then \emph<deletion!vertex>{deleting} all vertices in $V'$ along with all incident edges yields the graph $G \setminus V' := (V \setminus V',E',\delta|_{E'})$ with $E' = \{ e \in E \mid \delta(e) \cap V' = \emptyset \}$, or

  \item an edge set $E' \subseteq E$, then \emph<deletion!edge>{deleting} all edges in $E'$ yields the graph $G \setminus E' := (V, E \setminus E', \delta|_{E \setminus E'})$; the remainder $G \setminus E'$ may contain isolated vertices.

  \item For brevity, we write $G - v$ or $G - e$ instead of $G \setminus \{v\}$ or $G \setminus \{e\}$ for \emph[deletion]{deletion} of a single vertex $v \in V$ or a single edge $e \in E$,

  \item $G + e := (V,E \cup \{e\},\delta')$ for \emph<addition!edge>{addition} of an edge $e$ to $G$, where the incidence value $\delta'(e)$ of $e$ must be defined by the context, and

  \item $G + \{v_1,v_2\} := (V,E \cup \{e\},\delta')$ for explicit \emph<addition!edge>{addition} of a new edge $e \notin E$ to $G$ with incidence $\delta'(e) = \{v_1,v_2\}$ for its ends $v_1,v_2 \in V$.

  \end{enumerate}
\end{definition}

\begin{definition}[vertex pair, edge and subgraph contraction]
  \label{def:contraction}\symbol{GslashX}{$G \contr X$}{contraction of $X$}
  For a graph $G = (V,E,\delta)$ and a subset of vertices $X \subseteq V$ we denote by $G \contr X$ the graph obtained from $G$ by \emph[contraction]{contracting} all vertices $X$ into a new vertex $x \notin V$, which becomes incident to all edges priorly incident to any vertex in $X$. As we do not allow loops in graphs, all edges $e \in E$ with $\delta(e) \subseteq X$ are removed during contraction.\

In symbols, given $X \subseteq V$, we define $G \contr X := (V',E',\delta')$ with $V' := (V \setminus X) \cup \{ x \}$, $E' = E \setminus \{ e \in E \mid \delta(e) \subseteq X \}$ and
  $$\delta'(e) = \begin{cases}
    \delta(e)                           & \text{if } \delta(e) \cap X = \emptyset \,, \\
    \{x\} \cup (\delta(e) \setminus X) & \text{otherwise}\,.
  \end{cases}$$
  The \emph{contraction} of a single edge $e \in E$ is the contraction of its ends: $G \contr e := G \contr \delta(e)$.  If $X$ is a set of vertex sets, then $G \contr X$ is the result of contracting these vertex sets in any sequence. Given a subgraph $G' \subseteq G$, the \emph{contraction} of $G \contr G'$ is defined as the contraction of all vertices of $G'$.
\end{definition}

Figure~\ref{fig:example contraction} shows an example of a contraction of the vertex set $X$ into $x$. We defined contraction on sets of vertices, while other authors define contraction of edges. In our case, where loops are prohibited, these different views are indistinguishable.

While the following is clear in the definition above, we make an additional note that when contracting $X \subseteq V$ of a graph $G = (V,E)$, then $G \contr X$ contains exactly all edges of $G$ except those in $G[X]$.

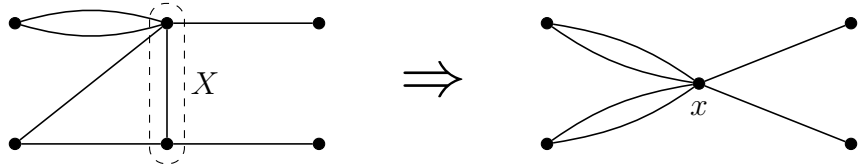
\begin{figure}\centering
  \tikzset{every picture/.style={scale=2, graphfinal}}
  \begin{tikzpicture}[yscale=0.8]

    \begin{scope}

      \node (n1) [odot] at (-1,0) {1};
      \node (n2) [odot] at (-1,1) {2};
      \node (n3) [odot] at (0,0) {3};
      \node (n4) [odot] at (0,1) {4};
      \node (n5) [odot] at (1,0) {5};
      \node (n6) [odot] at (1,1) {6};

      \draw[X] (n1) -- (n3);
      \draw[X] (n1) -- (n4);
      \draw[X] (n2) to[bend left=20] (n4);
      \draw[X] (n2) to[bend left=-20] (n4);
      \draw[X] (n3) -- (n4);
      \draw[X] (n5) -- (n3);
      \draw[X] (n6) -- (n4);

      \draw [dashed,rounded corners=6pt]
      ($(n4.center) + (125:2mm)$) rectangle ($(n3.center) + (305:2mm)$);

      \node at (0.25,0.5) {$X$};

    \end{scope}

    \node at (17.5mm,0.5) {\Huge$\Rightarrow$};

    \begin{scope}[xshift=35mm]

      \node (n1) [odot] at (-1,0) {1};
      \node (n2) [odot] at (-1,1) {2};
      \node (n3) [odot] at (0,0.5) {3};
      \node (n5) [odot] at (1,0) {5};
      \node (n6) [odot] at (1,1) {6};

      \node at (0,0.3) {$x$};

      \draw[X] (n1) to[bend left=15] (n3);
      \draw[X] (n1) to[bend left=-15] (n3);
      \draw[X] (n2) to[bend left=15] (n3);
      \draw[X] (n2) to[bend left=-15] (n3);
      \draw[X] (n5) -- (n3);
      \draw[X] (n6) -- (n3);

    \end{scope}

  \end{tikzpicture}
  \caption{Example of a contraction of $X$ into a new vertex $x$.}\label{fig:example contraction}
\end{figure}

% ------------------------------------------------------------------------------

\subsection{Paths and Connectivity}\label{sec:connectivity}

To consider connectivity and connected components of a graph, we need the notion of paths.

\begin{definition}[edge walk, path and cycle]
  \begin{enumerate}
  \item An \emph<edge!walk>[walk]{edge walk} from $v_0$ to $v_n$ is an alternating sequence $v_0,e_1,v_1,e_2,v_2,\ldots,v_{n-1},e_n,v_n$ in a graph $G = (V,E,\delta)$, where each edge $e_i \in E$ is incident to the vertices $v_{i-1} \in V$ and $v_i \in V$, namely $\delta(e_i) = \{ v_{i-1}, v_i \}$ for $i = 1,\ldots,n$. The \emph{length} of a walk is the number of edges it uses, and we say the walk \emph{starts} at $v_0$ and \emph{ends} at $v_n$.

  \item A \emph<edge!path>[path]{path} is an edge walk which contains no vertex more than once.

  \item A \emph<edge!cycle>[cycle]{cycle} is an edge walk, which starts and ends at the same vertex and contains no vertex more than once except $v_0 = v_n$, which is contained exactly twice.
  \end{enumerate}
\end{definition}

Other authors define a cycle more generally as an edge walk with $v_0 = v_n$ without requiring it to visit vertices once. They then call cycles with the restriction \emph{simple}. In this thesis, we do not need this distinction and assume all cycles to visit vertices once (except start and end).

\begin{definition}[connected vertices and graph, and component]
  In a graph $G = (V,E,\delta)$, a pair of vertices $v_1$ and $v_2$ are \emph{connected}, if a path starting at $v_1$ and ending at $v_2$ exists. The whole graph $G$ is called \emph<connected!graph>{connected} if all pairs of vertices are connected.
\end{definition}

\begin{definition}[connected component]
  A \emph[connected!component]{connected component} or just \emph[component]{component} of a graph $G$ is a connected subgraph $G' \subseteq G$, which is not contained in any connected subgraph of $G$ having more vertices or edges than $G'$. We denote the number of connected components of a graph $G$ by $\operatorname{comp}(G)$.
  \symbol{comp(G)}{$\operatorname{comp}(G)$}{number of components of $G$}
\end{definition}

\begin{definition}[vertex and edge cut, cut-vertex, bridge and bond]
  \begin{enumerate}
  \item A vertex subset $V' \subseteq V$ of a graph $G = (V,E,\delta)$ is a \emph<vertex!cut>{vertex cut} if deletion of all $v' \in V'$ increases the number of connected components, thus if $\operatorname{comp}(G \setminus V') > \operatorname{comp}(G)$.

  \item A vertex $v \in V$ of a graph $G = (V,E,\delta)$ is called a \emph[cut-vertex]{cut-vertex} if $\{v\}$ is a vertex cut.

  \item An edge subset $E' \subseteq E$ of a graph $G = (V,E,\delta)$ is called an \emph<edge!cut>{edge cut} if deletion of all $e' \in E'$ increases the number of connected components, thus if $\operatorname{comp}(G \setminus E') > \operatorname{comp}(G)$.

  \item An edge $e \in E$ of a graph $G = (V,E,\delta)$ is called a \emph[cut-edge]{cut-edge} or \emph[bridge]{bridge} if $\{e\}$ is an edge cut.

  \item A vertex or edge cut is called \emph<minimal!cut>{minimal}, if no item can be removed from the set without losing its property. A minimal edge cut is also called a \emph[bond]{bond}.

  \end{enumerate}
\end{definition}

Most edge cuts we discuss in this thesis are actually minimal, so most are bonds. However, as minimality is usually not their decisive property, we talk about edge cuts and explicitly establish minimality when needed.

The previous definition determines edge cuts as subsets of edges. An alternative approach to edge cuts is to separate the vertex set into two (usually disjoint) subsets and taking the edges between them. Other authors then continue by defining edge cuts as sets of vertices. We do not follow this practice, and called these edge cuts \emph{induced} by a vertex set.

\begin{definition}[induced edge cut]
  \begin{enumerate}
  \item For two vertex subsets $S,T \subseteq V$ of a graph $G = (V,E,\delta)$ we define $[S,T] := \{ e \in E \mid \delta(e) \cap S \neq \emptyset \text{ and } \delta(e) \cap T \neq \emptyset \}$, thus as all edges with one end in $S$ and the other in $T$.

  \item A vertex subset $S \subseteq V$ of a connected graph $G = (V,E,\delta)$ defines the edge cut $[S,V \setminus S]$, which is called the minimal edge cut \emph<induced!cut>{induced} by $S$.

  \end{enumerate}
\end{definition}

% ------------------------------------------------------------------------------

\subsection{Trees in Graphs}\label{sec:trees}

Trees and enumeration of trees were among the founding applications of graph theory. In this thesis we consider graphs with two disjoint spanning trees, and thus need to precisely specify properties of trees. In the following definition, we distinguish the term ``tree'' as a set of edges of a graph possibly containing further edges, and the term ``tree-graph'' as a graph with exactly a tree as edge set.

\begin{definition}[forest, tree-graph, tree, spanning tree and leaf]\label{def:tree}
  \begin{enumerate}
  \item A graph containing no cycle is called \emph<acyclic!graph>{acyclic} or a \emph[forest]{forest}. All forests are simple.

  \item A connected forest is called a \emph[tree-graph]{tree-graph}.\label{def:tree-graph}

  \item If an edge subset $F \subseteq E$ of a graph $G = (V,E,\delta)$ edge-induces a forest $G[F]$, we also call the edge set $F$ an \emph<edge!forest>{edge forest} in $G$. Likewise if $T \subseteq E$ edge-induces a tree-graph $G[T]$, we call $T$ a \emph<tree!edge>[tree]{tree} in $G$.\label{def:tree detail}

  \item A tree $T \subseteq E$ in a graph $G = (V,E,\delta)$ is called \emph<spanning!tree>{spanning} if $G[T]$ is a spanning subgraph of $G$. Thus a spanning tree of $G$ is an edge set connecting all vertices of $G$, and $G[T]$ is always a spanning tree-graph of $G[T]$ if $T$ is a tree in $G[T]$.\label{def:spanning tree}

  \item A vertex $v \in V$ of a forest $G[F]$ is called a \emph<vertex!leaf>{leaf} if $\deg_{G[F]}(v) = 1$. We also call the one edge incident to a leaf $v$ a \emph<edge!leaf>{leaf edge}.

  \end{enumerate}
\end{definition}

\begin{lemma}[two leaves in a tree]\label{lem:two leaves}
  Every tree-graph $G = (V,E,\delta)$ with $|V| \geq 2$ contains at least two leaves.
\end{lemma}
\begin{proof}
  Since $G$ is connected and has two vertices, it contains an edge. Consider a path $P$ of maximal length $k \geq 1$ in $G$ which starts at $v_0$ and ends at $v_k$. $P$ cannot be a cycle, since $G$ is acyclic. Hence, both $v_0$ and $v_k$ are each a leaf in $G$, connected only by the edge in the path. Otherwise one could extend $P$ at $v_0$ or $v_k$ and thus obtain a longer path.
\end{proof}

\begin{theorem}[spanning tree equivalences]\label{thm:tree}
  The following conditions are equivalent for a graph $G = (V,E,\delta)$ with $|V| \geq 1$:
  \begin{enumerate}
  \item $T = E$ is a spanning tree of $G$ and $G = G[T]$ is a tree-graph.
  \item $G$ is acyclic and connected.
  \item $G$ is connected and $|E| = |V| - 1$\,.\label{thm:tree connected}
  \item $G$ is acyclic and $|E| = |V| - 1$\,.\label{thm:tree acyclic}
  \item Every pair of vertices in $G$ is connected by a unique path.\label{thm:tree unique path}
  \item $G$ is connected and every edge $e \in E$ is a bridge in $G$.\label{thm:tree bridge}
  \end{enumerate}
\end{theorem}
\begin{proof}
  (i) and (ii) are equivalent by definition~\ref{def:tree}: an acyclic connected graph is a connected forest, which is a tree-graph and vice versa. The edge set of a tree-graph is by definition a spanning tree.

  (ii) implies (iii) and (iv): We start an induction at $|V| = 1$, since (iii) and (iv) are trivially true for the tree-graph with one vertex. Let $G$ be a graph with $|V| \geq 2$ and $v \in V$ a leaf, which exists due to lemma~\ref{lem:two leaves}. Deleting the leaf results in $G' = (V',E',\delta') := G - v$, which is a smaller graph that remains acyclic and connected. By induction we have $|E'| = |V'| - 1$, and since exactly one vertex and one edge were deleted, $|E| = |E'| + 1 = |V'| - 1 + 1 = |V| - 1$.

  (iii) implies (ii) and (iv): Assume that a cycle $C \subseteq E$ exists in $G$. The cycle connects $|C|$ vertices using $|C|$ edges. The other $|V| - |C|$ vertices require at least $|V| - |C|$ edges to connect to the cycle, however, as $|E| \overset{\text{(iii)}}{=} |V| - 1 < |V| = |C| + (|V| - |C|)$, not enough edges exist to connect the graph. Thus $G$ is acyclic.

  (iv) implies (ii) and (iii): Decompose the graph $G$ into its $k$ connected components $G_1,\ldots,G_k$ with $G_i = (V_i,E_i,\delta_i)$. Each component is acyclic and connected, and thus fulfills (ii), from which (iii) and (iv) follow, so we have $|E_i| = |V_i| - 1$ for all $i = 1,\ldots,k$. In total the graph contains $|E| = \sum_{i=1}^k |E_i| = |V| - k$ edges, hence $k = 1$ and $G$ must be connected.

  (ii) implies (v): As $G$ is connected, each pair of vertices is connected by at least one path. Let $P$ and $Q$ be two different paths between a pair of vertices, then these differ starting at a vertex $v$ and rejoin at $w$. The edges between $v$ and $w$ in $P$ and $Q$ together for a cycle, which contradicts that $G$ is acyclic, so exactly one path exists.

  (v) implies (ii): It is clear that $G$ is connected. Assume $C$ is a cycle in $G$, then every pair $v,w \in C$ is connected by two different paths in $C$, and also in $G$, which contradicts that only a unique path exists between any pair in $G$.

  (ii) implies (vi): Consider an edge $e \in E$ which is not a bridge, then $G - e$ remains connected. Let $P$ be a path in $G - e$ from one end of $e$ to the other, then $P$ extended with $e$ is a cycle in $G$. However, $G$ is acyclic, so every edge is a bridge.

  (vi) implies (ii): Assume $C$ is a cycle in $G$, then any edge $e \in C$ is not a bridge, since $G - e$ remains connected. So $G$ is acyclic and connected.
\end{proof}

\begin{theorem}[fundamental cycle $C_G(T,e)$]\label{thm:fundamental cycle}
  \symbol{C_G(T,e)}{$C_G(T,e)$}{fundamental cycle of $e$ and $T$}
  Given a graph $G = (V,E,\delta)$ and a spanning tree $T \subseteq E$, then every non-tree edge $e \in E \setminus T$ defines a cycle in $G[T] + e$. This cycle is unique with respect to $T$, and hence called the \emph<fundamental!cycle>{fundamental cycle} $C_G(T,e)$ of $e$ and $T$ in $G$.
\end{theorem}
\begin{proof}
  The ends of the non-tree edge $e$ are connected by a unique path $P$ in the spanning tree-graph $G[T]$ (theorem~\ref{thm:tree unique path}). Thus $P$ extended by $e$ is the fundamental cycle in $G$, and this cycle is unique as $P$ is unique with respect to $T$.
\end{proof}

\begin{theorem}[fundamental cut $D_G(T,e)$]\label{thm:fundamental cut}
  \symbol{D_G(T,e)}{$D_G(T,e)$}{fundamental cut of $e$ and $T$}
  Given a graph $G = (V,E,\delta)$ and a spanning tree $T \subseteq E$, then every tree edge $e \in T$ defines a minimal edge cut in $G[E \setminus T] + e$. This minimal edge cut is unique with respect to $T$, and hence called the \emph<fundamental!cut>{fundamental cut} $D_G(T,e)$ of $e$ and $T$ in $G$.
\end{theorem}
\begin{proof}
  As $e \in T$ is a bridge in the tree-graph $G[T]$ (theorem~\ref{thm:tree bridge}), $G[T] - e$ is composed of two connected components $G[V_1]$ and $G[V_2]$ with $V_1 \dotcup V_2 = V$. The thereby defined edge set $D := [V_1,V_2]$ in $G[E \setminus T] + e$ is a minimal edge cut in $G$. To show that $D$ is unique, assume $D' := [V'_1,V'_2]$ is a \emph{different} minimal edge cut contained in $G[E \setminus T] + e$, and let $\overline{D} := D \triangle D' = (D \cup D') \setminus (D \cap D')$.  We claim that $\overline{D}$ is an edge cut for $[(V_1 \cap V'_1) \cup (V'_2 \cap V_2),\, (V_1 \cap V'_2) \cup (V'_1 \cap V_2)] = [\overline{V}_1,\overline{V}_2]$. To see why, consider without loss of generality an edge $f$ from $(V_1 \cap V'_1)$ to $(V_1 \cap V'_2)$: then $f \in D'$ and $f \notin D$, hence $f \in \overline{D}$. The other three pairs are analogous, hence $\overline{D}$ is another different edge cut, if $D$ and $D'$ are different. This third cut is contained in $G[E \setminus T]$ (excluding $e$, as $e \in D \cap D'$). However, then $T \subseteq E \setminus \overline{D}$ connects all vertices in the graph $G - \overline{D}$, contradicting that $\overline{D}$ is an edge cut. Thus $D'$ cannot exist, and $D$ is unique with respect to $T$.
\end{proof}

\begin{remark}[clarification of edge sets of fundamental cycles and cuts]\label{rem:cycle cut sets}
  The definitions of fundamental cycle and cut assign for an edge $e$ and tree $T$ a cycle and cut. In the cycle case the edge $e$ is required to be \emph{outside} of the tree, and the cycle closed with $e$ contains edges from \emph{inside} $T$, while in the cut case the edge $e$ is \emph{inside} the tree, and the cut edges \emph{outside} the tree $T$.
  In symbols, given a graph $G = (V,E,\delta)$ and a spanning tree $T \subseteq E$, then
  \begin{enumerate}
  \item for \parbox{10ex}{$e \in E \setminus T$} we have $C_G(T,e) \subseteq T + e$, and
  \item for \parbox{10ex}{$e \in T$} we have $D_G(T,e) \subseteq (E \setminus T) + e$.
  \end{enumerate}
\end{remark}

\begin{theorem}[duality of fundamental cycle and cut]\label{thm:duality cycle cut}
  Given a graph $G = (V,E,\delta)$ and a spanning tree $T \subseteq E$, then for every tree edge $e \in T$ and non-tree edge $f \in E \setminus T$, we have
  \[
  e \in C(T,f) \quad\text{if and only if}\quad f \in D(T,e) \,.
  \]
\end{theorem}
\begin{proof}
  As $e \in T$, $G[T] - e$ is composed of two connected components $G[V_1]$ and $G[V_2]$ with $V_1 \dotcup V_2 = V$. See figure~\ref{fig:example dual cycle cut} for a sketch, which illustrates the following two proof directions.

  Assume $e \in C(T,f)$ with $e \neq f$, since $f \in D(T,f)$ is trivially true. Since $e \in C(T,f)$, the cycle is composed of $e$, $f$, and two unique paths in $T$ between their ends. Thus the two ends of $f$ are in different components, and hence $f \in D(T,e)$, because the cut is unique.

  Now assume $f \in D(T,e)$ with $e \neq f$, since $e \in C(T,e)$ is trivial. Since $f \in D(T,e)$, the two ends of $f$ are in different components. The ends of $e$ are also in different components, so one can find two paths in the components connecting their ends. All together these form a cycle, which is the cycle $C(T,f)$ because it is unique with respect to $T$, hence $e \in C(T,f)$.
\end{proof}

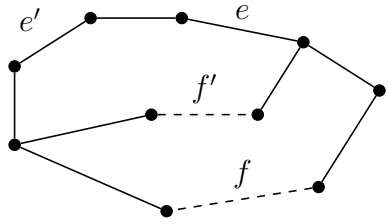
\begin{figure}\centering
  \tikzset{every picture/.style={scale=2, graphfinal}}
  \begin{tikzpicture}[yscale=0.8]

    \begin{scope}

      \node (n1) [odot] at (0,-0.2) {1};
      \node (n2) [odot] at (-1,-0.4) {2};
      \node (n3) [odot] at (-2,0.15) {3};
      \node (n4) [odot] at (-2,0.8) {4};
      \node (n5) [odot] at (-1.5,1.2) {5};
      \node (n6) [odot] at (-0.9,1.2) {6};
      \node (n7) [odot] at (-0.1,1.0) {7};
      \node (n8) [odot] at (0.4,0.6) {8};

      \node (n9) [odot] at (-1.1,0.4) {9};
      \node (n10) [odot] at (-0.4,0.4) {10};

      \draw[X] (n2) -- (n3) -- (n4) -- node[auto] {$e'$} (n5) -- (n6) -- node[above] {$e$} (n7) -- (n8) -- (n1);
      \draw[X,dashed] (n1) -- node[above] {$f$} (n2);

      \draw[X] (n3) -- (n9);
      \draw[X,dashed] (n9) -- node[above] {$f'$} (n10);
      \draw[X] (n10) -- (n7);

    \end{scope}

  \end{tikzpicture}
  \caption{Example of duality of fundamental cycle and cut theorem~\ref{thm:duality cycle cut}.}\label{fig:example dual cycle cut}
\end{figure}

% ------------------------------------------------------------------------------

\subsection{Directed Graphs}\label{sec:directed graphs}

When considering exchange graphs on game, we require directed graphs in intermediate steps.

\begin{definition}[directed graph, vertex, arc, incidence, and ends]
  A \emph<graph!directed>{directed graph} $G = (V,E,\delta)$ consists of a set of \emph[vertex]{vertices} $V$, a set of \emph[arc]{arcs} $E$, and an \emph<incidence!function>{incidence function} $\delta : E \rightarrow V \times V$, where $V \times V$ is the set of ordered pairs of $V$.  For an edge $e \in E$ with $\delta(e) = (v_1,v_2)$, the first vertex $v_1$ is the \emph<tail!arc>{tail} and second vertex $v_2$ is the \emph<head!arc>{head} of $e$, together they are also called the \emph<arc!ends>{ends} of $e$. We also say that $e$ is an edge \emph{from} its tail $v_1$ \emph{to} its head $v_2$.
\end{definition}

While our definition of directed graphs allows loop $(v,v)$, we will never use them in this thesis.

\begin{definition}[isomorphisms of directed graphs]\label{def:isomorphism directed}
  \begin{enumerate}

  \item An \emph<isomorphism!graph>{isomorphism} from a directed graph $G_1=(V_1,E_1,\delta_1)$ to a directed graph $G_2=(V_2,E_2,\delta_2)$ consists of a bijection $\varphi_v : V_1 \rightarrow V_2$ on the vertex sets and a bijection $\varphi_e : E_1 \rightarrow E_2$ on the arc sets, such that incidence of vertices and arcs remains the same, namely $(\varphi_v \times \varphi_v)( \delta_1(e) ) = \delta_2( \varphi_e(e) )$ for all arcs $e \in E_1$, where $\varphi_v \times \varphi_v$ maps both ends of $e$.

  \item Two undirected or directed graphs $G_1$ and $G_2$ are called \emph<isomorphic!graph>{isomorphic}, written $G_1 \cong G_2$, if an isomorphism from $G_1$ to $G_2$ exists.
    \glsadd{symb:G1=G2}

  \end{enumerate}
\end{definition}

Obviously, if $\varphi_v$ and $\varphi_e$ define an isomorphism as described in the definition, then it implies \( (\varphi_v^{-1} \times \varphi_v^{-1})(\delta_2(e)) = \delta_1( \varphi_e^{-1}(e) ) \) for all arcs $e \in E_2$ due to the inverse functions of the bijections.

% ------------------------------------------------------------------------------

\subsection{Matroids}\label{sec:matroids}

Spanning trees of undirected graphs are among the most natural cases of the more general dependence structure called a \emph{matroid}~\cite{whitney1935abstract}. In this subsection we review basic matroid definitions and theorems, and refer to the standard and introductory textbooks on matroid theory \cite{welsh1976matroid,oxley2011matroid,gordon2012matroids} for more details and proofs. The focus of this thesis, however, remains primarily on graphs (graphical matroids), since the problem behind Alice and Bob's game of unique exchange has not been solved even in the graphical case. Nevertheless, the more general perspective of matroid theory may help.

\begin{definition}[matroid, independent and dependent set, base, circuit]
  A \emph[matroid]{matroid} $M = (E,\mathcal{I})$ consists of a finite \emph[matroid!ground set]{ground set} $E$, containing its \emph{elements}, and a collection $\mathcal{I}$ of subsets of $E$, called \emph[matroid!independent sets]{independent sets}, such that:
  \begin{enumerate}
  \item $\emptyset \in \mathcal{I}$,
  \item if $A \in \mathcal{I}$ and $B \subseteq A$ then $B \in \mathcal{I}$, and \hfill(\emph{closed under subsets})
  \item if $A, B \in \mathcal{I}$ with $|A| < |B|$, then there is an element $x \in B \setminus A$, such that $A \cup \{x\} \in \mathcal{I}$. \\\null\hfill(\emph{independence augmentation})
  \end{enumerate}
  All subset of $E$ which are not independent are called \emph[matroid!dependent sets]{dependent}. A \emph<base!matroid>{base} of $M$ is a maximal independent set, while a \emph<circuit!matroid>{circuit} of $M$ is a minimal dependent set. The collection of bases of $M$ is denoted by $\matheu{B}(M)$ and the collection of circuits by $\matheu{C}(M)$.
  \symbol{B(M)}{$\matheu{B}(M)$}{bases of a matroid $M$}
  \symbol{C(M)}{$\matheu{C}(M)$}{circuits of a matroid $M$}
\end{definition}

Due to the generality of matroids, there are multiple equivalent ways to define them. The equivalence of these definitions is often not obvious, which is why some authors call the equivalent axiomatic definitions \emph[cryptomorphism]{cryptomorphisms}. We only state other definitions as theorems, and refer to the standard textbooks for proofs.

\begin{theorem}[circuit axioms \R{\cites[thm.\,1.1.4]{oxley2011matroid}}]
  A collection $\matheu{C}$ of subsets of a finite set $E$ defines the circuits of a matroid $M = (E,\mathcal{I})$ with $\mathcal{I} = \{ A \subseteq E \mid \nexists\, C \in \matheu{C} : C \subseteq A \}$, if and only if it satisfies the following conditions:
  \begin{enumerate}
  \item $\emptyset \notin \matheu{C}$,
  \item if $C_1,C_2 \in \matheu{C}$ and $C_1 \subseteq C_2$, then $C_1 = C_2$, and \hfill(\emph{clutter})
  \item if $C_1, C_2 \in \matheu{C}$ with $C_1 \neq C_2$ and $x \in C_1 \cap C_2$ then there is a circuit $C_3 \in \matheu{C}$ \\ such that $C_3 \subseteq (C_1 \cup C_2) \setminus \{ x \}$. \hfill(\emph{weak circuit elimination})\label{thm:weak circuit elimination}
  \end{enumerate}
\end{theorem}

\begin{theorem}[base axioms \R{\cites[cor.\,1.2.5]{oxley2011matroid}}]
  A collection $\matheu{B}$ of subsets of a finite set $E$ defines the bases of a matroid $M = (E,\mathcal{I})$ with $\mathcal{I} = \{ A \subseteq E \mid \exists\, B \in \matheu{B} : A \subseteq B \}$, if and only if it satisfies the following conditions:
  \begin{enumerate}
  \item $\matheu{B} \neq \emptyset$, and \hfill(\emph{non-trivial})
  \item if $B_1, B_2 \in \matheu{B}$ and $x \in B_1 \setminus B_2$, then there is an element $y \in B_2 \setminus B_1$ \\ such that $(B_1 \setminus \{ x \}) \cup \{ y \} \in \matheu{B}$. \hfill(\emph{weak base exchange})
  \end{enumerate}
\end{theorem}

From the weak base exchange axiom many different equivalent axioms or lemmata were derived. For our base exchange game, we require a symmetric axiom, sometimes also called ``strong'' base exchange.

\begin{lemma}[strong (symmetric) base exchange \R{\cite{brualdi1969comments}}]\label{lem:strong base exchange}
  If $M$ is a matroid, and $B_1, B_2 \in \matheu{B}(M)$ are two bases, then for every $x \in B_1$ there exists an $y \in B_2$, such that $(B_1 \setminus \{ x \}) \cup \{ y \} \in \matheu{B}(M)$ and $(B_2 \setminus \{ y \}) \cup \{ x \} \in \matheu{B}(M)$.
\end{lemma}

Due to symmetric exchanges operation on bases of graphs or matroids, we are most interested in these.

\begin{definition}[rank of a subset $r_M(A)$ and rank of a matroid $r(M)$]
  \symbol{r_M(X)}{$r_M(X)$}{rank of $X$ in the matroid $M$}
  If $M = (E,\mathcal{I})$ is a matroid and $X \subseteq E$, then the \emph{rank} of $X$ in $M$ is the size of the largest independent set in $X$. Formally the \emph{rank function} $r_M : 2^E \rightarrow \mathbb{N}_0$ is defined as $r_M(X) := \max \{ |A| \mid A \in \mathcal{I} \text{ and } A \subseteq X \}$.

  The \emph<rank!matroid>{rank} of a matroid $M = (E,\mathcal{I})$ is plainly $r_M(E)$, and also written as $r(M)$.
\end{definition}

\begin{theorem}[rank function properties \R{\cites[cor.\,1.3.4]{oxley2011matroid}}]
  If $E$ is a finite set and $r : 2^E \rightarrow \mathbb{N}_0$ a function on $E$, then $r$ defines the rank function of a matroid on $E$ if and only if it satisfies the following conditions:
  \begin{enumerate}
  \item If $A \subseteq E$, then $0 \leq r(A) \leq |A|$,
  \item if $A \subseteq B \subseteq E$, then $r(A) \leq r(B)$, and
  \item if $A,B \subseteq E$, then $r(A \cup B) + r(A \cap B) = r(A) + r(B)$
  \end{enumerate}
\end{theorem}

Using the rank function of a subset, one can classify subsets of elements via their rank, similar to vector geometries. For this thesis, the number of elements in a base and circuit are maybe most important.

\begin{theorem}[flat, rank of a hyperplane, base, and circuit \R{\cites[lem.\,1.3.5]{oxley2011matroid}}]
  If $M = (E,\mathcal{I})$ is a matroid with rank function $r_M$, then
  \begin{enumerate}
  \item a subset $A \subseteq E$ is called \emph<closed!matroid>{closed} or a \emph<flat!matroid>{flat} if $A = \{ x \in E \mid r_M(A \cup \{ x \}) = r_M(A) \}$,
  \item a closed subset $A \subseteq E$ is called a \emph<hyperplane!matroid>{hyperplane} if and only if $r_M(A) = r(M) - 1$,
  \item a subset $A \subseteq E$ is a base if and only if $|A| = r_M(A) = r(M)$, and
  \item a subset $A \subseteq E$ is a circuit if and only if $A \neq \emptyset$ and for all $x \in A$, $r_M(A \setminus \{x\}) = |A| - 1 = r_M(A)$.
  \end{enumerate}
\end{theorem}

After the multiple equivalent definitions of matroids, we can provide a series of canonical examples.

\begin{example}[uniform matroid $U_{k,n}$]
  $U_{k,n} = (E,\mathcal{I})$ with $E = \{1,\ldots,n\}$ and $\mathcal{I} = \{ A \subseteq E \mid |A| \leq k \}$ is a matroid, called the \emph<uniform!matroid>{uniform} matroid.
\end{example}

 In $U_{k,n}$ all possibly subsets of $n$ elements containing at most $k$ items are independent, while all larger subsets are dependent. Thus each subset with exactly $k$ items is a base, each subset with $k+1$ items is a circuit, and $r(U_{k,n}) = k$.

\begin{example}[vector matroid {$\matheu{M}[A]$} of a matrix]
  Given a $m \times n$ matrix $A$ over a field $F$, then $A$ defines a matroid $(E,\mathcal{I})$ as follows: let $E$ be the set of columns of $A$ (or $n$ representatives thereof) and $\mathcal{I}$ be the collection of all subsets of $E$ containing columns of $A$, which are linearly independent as vectors of the vector space $F^m$. This matroid is denoted by $\matheu{M}[A]$, and is called the \emph<vector!matroid>{vector matroid} of $A$.
\end{example}

The columns of $A$ define the column space $\langle A \rangle$ as a subspace of $F^m$ with dimension $\dim_F(\langle A \rangle) = \operatorname{rank}_F(A) = r(\matheu{M}[A])$. The matroid bases of $\matheu{M}[A]$ are all vectorial basis of the column space of $A$ (maximal vector sets that are linearly independent), while the circuits of $\matheu{M}[A]$ are all minimal sets of linearly dependent column vectors in $A$.

For our game scenario, the following class of matroids connects graph theory and matroid theory.

\begin{example}[cycle matroid {$\matheu{M}[G]$} of a graph]
  Given a graph $G = (V,E,\delta)$ with incidence matrix $H$, then $H$ defines a vector matroid $\matheu{M}[H]$, which is also called the \emph<cycle!matroid>{cycle matroid} of $G$ and denoted by $\matheu{M}[G]$.
\end{example}

In figure~\ref{fig:example graphic} the cycle matroid of an example graph $G$ is shown. It is the vector matroid of the incidence matrix $H$, also shown in the figure, within which each row represents a vertex of the graph, and each column specifies the adjacency of two vertices, or, equivalently, an edge in the graph. The matroid hence has one ground element for each edge of the graph, the circuits $\matheu{C}$ are all simple cycles, and the bases $\matheu{B}$ all possible spanning trees of $G$. The set of independent sets $\mathcal{I}$ contains all bases, and all subsets of bases, down to the empty set.

\begin{figure}\centering
\hfill%
\begin{tikzpicture}[graphfinal,
  baseline=(current bounding box.center),
  scale=1.5, auto, label distance=-2pt]

  \node [odot,label={180:$v_1$}] (1) at (-1,0) {};
  \node [odot,label={90:$v_2$}] (2) at (0,1) {};
  \node [odot,label={0:$v_3$}] (3) at (1,0) {};
  \node [odot,label={270:$v_4$}] (4) at (0,-1) {};

  \node at (-1,1) {$G$};

  \draw (1) to node[inner sep=1pt] {$a$} (2);
  \draw (2) to node[inner sep=1pt] {$b$} (3);
  \draw (3) to node[inner sep=1pt] {$c$} (4);
  \draw (4) to node[inner sep=1pt] {$d$} (1);
  \draw (1) to[bend left=15] node[inner sep=3pt] {$e$} (3);
  \draw (1) to[bend left=-15] node[inner sep=3pt,swap] {$f$} (3);

  \begin{scope}[
    yshift=-28mm,
    every matrix/.style={
      row sep=2pt, column sep=2pt, inner sep=0pt,
      every node/.style={minimum size=14pt, text depth=0pt},
      matrix of math nodes,
    },
    ]

    \matrix (m) [ left delimiter=(,right delimiter=) ]
    {
      1 & 0 & 0 & 1 & 1 & 1 \\
      1 & 1 & 0 & 0 & 0 & 0 \\
      0 & 1 & 1 & 0 & 1 & 1 \\
      0 & 0 & 1 & 1 & 0 & 0 \\
    };

    \matrix (ma) at (m.north) [yshift=3mm]
    {
      a & b & c & d & e & f \\
    };
    \node at (m.west) [xshift=-8mm] {$H = $};
    \matrix (me) at (m.east) [xshift=6mm]
    {
      v_1 \\ v_2 \\ v_3 \\ v_4 \\
    };
  \end{scope}

\end{tikzpicture}
\hfill%
\begin{minipage}{0.53\linewidth}
\begin{align*}
E = \{           & a,b,c,d,e,f \} \,,                                     \\[\medskipamount]
\matheu{C} = \{  & \{e,f\}, \{a,b,e\}, \{a,b,f\}, \{c,d,e\}, \{c,d,f\},   \\
                 & \{a,b,c,d\} \} \,,                                     \\[\medskipamount]
\matheu{B} = \{  & \{a,b,c\}, \{a,b,d\}, \{a,c,d\}, \{a,c,e\}, \{a,c,f\}, \\
                 & \{a,d,e\}, \{a,d,f\}, \{b,c,d\}, \{b,c,e\}, \{b,c,f\}, \\
                 & \{b,d,e\}, \{b,d,f\} \} \,,                            \\[\medskipamount]
\mathcal{I} = \{ & \emptyset, \{a\}, \{b\}, \{c\}, \{d\}, \{e\}, \{f\},   \\
                 & \{a,b\}, \{a,c\}, \{a,d\}, \{a,e\}, \{a,f\}, \{b,c\},  \\
                 & \{b,d\}, \{b,e\}, \{b,f\}, \{c,d\}, \{c,e\}, \{c,f\},  \\
                 & \{d,e\}, \{d,f\} \} \cup \matheu{B} \,.
\end{align*}
\end{minipage}
\hfill\null%

\caption{A graph $G$, its incidence matrix $H$ and cycle matroid $\matheu{M}[G]$.}\label{fig:example graphic}
\end{figure}
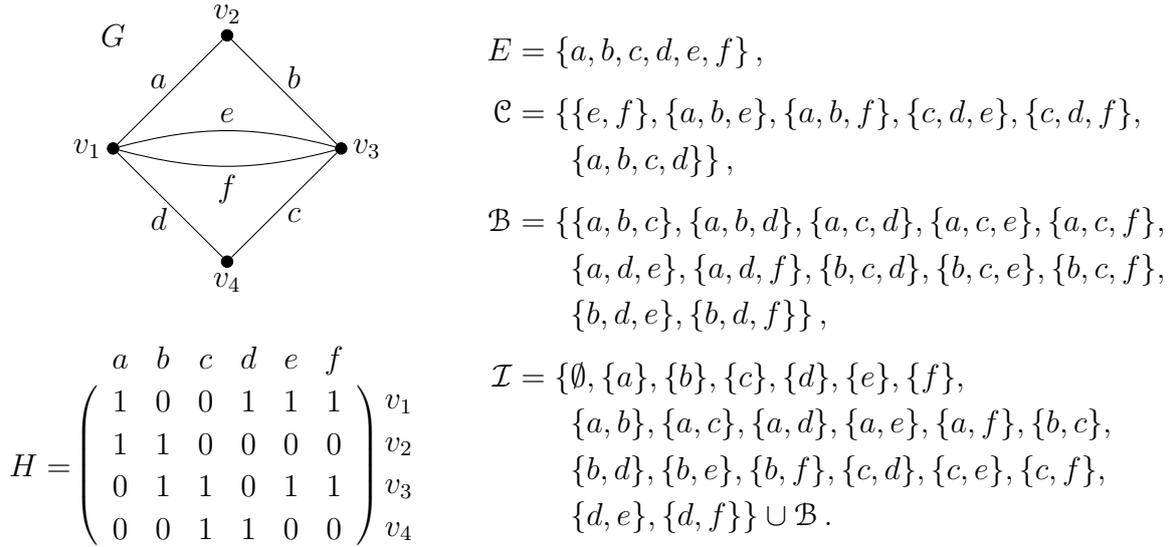

While only planar graphs have a \emph{dual} graph, which is defined by connecting adjacent faces of the planar graph, all matroids have a dual.

\begin{theorem}[dual of a matroid \R{\cites[thm.\,2.1.1]{oxley2011matroid}}]\label{thm:dual matroid}
  Let $M$ be a matroid with bases $\matheu{B}(M)$, then the collection $\matheu{B}^*(M) = \{ E \setminus B \mid B \in \matheu{B}(M) \}$ of subsets of $E$ defines a matroid $M^*$ on the elements $E$ with bases $\matheu{B}^*(M)$. This matroid $M^*$ is called the \emph<dual!matroid>{dual} matroid of $M$, and $(M^*)^* = M$.
\end{theorem}

As each matroid $M$ and its dual $M^*$ are on the same ground set $E$, one can regard special sets of $M^*$ as special sets in $M$, whose names are then prefixed with \emph{co-}.

\begin{definition}[cobase, cocircuit, and cohyperplane]
  Let $M$ be a matroid and $M^*$ its dual, then
  \begin{enumerate}
  \item the bases of $M^*$ are called \emph<cobase!matroid>{cobases} of $M$, denoted by $\matheu{B}^*(M)$,
    \symbol{B^*(M)}{$\matheu{B}^*(M)$}{cobases of a matroid $M$}
  \item the circuits of $M^*$ are called \emph<cocircuit!matroid>{cocircuits} of $M$, denoted by $\matheu{C}^*(M)$, and
    \symbol{C^*(M)}{$\matheu{C}^*(M)$}{cocircuits of a matroid $M$}
  \item the hyperplanes of $M^*$ are called \emph<cohyperplane!matroid>{cohyperplanes} of $M$.
  \end{enumerate}
\end{definition}

As with graphs, matroids are considered isomorphic if and only if a bijection exists which preserves the structure, in this case the independence structure.

\begin{definition}[isomorphic matroids]
  An \emph<isomorphism!matroid>{isomorphism} from a matroid $M_1 = (E_1,\mathcal{I}_1)$ to a matroid $M_2 = (E_2,\mathcal{I}_2)$ is a bijection $f : E_1 \rightarrow E_2$ on the element sets such that for all $X \subseteq E_1$ the image $f(E_1) \in \mathcal{I}_2$ if and only if $E_1 \in \mathcal{I}_1$. Two matroids $M_1$ and $M_2$ are called \emph<isomorphic!matroid>{isomorphic}, written $M_1 \cong M_2$, if an isomorphism from $M_1$ to $M_2$ exists.
\end{definition}

The variety of structures fulfilling an axiomatic matroid definition is surprisingly large, and one of the most interesting and important fields in matroid theory is the classification of matroids by their internal structure. These classes then have special properties that can be used to prove more elaborate theorems. The following standard definitions suffice for this thesis.

\begin{definition}[graphic, cographic, representable, binary, and regular matroids]
  \begin{enumerate}
  \item A matroid $M$ is called \emph<graphic!matroid>{graphic} if it is isomorphic to the cycle matroid of a graph.
  \item A matroid $M$ is called \emph<cographic!matroid>{cographic} if it is isomorphic to the dual of a graphic matroid.
  \item A matroid $M$ is called \emph<representable!matroid>{$F$-representable}, if $M \cong \matheu{M}[A]$ for a matrix $A$ over a field $F$.
  \item A matroid $M$ is called \emph<binary!matroid>{binary}, if $M$ is $\mathbb{F}_2$-representable where $\mathbb{F}_2$ is the field with two elements.
  \item A matroid $M$ is called \emph<regular!matroid>{regular} or \emph<unimodular!matroid>{unimodular}, if $M$ is representable over any field.
  \end{enumerate}
\end{definition}

\begin{remark}[graphs vs. graphic matroids]
  If $G$ is a connected graph, and $M = \matheu{M}[G]$ the corresponding graphic matroid, then
  \begin{enumerate}
    \item the set of bases $\matheu{B}(M)$ is the set of all spanning trees of $G$,
    \item the set of circuits $\matheu{C}(M)$ is the set of all (simple) cycles of $G$,
    \item the set of cocircuits $\matheu{C}^*(M)$ is the set of all minimal cuts of $G$.
  \end{enumerate}
\end{remark}

Obviously, Alice and Bob's graph game is played on a graphic matroid. Other than in graphs, circuits, cocircuits, bases and cobases are the \emph{primary} structure objects in matroids. Matroids hence inherently have fundamental circuits and ``cuts'', which are better called cocircuits.

\begin{theorem}[fundamental circuit $C_M(B,e)$]\label{def:matroid funcircuit}
  \symbol{C_M(B,e)}{$C_M(B,e)$}{fundamental circuit of $e$ and $B$}
  Given a matroid $M = (E,\mathcal{I})$ and a base $B \in \matheu{B}(M)$, then every non-base element $e \in E \setminus B$ closes a circuit contained in $B + e$. This circuit is unique within $B + e$, contains $e$, and is called the \emph[fundamental!circuit (matroid)][matroid!fundamental circuit]{fundamental circuit} $C_M(B,e)$ of $x$ and $B$ in $M$.
\end{theorem}
\begin{proof}
  As $B$ is maximally independent and $e \notin B$, $B + e$ contains a circuit. Let $C_1,C_2 \subseteq B + e$ be two different circuits, then each contains $e$. Furthermore, due to theorem~\ref{thm:weak circuit elimination}, $(C_1 \cup C_2) - x$ contains a circuit. However, $(C_1 \cup C_2) - x \subseteq B$, so this cannot be a circuit, and hence $C_1$ cannot be different from $C_2$.
\end{proof}

\begin{theorem}[fundamental cocircuit $D_M(B,e)$]\label{def:matroid funcut}
  \symbol{D_M(B,e)}{$D_M(B,e)$}{fundamental cocircuit of $e$ and $B$}
  Given a matroid $M = (E,\mathcal{I})$ and a base $B \in \matheu{B}(M)$, then every base element $e \in B$ defines a cocircuit contained in $E \setminus B + e$. This cocircuit is unique within $(E \setminus B) + e$, contains $e$, and is called the \emph[fundamental!cocircuit (matroid)][matroid!fundamental cocircuit]{fundamental cocircuit} $D_M(B,e)$ of $e$ and $B$ in $M$.
\end{theorem}
\begin{proof}
  Consider the dual matroid $M^*$, wherein $B$ is a cobase, hence $B^* = E \setminus B$ is a base (theorem~\ref{thm:dual matroid}). Then $C_{M^*}(B^*,e)$ is the unique fundamental circuit of $e \in E \setminus B^*$ and $B^*$ in $M^*$. So $C_{M^*}(B^*,e)$ is a uniquely defined cocircuit in $M$. In short: $D_M(B,e) = C_{M^*}(E \setminus B, e)$.
\end{proof}

Some other authors denote the fundamental cocircuit with $C_M^*(B,e) := C_{M^*}(E \setminus B, e)$, but we prefer the analogon to the graph definition.

\begin{theorem}[duality of fundamental circuit and cocircuit]\label{thm:duality circuit cocircuit}
  Given a matroid $M = (E,\mathcal{I})$, a base $B \in \matheu{B}(M)$, then for every base element $e \in B$ and non-base element $f \in E \setminus B$, we have
  \[
  e \in C_M(B,f) \quad\text{if and only if}\quad f \in D_M(B,e) \,.
  \]
\end{theorem}
\begin{proof}
  Assume $e \in C_M(B,f)$ and consider the hyperplane $H$ spanned by $B \setminus f$. Then $e \not\in H$, for otherwise $f$ would be in the closure $\operatorname{cl}(H) = \{ x \in E \mid r(H) = r(H \cup \{x\})$ of $H$ and hence in $H$. Thus $e$ and $f$ are in the cocircuit $E \setminus H = D_M(B,e)$. The other implication follows by duality.
\end{proof}

\begin{definition}[deletion, restriction, and contraction]
  If $M = (E,\mathcal{I})$ is a matroid and $E' \subseteq E$ a subset of elements, then \emph{deletion} of $E'$ from $M$ yields a matroid $M - E' := (E \setminus E', \{ A \in \mathcal{I} \mid A \cap E' = \emptyset \})$, while \emph{restriction} of $E'$ to $M$ yields $M|_{E'} := M - (E \setminus E') = (E', \{ A \subseteq E' \mid A \in \mathcal{I} \})$, and \emph{contraction} of $E'$ from $M$ yields a matroid $M \contr E' := (M^* - E')^*$.
\end{definition}

%%%%%%%%%%%%%%%%%%%%%%%%%%%%%%%%%%%%%%%%%%%%%%%%%%%%%%%%%%%%%%%%%%%%%%%%%%%%%%%%
\clearpage

\section{Bispanning Graphs}\label{sec:bispanning}

\begin{figure}[p]
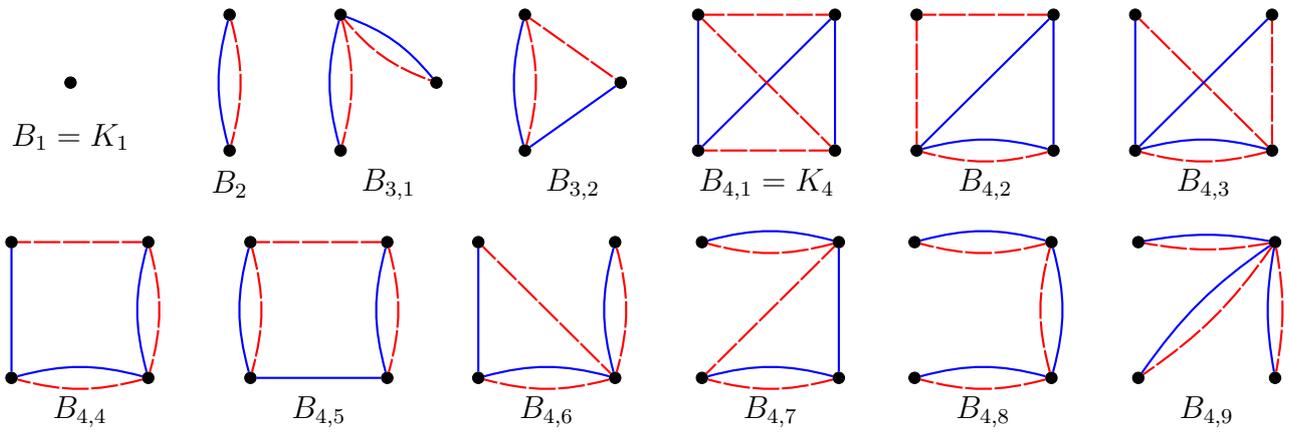
\centering
  \tikzset{every picture/.style={scale=1.8, graphfinal}}
  \tikzset{baseline=(n1.center)}



  \caption{All non-isomorphic bispanning graphs with at most four vertices.}\label{fig:small-bispanning}
\end{figure}

\begin{figure}[p]
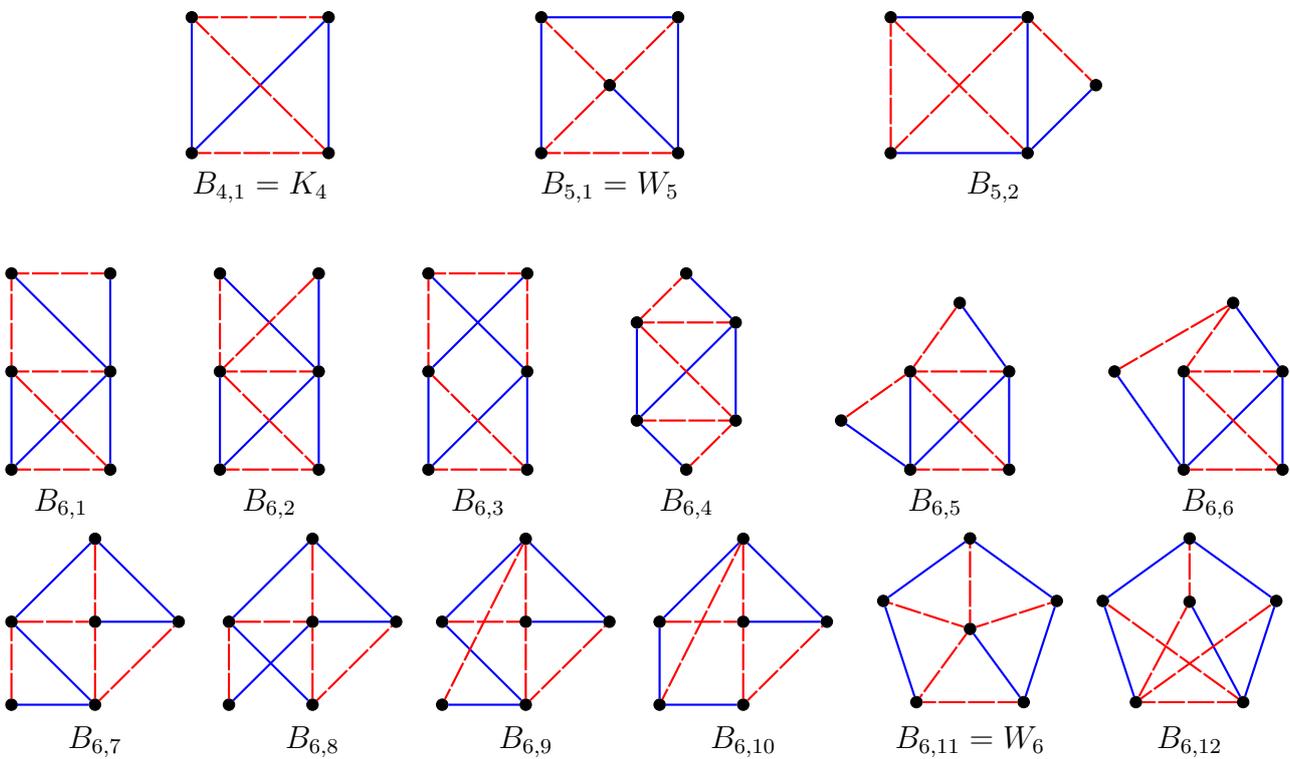
\centering
  \def\vsep{\bigskip\medskip}
  \tikzset{every picture/.style={scale=1.8, graphfinal}}

  \hfill%
  %
  \medskip

  \caption{All non-isomorphic simple bispanning graphs with at most six vertices.}\label{fig:small simple bispanning}
\end{figure}

The underlying objects of interest in this thesis are bispanning graphs, which are the board on which Alice and Bob play the unique exchange game.

\begin{definition}[bispanning graph]\label{def:bispanning}
  A graph $G = (V,E,\delta)$ is called \emph<bispanning!graph>[bispanning graph]{bispanning} if its edge set can be decomposed of two disjoint spanning trees $S$ and $T$, so if $E = S \cup T$ with $S \cap T = \emptyset$. We write $S \dotcup T$ for two disjoint spanning trees.
\end{definition}

See figures~\ref{fig:small-bispanning} and \ref{fig:small simple bispanning} for many examples of bispanning graphs. To refer to specific bispanning graph instances, we label small non-isomorphic bispanning graphs with $B_{n,k}$ where $n$ is the number of vertices. Table~\ref{tab:num bispanning} shows the total number of non-isomorphic bispanning graphs for small vertex sets.

To better refer to $S$ and $T$ in the figures, we always color \textcolor{blue}{$S$ blue} (solid) and \textcolor{red}{$T$ red} (and dashed infrequently). In general one can swap both, except when discussing examples in detail.

Immediately from the basic definition of bispanning graphs, many properties follow. We summarize straight-forward ones from previous authors~\cites[thm.\,6]{merkel2009basentauschspiel}[lem.\,2.2]{baumgart2009ranking} in the following theorem.

\begin{theorem}[basic properties of bispanning graphs]
  \begin{enumerate}
  \item If $G = (V,E,\delta)$ is a bispanning graph, then $|E| = 2 |V| - 2$.\label{thm:bispanning edge count}

  \item $K_1$ and $K_4$ are the two smallest simple bispanning graphs.

  \item No three edges are pairwise parallel in a bispanning graph.\label{thm:bispanning no three parallel}

  \item Every vertex in a bispanning graph has degree at least 2.\label{thm:bispanning deg 2}

  \item Every bispanning graph has a vertex of degree 2 or 3.\label{thm:bispanning deg23}

  \item If all $v \in V$ have $\deg(v) \geq 3$, then there are at least four vertices with degree 3.

  \end{enumerate}
\end{theorem}
\begin{proof}
  \begin{enumerate}
  \item As $E = S \dotcup T$ and $S$ and $T$ are spanning trees, both $|S| = |T| = |V| - 1$ due to theorem~\ref{thm:tree} and thus $|E| = 2 |V| - 2$.

  \item One can verify that $K_1$ and $K_4$ are bispanning by regarding figure~\ref{fig:small-bispanning}. For $|V| = 4$, (i)~implies $|E| = 6$, and $K_4$ is the only simple graph with these properties. Any simple graph with less than four vertices has at most three edges. Due to (i) any bispanning graph has an even number of edges, thus only $|V| = 1$ solves $|E| = 2 |V| - 2$.

  \item Three pairwise parallel edges would obviously form a cycle of length two in one of the disjoint spanning trees.

  \item In any bispanning graph $(V,S \dotcup T,\delta)$ both $G[S]$ and $G[T]$ are connected subgraphs, thus every vertex has at least two incident edges.

  \item Consider a bispanning graph $G = (V,E,\delta)$ with $\deg_G(v) \geq 4$ for all $v \in V$. Then theorem~\ref{thm:handshake} implies $2 |E| = \sum_{v \in V} \deg_G(v) \geq 4 |V|$. However, (i) requires $2 |E| = 4 |V| - 4$, which contradicts $2 |E| \geq 4 |V|$.

  \item Let $X = \{ v \in V \mid \deg_G(v) = 3 \} \subseteq V$ be all vertices of degree three in a bispanning graph $G = (V,E,\delta)$. Again, theorem~\ref{thm:handshake} and (i) imply $2 |E| = 4 |V| - 4 = \sum_{v \in V} \deg_G(v) = 3 |X| + \sum_{v \in V \setminus X} \deg_G(v) \geq 3 |X| + 4 (|V| - |X|) = 4 |V| - |X|$. Thus $|X| \geq 4$.
\hfill\proofSymbol
\end{enumerate}
\end{proof}

\begin{table}\centering
  \begin{tabular}{crrr}
    \toprule
    $|V|$ & bispanning    & simple           & atomic      \\ \midrule
    1     & 1             & 1                & -           \\
    2     & 1             & 0                & 1           \\
    3     & 2             & 0                & 0           \\
    4     & 9             & 1                & 1           \\
    5     & 46            & 2                & 1           \\
    6     & 380           & 12               & 4           \\
    7     & 4\,229        & 92               & 15          \\
    8     & 61\,103       & 1\,010           & 109         \\
    9     & 1\,077\,101   & 14\,957          & 1\,075      \\
    10    & 22\,364\,980  & 275\,748         & 14\,506     \\
    11    & 532\,825\,027 & 6\,000\,780      & 236\,243    \\
    12    &               & 149\,469\,333    & 4\,491\,490 \\
    13    &               & 4\,176\,699\,079 &             \\ \bottomrule
  \end{tabular}
  \caption{The number of non-isomorphic general, simple, and atomic bispanning graphs for small numbers of vertices.}\label{tab:num bispanning}
\end{table}

Intuitively, bispanning graphs have to be ``locally dense'' enough to allow the two disjoint spanning trees to reach all vertices, but sparse enough such that disjoint edge sets do not contain cycles, and hence are trees. Obviously, each vertex needs at least two edges.  Remarkably, the same ``connectivity factor'' carries over to \emph{sets of vertices}, as independently discovered in 1961 by Nash-Williams and Tutte.  They formulated the same idea in two different ways: Nash-Williams used vertex partitions to separate the graph, while Tutte uses edge cuts.  The original proofs of the two theorems are rather intricate. We omit them here and refer a shorter proof~\cite{chen1994short} and to Diestel's textbook \cite[sect.~2.4]{diestel2010graph}. The original theorems by Nash-Williams and Tutte apply to an arbitrary number $k$ of disjoint spanning trees, but we only require the case of $k=2$.

\begin{theorem}[arboricity of a graph (vertex partition version) \R{\cite{nash1961edge}}]\label{thm:nash-williams}
  An undirected graph $G=(V,E,\delta)$ is a bispanning graph, if and only if $|E| = 2 |V| - 2$ and
  $$|E_P| \geq 2 (|P| - 1) \quad\text{for every partition $P$ of $V$,}$$
  where $E_P$ is the set of edges with ends in different members of the partition $P$ and $|P|$ is the number of members.
  \symbol{EP}{$E_P$}{edges between members of $P$}
\end{theorem}

\begin{theorem}[arboricity of a graph (edge cut version) \R{\cite{tutte1961problem}}]\label{thm:tutte bispanning}
  An undirected graph $G=(V,E,\delta)$ is a bispanning graph, if and only if $|E| = 2 |V| - 2$ and
  $$|X| \geq 2 ( \operatorname{comp}(G \setminus X) - 1) \quad\text{for all edge subsets $X \subseteq E$.}$$
\end{theorem}

Depending on the situation, one can chose either theorem for proving that a graph is bispanning. We will mostly be using Nash-Williams' theorem, though one can surely rewrite the proofs to use Tutte's.

% ------------------------------------------------------------------------------

\subsection{Block Matroids}

The generalization of bispanning graphs to matroids are called block matroids and these follow a very similar definition.

\begin{definition}[block matroid]\label{def:block matroid}
  A matroid $M = (E,\mathcal{I})$ is called a \emph{block matroid} if its element set $E$ is the disjoint union of two bases.
\end{definition}

Naturally, the cycle matroid of every bispanning graph is a block matroid.  Examples for \emph{non-graphic} block matroids are $U_{2,4}$, or more generally $U_{k,2k}$ for $k \geq 2$, and the special regular matroids $R_{10}$ and $R_{12}$ (see appendix of \cite{oxley2011matroid}). Furthermore, the dual of a block matroid is also a block matroid, hence the duals of non-planar bispanning graphs yield more examples for non-graphic block matroids.

Since $r(M)$ is the size of a base, $|E| = 2 \, r(M)$ holds for every block matroid. But this criterion is not sufficient: consider the matroid with bases $\{ \{a,b\}, \{a,c\}, \{a,d\} \}$ as a counterexample. It has rank two, but the ground set $\{a,b,c,d\}$ has four elements which are not the disjoint union of any two bases, hence it is not a block matroid despite fulfilling the equality.

Hence, one needs a more elaborate characterization for block matroids. Edmonds showed that the principle behind the theorems of Nash-Williams and Tutte (\ref{thm:nash-williams} and \ref{thm:tutte bispanning}) carries over to block matroids.
\begin{theorem}[block matroid packing \R{\cites{edmonds1965lehman,edmonds1965minimum}[cor.\,8.2.58]{west2001introduction}}]
  A matroid $M = (E,\mathcal{I})$ is a block matroid, if and only if $|E| = 2 \, r(M)$ and
  $$|E| - |A| \geq 2 (r_M(E) - r_M(A)) \rlap{\qquad for all $A \subseteq E\,.$}$$
\end{theorem}

\begin{corollary}[simplified block matroid packing]
  A matroid $M = (E,\mathcal{I})$ with $|E| = 2 \, r(M)$ is a block matroid, if and only if $|E| = 2 \, r(M)$ and
  $$|A| \leq 2 \, r_M(A) \rlap{\qquad for all $A \subseteq E\,.$}$$
\end{corollary}

Our counterexample above violates the theorem since $|\{b,c,d\}| = 3 \nleq 2 \cdot 1 = 2 \, r_M(\{b,c,d\})$.

From the definition of block matroids, $E = B_1 \dotcup B_2$, one could suspect them to be self-dual ($M \cong M^*$) by definition. However, this would require all bispanning graphs to be planar (see \cite[sect.~2.3]{oxley2011matroid}), which is not the case: $B_{6,12}$ in figure~\ref{fig:small simple bispanning} is the smallest non-planar bispanning graph. The error in this hypothesis, is that while $E$ must be disjoint union for some pair(s) of bases, this need not hold for all bases and their complement. Hence, a block matroid may contain bases of which the complement is not a base, and this relation is flipped in the dual matroid. Hence, the dual of $\matheu{M}[B_{6,12}]$ is non-graphic but a block matroid.

% ------------------------------------------------------------------------------

\subsection{Inductive Construction with Double-Attach and Edge-Split-Attach}

Another special property of bispanning graphs is that they yield to a simple inductive construction method involving just two operations: \emph{double-attach} and \emph{edge-split-attach}. Using these two operations all bispanning graphs can be constructed from a single vertex. This constructive property makes the class of bispanning graphs tractable to inductive proofs, since one only needs to show that a property remains true under both operations. Alternatively, structural proofs can decompose a bispanning graph using the two operations in a similar way as in the proof of theorem~\ref{thm:inductive construction}: by reducing either at a vertex of degree two or three.

In some sense, the two operations \emph{double-attach} and \emph{edge-split-attach} are just slightly more complex than the two operations in the inductive construction of the well-studied \emph{series-parallel networks}~\cites{duffin1965topology}[sect.~5.4]{oxley2011matroid}, or the single operation of the open ear decomposition of $2$-vertex-connected graphs.

\begin{theorem}[double-attach and edge-split-attach operations\R{\cite[sect.~2.2.1]{baumgart2009ranking}}]\label{thm:inductive operations}
  If $G = (V,E,\delta)$ is a bispanning graph and $v \notin V$ a new vertex, then a larger bispanning graph with $|V|+1$ vertices and $|E|+2$ edges can be created using either of the following operations:
  \begin{enumerate}
  \item Let $x,y \in V$ be two not necessarily distinct vertices in $G$, then $$G' := (V \cup \{v\},E) + \{x,v\} + \{y,v\}$$ is a bispanning graph. We call this a \emph<double-attach!operation>{double-attach} operation of $v$ at $x,y$ (see figure~\ref{fig:double-attach}).\label{thm:double-attach}

  \item Let $\bar{e} \in E$ be an edge with $\delta(\bar{e}) = \{x,y\}$ and $z \in V$ a third not necessarily distinct vertex, then $$G'' := (V \cup \{v\},E) - \bar{e} + \{x,v\} + \{y,v\} + \{z,v\}$$ is a bispanning graph. We call this an \emph<edge-split-attach!operation>{edge-split-attach} operation of $\bar{e}$ at $z$, as illustrated in figure~\ref{fig:edge-split-attach}.\label{thm:edge-split-attach}

  \end{enumerate}
\end{theorem}
\begin{proof} Let $S \dotcup T$ be two disjoint spanning trees of $G$.
  \begin{enumerate}
  \item If $e_x$ and $e_y$ are the two new edges of $G'$ with $\delta'(e_x) = \{x,v\}$ and $\delta'(e_y) = \{y,v\}$, then $S + e_x$ and $T + e_y$ are two disjoint spanning trees of $G'$.

  \item Let $e_x$, $e_y$, and $e_z$ be the three new edges of $G''$ with $\delta''(e_x) = \{x,v\}$, $\delta''(e_y) = \{y,v\}$, and $\delta''(e_z) = \{z,v\}$. Either $S$ or $T$ contains $\bar{e}$, so assume without loss of generality $\bar{e} \in S$, otherwise relabel. Then $S - \bar{e} + e_x + e_y$ and $T + e_z$ are two disjoint spanning trees of $G''$.

  \end{enumerate}
\end{proof}
\begin{figure}[b]\centering
  \tikzset{every picture/.style={scale=2.5, graphfinal}}

  \hfill%
  \subcaptionbox{double-attach operation\label{fig:double-attach}}[66mm]{%
    \centering
    \begin{tikzpicture}[yscale=0.9]

      \draw[fill=black!5] (0.5,0) ellipse (9mm and 3mm);
      \node at (0.5,-0.1) {$G$};

      \node (x) [vdot] at (0,0) {$x$};
      \node (y) [vdot] at (1,0) {$y$};
      \node (v) [vdot] at (0.5,1) {$v$};

      \draw[B] (x) -- node [vlabel] {$e_x$} (v);
      \draw[R] (y) -- node [vlabel,swap] {$e_y$} (v);

    \end{tikzpicture}
  }%
  \hfill%
  \subcaptionbox{edge-split-attach operation\label{fig:edge-split-attach}}[70mm]{%
    \centering
    \begin{tikzpicture}[yscale=0.9,xscale=0.9]

      \draw[fill=black!5] (1,0) ellipse (15mm and 3mm);
      \node at (1.35,-0.1) {$G$};

      \node (x) [vdot] at (0,0) {$x$};
      \node (y) [vdot] at (1,0) {$y$};
      \node (z) [vdot] at (2,0) {$z$};
      \node (v) [vdot] at (1,1) {$v$};

      \draw[B,dashed] (x) -- node [vlabel,above,inner sep=3pt] {$\bar{e}$} (y);
      \draw[B] (x) -- node [vlabel] {$e_x$} (v);
      \draw[B] (y) -- node [vlabel,swap] {$e_y$} (v);
      \draw[R] (z) -- node [vlabel,swap] {$e_z$} (v);

    \end{tikzpicture}
  }%
  \hfill\null%
  \caption{The double-attach operation and edge-split-attach operations.}
\end{figure}

\begin{theorem}[inductive construction of bispanning graphs\R{\cite[thm.\,2.6]{baumgart2009ranking}}]\label{thm:inductive construction}
  Any bispanning graph $G = (V,E,\delta)$ can be constructed using the operations ``double-attach'' and ``edge-split-attach'' starting with a single vertex.
\end{theorem}
\begin{proof}
  We prove this by induction over the number of vertices $|V|$. The only bispanning graph with $|V| = 1$ is an isolated vertex and serves as a basis for the construction. For $|V| = 2$ the only bispanning graph is a pair of vertices with two parallel edges (see figure~\ref{fig:small-bispanning}), and this graph can be constructed from a single vertex using one ``double-attach'' operation.

  Now consider a bispanning graph $G = (V,E,\delta)$ with $|V| \geq 3$ and two disjoint spanning trees $S \dotcup T$. Due to theorem~\ref{thm:bispanning deg23}, the graph $G$ contains a vertex of degree two or three. If $G$ contains a vertex $v$ of degree two, then $v$ is connected to the remainder of $G$ by only two edges $e_1$ and $e_2$. Let $e_1 \in S$ and $e_2 \in T$ without loss of generality. Removal of $v$ together with $e_1$ and $e_2$ yields a graph $G' := G - v$ with $|V| - 1$ vertices and $|E| - 2$ edges, which admits the two disjoint spanning trees $S - e_1$ and $T - e_2$, as $e_1$ and $e_2$ are leaves in the trees. Thus applying the induction hypothesis to $G'$ assures that a sequence of the two operations exists for $G'$, and this sequence can be extended with a ``double-attach'' operation of $v$ at the other ends of $e_1$ and $e_2$, which results in $G$.

  If $G$ contains no vertex of degree two, then $G$ contains a vertex $v$ of degree three, which is connected to the remainder of $G$ by exactly three edges $e_1$, $e_2$, and $e_3$. Two of the three edges must be contained in the same disjoint spanning tree of $G$. Without loss of generality, we can assume the edges and trees labeled as $e_1, e_2 \in S$ and $e_3 \in T$. Let $x$ and $y$ be the other ends of $e_1$ and $e_2$ and consider the graph $G' := G - v + \{x,y\}$, which is the reversal of the ``edge-split-attach'' operation (theorem~\ref{thm:edge-split-attach}). As $e_3$ is a leaf edge in $T$, $T - e_3$ is a spanning tree of $G'$, and $S - e_1 - e_2 + \{x,y\}$ is a second disjoint spanning tree of $G'$. Hence, $G'$ is a bispanning graph with $|V| - 1$ vertices and the induction hypothesis can be applied. A sequence of operations for $G'$ can thus be extended with a ``edge-split-attach'' operation to construct $G$.
\end{proof}

\begin{remark}[properties of inductive construction]
  The inductive construction process and the proof of theorem~\ref{thm:inductive construction} show important properties of bispanning graphs, however, we must highlight some short-comings of the process:
  \begin{enumerate}
  \item While the construction process implicitly defines two disjoint spanning trees, the proof \emph{assumes known} disjoint spanning trees. Hence, it does not yield an algorithm to find them, and we will deal with this in the next section.

  \item Different construction sequences can construct isomorphic bispanning graphs. The different sequences leading to the same bispanning graph may even contain different types of operations.

    The simplest example is $B_{3,2}$ (see figure~\ref{fig:small-bispanning}, page~\pageref{fig:small-bispanning}), which can either be constructed using two double-attach operations, or one double-attach followed by one edge-split-attach.

  \item Contrarily to other matroid constructions, the two operations are not dual to each other.
  \end{enumerate}
\end{remark}

% ------------------------------------------------------------------------------

\subsection{Finding Two Disjoint Spanning Trees}

In this section we discuss an algorithm for finding two disjoint spanning trees in a graph $G = (V,E,\delta)$, or determining that no such exist. The algorithm thus tests whether a given uncolored graph is bispanning, and delivers $E = S \dotcup T$.

This problem immediately appears similar to the classical minimum-cost spanning tree problem~\cite{boruuvka1926jistem,jarnik1930jistem,kruskal1956shortest,prim1957shortest}, which however only asks for a single spanning tree but additionally minimizes the tree's edge weight sum. The well-known solutions for the minimum-cost spanning tree by Kruskal and Jarn\'ik-Prim are textbook examples of greedy algorithms.

In this thesis, we review an algorithm by Roskind and Tarjan \cites[chapter~3]{roskind1983edge}{roskind1985note}, which can in general find $k$ disjoint spanning trees in $\mathcal{O}(k^2 |V|^2)$ time for unit weight edges, or minimum-cost spanning trees in $\mathcal{O}(|E| \log |E| + k^2 |V|^2)$ time. It is based on greedily finding augmenting edge swap sequences, which alternate between existing disjoint forests, similar to the bipartite matching algorithm by Berge~\cite{berge1957two}. For our bispanning graph application, we present precise pseudo-code of an optimized and simplified version with $k = 2$, which additionally takes parallel edges and a pre-existing coloring into account. The resulting algorithm thus runs in $\mathcal{O}(|V|^2)$ time.

A very similar algorithm is given in \cite{andres2014base}. It too searches for augmenting edge swap sequences using fundamental cycle/cuts. Due to repeated searches for cycles and cuts, it is only bounded by $\mathcal{O}(|E|^3) = \mathcal{O}(|V|^3)$ time, however, it is also considerable simpler. Jochim surveys in total three algorithms for the ``partitioning'' problems of bispanning graphs~\cite{jochim2014algorithmen}.

The algorithm by Roskind and Tarjan is presented as three subroutines in algorithms~\ref{alg:algbispanning}--\ref{alg:bfs-augment}. The basic idea behind the algorithm is to keep two disjoint forests as edge sets, which are initially empty, and attempt to add new edges one at a time to the first one. If this would create a cycle, test for each edge of the cycle whether one could swap it into the other forest. If this in turn creates a cycle, try to resolve it again. This recursive resolution search space is explored breadth first and resembles a tree-graph of alternating swap sequences. When finally one edge is found that can be swapped without violating the two forests' property, the whole chain of swap sequences is executed. If no such edge can be found, then the original edge cannot be added to either disjoint forest.

To get better time complexity, Roskind and Tarjan, do not perform repeated searches for a cycle in the forests. Instead, they construct two ``colored'' trees from the disjoint edge sets, with the trees rooted at one of the new edge's vertices. The trees are constructed using a colored breadth-first search (BFS) and stored as predecessor edges for each vertex. Then, the test for a cycle can be implemented using a union-find data structure, and the cycle itself can be found using the corresponding colored tree by walking the predecessors to a known vertex. This trick yields the quadratic time complexity.

The main routine \FuncSty{FindBispanningTrees} in algorithm \ref{alg:algbispanning} takes a graph $G$ with an optionally pre-filled edge color array \DataSty{color}, containing values \textsl{blue}, \textsl{red}, and any other value like \textsl{black} for unmarked edges. The pre-filled array can be used, for example, to keep most edge colors the same, when a user wishes to add an edge to the graph. The array \DataSty{color} defines two edge sets, \textsl{blue} and \textsl{red}, which are certified to be forests throughout the algorithm by using two union-find data structure $U_{\textsl{blue}}$ and $U_{\textsl{red}}$. In lines 3--7, the initial colors are checked for cycles and added to the union-find data structure. After this initialization, the main loop follows in line 8--16, wherein each edge is added to one of the forests; if an edge cannot be added, then the algorithm terminates and returns \textsl{false}. When considering an edge $e$, the algorithm first checks whether it was pre-colored (line 9), then tests whether it can be added to either forest without further ado using the union-find data structures (lines 10--13), which is a simple optimization (no costly colored BFS are done). If this is not possible, the \FuncSty{AugmentTree} routine is called to find an augmenting color swap sequence.

The initial steps of \FuncSty{AugmentTree} (algorithm~\ref{alg:bfs-augment}) are to construct two colored breadth-first trees rooted at vertex $v_0$, which is an arbitrary end of the edge $e_0$, that should be added to the forests. The subroutine \FuncSty{ColoredBFS} in algorithm \ref{alg:bfs-color} is a standard breadth-first search modified to consider only edges of a specific color (\textsl{blue} or \textsl{red}). The array \DataSty{label} is used to remember the previous edge in the augmenting color swap sequence that will be constructed, hence \DataSty{label} resembles a tree on the edge set. The queue $Q$ contains unseen edges for the breadth-first exploration of this edge set.

To explore the edge swap sequence space, an edge $e$ is taken from $Q$ and inspected whether it can be added to a forest set. If $e$ is \textsl{blue} then, the forest \textsl{red} is the destination, otherwise, if $e$ is \textsl{red} or \textsl{black}, then it should become \textsl{blue}. Then the corresponding union-find data structure is queried (line 9), whether $e$ would close a cycle. If it does not, then $e$ is the final edge is a valid color swap sequence, that was previously saved in the \DataSty{label} edge tree. It only remains to walk the tree branch backwards to the root, swapping colors along the way (lines 11--14). However, if $e$ closes a cycle in the other edge set, then the swap sequence space needs further exploration along the cycle. The cycle edges are found by moving along the corresponding colored BFS tree backwards until either the root $v_0$ or a vertex is found that is already inspected (lines 21--24). The starting vertex $x$ for the walk backwards in $\DataSty{pred}_c$ is the end of $e$, which was not reached by the BFS exploration yet (lines 17--19). The walk backwards needs to push edges on a stack, and then push them in reverse order into the BFS queue, because the exploration must be done breadth-first from $e_0$.

\begin{algorithm}[p]
  \caption{Find two disjoint spanning trees}\label{alg:algbispanning}
  \SetKwFunction{FindBispanningTrees}{FindBispanningTrees}
  \SetKwFunction{AugmentTree}{AugmentTree}
  \lSetKwArray{pred}{pred}
  \lSetKwArray{color}{color}
  \def\blue{\textsl{blue}}
  \def\red{\textsl{red}}

  \Function{\FindBispanningTrees{$G = (V,E,\delta),\color[]$}}
  {
    \KwIn{A graph $G$, and an optionally pre-initialized edge color array $\color[]$.}

    $U_\blue := \KwSty{new}\ \FuncSty{UnionFind}()$,\quad $U_\red := \KwSty{new}\ \FuncSty{UnionFind}()$ \tcp*{Union-find for two forests.}

    \ForEach(\tcp*[f]{Initially check each edge $e$.}){$e \in E$ \KwSty{with} $\{ v_1,v_2 \} := \delta(e)$}{
      \uIf{$\color[e] \in \{ \blue, \red \}$ \KwSty{and} $U_{\color[e]}.\FuncSty{find}(v_1) \neq U_{\color[e]}.\FuncSty{find}(v_2)$}{
        $U_{\color[e]}.\FuncSty{union}(v_1,v_2)$ \tcp*{If edge $e$ is already colored and independent.}
      }
      \Else{
        $\color[e] := \textsl{black}$ \tcp*{If edge $e$ is uncolored or closes a cycle.}
      }
    }

    \ForEach(\tcp*[f]{Goal: add $e$ to either forest.}){$e \in E$ \KwSty{with} $\{ v_1,v_2 \} := \delta(e)$}{
      \lIf{$\color[e] \in \{ \blue, \red \}$}{
        \KwSty{continue} \tcp*[f]{If edge $e$ is already colored.}%
      }
      \uElseIf{$U_\blue.\FuncSty{find}(v_1) \neq U_\blue.\FuncSty{find}(v_2)$}{
        $\color[e] := \blue$,\quad $U_\blue.\FuncSty{union}(v_1,v_2)$ \tcp*{Easy case: $e$ is independent in blue forest.}
      }
      \uElseIf{$U_\red.\FuncSty{find}(v_1) \neq U_\red.\FuncSty{find}(v_2)$}{
        $\color[e] := \red$,\quad $U_\red.\FuncSty{union}(v_1,v_2)$ \tcp*{Easy case: $e$ is independent in red forest.}
      }
      \Else{
        \If(\tcp*[f]{Otherwise: try to find an}){\KwSty{not} $\AugmentTree(G,\color[],e,U_\blue,U_\red)$}{
          \Return \textsl{false} \tcp*{augmenting swap sequence.}
        }
      }
    }
    \Return \textsl{true} \tcp*{All edges in one of the forests.}
  }
  \KwOut{\textsl{true} if $G$ is bispanning and $\color[]$ contains two disjoint trees, \textsl{false} otherwise.}

\end{algorithm}

\begin{algorithm}[p]
  \caption{Calculate colored breadth-first search tree}\label{alg:bfs-color}
  \SetKwFunction{ColoredBFS}{ColoredBFS}
  \lSetKwArray{color}{color}
  \lSetKwArray{pred}{pred}
  \def\blue{\textsl{blue}}
  \def\red{\textsl{red}}

  \Function{\ColoredBFS{$G = (V,E,\delta),r \in V,\color[],c \in \{ \blue,\red \}$}}
  {
    \KwIn{A graph $G$, a root $r$, an edge color array $\color[]$, and a color filter $c$.}

    \lForEach{$v \in V$}{
      $\pred[v] := \bot$ \tcp*[f]{Reset predecessor array,}
    }

    $Q := \KwSty{new}\ \FuncSty{Queue}(\{ r \})$,\quad $\pred[r] := \top$ \tcp*{mark root with sentinel, and initialize queue.}

    \While{$Q.\FuncSty{notEmpty}()$}{
      $v := Q.\FuncSty{popTop}()$ \tcp*{Process graph breadth first.}
      \ForEach(\tcp*[f]{Iterate over neighbors,}){$e \in \{ e \in E \mid v \in \delta(e) \}$ \KwSty{with} $\{v,w\} := \delta(e)$}{
        \If(\tcp*[f]{if neighbor is unseen and color matches,}){$\pred[w] = \bot$ \KwSty{and} $\color[e] = c$}{
          $\pred[w] := e$,\quad $Q.\FuncSty{append}(w)$ \tcp*{set predecessor edge and queue for visit.}
        }
      }
    }
    \Return $\pred$ \;
    \KwOut{The array $\pred[]$ contains predecessor edges forming a breadth-first tree of color $c$ in $G$ rooted at $r$.}
  }
\end{algorithm}

\begin{algorithm}[p]
  \caption{Cyclic search for an augmenting path}\label{alg:bfs-augment}
  \SetKwFunction{ColoredBFS}{ColoredBFS}
  \SetKwFunction{AugmentTree}{AugmentTree}
  \lSetKwArray{pred}{pred}
  \lSetKwArray{color}{color}
  \lSetKwArray{label}{label}
  \def\blue{\textsl{blue}}
  \def\red{\textsl{red}}

  \Function{\AugmentTree{$G = (V,E,\delta),\color[],e_0 \in E,U_\blue,U_\red$}}
  {
    \KwIn{A graph $G$, an array $\color[]$, two union-find data structures $U_\blue$ and $U_\red$, and a new edge $e_0$ with $\{ v_0, w_0 \} := \delta(e_0)$.}

    $\pred_\blue := \ColoredBFS(G,v_0,\blue)$ \tcp*{Find predecessor BFS trees for}
    $\pred_\red := \ColoredBFS(G,v_0,\red)$ \tcp*{walking cycles back instead of searching.}

    \lForEach(\tcp*[f]{Clear labels.}){$e \in E$}{
      $\label[e] := \bot$
    }

    $Q := \KwSty{new}\ \FuncSty{Queue}(\{ e_0 \})$ \tcp*{Initialize BFS queue with root edge.}

    \While{$Q.\FuncSty{notEmpty}()$}{
      $e := Q.\FuncSty{popTop}()$ \KwSty{with} $\{ v,w \} := \delta(e)$
      \tcp*{Inspect edge in swap sequence space, and}

      $c := (\KwSty{if } \color[e] = \blue \KwSty{ then } \red \KwSty{ else } \blue)$
      \tcp*{try to add it to the other forest.}

      \uIf{$U_c.\KwSty{find}(v) \neq U_c.\KwSty{find}(w)$}{
        $U_c.\KwSty{union}(v,w)$
        \tcp*{Edge $e$ can be swapped without closing a cycle,}

        \While(\tcp*[f]{which finished a valid color swap sequence!}){$e \neq e_0$}{
          $\KwSty{swap}(c, \color[e])$
          \tcp*{Hence perform swap sequence along predecessors}
          $e := \label[e]$
          \tcp*{in label up to root edge, swapping $c$ along the way,}
        }
        $\color[e] := c$ \tcp*{and finally coloring $e_0$.}
        \Return \textsl{true}
      }
      \Else(\tcp*[f]{If $e$ closes a cycle, let $x$}){
        \lIf(\tcp*[f]{be the end vertex of $e$}){$v \neq v_0$ \KwSty{and} $\label[\pred_c[v]] = \bot$}{
          $x := v$
        }
        \lElseIf(\tcp*[f]{previously unreached}){$w \neq v_0$ \KwSty{and} $\label[\pred_c[w]] = \bot$}{
          $x := w$
        }
        \lElse(\tcp*[f]{by edges in the \DataSty{label} tree.}){
          \KwSty{abort}(``This cannot occur.'')
        }

        $S := \KwSty{new}\ \FuncSty{Stack}()$ \;

        \While(\tcp*[f]{Collect edges in cycle back to}){$x \neq v_0$ \KwSty{and} $\label[\pred_c[x]] = \bot$}{
          $e' := \pred[x]$ \KwSty{with} $\{ v', w' \} := \delta(e')$
          \tcp*{the last previously inspected edge, and}

          $S.\KwSty{push}(e')$ \tcp*{save these on $S$.}

          $x := (\KwSty{if } x = v' \KwSty{ then } w' \KwSty{ else } v')$ \;
        }

        \While{$S.\FuncSty{notEmpty}()$}{
          $e' := S.\FuncSty{popTop}()$ \;
          $\label[e'] := e$
          \tcp*{Label edges in cycle with $e$ in swap sequence, and}

          $Q.\FuncSty{push}(e')$
          \tcp*{queue for inspection in order outgoing from clump.}
        }
      }
    }
    \Return \textsl{false} \;
    \KwOut{\textsl{true} if $e_0$ was colored, \textsl{false} is no augmenting color swap sequence was found.}
  }
\end{algorithm}

%%%%%%%%%%%%%%%%%%%%%%%%%%%%%%%%%%%%%%%%%%%%%%%%%%%%%%%%%%%%%%%%%%%%%%%%%%%%%%%%
\clearpage

\section{Exchange Graphs and Games}\label{sec:exchange graph}

In this section we formalize the swapping of edges in the games on bispanning graphs. We first define symmetric edge exchanges and then exchange graphs where vertices are configurations of disjoint spanning trees and edges correspond to possible edge exchanges. The focus of the thesis is on exchange graphs containing only unique edge exchanges or ``forced moves'', where Alice leaves Bob no choice in which edge to color.

For this section we generally assume that a particular bispanning graph $G = (V,E,\delta)$ is given, though we reiterate this premise in the theorems to keep them self-consistent.

% ------------------------------------------------------------------------------

\subsection{Unique and Unrestricted Symmetric Edge Exchanges}

We first define and prove the existence of unrestricted (not necessarily unique) and unique symmetric edge exchanges, without implying an order on the pair of trees.

\begin{theorem}[(unrestricted) symmetric edge exchange in bispanning graphs]\label{thm:edge exchange}
  Given two disjoint spanning trees $S \dotcup T$ in a bispanning graph $G$, then for every edge $e \in S$ there exists at least one edge $f \in T$ , such that $S - e + f$ and $T + e - f$ are a pair of disjoint spanning trees of $G$.

  Any edge $f \in D(S,e) \cap C(T,e)$ with $f \neq e$ can be selected, and we call this operation a \emph<edge!exchange>[edge exchange]{(unrestricted) symmetric edge exchange} or simply \emph{edge exchange} $(e,f)$ on $(S,T)$ with $(e,f) \in S \times T$.
\end{theorem}
\begin{proof}
  Removing $e$ from $S$ opens the cut $D(S,e) \subseteq T + e$ and adding $e$ to $T$ closes the cycle $C(T,e) \subseteq T + e$ (see theorems~\ref{thm:fundamental cut}, \ref{thm:fundamental cycle}). Since $e$ is in both cycle and cut, their intersection contains at least another edge, because the cycle crosses the cut an even number of times. Hence $|D(S,e) \cap C(T,e)| = 2 k$ for some integer $k \geq 1$. Select any $f \in D(S,e) \cap C(T,e)$ with $f \neq e$, then $f \in T$. Adding $f$ to $S - e$ yields a tree due to theorem~\ref{thm:tree connected}, and removing $f$ from $T + e$ yields a tree due to theorem~\ref{thm:tree acyclic}.
\end{proof}

\begin{definition}[unique symmetric edge exchange in bispanning graphs]\label{def:unique edge exchange}
  Given two disjoint spanning trees $S \dotcup T$ in a bispanning graph $G$, if for an edge $e \in S$ there exists \emph{only a single edge} $f \in T$, such that $S - e + f$ and $T + e - f$ are a pair of disjoint spanning trees of $G$, then the symmetric edge exchange $(e,f)$ on $(S,T)$ is called \emph{unique}.

  Equivalently, an edge exchange $(e,f)$ on $(S,T)$ with $(e,f) \in S \times T$ is unique, if and only if $D(S,e) \cap C(T,e) = \{ e,f \}$, which yields only one choice for $f \in T$.
\end{definition}

The (unrestricted) symmetric edge exchanges defined in theorem~\ref{thm:edge exchange} correspond to moves by Alice and Bob, where Bob can sometimes choose freely between multiple edges to repair the two disjoint spanning trees. Unique symmetric edge exchanges (of definition~\ref{def:unique edge exchange}) are restricted such that Bob has only one possible edge to choose. We will further explain this decisive difference using figure~\ref{fig:example edge exchange}, which shows a bispanning graph and two edge exchanges: one non-unique and one unique.

\pagebreak[4]

First regard $e \in T$: by coloring $e$ blue Alice may close the cycle $C(S,e)$, which is highlighted by the thick blue cycle, and opens the cut $D(T,e)$ marked by the dashed blue line. Bob can then select any of the edges $\{ f_1, f_2, f_3 \}$ to break the cycle and amend the cut. Note that both cycle and cut contain another edge, which does not suffice as it is not in the intersection of cycle and cut. This is an example of a non-unique edge exchange $(e,f_i)$ on $(T,S)$.

However, if Alice chooses to color $e' \in S$ red, and thereby close the red cycle and open the red cut, Bob has no choice but to color $f'_1$ blue. This is an example for a ``forced move'' by Alice, or in words of the new definition: a unique edge exchange $(e',f'_1)$ on $(S,T)$.

We have now modeled the basic edge swap operation of the bispanning graph games, however, we still need to assign the players ownership of the trees. Therefore, we have to keep an order of the disjoint trees: $S$ will be Alice's tree and $T$ Bob's, and a state of the game is an ordered pairs $(S,T)$. Depending on the rules, Alice may then pick an edge $e$ from $S$ or $T$, or just from $T$. Since the game state is an ordered pair, we must define two distinct edge exchange operations:

\begin{figure}
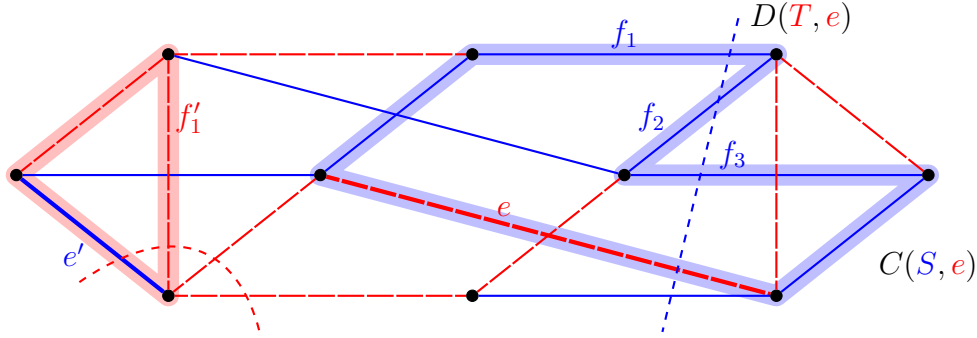
\centering
  \tikzset{every picture/.style={scale=2, graphfinal}}

  \caption{A bispanning graph with a non-unique edge exchange and a unique edge exchange.}\label{fig:example edge exchange}
\end{figure}

\begin{definition}[ordered symmetric $S$/$T$ edge exchange in bispanning graphs]\label{def:ordered edge exchange}
  For an ordered pair of disjoint trees $(S,T)$ and an edge exchange $(e,f) \in S \times T$, we call the transition to $S - e + f$ and $T + e - f$ a \emph{(symmetric) $S$ edge exchange} and write
  $$(S,T) \xchgarrow{(e,f)}{S} (S - e + f, T + e - f) \,.$$
  Likewise, for an edge exchange $(f,e) \in T \times S$, we call the transition to $S + f - e$ and $T - f + e$ a \emph{(symmetric) $T$ edge exchange} and write
  $$(S,T) \xchgarrow{(f,e)}{T} (S + f - e, T - f + e) \,.$$
  Furthermore, if the edge exchange is unique, we add the subscript ``UE'' to the transition arrow, yielding $\UExchgarrow{(f,e)}{T}$ in the second case above.
\end{definition}

The last definition raises the question why a distinction has to be made between $S$ and $T$ exchanges.  Indeed, for unrestricted exchanges the distinction between $S$ and $T$ edge exchanges is not important, since there is no restriction on the edge exchange and the resulting tree pair is identical.  However, for unique edge exchanges \emph{the distinction matters}, since one exchange can be unique and other non-unique.

Reconsider $e' \in S$ and $f'_1 \in T$ in figure~\ref{fig:example edge exchange}: while $(e',f'_1)$ is a unique edge exchange, because $D(S,e') \cap C(T,e') = \{ e', f'_1 \}$, the edge exchange $(f'_1,e')$ is not unique: $D(T,f'_1) \cap C(S,f'_1)$ contains $e'$, $f'_1$, $f_1$ and many other blue edges to the left of $f_1$.

While the previous definition is very technical, instead of $S$ and $T$ in subscript we will illustrate the edge exchange arrows using colors: \textcolor{blue}{blue for $S$} and \textcolor{red}{red for $T$}. Also, for an unlabeled edge exchange $(e,f)$ one can deduce whether it is a $S$ or $T$ edge exchange by regarding if $e \in S$ or $e \in T$.

A first remarkable property of unique exchanges is that they can be undone:

\begin{theorem}[reversibility of unique symmetric edge exchanges]\label{thm:reversibility unique exchange}
  If a $S$ edge exchange $(e,f) \in S \times T$ is unique for a pair of disjoint spanning trees $(S,T)$ in a bispanning graph $G$, then $(f,e)$ is a unique $S$ edge exchange for $(S - e + f, T + e - f)$. In symbols:
$$(S,T) \UExchgarrow{(e,f)}{S} (S - e + f, T + e - f) \text{\; if and only if \;} (S - e + f, T + e - f) \UExchgarrow{(f,e)}{S} (S,T) \,.$$
Likewise, the same is true for unique $T$ edge exchanges:
$$(S,T) \UExchgarrow{(f,e)}{T} (S + f - e, T - f + e) \text{\; if and only if \;} (S + f - e, T - f + e) \UExchgarrow{(e,f)}{T} (S,T) \,.$$
\end{theorem}
\begin{proof}
  As $S' := (S + f) - e$ is a tree, adding the edge $e$ closes a unique cycle $C(S',e)$. This cycle $C(S',e) = C((S + f) - e, e) = C(S,f)$, since re-adding $e$ to $S'$ closes the same cycle $C(S,f)$ as was closed by $f$ and broken by removing $e$. Likewise, $T' := (T - f) + e$ is a tree and removing the edge $e$ opens a unique cut $D(T',e)$. This cut $D(T',e) = D((T - f) + e, e) = D(T,f)$, since removing $e$ from $T'$ re-opens the same cut $D(T,f)$ as was bridged by $f$ and amended by $e$.
  Thus $C(S',e) \cap D(T',e) = C(S,f) \cap D(T,f) = \{ e,f \}$, and $(f,e)$ is a unique edge exchange.

  The same argument, interchanging $S'$, $T'$ and $S$, $T$ and $e$, $f$, shows the theorem for unique $T$ edge exchanges.
\end{proof}

The previous theorem allows us to interpret unique edge exchanges as an undirected relation between a pair of two disjoint spanning trees $(S,T)$ and $(S',T')$: they are linked by a unique edge exchange if some $(e,f) \in S \times T$ or $(f,e) \in (S',T')$ can accomplish the unique edge exchange. The previous theorem guarantees that the other pair exists and is unique.

\begin{definition}[undirected unique symmetric edge exchange in bispanning graphs]
  For an ordered pair of disjoint trees $(S,T)$ and a unique $S$ edge exchange $(e,f) \in S \times T$ we write the transition bidirectional due to theorem~\ref{thm:reversibility unique exchange}  as
  $$(S,T) \UEbixchgarrow{(e,f)}{S} (S - e + f, T + e - f) \,,$$
  which implies
  $$(S - e + f, T + e - f) \UEbixchgarrow{(f,e)}{S} (S,T) \,.$$
  Analogously, transitions due to unique $T$ edge exchange are written with bidirectional edges.
\end{definition}

As with arcs, instead of $S$ and $T$ in subscript, we will illustrate symmetric edge exchanges as lines using colors: \textcolor{blue}{blue for $S$} and \textcolor{red}{red for $T$}, and label the edge lines with $(e,f)$. These edge labels can be read left to right: if an exchange is labeled $(e,f) \in S \times T$ (a unique $S$ edge exchange), then $e$ must be in $S$ on the left side of the exchange and on the right side in $T' = T + e - f$. The same bidirectional edge label can also be read right to left: $f$ must be in $S' = S - e + f$ on the right of the exchange and on the left in $S$. So the edge exchange label $(e,f)$ must be read in the same direction as the transition.

% ------------------------------------------------------------------------------

\subsection{Directed Exchange Graphs}\label{sec:directed exchange graphs}

Having precisely defined restricted and unique symmetric edge exchanges, we now construct graphs of these edge exchanges. We call the graphs \emph<exchange!graph>[exchange graph]{exchange graphs}, while other authors also name them \emph{tree pair graphs} or \emph{base pair graphs}, since their vertices contain pairs of trees or matroids bases.  To match other authors~\cite{merkel2009basentauschspiel,andres2014base}, the exchange graphs we focus are called $\tau_2$, $\tau_3$, and $\tau_4$; $\tau_1$ was previously used to represent exchanges of whole edge subsets.

\begin{definition}[(unrestricted) exchange graph $\vec{\tau}_2$ of bispanning graphs]
  Given a bispanning graph $G = (V,E,\delta)$, we define a new directed graph $\vec{\tau}_2(G) = (V_{\tau(G)},E_{\vec{\tau}_2(G)},\delta_{\vec{\tau}_2(G)})$ with
  \exchangegraphsymbol{1tau2}{$\vec{\tau}_2(G)$}{(unrestricted) directed}
  \begin{enumerate}
  \item vertex set $V_{\tau(G)} = \{ (S,T) \mid E = S \dotcup T \text{ are disjoint spanning trees of $G$} \}$,
  \item arc set $\begin{aligned}[t] E_{\vec{\tau}_2(G)} =\; & \{ (e,f,S,T) \mid (e,f) \in S \times T \text{ is a $S$ edge exchange for } (S,T) \} \\ \cup\;& \{ (e,f,S,T) \mid (e,f) \in T \times S \text{ is a $T$ edge exchange for } (S,T) \} \,, \end{aligned}$
  \item and incidence\\[\medskipamount]
    \smash[t]{\(
    \delta_{\vec{\tau}_2(G)}( (e,f,S,T) ) = \begin{cases}
      (\, (S,T), (S - e + f, T + e - f) \,) & \!\text{if } (e,f) \in S \times T \,, \\
      (\, (S,T), (S + e - f, T - e + f) \,) & \!\text{if } (e,f) \in T \times S \,.
    \end{cases}
    \)}
  \end{enumerate}
\end{definition}

Intuitively, the vertices in the graph $\vec{\tau}_2(G)$ and all other exchange graphs specify all possible configurations of the disjoint spanning trees in the base bispanning graph $G$. An arc in $\vec{\tau}_2(G)$ marks a possible transition from one configuration to another by swapping $(e,f)$.

This definition of $\vec{\tau}_2(G)$ corresponds to the game where Alice picks an edge $e \in S \cup T$ and closes a cycle in either $S$ or $T$ (and opens a cut in the other tree) by inverting the color of $e$. Bob then breaks the cycle by inverting the color an edge in the other tree. The graph $\vec{\tau}_2(G)$ places no restrictions on whether Bob has a choice in selecting the edge or not.

$\vec{\tau}_2(G)$ is a pretty dense graph, since obviously theorem~\ref{thm:edge exchange} guarantees that for every tree pair $(S,T)$ and every edge $e \in S \cup T$ there exists at least one exchange edge $f$. We thus make the graph sparser by allowing only unique edge exchanges:

\begin{definition}[directed unique exchange graph $\vec{\tau}_3$ of bispanning graphs]\label{def:directed tau}
  Given a bispanning graph $G = (V,E,\delta)$, we define a new directed graph $\vec{\tau}_3(G) = (V_{\tau(G)},E_{\vec{\tau}_3(G)},\delta_{\vec{\tau}_3(G)})$ with
  \exchangegraphsymbol{1tau3}{$\vec{\tau}_3(G)$}{directed unique}
  \begin{enumerate}
  \item vertex set $V_{\tau(G)} = \{ (S,T) \mid E = S \dotcup T \text{ are disjoint spanning trees of $G$} \}$,
  \item arc set \\
    $
    \begin{aligned}[t] E_{\vec{\tau}_3(G)}
        =\; & \{ (e,f,S,T) \mid (e,f) \in S \times T \text{ is a unique $S$ edge exchange for } (S,T) \} \\
      \cup\;& \{ (e,f,S,T) \mid (e,f) \in T \times S \text{ is a unique $T$ edge exchange for } (S,T) \} \\
        =\; & \{ (e,f,S,T) \mid (e,f) \in S \times T \text{ and } D(S,e) \cap C(T,e) = \{e,f\} \} \\
      \cup\;& \{ (e,f,S,T) \mid (e,f) \in T \times S \text{ and } D(T,e) \cap C(S,e) = \{e,f\} \} \,,
    \end{aligned}
    $
  \item and incidence \\[\medskipamount]
    \smash[t]{\(
      \delta_{\vec{\tau}_3(G)}( (e,f,S,T) ) = \begin{cases}
        (\, (S,T), (S - e + f, T + e - f) \,) & \!\text{if } (e,f) \in S \times T \,, \\
        (\, (S,T), (S + e - f, T - e + f) \,) & \!\text{if } (e,f) \in T \times S \,.
      \end{cases}
      \)}
  \end{enumerate}
\end{definition}

Again, the vertices in the graph $\vec{\tau}_3(G)$ and all other exchange graphs specify all possible configurations of the disjoint spanning trees in the base bispanning graph $G$. An arc in $\vec{\tau}_3(G)$ marks a possible ``forced'' transition from one configuration to another by swapping $(e,f)$. This exchange graph is the main focus of this thesis.

The exchange graph of $\vec{\tau}_3(G)$ corresponds to the game where Alice deliberately picks an edge $e \in S \cup T$, which leaves Bob no choice in which edge to color to break the cycle and amend the cut. These ``forced moves'' are represented exactly by unique edge exchanges.

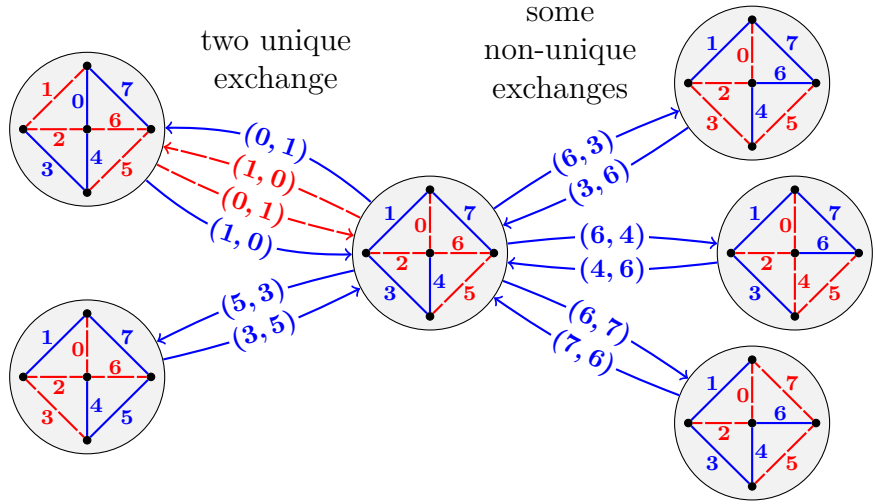
\begin{figure}\centering
  \tikzset{every picture/.style={scale=1.2}}
  \begin{tikzpicture}[
    taugraph,
    BigNode/.append style={
      minimum size=17mm,
    },
    ]

    \node at (0,0) (n0256) [BigNode] {0,2,5,6};

    \node at (180-20:4) (n1256) [BigNode] {1,2,5,6};
    \node at (180+20:4) (n0236) [BigNode] {0,2,3,6};

    \node at (28:4) (n0235) [BigNode] {0,2,3,5};
    \node at (0:4) (n0245) [BigNode] {0,2,4,5};
    \node at (-28:4) (n0257) [BigNode] {0,2,5,7};

    % 0 1 2 3 4 5 6 7
    \CompleteFourWheel{n0256}{R,B,R,B,B,R,R,B}

    % 0 1 2 3 4 5 6 7
    \CompleteFourWheel{n1256}{B,R,R,B,B,R,R,B}
    % 0 1 2 3 4 5 6 7
    \CompleteFourWheel{n0236}{R,B,R,R,B,B,R,B}

    \draw[R,<-] (n1256) to [bend left=7] node [BigLabel] {$(1,0)$} (n0256);
    \draw[R,->] (n1256) to [bend left=-7] node [BigLabel] {$(0,1)$} (n0256);

    \draw[B,<-] (n1256) to [bend left=21] node [BigLabel] {$(0,1)$} node[black,above=5mm,align=center] {two unique\\exchange} (n0256);
    \draw[B,->] (n1256) to [bend left=-21] node [BigLabel] {$(1,0)$} (n0256);

    \draw[B,<-] (n0236) to [bend left=7] node [BigLabel] {$(5,3)$} (n0256);
    \draw[B,->] (n0236) to [bend left=-7] node [BigLabel] {$(3,5)$} (n0256);

    % 0 1 2 3 4 5 6 7
    \CompleteFourWheel{n0257}{R,B,R,B,B,R,B,R}
    % 0 1 2 3 4 5 6 7
    \CompleteFourWheel{n0245}{R,B,R,B,R,R,B,B}
    % 0 1 2 3 4 5 6 7
    \CompleteFourWheel{n0235}{R,B,R,R,B,R,B,B}

    \draw[B,->] (n0256) to [bend left=7] node [BigLabel] {$(6,3)$} node [black,above=5mm,align=center,xshift=-3mm] {some\\non-unique\\exchanges} (n0235);
    \draw[B,->] (n0256) to [bend left=7] node [BigLabel] {$(6,4)$} (n0245);
    \draw[B,->] (n0256) to [bend left=7] node [BigLabel] {$(6,7)$} (n0257);

    \draw[B,<-] (n0256) to [bend left=-7] node [BigLabel] {$(3,6)$} (n0235);
    \draw[B,<-] (n0256) to [bend left=-7] node [BigLabel] {$(4,6)$} (n0245);
    \draw[B,<-] (n0256) to [bend left=-7] node [BigLabel] {$(7,6)$} (n0257);

  \end{tikzpicture}
  \caption{Excerpt of the exchange graph $\vec{\tau}_2(W_5)$ with a set of non-unique exchanges and a unique exchange.}\label{fig:example directed tau}
\end{figure}

Figure~\ref{fig:example directed tau} shows an excerpt of the exchange graph $\vec{\tau}_2(W_5)$ as an example of the structure of exchange graphs. In each vertex of $\vec{\tau}_2$ the pair of disjoint trees $(S,T)$ is represented by the small colored bispanning graphs with numbered edges. An arc between vertices is labeled with the edge exchange $(e,f)$ and colored \textcolor{blue}{blue} if it is \textcolor{blue}{$S$ edge exchange} or \textcolor{red}{red} if it is a \textcolor{red}{$T$ edge exchange}. This color is the one Alice must change the edge $e$ to for the particular transition.

Of the many edge exchanges shown in figure~\ref{fig:example directed tau} some are unique and some non-unique. Notice that if $(e,f)$ is a unique edge exchange, then no other arc $(e,f')$ with $f' \neq f$ can exit the vertex; the same is not true for non-unique edge exchanges, as can be seen in the figure.

The last directed exchange graph further limits which type of unique edge exchanges are allowed:

\begin{definition}[left-unique exchange graph $\vec{\tau}_4$ of bispanning graphs]\label{def:directed tau4}
  Given a bispanning graph $G = (V,E,\delta)$, we define a new directed graph $\vec{\tau}_4(G) = (V_{\tau(G)},E_{\vec{\tau}_4(G)},\delta_{\vec{\tau}_4(G)})$ with
  \exchangegraphsymbol{1tau4}{$\vec{\tau}_4(G)$}{directed left-unique}
  \begin{enumerate}
  \item vertex set $V_{\tau(G)} = \{ (S,T) \mid E = S \dotcup T \text{ are disjoint spanning trees of $G$} \}$,
  \item arc set\\
    \(
    E_{\vec{\tau}_4(G)} = \{ (e,f,S,T) \mid (e,f) \in S \times T \text{ is a unique $S$ edge exchange for } (S,T) \} \,,
    \)
  \item and incidence $\delta_{\vec{\tau}_4(G)}( (e,f,S,T) ) = (\, (S,T), (S - e + f, T + e - f) \,) \,.$
  \end{enumerate}
\end{definition}

This last exchange graph $\vec{\tau}_4(G)$ corresponds to the game where Alice deliberately picks an edge $e$ only from her tree $S$, which leaves Bob no choice in which edge in $T$ to color to break the cycle and amend the cut. Again, these ``forced moves'' are represented exactly by unique edge exchanges, but this time the forcing edge may only be selected from $S$. This restriction to ``left-unique'' edge exchanges truly changes the structure of $\vec{\tau}_4$~\cite{andres2014base}. The analogous restriction to selecting edges only from $T$ for unique exchanges is symmetric and results in the same exchange graphs (one can just swap the roles of $S$ and $T$). In figures of $\vec{\tau}_4(G)$ the restriction to one kind of unique edge exchange can be clearly seen, as all edges are of the same color: blue for left-unique exchanges and red for right-unique ones.

% ------------------------------------------------------------------------------

\subsection{Undirected Exchange Graphs}\label{sec:undirected exchange graphs}

We defined all three exchange graphs $\vec{\tau}_2$, $\vec{\tau}_3$, and $\vec{\tau}_4$ in section~\ref{sec:directed exchange graphs} as directed graphs due to the intuitive nature of a directed edge exchange. However, due to theorem~\ref{thm:reversibility unique exchange} and that unrestricted symmetric exchanges are reversible as well, all arcs in the exchange graphs have a twin arc doing the transition in reverse.

We can thus reduce the directed exchange graphs into an undirected form by combining the twin arcs into an undirected edge. The formal method we choose to do this is to define edges only for trees $S < T$, where $<$ is an arbitrary total order on the set of spanning trees of $G$. Since of the twin arcs exactly one has $S < T$ and the other $S > T$, the undirected graph will represent both if just one is included:

\begin{definition}[undirected exchange graphs $\tau_2(G)$, $\tau_3(G)$, and $\tau_4(G)$]\label{def:undirected tau}
  Given a bispanning graph $G = (V,E,\delta)$ and a directed exchange graph $\vec{\tau}_i(G) = (V_\tau,E_{\vec{\tau}_i(G)},\delta_{\vec{\tau}_i(G)})$ with $i \in \{ 2,3,4 \}$, we define a new undirected graph $\tau_i(G) = (V_{\tau(G)},E_{\tau_i(G)},\delta_{\tau_i(G)})$ with
  \exchangegraphsymbol{2tau2}{$\tau_2(G)$}{(unrestricted) undirected}
  \exchangegraphsymbol{2tau3}{$\tau_3(G)$}{undirected unique}
  \exchangegraphsymbol{2tau4}{$\tau_4(G)$}{undirected restricted unique}
  \begin{enumerate}
  \item the same vertex set $V_{\tau(G)} = \{ (S,T) \mid E = S \dotcup T \text{ are disjoint spanning trees of $G$} \}$,
  \item edge set $E_{\tau_i(G)} = \{ (e,f,S,T) \in E_{\vec{\tau}_i(G)} \mid S < T \}$
  \item and incidence\\[\medskipamount]
    \smash[t]{\(
      \delta_{\tau_i(G)}( (e,f,S,T) ) = \begin{cases}
        \{\, (S,T), (S - e + f, T + e - f) \,\} & \!\text{if } (e,f) \in S \times T \,, \\
        \{\, (S,T), (S + e - f, T - e + f) \,\} & \!\text{if } (e,f) \in T \times S \,.
      \end{cases}
      \)}
  \end{enumerate}
\end{definition}

This compact meta-definition yields undirected exchange graphs $\tau_2(G)$ containing all unrestricted symmetric edge exchanges, $\tau_3(G)$ containing all unique $S$ and $T$ symmetric edge exchanges, $\tau_4(G)$ contains all unique $T$ symmetric edge exchanges.

\begin{remark}[directed exchange graph $\vec{\tau}_3(G)$ or undirected exchange graph $\tau_3(G)$]
  We defined the directed exchange graph $\vec{\tau}_3(G)$ with vertex set $V_{\tau(G)}$ and arc set $E_{\vec{\tau}_3(G)}$, and the undirected exchange graph $\tau_3(G)$ with the same vertex set $V_{\tau(G)}$ and undirected edge set $E_{\tau_3(G)}$. In the remaining thesis we will mostly reference $\tau_3(G)$ in the text, but \emph{prove theorems using} $\vec{\tau}_3(G)$, because the directed version allows us to better reason about (directional) unique symmetric edge exchanges. The theorem about $\vec{\tau}_3(G)$ carry over to $\tau_3(G)$ appropriately.
\end{remark}

Other authors more laxly directly define undirected exchange graphs as follows:

\begin{definition}[simple undirected exchange graphs $\bar{\tau}_2(G)$, $\bar{\tau}_3(G)$, and $\bar{\tau}_4(G)$]\label{def:simple tau}
  Given a bispanning graph $G = (V,E,\delta)$, we define three simple undirected graphs $\bar{\tau}_i(G) = (V_{\tau(G)},E_{\bar{\tau}_i(G)})$, $i \in \{2,3,4\}$ with
  \exchangegraphsymbol{3tau2}{$\bar{\tau}_2(G)$}{(unrestricted) simple}
  \exchangegraphsymbol{3tau3}{$\bar{\tau}_3(G)$}{simple unique}
  \exchangegraphsymbol{3tau4}{$\bar{\tau}_4(G)$}{simple restricted unique}
  \begin{enumerate}
  \item vertex set $V_{\tau(G)} = \{ (S,T) \mid E = S \dotcup T \text{ are disjoint spanning trees of $G$} \}$, and
  \item edge sets \\[\medskipamount]
    $\begin{aligned}[t] E_{\bar{\tau}_2(G)} = \{ & \{ (S,T), (T'_1,T'_2) \} \mid \\
      &\exists (e,f) \in S \times T : (T'_1,T'_2) = (S - e + f, T + e - f) \} \,,
    \end{aligned}$ \\[\medskipamount]
    $\begin{aligned}[t] E_{\bar{\tau}_3(G)} = \{ & \{ (S,T), (T'_1,T'_2) \} \mid \\
      &\exists (e,f) \in S \times T : (T'_1,T'_2) = (S - e + f, T + e - f) \\
      &\text{and}\quad D(S,e) \cap C(T,e) = \{ e,f \} \quad\text{or}\quad D(S,f) \cap C(T,f) = \{ e,f \} \quad\} \,,
    \end{aligned}$ \\[\medskipamount]
    $\begin{aligned}[t] E_{\bar{\tau}_4(G)} = \{ & \{ (S,T), (T'_1,T'_2) \} \mid \\
      &\exists (e,f) \in S \times T : (T'_1,T'_2) = (S - e + f, T + e - f) \\
      &\text{and}\quad D(S,e) \cap C(T,e) = \{ e,f \} \quad\} \,.
    \end{aligned}$
  \end{enumerate}
\end{definition}

The graphs $\bar{\tau}_2(G)$, $\bar{\tau}_3(G)$, and $\bar{\tau}_4(G)$ from the preceding definition are ``simple versions'' of the graphs $\tau_2(G)$, $\tau_3(G)$, and $\tau_4(G)$, and can be gained from them by joining all parallel edges in the corresponding $\tau_i(G)$. There is, however, an important difference to our definition: the simple exchange graphs do not contain the information whether the unique edge exchanges is caused by either $(e,f)$ or $(f,e)$, i.e., whether the unique exchange is a $S$ or $T$ edge exchange.

\begin{figure}\centering
  \tikzset{every picture/.style={scale=1.2}}
  \begin{tikzpicture}[
    taugraph,
    BigNode/.append style={
      minimum size=17mm,
    },
    ]

    % ----------------
    \node at (0,0) (n0256) [BigNode] {0,2,5,6};
    % 0 1 2 3 4 5 6 7
    \CompleteFourWheel{n0256}{R,B,R,B,B,R,R,B}

    % --------------------------------------------
    \node at (180-20:4) (n1256) [BigNode] {1,2,5,6};
    % 0 1 2 3 4 5 6 7
    \CompleteFourWheel{n1256}{B,R,R,B,B,R,R,B}

    \draw[B] (n1256) to[bend left=8] node [BigLabel] {$(1,0)$} (n0256);
    \draw[R] (n1256) to[bend left=-8] node [BigLabel] {$(0,1)$} (n0256);

    % --------------------------------------------
    \node at (180+20:4) (n0236) [BigNode] {0,2,3,6};
    % 0 1 2 3 4 5 6 7
    \CompleteFourWheel{n0236}{R,B,R,R,B,B,R,B}

    \draw[B] (n0236) to node [BigLabel] {$(3,5)$} (n0256);

    % --------------------------------------------
    \node at (0-28:4) (n0245) [BigNode] {0,2,4,5};
    % 0 1 2 3 4 5 6 7
    \CompleteFourWheel{n0245}{R,B,R,B,R,R,B,B}

    \draw[R] (n0256) to node [BigLabel] {$(4,6)$} (n0245);

    % --------------------------------------------
    \node at (0:4) (n0356) [BigNode] {0,3,5,6};
    % 0 1 2 3 4 5 6 7
    \CompleteFourWheel{n0356}{R,B,B,R,B,R,R,B}

    \draw[B] (n0256) to node [BigLabel] {$(2,3)$} (n0356);

    % --------------------------------------------
    \node at (0+28:4) (n0257) [BigNode] {0,2,5,7};
    % 0 1 2 3 4 5 6 7
    \CompleteFourWheel{n0257}{R,B,R,B,B,R,B,R}

    \draw[R] (n0256) to node [BigLabel] {$(7,6)$} (n0257);

  \end{tikzpicture}
  \caption{Excerpt of the undirected unique exchange graph $\tau_3(W_5)$.}\label{fig:example undirected tau}
\end{figure}
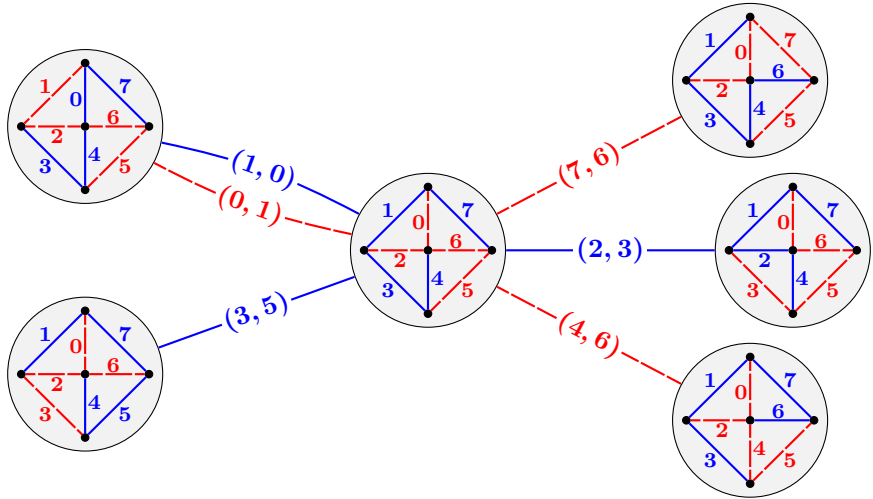

Figure~\ref{fig:example undirected tau} shows an excerpt of the undirected exchange graph $\tau_3(W_5)$, which we will further study in section~\ref{sec:mend-brokenUEs}. The figure shows all unique edge exchanges for the center vertex. Compare the unique edges exchanges on the right with those shown in the previous figure~\ref{fig:example directed tau}: the twin arcs are grouped as an edge.

\begin{remark}[guide to reading unique exchange graph figures]
  In each vertex of the exchange graph the pair of disjoint trees $(S,T)$ is represented by the small colored bispanning graphs with numbered edges.

  Undirected edges between vertices are labeled with $(e,f)$, which represents unique exchanges in both directions. When following the edge left to right (in reading direction of the label) the edge symbolizes an edge exchange $(e,f)$, and when following the edge right to left (opposite to reading direction of the label) the edge symbolizes the reverse edge exchange $(f,e)$.

  An arc is colored \textcolor{blue}{blue} if it is a \textcolor{blue}{$S$ edge exchange}, and \textcolor{red}{red} if it is a \textcolor{red}{$T$ edge exchange}. This color is the one Alice must change the first edge to, for the particular transition to occur. In larger exchange graphs parallel blue and red edges are combined using interleaved blue/red dashes, and labeled with $\{e,f\}$ as both $S$ and $T$ edge exchanges are possible.
\end{remark}

% ------------------------------------------------------------------------------

\subsection{Basic Theorems and Observations on Exchange Graphs}

In this section we collect some straight-forward insights into unique symmetric edges exchange, and prove basic theorems about them. Figures \ref{fig:tau-x3}, \ref{fig:tau-x4}, and \ref{fig:tau-k4} show the complete unique exchange graphs $\tau_3$ for all bispanning graphs with three or four vertices. Remarkably, the exchange graphs of all these small graphs are isomorphic except for the one of $K_4$. In section~\ref{sec:reduce vdeg3}, the unique exchange graph of $W_5$ is shown in figure~\ref{fig:tau-w5} (page~\pageref{fig:tau-w5}). Table~\ref{tab:properties small bispanning} shows some basic properties about small bispanning graphs with up to six vertices, and their exchange graph.

\begin{figure}
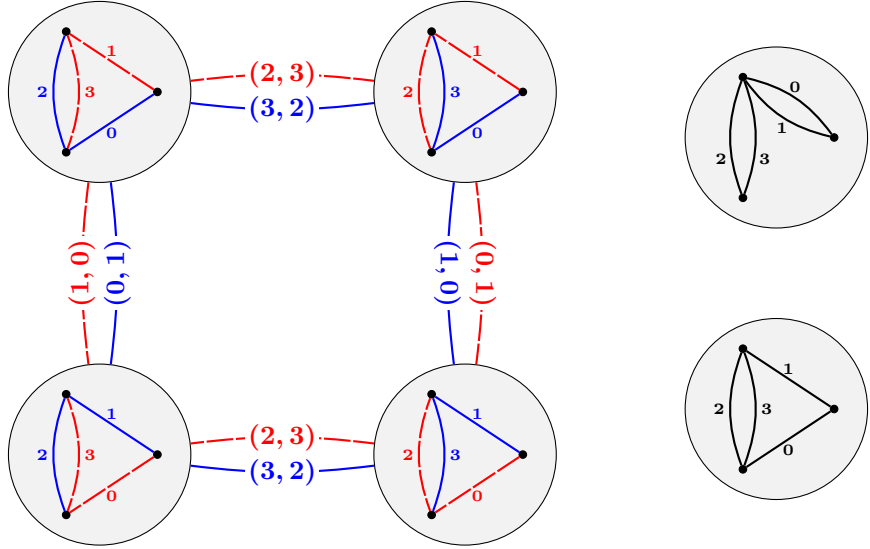
\centering

  \caption{The unique exchange graph $\tau_3(G)$ for $G \in \{ B_{3,1}, B_{3,2} \}$.}\label{fig:tau-x3}
\end{figure}

The first observation is that each leaf edge of a spanning tree yields a unique exchange.

\begin{lemma}[leafUEs: unique exchanges due to leaf edges]\label{lem:leafUE}
  If $(S,T) \in V_{\tau(G)}$ is a pair of disjoint spanning trees of a bispanning graph $G = (V,E,\delta)$, and $v \in V$ is a leaf in the tree-graph $G[S]$ incident only to $e \in S$, then there exists a unique $S$ edge exchange $(e,f) \in S \times T$ where $\{ e, f \} = D(S,e) \cap C(T,e)$.

  We call such a unique exchange a \emph[leafUE]{leafUE}.

  Analogously, if $v \in V$ is a leaf in $G[T]$ incident only to $e \in T$, there exists a unique $T$ edge exchange $(e,f) \in T \times S$ where $\{ e, f \} = D(T,e) \cap C(S,e)$.
\end{lemma}
\begin{proof}
  As $v$ is a leaf in $G[S]$, it is adjacent to exactly one edge in $S$. Therefore, $D(S,e) = \{ e \in E \mid v \in \delta(e) \}$, since $v$ is isolated in $G[S] - e$. As the cycle $C(T,e)$ visits $v$ only once, we have $| D(S,e) \cap C(T,e) | = 2$, and thus shown that $e$ yields a unique $S$ edge exchange. Analogous arguments show that $v$ yields a unique $T$ edge exchange if $v$ is a leaf in $G[T]$.
\end{proof}

\begin{theorem}[minimum degree of $\tau_3(G)$]
  For every bispanning graph $G = (V,E,\delta)$ with $|V| \geq 3$, every vertex in $\vec{\tau}_3(G)$ and $\tau_3(G)$ is incident to at least four arcs or edges, in symbols $\deg_{\vec{\tau}_3(G)}(v) \geq 4$ and $\deg_{\tau_3(G)}(v) \geq 4$ for all $v$.
\end{theorem}
\begin{proof}
  We only need to prove the case for $\vec{\tau}_3(G)$. Due to lemma~\ref{lem:two leaves}, the two disjoint spanning trees $S \dotcup T = E$ in $G$ both have at least two leaves. Each of the four leaves yields a different unique exchange, due to lemma~\ref{lem:leafUE}.
\end{proof}

Recently, McGuinness proved a similar theorem for all regular matroids~\cite{mcguinness2014base}: for every base pair $B$ and $B'$ of a regular matroid $M$ there exists at least one element $e \in B$ such that there is a unique element $f \in B'$ for which $B - e + f$ and $B' - f + e$ are bases of $M$. The proof is very involving and contains no hints about whether the corresponding exchange graph is connected.

Beyond the observation that leaves always yield unique exchanges, other simple subgraphs, as sketched in figure~\ref{fig:small-cycle-cut}, also yield unique exchanges straight-forwardly.

\begin{remark}[unique exchanges of parallel edges and edges at degree two vertices]
  Parallel edges and edges at degree two vertices are obviously swappable only with their partners, as their cycle and cut have size \emph{exactly two}, respectively. The bispanning graphs in figure \ref{fig:tau-x4} show many examples of these subgraphs.

  Due to these \emph{fixed exchange partners}, which are \emph{independent} of the remaining graph, one can swap the pair at each vertex of $\tau_3(G)$.  Hence, $\tau_3(G)$ is composed of two isomorphic copies of the $\tau_3$ of the remaining graph, and arcs between the two for the pair at each vertex.  We do not state an explicit theorem about this, since it is a special case of the theorem~\ref{thm:composite tau decompose}.
\end{remark}
As mentioned in the previous remark, parallel edges and edges at degree two vertices have a cycle and cut of size two. While every cycle of size two is always a pair of parallel edges, cuts of size two can occur in other situations as well (see figure~\ref{fig:vconn-econn}, page~\pageref{fig:vconn-econn}).

The success of the investigation of size two raises the question which subgraphs have cycles and cuts of size three. Cycles of size three are exactly \emph{triangles} in the bispanning graph, and vertices of degree three always yield a cut of size three; though, as before, cuts of size three can occur otherwise as well.

Due to the pidgin hole principle, of the three edges in a triangle, one edge $e_1$ must be colored differently from the other two $e_2$ and $e_3$ (see figure~\ref{fig:small-cycle-cut three}). Changing the color of $e_1$ forms a cycle with the others, of which exactly one is in the cut opened by $e_1$ in the other color. Every triangle yields a unique exchange by coloring $e_1$ as the cycle has size three. However, which of the two edges $e_2$ and $e_3$ is the exchange partner depends on the remaining graph, namely the cut opened by swapping $e_1$. Contrarily to cycles and cuts of size two, the exchange partner in the triangle need not always yield a unique exchange of the other type.

Similarly, of the three edges at a vertex of degree three, one edge $e_1$ must be colored differently from the other two $e_2$ and $e_3$ (see figure~\ref{fig:small-cycle-cut three} again). Changing the color of $e_1$ forms a cycle with one of the others, which yields a unique exchange to that edge. As with triangles, which of the two edges is the exchange partner depends on the remaining graph, namely which of the edges $e_2$ and $e_3$ is connected to $e_1$ in the corresponding tree. As with triangles, the exchange partner need not always yield a reverse unique exchange of the other type.

Having regarded cycles of length two and three, the obvious question is ``What about cycle of length four?'' We call these \emph{squares}, and since they have four edges, the pidgin hole principle only helps little. If we assume that three edges are the same color, the one differently colored edge $e_1$ may yield a unique exchange (see figure~\ref{fig:small-cycle-cut four} again). There are only three non-isomorphic ways $e_1$ can be connected to the other two vertices of the square in the same tree, labeled as cases $A$, $B$, and $C$ in the figure.  For squares, only the cases $A$ and $B$ yield a unique edge exchange, while $C$ is the \emph{prototypical subgraph} in which $e_1$ does not yield a unique exchange.

The musings about triangles and squares suggest that graphs with large girth (length of smallest cycle), like triangle-free or square-free graphs, may be a counter example for connectivity of $\tau_3$. One could even assume that no triangle-free or square-free bispanning graphs exist. This, however, is not the case: the smallest triangle-free bispanning graph has seven vertices and can be see in the center of figure~\ref{fig:bispanning7}, page~\pageref{fig:bispanning7}, while the smallest square-free bispanning graphs have 18 vertices (see figure~\ref{fig:bispanning-squarefree}, page~\pageref{fig:bispanning-squarefree}, for one of the eight), and the $\tau_3$ of both are connected.

\begin{table}
  \def\tabcolsep{4.2pt}
  \def\skipamount{7.3pt}
  \begin{minipage}[t]{0.5\linewidth}\centering\vspace{0pt}%
\begin{tabular}{lcccc}
      \toprule
      graph $G$        & deg seq     & $|V_{\tau(G)}|$ & $|E_{\tau_3(G)}|$ & degr  \\\midrule
      $B_1 = K_1$      & 0           & 1               & 0                 & 0     \\[\skipamount]
      $B_2$            & 2,2         & 2               & 2                 & 2     \\[\skipamount]
      $B_{3,1}$        & 4,2,2       & 4               & 8                 & 4     \\
      $B_{3,2}$        & 3,3,2       & 4               & 8                 & 4     \\[\skipamount]
      $B_{4,1} = K_4$  & 3,3,3,3     & 12              & 24                & 4     \\
      $B_{4,2}$        & 4,3,3,2     & 8               & 24                & 6     \\
      $B_{4,3}$        & 4,4,2,2     & 8               & 24                & 6     \\
      $B_{4,4}$        & 4,3,3,2     & 8               & 24                & 6     \\
      $B_{4,5}$        & 3,3,3,3     & 8               & 24                & 6     \\
      $B_{4,6}$        & 5,3,2,2     & 8               & 24                & 6     \\
      $B_{4,7}$        & 4,3,3,2     & 8               & 24                & 6     \\
      $B_{4,8}$        & 4,4,2,2     & 8               & 24                & 6     \\
      $B_{4,9}$        & 6,2,2,2     & 8               & 24                & 6     \\\bottomrule
    \end{tabular}%
  \end{minipage}%
  \begin{minipage}[t]{0.5\linewidth}\centering\vspace{0pt}%
    \begin{tabular}{lcccc}
      \toprule
      graph $G$        & deg seq     & $|V_{\tau(G)}|$ & $|E_{\tau_3(G)}|$ & degr  \\ \midrule
      $B_{5,1} = W_5$  & 4,3,3,3,3   & 28              & 80                & 4--8  \\
      $B_{5,2}$        & 4,4,3,3,2   & 24              & 72                & 6     \\[\skipamount]
      $B_{6,1}$        & 5,4,3,3,3,2 & 48              & 192               & 8     \\
      $B_{6,2}$        & 5,5,3,3,2,2 & 48              & 192               & 8     \\
      $B_{6,3}$        & 4,4,3,3,3,3 & 72              & 224               & 4--7  \\
      $B_{6,4}$        & 4,4,4,4,2,2 & 48              & 192               & 8     \\
      $B_{6,5}$        & 5,4,4,3,2,2 & 48              & 192               & 8     \\
      $B_{6,6}$        & 4,4,4,3,3,2 & 48              & 192               & 8     \\
      $B_{6,7}$        & 4,4,4,3,3,2 & 56              & 216               & 6--10 \\
      $B_{6,8}$        & 5,4,3,3,3,2 & 56              & 216               & 6--10 \\
      $B_{6,9}$        & 4,4,4,3,3,2 & 56              & 216               & 6--10 \\
      $B_{6,10}$       & 4,4,3,3,3,3 & 68              & 238               & 4--10 \\
      $B_{6,11} = W_6$ & 5,3,3,3,3,3 & 60              & 240               & 6--10 \\
      $B_{6,12}$       & 4,4,3,3,3,3 & 72              & 240               & 6--8  \\\bottomrule
    \end{tabular}%
  \end{minipage}
  \caption[Properties of small bispanning graphs and their exchange graphs.]{Properties of small bispanning graphs and their exchange graphs, where the columns ``deg seq'' lists the vertex degree sequence of $G$, and ``degr'' the minimum and maximum degree of vertices in $\tau_3(G)$.}\label{tab:properties small bispanning}
\end{table}

\begin{figure}
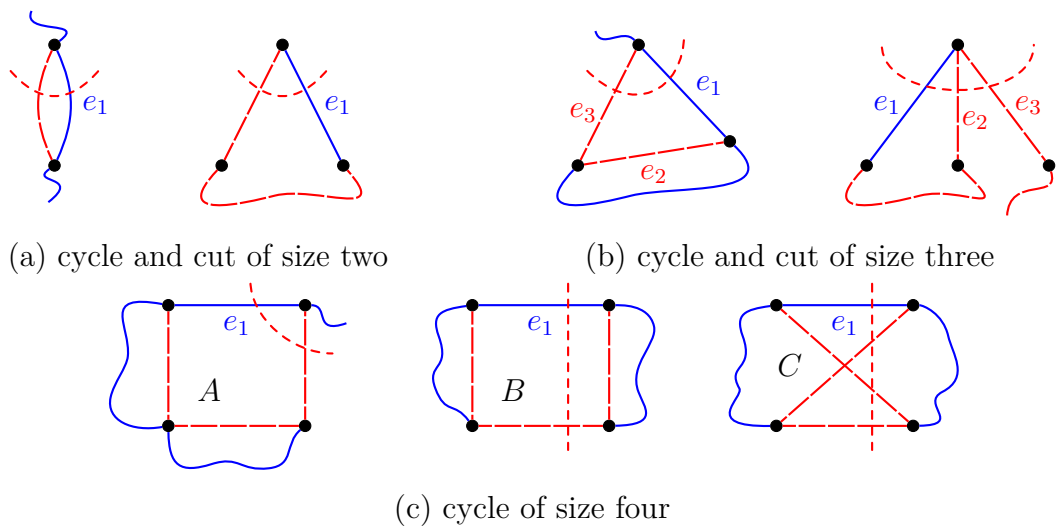
\centering
  \tikzset{every picture/.style={
      scale=2, graphfinal,
      yscale=0.8,
      cut/.style={dashed, line cap=round}
    }}

  \hfill%
  \subcaptionbox{cycle and cut of size two\label{fig:small-cycle-cut two}}{%

  }
  \hfill%
  \subcaptionbox{cycle and cut of size three\label{fig:small-cycle-cut three}}{%
    %
  }
  \hfill\null%

  \hfill%
  \subcaptionbox{cycle of size four\label{fig:small-cycle-cut four}}{%
    %
  }
  \hfill\null%

  \caption[Subgraphs with small cycle and cut sizes.]{Subgraphs with small cycle and cut sizes.The dashed red lines show $D(e_1,T)$ in all graphs.}\label{fig:small-cycle-cut}
\end{figure}

\begin{figure}[p]
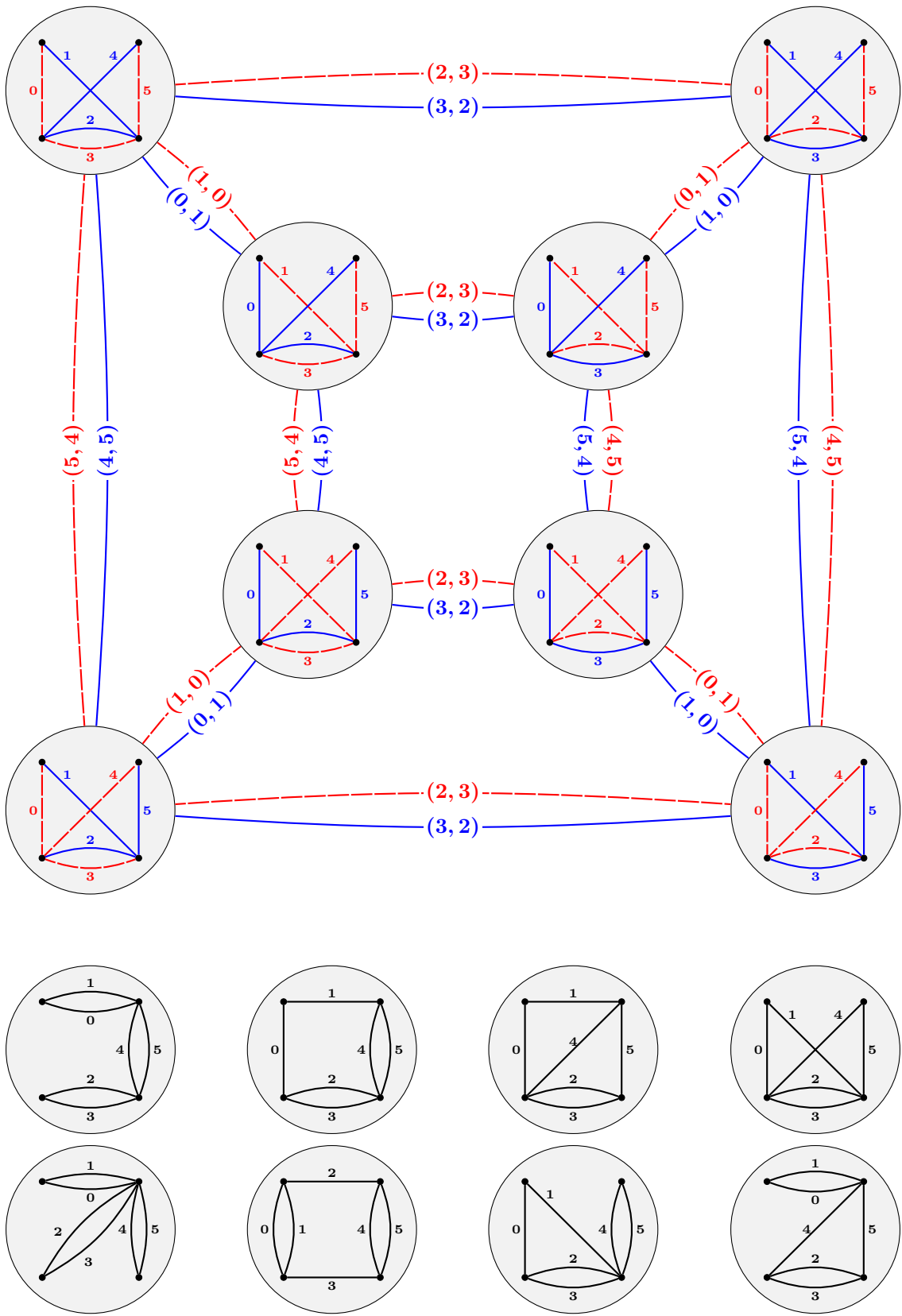
\centering
  \begin{leftfullpage}
  %
  \caption{The unique exchange graph $\tau_3(G)$ for $G = (V,E,\delta)$ with $|V| = 4$ and $G \neq K_4$.}\label{fig:tau-x4}
  \end{leftfullpage}
\end{figure}

\begin{figure}[p]
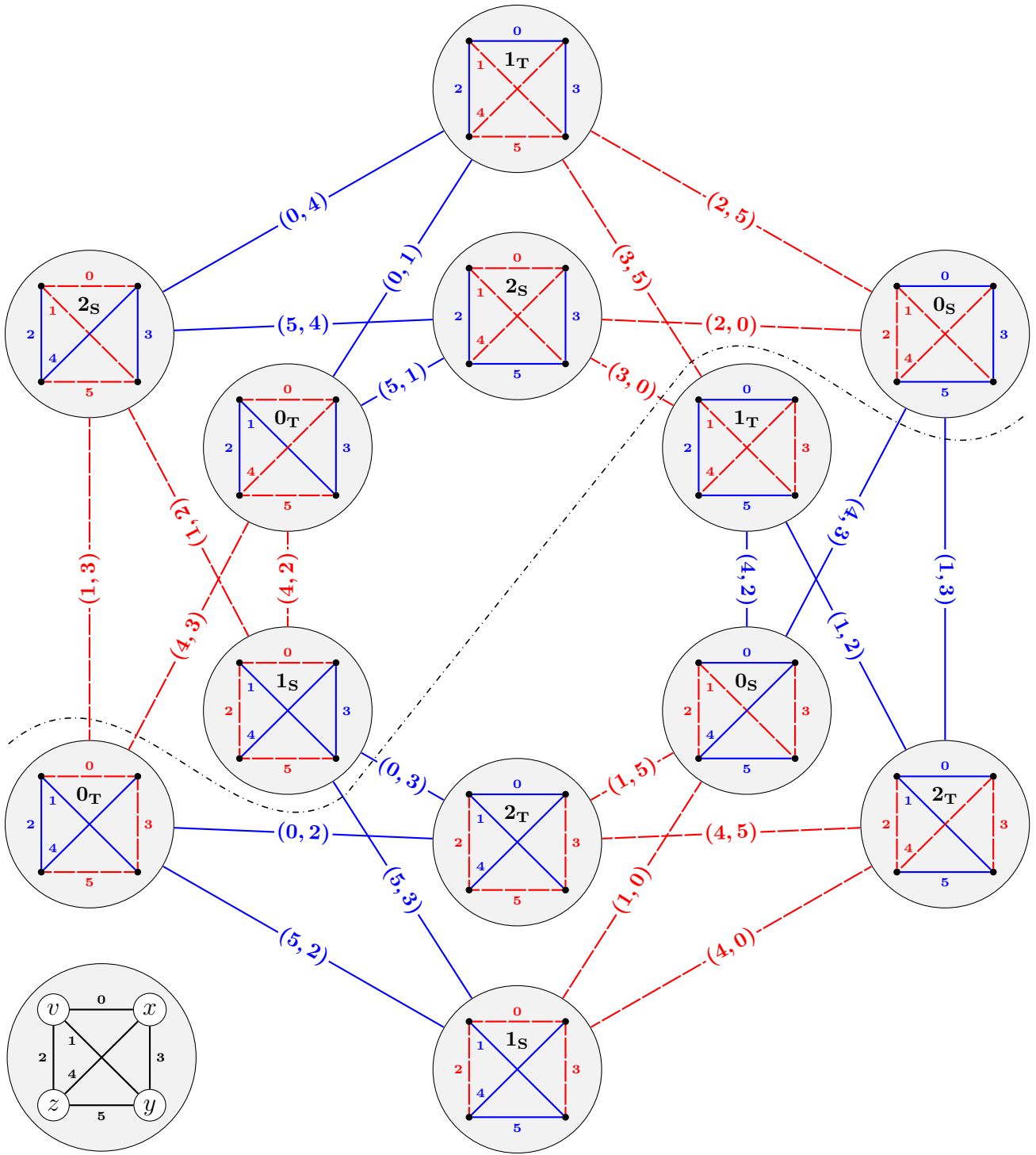
\centering
  \begin{fullpage}
  %
  \caption{The unique exchange graph $\tau_3(K_4)$.}\label{fig:tau-k4}
  \end{fullpage}
\end{figure}

% ------------------------------------------------------------------------------

\subsection{Conjectures by White on Base Exchanges}

While we developed the various exchange graphs $\tau_i(G)$ in the previous subsections specifically for bispanning graphs, one can immediately transfer their definitions to block matroids. This transfer is possible because the generalization of symmetric edge exchanges (theorem~\ref{thm:edge exchange}) is given by the strong symmetric base exchange lemma~\ref{lem:strong base exchange}.

In 1980, White put forth a range of conjectures concerning unique and other types of base exchanges on matroids~\cite{white1980unique}. He describes six different classes of base exchanges, and discusses in which classes of matroids any pair of base sequences is related using only the specific exchange type. The paper contains proofs of some of the straight-forward classifications; but also far reaching conjectures, of which some were proven under special premises, but none were finally disproved. In this section we review parts of the paper, reformat the definitions and theorem, reference recent results, and draw links to our definitions.

White lets $A$ and $B$ be two bases of a matroid $M$. For any element $x \in A$, he denotes $E(x; A,B) := \{ y \in B \mid A - x + y \text{ and } B - y + x \text{ are bases of M} \}$. Due to lemma~\ref{lem:strong base exchange}, $|E(x; A,B)| \geq 1$ for all possible parameters. Furthermore, the same property extends to \emph{subsets} of elements: if $X \subseteq A$, then $E(X; A,B) := \{ Y \subseteq B \mid (A \setminus X) \cup Y \text{ and } (B \setminus Y) \cup X \text{ are bases of M} \}$ is non-empty for all possible parameters~\cite{brylawski1973some,greene1973multiple,woodall1974exchange}. Obviously, the case $|E(x; A,B)| = 1$ in graphic matroids corresponds directly to our unique edge exchanges (definition~\ref{def:unique edge exchange}).

After $E(x; A,B)$, he considers sequences $B = \arr{B_1,B_2,\ldots,B_m}$ of $m$ bases of a matroid $M$. He then defines, if $E(x; B_i,B_j) = \{ y \}$ for some $1 \leq i < j \leq m$ and $x \in B_i$, that the base sequence
\begin{equation}\label{eq:white-ue}
B' = \arr{B_1,\ldots,B_{i-1},B_i - x + y,B_{i+1},\ldots,B_{j-1},B_j - y + x,B_{j+1},\ldots,B_m}
\end{equation}
is obtained from $B$ by a \emph{unique (single-element) exchange}. This can be seen as a relation of $B$ and $B'$, and he writes $\simeq_1$ for the transitive closure of this relation. In a second step, he further allows the relation to also \emph{permute the order} of the bases in the sequence, and denotes this transitive closure with $\simeq_2$.

As a third variant, he generalizes from single-element unique exchanges to subsets of items: if $E(X; B_i,B_j) = \{ Y \}$ for some $1 \leq i < j \leq m$ and $X \subseteq B_i$, then the base sequence
$$B' = \arr{B_1,\ldots,B_{i-1},(B_i \setminus X) \cup Y,B_{i+1},\ldots,B_{j-1},(B_j \setminus Y) \cup X,B_{j+1},\ldots,B_m}$$
is obtained from $B$ by a \emph{unique subset exchange}. Again, sequences that can be obtained from each other are in relation, and the transitive closure is denoted as $\simeq_3$.

Going beyond unique exchanges, he defines relations $\sim_1$, $\sim_2$, and $\sim_3$, which each drop the constraint $|E(\cdot; B_i,B_j)| = 1$. In more detail: if $B'$ can be obtained from $B$ as stated above, due to $|E(\cdot; B_i,B_j)| \geq 1$, then this operation is called a \emph{symmetric (single-element) exchange} or \emph{symmetric subset exchange}. The corresponding transitive closures are $\sim_1$ of symmetric exchanges, $\sim_2$ of symmetric exchange and permutations, and $\sim_3$ of symmetric subset exchanges.

Before we continue reviewing White's results on these exchange types, we have to consider how they are related to the unique exchanges we defined in the previous subsections. For bispanning graphs or block matroids we are interested in base sequences $(B_1,B_2)$ with $m=2$ such that $B_1 \dotcup B_2 = E$. These are a very special subset of White's definition, but they are included. The precondition to equation~\eqref{eq:white-ue} requires $i < j$, hence, White's definition of \emph{unique (single-element) exchanges} and of $\simeq_1$ actually corresponds only to the left-unique exchanges of $\tau_4(G)$ (definition~\ref{def:directed tau4}), where $e$ is required to be in $S$. The laxer relation $\simeq_2$ allows permutation of the two bases, but this too does not directly correspond to $\tau_3(G)$, where $e$ can be chosen from $S$ or $T$. By allowing permutation, Alice and Bob could just swap trees in one step, instead of performing a sequence of single-element edge exchanges. However, if $\tau_3(G)$ turns out to be connected for all $G$, then this \emph{does} imply that all complementary base pairs are related by White's $\simeq_2$, as permutation of the bases can then be emulated using single-element exchanges. In summary, neither $\simeq_1$ nor $\simeq_2$ is directly equivalent to the unique exchange graph $\tau_3(G)$, but lies in between. It is close to $\simeq_1$, but allows the unique exchange to be initiated by either base, and equal to $\simeq_2$ if $\tau_3(G)$ is connected.

But White is interested in more general matroids than block matroids. And hence, he asks if \emph{all} base pairs are related, while we are only interested in complementary pairs $(B_1,B_2)$, and whether they are related to $(B_2,B_1)$. But our restriction to complementary tree pairs is rather artificial: if $B_1 \cap B_2 \neq \emptyset$, then one can contract and delete pairs from the intersection to gain a full exchange problem on a smaller bispanning graph. Until the full exchange problem is solved, we suggest putting the partial one to the side, since it is unclear how the contraction and deletion may affect a swap sequence. Furthermore, the extreme case $B_1 \cap B_2 = E \setminus \{e,f\}$ asks for a unique swap sequence for any pair of edges.

White proves that $\simeq_1$, $\simeq_2$, $\sim_1$, $\sim_2$, and $\sim_3$ are equivalence relations. Reflexivity and transitivity are clear, only symmetry is difficult for $\simeq_1$ and $\simeq_2$; for $\sim_\alpha$ it is also trivial. The proof of symmetry of $\simeq_1$ and $\simeq_2$ corresponds to the matroid generalization of theorem~\ref{thm:reversibility unique exchange}. Whether $\simeq_3$ is symmetric is unknown.

To classify matroids using the exchange types, White calls two base sequences $B$ and $B'$ \emph{compatible} if the multiplicity of all elements in the contained bases are equal, which is an obvious necessary condition for the sequences to be related by unique and other exchange types. He defines $\operatorname{UE}(1)$, $\operatorname{UE}(2)$, $\operatorname{UE}(3)$, $\operatorname{TE}(1)$, $\operatorname{TE}(2)$, and $\operatorname{TE}(3)$ to be the class of all matroids, within which \emph{every} base sequence with $m \geq 2$ is related to \emph{every compatible} base sequence by $\simeq_1$, $\simeq_2$, $\simeq_3$, $\sim_1$, $\sim_2$, and $\sim_3$, respectively. He declares matroids in $\operatorname{UE}(\alpha)$ to satisfy the \emph{unique exchange property} of type $\simeq_\alpha$, and in $\operatorname{TE}(\alpha)$ to satisfy the \emph{transitive exchange property} of type $\sim_\alpha$.  Additionally, White defines $\operatorname{UE}(\alpha)'$ to be the class of matroids such that every base sequence with $m = 2$ is related to every compatible sequence by $\simeq_\alpha$. For the classes, the inclusions $\operatorname{UE}(\alpha) \subseteq \operatorname{UE}(\alpha)'$, and $\operatorname{UE}(\alpha) \subseteq \operatorname{TE}(\alpha)$ hold for $\alpha = 1,2,3$, and $\operatorname{UE}(\alpha) \subseteq \operatorname{UE}(\alpha+1)$, $\operatorname{UE}(\alpha)' \subseteq \operatorname{UE}(\alpha + 1)'$, and $\operatorname{TE}(\alpha) \subseteq \operatorname{UE}(\alpha+1)$ hold for $\alpha = 1,2$.

\begin{remark}[theorems by White in \cite{white1980unique}]
White proves the following theorems about the classes $\operatorname{UE}(\alpha)$ and $\operatorname{TE}(\alpha)$:
\begin{enumerate}
\item All classes $\operatorname{UE}(\alpha)$, $\operatorname{UE}(\alpha)'$, and $\operatorname{TE}(\alpha)$ for $\alpha = 1,2,3$, are closed under taking minors (hereditary), self-dual, and closed under direct sum.~\hfill\cite[prop.\,5]{white1980unique}\label{thm:white1}

\item $\operatorname{UE}(1) = \operatorname{UE}(1)'$ is the class of cycle matroids of series-parallel networks.~\hfill\cite[thm.\,6]{white1980unique}\label{thm:white2}

\item $\operatorname{UE}(3)'$ is a subset of the set of binary matroids, but not equal.~\hfill\cite[prop.\,7]{white1980unique}\label{thm:white3}

\end{enumerate}
\end{remark}

We can now consider how these results carry over to our scenarios.  For better comparison, we denote with $\operatorname{BUE}_3$ the set of all block matroids, including cycle matroids of bispanning graphs, for which $\tau_3(M)$ is connected, and with $\operatorname{BUE}_4$ those, for which $\tau_4(M)$ is connected.

However, in $\tau_\alpha(G)$, vertices in the graph are required to be \emph{complementary} base pairs $(B_1,B_2)$, and $\operatorname{BUE}_\alpha$ only contains block matroids. This requirement voids much of White's theorems for our exchange game. Of theorem~\ref{thm:white1}, only self-duality and closure under direct sums remain valid. The important property of being closed under minors is not valid for $\operatorname{BUE}_\alpha$, because deletion or contraction of a base in a block matroid does not always yield a block matroid. Likewise, in White's proof of theorem~\ref{thm:white1} the base pairs do not remain complementary under projection to/from a minor.

The proof of theorem~\ref{thm:white2} is completely dependent on the property of $\operatorname{UE}(1)$ and $\operatorname{UE}(1)'$ being closed under minors, and hence does not carry to complementary base pairs on block matroids.  One could now think that $\operatorname{BUE}_4 \subseteq \operatorname{UE}(1)'$, but this is false. A simple set of counterexamples are $\tau_4(W_n)$, which are connected for all $n \geq 4$~\cite[thm.\,23]{andres2014base}, but no $W_n$ is a series-parallel network.  The error in this hypothesis is that $\operatorname{UE}(1)'$ contains only matroids for which \emph{all} base pairs are related by $\simeq_1$, while in $\operatorname{BUE}_4$ \emph{only complementary} base pairs have to be related. Hence, $\operatorname{BUE}_4$ can contain $\matheu{M}(W_n)$, while $\operatorname{UE}(1)'$ does not. On the other hand, $\operatorname{UE}(1)' \nsubseteq \operatorname{BUE}_4$, since $\operatorname{UE}(1)'$ contains non-block matroids.

Theorem~\ref{thm:white3} seems to have little relevance, as unique subset exchanges are probably very difficult to transform into unique single-element exchanges.

\begin{remark}[conjectures by White in \cite{white1980unique}]
White conjectures the following about the classes $\operatorname{UE}(\alpha)$ and $\operatorname{TE}(\alpha)$:
\begin{enumerate}
\item All regular matroids are in $\operatorname{UE}(2)$.~\hfill\cite[conj.\,8]{white1980unique}\label{conj:white1}

\item For every two bases $B_1$ and $B_2$ of a regular matroid, there exists $x \in B_1$ such that $|E(x; B_1,B_2)| = 1$.~\hfill\cite[rem.\,10]{white1980unique}\label{conj:white2}

\item $\operatorname{TE}(1) = \operatorname{TE}(2) = \operatorname{TE}(3)$ is the set of all matroids.~\hfill\cite[conj.\,12]{white1980unique}\label{conj:white3}

\item The relation $\sim_3$ is equal to $\sim_1$, hence every symmetric subset exchange can be decomposed into single-element exchanges.~\hfill\cite[conj.\,13]{white1980unique}\label{conj:white4}

\end{enumerate}
\end{remark}
The exchange game played by Alice and Bob is a special variant of White's conjecture~\ref{conj:white1}. The main difference being that only complementary base pairs need to be swapped using unique exchanges, while in $\operatorname{UE}(2)$ any base pairs need to be in relation. However, the questions of whether all regular block matroids are in $\operatorname{BUE}_3$, and whether all regular matroids are in $\operatorname{UE}(2)$, probably have the same challenges at their core. As it is not even known whether all graphic block matroids are in $\operatorname{BUE}_3$, we consider the broad structure of $\tau_3(G)$ in this thesis.

Conjecture~\ref{conj:white2} is a side-remark of White and has been proven by McGuinness~\cite{mcguinness2014base} using Seymour's decomposition theorem for regular matroids~\cite{seymour1980decomposition}.

The most broad conjecture~\ref{conj:white3} turned out to be the most intensely studied one, since it has relationships to many other branches of mathematics like algebraic geometry.  Blasiak showed in 2008 that the conjecture is true for all graphic matroids~\cite{blasiak2008toric}. Bonin extended this in 2013 to all sparse paving matroids~\cite{bonin2013basis}. Thereafter, Laso{\'n} and Micha{\l}ek proved the conjecture for strongly base orderable matroids~\cite{lason2014toric}, and up to saturation (see \cite[sect.~4.3]{lason2015coloring} for details).

There has also been a lot of work on conjecture~\ref{conj:white4}, some which overlaps with attempts to prove conjecture~\ref{conj:white3}. The underlying question, of whether a symmetric subset exchange can be decomposed into a serial single-element exchange is attributed to Gabow~\cite{gabow1976decomposing}, and we will discuss it in the next section, in the context of cyclic base orderings.

% ------------------------------------------------------------------------------

\subsection{Cyclic Base Orderings}

Let us go back to Alice and Bob's exchange game: translated into the nomenclature from this chapter, the players are tasked to find a path through the exchange graph $\tau_3(G)$ from one vertex $(S,T)$ to its complement $(T,S)$. For each edge in the path, one pair of elements $(e,f)$ is exchanged between the disjoint spanning trees. For Alice to win with certainty, she has to find a path wherein Bob has no choice of $f$, and in $\tau_3(G)$, all edges imply this constraint. But this setting is only a special case of a wider range of possible exchange games.

The same exchange paths can be sought for in the context of base pairs of matroids, including pairs that are not necessarily complementary. Many authors have raised different questions and gained various results in this area. Gabow first considered whether symmetric \emph{subset} base exchanges can be decomposed~\cite{gabow1976decomposing}. He succeeded in decomposing them into smaller independent subset base exchanges, and poses the question whether they can be decomposed into single-element exchanges. While the question has been answered for many large classes of matroids, to date, the answer is still unknown in the general case.

A much more intriguing conjecture by Gabow, even beyond the decomposition of a subset exchange into single-element exchanges, is to order the ground set of elements such that all $r$ cyclically consecutive elements form a base. These base orderings have an immediate relationship to the path through $\tau_2(G)$ or $\tau_3(G)$ which Alice is seeking. But little is known about paths with the unique exchange property, hence we first neglect the unique exchange property:
\begin{conjecture}[cyclic base ordering \R{\cite{gabow1976decomposing}}]
  If $B_1$ and $B_2$ are disjoint bases of a matroid $M$ with rank $r := r(M)$ and ground set $E$, then an ordering $\arr{ b_1,\ldots, b_r, b_{r+1}, \ldots, b_{2r} }$ with $B_1 = \{ b_1,\ldots,b_r \}$ and $B_2 = \{ b_{r+1},\ldots,b_{2r} \}$ exists such that every set of $r$-cyclically consecutive elements is a base of $M$.

  An alternative, slightly stronger conjecture, is that all elements of $E$ can be ordered such that all $r$-cyclically consecutive elements are a base of $M$.
\end{conjecture}
Farber, Richter, and Shank in 1985~\cite{farber1985edge} proved for graphic matroids that subset exchanges can be composed into single-element exchange by showing that $\tau_2(G)$ is connected.  Kajitani, Ueno, and Miyano~\cite{kajitani1988ordering}, Edmonds (recorded for posterity by Wiedemann)~\cite{wiedemann2006cyclic}, and later Cordovil and Moreira~\cite{cordovil1993bases} showed that cyclic base orderings can be found for all graphical matroids, which implies that subset exchanges can be decomposed. Farber proved that subset base exchanges can be decomposed for transversal matroids in 1989~\cite{farber1989basis}. Van den Heuvel and Thomass\'e proved the stronger cyclic base ordering conjecture for the case when $|E|$ and $r(M)$ are coprime~\cite{heuvel2012cyclic}.  Recently, Bonin showed the conjecture for all sparse paving matroids~\cite{bonin2013basis}, Kotlar and Ziv proved it for any matroid of rank $4$~\cite{kotlar2013serial}. Later, Kotlar extended this to any matroid of rank $5$~\cite{kotlar2013circuits}, and furthermore proves that at least three consecutive symmetric single-element base exchanges exist for any pair of disjoint bases.

\def\carr#1#2{\arr{#1\!\mid\!#2}}
For bispanning graphs, Baumgart gives a constructive method to calculate a cyclic base (edge) ordering using the inductive construction consisting of operations in theorem~\ref{thm:inductive operations} (page~\pageref{thm:inductive operations}). Since this method is reused in section~\ref{sec:reduce vdeg3}, we review the proof in detail. For bispanning graphs, we insert a ``$\mid$'' in the notation of a base ordering to better separate edges of a disjoint pair of spanning trees.
\begin{theorem}[cyclic base ordering of bispanning graphs\R{\cite[thm.\,5.2]{baumgart2009ranking}}]\label{thm:cbo-bispanning}
For every bispanning graph $G = (V,E,\delta)$ and every pair of disjoint spanning trees $E = S \dotcup T$, there exists a cyclic base ordering $\carr{ s_1, \ldots, s_m }{ t_1, \ldots, t_m }$ where $m = |S| = |T|$, $S = \{ s_1,\ldots,s_m \}$ and $T = \{ t_1,\ldots,t_m \}$.
\end{theorem}
Before we prove the theorem by constructing such a base ordering, let us consider the relationship to paths in $\tau_2(G)$. A cyclic base ordering $\carr{ s_1, \ldots, s_m }{ t_1, \ldots, t_m }$ directly corresponds to the path $\arr{ (s_1,t_1), (s_2,t_2), \ldots, (s_m,t_m) }$ through the graph from $(S,T)$ to $(T,S)$. We call the second representation the corresponding \emph[edge swap sequence]{edge swap sequence}.

\def\myalt#1{\langle #1 \rangle}
\begin{proof}[of~\ref{thm:cbo-bispanning}]
  The proof uses induction over the number of vertices $n = |V|$. The trivial bispanning graph $B_1$ in figure~\ref{fig:small-bispanning} (page~\pageref{fig:small-bispanning}) with $|V| = 1$ has no edges, such that $\carr{\;}{\;}$ is a valid cyclic base ordering. The bispanning graph $B_2$ in figure~\ref{fig:small-bispanning} with $|V| = 2$ has two parallel edges $e_1 \in S$ and $e_2 \in T$, such that $\carr{e_1}{e_2}$ is a valid cyclic base ordering.

  Consider a bispanning graph with $n$ vertices, then due to theorem~\ref{thm:bispanning deg23} it either has a vertex of degree two or three. If it has a vertex $v$ of degree two, which is attached to the vertices $x$ and $y$ by the edges $e_x$ and $e_y$ (see figure~\ref{fig:double-attach}, page~\pageref{fig:double-attach}), then one can reduce the graph to $G' = G - v$, effectively cutting off the vertex and reversing a double-attach operation. Without loss of generality we can assume $e_x \in S$ and $e_y \in T$ as in the figure. The resulting graph $G'$ is bispanning, and by induction hypothesis, the bispanning graph $G'$ with disjoint spanning trees $S' = \{s'_1,\ldots,s'_{m-1}\}$ and $T' = \{t'_1,\ldots,t'_{m-1}\}$ has a cyclic base ordering $\carr{ s'_1,\ldots,s'_{m-1} }{ t'_{1},\ldots,t'_{m-1} }$. As $v$ is a leaf of both trees in $G$, we can insert $e_x$ and $e_y$ in front of $s'_i$ and $t'_i$ for any index $i \in \{ 1,\ldots,m-1 \}$, or after $s'_{m-1}$ and $t'_{m-1}$. In the resulting base ordering, all edges are considered exactly once, and every $m$ consecutive edges form a tree, because the $m-1$ edges of $G'$ form a non-spanning tree of $G$, which is completed to a spanning tree by either $e_x$ or $e_y$ as a leaf edge.

  Next consider the case where $G$ has a vertex $v$ of degree three, which is attached to the vertices $x$, $y$, and $z$ by the edges $e_x$, $e_y$, and $e_z$ (see figure~\ref{fig:edge-split-attach}, page~\pageref{fig:edge-split-attach}). Without loss of generality we can assume $e_x, e_y \in S$ and $e_z \in T$ as in the figure. Then one can reduce the graph to $G' = G - v + \{ x, y \}$, effectively cutting off the vertex and reversing an edge-split-attach operation. The resulting graph $G'$ is bispanning, and by induction hypothesis, the bispanning graph $G'$ with disjoint spanning trees $S' = \{s'_1,\ldots,s'_{m-1}\}$ and $T' = \{t'_1,\ldots,t'_{m-1}\}$ has a cyclic base ordering $\carr{ s'_1,\ldots,s'_{m-1} }{ t'_{1},\ldots,t'_{m-1} }$. The split edge $\bar{e} = \{x,y\} \in S'$ is located as $s'_i = \bar{e}$ at an index $i \in \{ 1,\ldots,m-1\}$ in the reduced graph. This edge $\bar{e}$ has to be replaced in the sequence with $\arr{ e_x, e_y }$ or $\arr{ e_y, e_x }$, and the edge $e_z$ inserted before or after $t'_i$, resulting in $\arr{ e_z, t'_i }$ or $\arr{ t'_i, e_z }$. We denote with $\myalt{e_x,e_y} \in \{ \arr{e_x,e_y}, \arr{e_y,e_x} \}$ a placeholder for either choice, and select the first edge with $\myalt{e_x,e_y}_1$ and the second with $\myalt{e_x,e_y}_2$, depending on the final choice of ordering $e_x$ and $e_y$. The resulting base ordering is
$$O = \carr{ s'_1,\ldots,s'_{i-1}, \myalt{e_x, e_y}, s'_{i+1},\ldots,s_{m-1} }{ t'_1,\ldots,t'_{i-1},\myalt{e_z, t'_i},t'_{i+1},\ldots,t'_{m-1} }\,,$$
where the order of the two pairs has yet to be determined. Of the four possible combinations it turns out that only exactly two form valid cyclic base orderings.

  First fix the order $\arr{e_z, t'_i}$ in the sequence $O$ above and consider the two possible edge subsets $X = \arr{\myalt{ e_x, e_y }_2, s'_{i+1},\ldots,s_{m-1}, t'_1,\ldots,t'_{i-1}, e_z}$ containing $m$ edges, where $\myalt{e_x,e_y}_2 \in \{ e_x,e_y \}$. By induction, we know $X - e_z - \myalt{e_x, e_y}_2 + \bar{e}$ is a spanning tree of $G'$. However, due to the splitting of $\bar{e}$, $X - \myalt{e_x, e_y}_2$ is no longer a tree of $G$: it contains the two components of $G'[S'] - \bar{e}$ and $v$ is attached to one of them via $e_z$. As $v$ is connected via $e_z$ either to $x$ or to $y$, the edge choice $\myalt{e_x, e_y}_2$ has to be the one $e_z$ is \emph{not} connected to. Formally this is the edge $\{e_x, e_y\} \setminus C(\{ e_x, e_y s'_{i+1},\ldots,s_{m-1}, t'_1,\ldots,t'_{i-1} \},e_z)$, hence the edge not contained in the fundamental cycle closed by $e_z$. Selecting this edge as $\myalt{e_x, e_y}_2$ makes $X$ a spanning tree of $G$.

  Let $\myalt{e_x, e_y}_1$ be the other edge, then the cyclic complement of $X$, $X' = \arr{t'_i, t'_{i+1},\ldots,t'_{m-1}, \linebreak s'_1,\ldots,s'_{i-1}, \myalt{e_x, e_y}_1}$ is also a spanning tree of $G$, since the $m-1$ edges of $G'$ connect all vertices of $G$ except $v$, and $v$ is connected by $\myalt{e_1, e_y}_1$.

  To complete the claim that proposed edge ordering $O$ is cyclic, consider any other $m$ consecutive elements $B$. If $B$ contains both $e_x$ and $e_y$ then it cannot contain $e_z$, and $B$ corresponds to the tree $B - e_x - e_y + \bar{e}$ of $G'$ with $\bar{e}$ split into $e_x$ and $e_y$. If $B$ contains neither $e_x$ and $e_y$ then it contains $e_z$, and $B$ corresponds to the tree $B - e_z$ of $G'$ with the additional leaf edge $e_z$. The cases where $B$ contains either $e_x$ or $e_y$, but not both have already been handled above.

  Hence, $O$ is a cyclic base ordering with $\arr{e_z, t'_i}$, where the order of $\myalt{e_x,e_y}$ is determined from the cycle structure of the graph. By selecting the order $\arr{t'_i, e_z}$ in the sequence $O$, another cyclic base order can be determined using the same arguments. In the second case, the resulting order of $\myalt{e_x, e_y}$ need not be the same or the opposite of the first choice $\arr{e_z, t'_i}$.
\end{proof}

As an example, consider the construction of cyclic base orderings for $W_5$ in figure~\ref{fig:calc-cbo}. A pair of disjoint bispanning trees of $W_5$ is given, and in the first row these are decomposed stepwise at vertices of degree two or three. The deleted vertices are circled. In the first step, the edge-split-attach operation involving edges $0$, $1$, and $7$ is reversed, and the split edge in the reduced graph $G_4$ is labeled with $\{1,7\}$. This special labeling helps later during the expansion, but in the context of $G_4$, $\{1,7\}$ represents a single edge. Then the edge-split-attach operation involving edges $2$, $3$, and $\{1,7\}$ is reversed. And then two double-attach operations are reversed, which reduces $G_3$ down to $G_1$.

Thereafter, the reduction steps are undone, and during recomposition one can calculate a set of cyclic base orderings. The last two double-attach operations yield two possible cyclic base orderings (in the figure listed below $G_3$), since the two edge pairs can be ordered arbitrarily. Then the edge-split-attach operation $G_3$ to $G_4$ is redone: the edge $\{2,3\}$ is split into $2$ and $3$, and the new vertex attached via $\{1,7\}$. The edge ``$t'_i$'' in the proof above is edge $4$ in this step of the example. This yields the two possible edge orderings $\arr{4, \{1,7\}}$ and $\arr{\{1,7\}, 4}$ for the edges labeled $\arr{t'_i,e_z}$ in the proof. The edges $\myalt{2,3}$ can only be correctly ordered as $\arr{2,3}$, since in the first case $\arr{\{1,7\}, 5, 3}$ and in the second case $\arr{2,6,\{1,7\}}$ are a cycle. This yields two cyclic base orderings based on $\carr{ 6,4 }{ 5, \{2,3\} }$ for $G_3$. Analogously, two more can be derived from $\carr{ 4,6 }{ \{2,3\}, 5 }$. In total, this yields four cyclic base ordering for $G_4$.

Then the first edge-split-attach $G_4$ to $G_5$ is redone: the edge $\{1,7\}$ is split into $1$ and $7$, where the new vertex is attached via edge $0$. Each of the known cyclic base orderings of $G_4$ can be expanded into two cyclic base orderings of $G_5$, wherein the partner of $\{1,7\}$ is either $2$ or $3$, depending on the position of $\{1,7\}$. The order of $\myalt{0,2}$ and $\myalt{0,3}$ can be prescribed, and the cycle structure determines the order of $\myalt{1,7}$ at the corresponding position in the cyclic base orderings.

\begin{figure}[t]
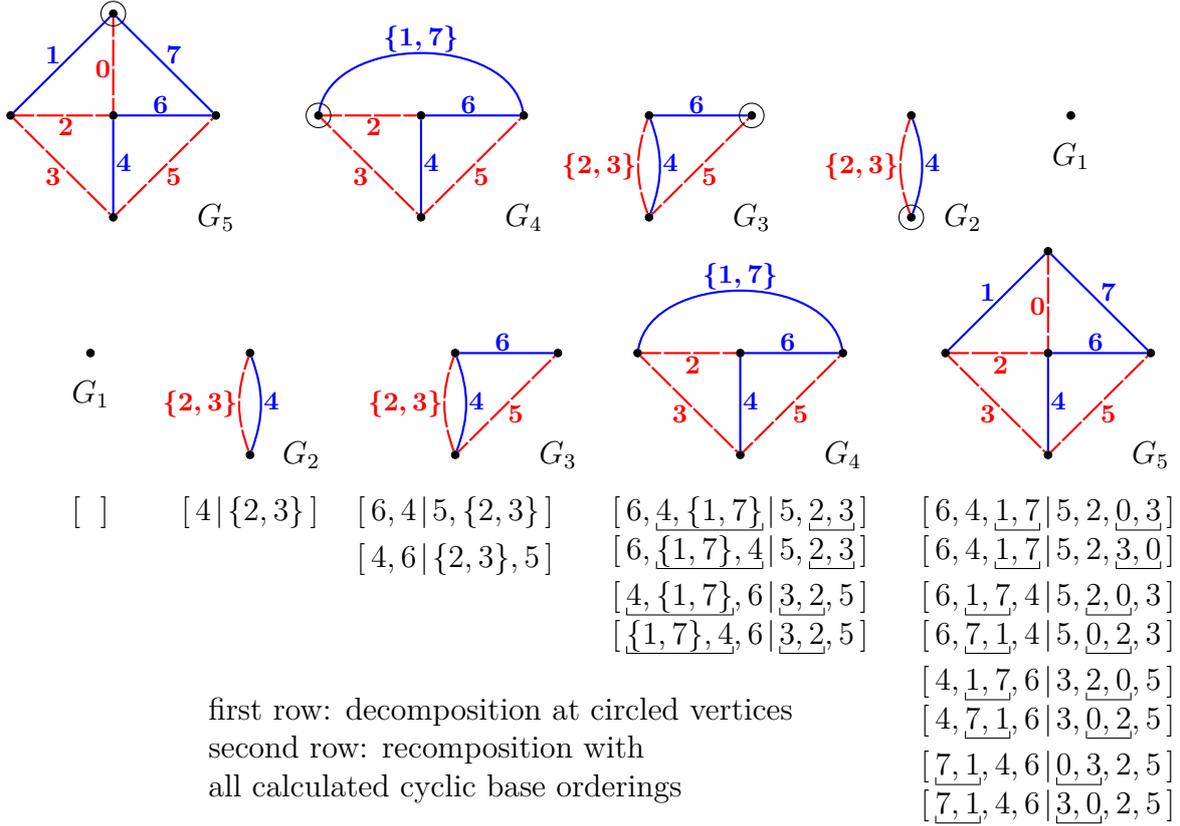
\centering

  \def\myunderbraket#1{\underbracket[0.4pt][2.5pt]{
      \smash{#1}%
    }}


  \caption{Calculation of cyclic base orderings for a pair of bispanning trees of $W_5$.}\label{fig:calc-cbo}
\end{figure}

As the previous theorem~\ref{thm:cbo-bispanning} showed that one can construct a cyclic base ordering for any bispanning graph, we raise the question whether one can construct a cyclic base ordering such that each exchange step in the cycle is a unique exchange (see figure~\ref{fig:uecbo-w5} for two examples). We define a \emph{unique exchange} cyclic base ordering via an edge swap sequence instead of as an edge sequence, since it better exposes the unique exchanges.

\begin{definition}[unique exchange cyclic base ordering of bispanning graphs]\label{def:ucbo-bispanning}
For a bispanning graph $G = (V,E,\delta)$ and a pair of disjoint spanning trees $(S,T)$ of $G$ with $m = |S| = |T|$, a \emph{unique exchange cyclic base ordering} (UECBO) of $G$ and $(S,T)$ is an ordering of  $S = \{ s_1,\ldots,s_m \}$ and $T = \{ t_1,\ldots,t_m \}$ into an edge swap sequence $\arr{ (e_1,f_1), (e_2,f_2), \ldots, (e_m,f_m) }$, such that for $i=1,\ldots,m$ each edge swap $(e_i,f_i) \in \{ (s_i,t_i), (t_i,s_i) \}$ is a unique $S$ or $T$ exchange for $(\{ s_i,\ldots,s_m,t_1,\ldots,t_{i-1} \}, \{ t_i,\ldots,t_m,s_1,\ldots,s_{i-1} \})$.

The edge swap sequence corresponds to the cyclic base ordering $\carr{ s_1, \ldots, s_m }{ t_1, \ldots, t_m }$.
\end{definition}

Restricting cyclic base orderings to unique exchanges makes their construction much more difficult, and obviously equivalent to finding a path of length $\frac{|E|}{2} = |S| = |T|$ between complementary vertices in $\tau_3(G)$. This is a restriction from finding a path of any length as required in Alice and Bob's game.

Besides indirectly regarded in Andres et al.~\cite{andres2014base}, we found no authors considering the problem of cyclic base orderings with unique exchanges. Baumgart~\cite[ch.~5.3]{baumgart2009ranking} considers a special subclass of cyclic base orderings called \emph{subsequence-interchangeable} base orderings, but could not prove that such exist for all bispanning graphs. His proof construction lacks a similar final step as our discussion in section~\ref{sec:reduce vdeg3}.

Every path in $\tau_3(G)$ of length $\frac{|E|}{2}$ between complementary tree pairs can be rewritten as a unique exchange cyclic base ordering. With a computer program, we verified that such a path exists for every pair of disjoint trees in all bispanning graphs with up to 12 vertices. Hence, we give the following conjecture:
\begin{conjecture}[existence of unique exchange cyclic base ordering]
  For every bispanning graph $G = (V,E,\delta)$ and every pair of disjoint spanning trees $(S,T)$ of $G$, a \emph{unique exchange cyclic base ordering} exists.
\end{conjecture}
A straight-forward attempt to prove the conjecture would be to consider whether always one of the two cyclic base orderings constructed in the proof of theorem~\ref{thm:cbo-bispanning} retains the unique exchange property. But this is not the case, and we only want to note a counter example here, and discuss the deeper problem in section~\ref{sec:reduce vdeg3}. If $G = W_5$ is numbered as in figure~\ref{fig:calc-cbo}, with the initial disjoint spanning trees $S = \{ 0,1,3,5 \}$ and $T = \{ 2,4,6,7 \}$, and is contracted at the degree vertex with edges $5$, $6$, and $7$ into the reduced graph $K_4$ with spanning trees $S = \{ 0,5,\{1,3\} \}$ and $T = \{ 4,6,7 \}$, then $K_4$ has the unique exchange cyclic base ordering $\carr{ \{1,3\},0,5 }{ 7,4,6 }$. The two valid expansions for the edge-split-attach operation are the cyclic base orderings $\carr{ 1,3,0,5 }{ 2,7,4,6 }$ and $\carr{ 1,3,0,5 }{ 7,2,4,6 }$, of which neither is a unique exchange cyclic base ordering.

We close this section with a simple theorem which implies that $\tau_3(G)$ graphs are highly ``symmetrical''.
\begin{theorem}[reversibility of unique exchange cyclic base orderings]\label{thm:uecbo-reversibility}
  If $G = (V,E,\delta)$ is a bispanning graph, $(S,T)$ a pair of disjoint spanning trees of $G$, and $\arr{ (e_1,f_1), \ldots, (e_m,f_m) }$ a unique exchange cyclic base ordering of $G$ and $(S,T)$, then $\arr{ (f_m,e_m), \ldots, (f_1,e_1) }$ is also a unique exchange cyclic base ordering of $G$ and $(S,T)$ (this is the \emph{same tree pair}), where $m = \frac{|E|}{2}$.
\end{theorem}
\begin{proof}
  The given edge swap sequence $\arr{ (e_1,f_1), (e_2,f_2), \ldots, (e_m,f_m) }$ exchanges $(S,T)$ for $(T,S)$, wherein each swap is a unique exchange. Due to theorem~\ref{thm:reversibility unique exchange}, each unique exchange is individually reversible, hence we can reverse the whole path: $\arr{ (f_m,e_m), (f_{m-1},e_{m-1}), \ldots, (f_1,e_1) }$ is an edge swap sequence from $(T,S)$ to $(S,T)$ of unique exchanges. By switching the disjoint edge trees in all pairs along the path, we get the edge swap sequence $\arr{ (f_m,e_m), (f_{m-1},e_{m-1}), \ldots, (f_1,e_1) }$ from $(S,T)$ to $(T,S)$ wherein each unique $S$ exchange is now a unique $T$ exchange and vice versa.
\end{proof}
While the proof of the previous theorem is simple, it shows that for every unique exchange cyclic base ordering, there is a symmetric ``twin''. For example, figure~\ref{fig:uecbo-w5} shows two such twins: the first cyclic base ordering, $\carr{7,1,4,6}{0,3,2,5}$, can be reversed to yield the second one, $\carr{6,4,1,7}{5,2,3,0}$. Intuitively, this makes the $\tau_3(G)$ highly ``symmetrical'' in some sense. Furthermore, it is remarkable, that there seems to be no relationship between the first two unique exchanges in the twin sequence.

\begin{figure}
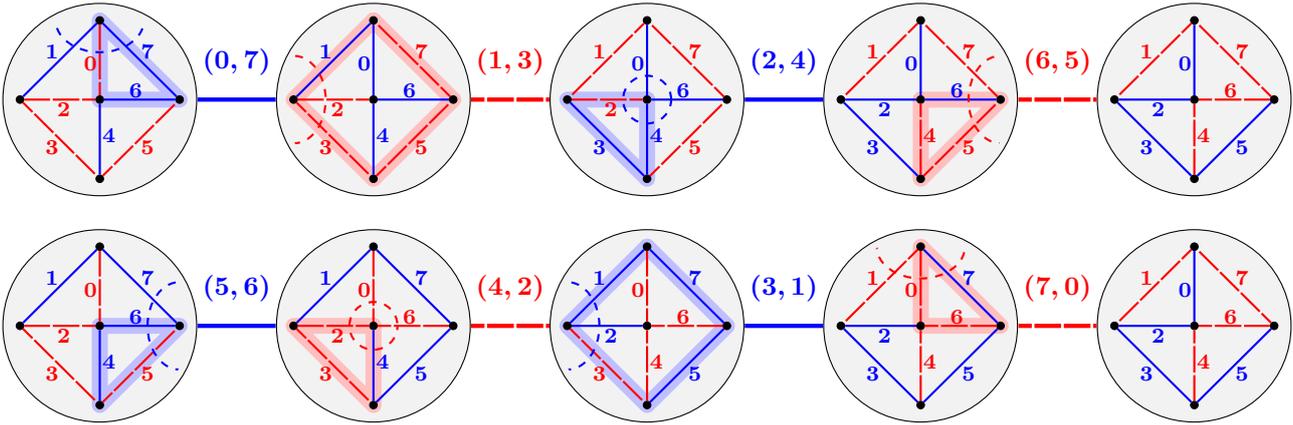
\centering


  \caption{Two unique exchange cyclic base orderings of $W_5$.}\label{fig:uecbo-w5}
\end{figure}

%%%%%%%%%%%%%%%%%%%%%%%%%%%%%%%%%%%%%%%%%%%%%%%%%%%%%%%%%%%%%%%%%%%%%%%%%%%%%%%%
\clearpage

\section{Composing Exchange Graphs}\label{sec:decomposing}

In this section we develop a composition strategy to break bispanning graphs into smaller ones, and to combine the unique exchange graph of the parts back to gain the exchange graph of the whole. In this context we refer to the bispanning graph as the \emph{underlying} graph.

% ------------------------------------------------------------------------------

\subsection{Classifying by Vertex- and Edge-Connectivity}

Our first approach is to split the underlying graph at small vertex or edge cuts. The idea is that only few edges cross these cuts and hence the cycle and cuts passing the split become manageable. Hence, we are interesting in the connectivity of bispanning graphs.

\begin{definition}[$k$-vertex- and $k$-edge-connected and connectivity \R{\cites[ch.\,4]{west2001introduction}}]
  \begin{enumerate}
  \item A connected graph $G = (V,E,\delta)$ is called \emph{$k$-vertex-connected}\index{graph!vertex-connected@$k$-vertex-connected}\index{vertex-connected@$k$-vertex-connected}, if $|V| > k$ and $G \setminus V'$ remains connected for all vertex subsets $V' \subseteq V$ with $|V'| < k$.

  \item A connected graph $G = (V,E,\delta)$ is called \emph{$k$-edge-connected}\index{graph!edge-connected@$k$-edge-connected}\index{edge-connected@$k$-edge-connected}, if $|E| > k$ and $G \setminus E'$ remains connected for all edge subsets $E' \subseteq E$ with $|E'| < k$.

  \item The greatest integer $k$ such that a given graph $G$ is $k$-vertex-connected is called the \emph<vertex-connectivity!graph>{vertex-connectivity} $\operatorname{vconn}(G)$ of $G$.
    \symbol{vconn(G)}{$\operatorname{vconn}(G)$}{vertex-connectivity of $G$}

  \item The greatest integer $k$ such that a given graph $G$ is $k$-edge-connected is called the \emph<edge-connectivity!graph>{edge-connectivity} $\operatorname{econn}(G)$ of $G$.
    \symbol{econn(G)}{$\operatorname{econn}(G)$}{edge-connectivity of $G$}

  \end{enumerate}
\end{definition}

The previous definition may seem confusing due to quantification over all sets with $|V'| < k$ or $|E'| < k$, however, when considering the contrapositive things are clear: a graph is not $k$-vertex/edge-connected if a vertex/edge set exists such that removal of the set disconnects the graph.
% Due to the definition, every $k$-vertex- or $k$-edge-connected graph is also $(k-1)$-vertex- or $(k-1)$-edge-connected if $k > 0$.
The following relationship between vertex- and edge-connectivity helps reduce the number of combinations:
\begin{theorem}[relation of vertex- and edge-connectivity \R{\cites[thm.\,5]{whitney1932congruent}{west2001introduction}}]\label{thm:vconn<=econn}
In every graph $G = (V,E,\delta)$ with $|V| \geq 2$ the relation $\operatorname{vconn}(G) \leq \operatorname{econn}(G)$ holds.
\end{theorem}
\begin{proof}\cite[thm.\,5]{whitney1932congruent}
  Consider a minimal edge cut $[A,V \setminus A]$ of $k := \operatorname{econn}(G)$ edges. If every vertex in $A$ is adjacent to every vertex in $V \setminus A$, then $k = |\, [A,V \setminus A] \,| \geq |A| \cdot |V \setminus A| \geq |V| > \operatorname{vconn}(G)$. Otherwise let $x_1,\ldots,x_k \in A$ and $y_1,\ldots,y_k \in V \setminus A$ be the ends of the $k$ edges in $[A,V \setminus A]$, and $\bar{x} \in A$ and $\bar{y} \in V \setminus A$ be two non-adjacent vertices. In every pair $(x_i,y_i)$ either $x_i \neq \bar{x}$ or $y_i \neq \bar{y}$. Collect in $T$ either $x_i \neq \bar{x}$ or $y_i \neq \bar{y}$ from every pair $(x_i,y_i)$, where one can freely choose if both are unequal. As $T$ contains an end from every edge in $[A,V \setminus A]$, removing $T$ from $G$ also removes all edges in the cut. Thus $V \setminus T$ is not connected with $|T| \leq k$, and therefore $k \geq |T| \geq \operatorname{vconn}(G)$.
\end{proof}

\begin{figure}
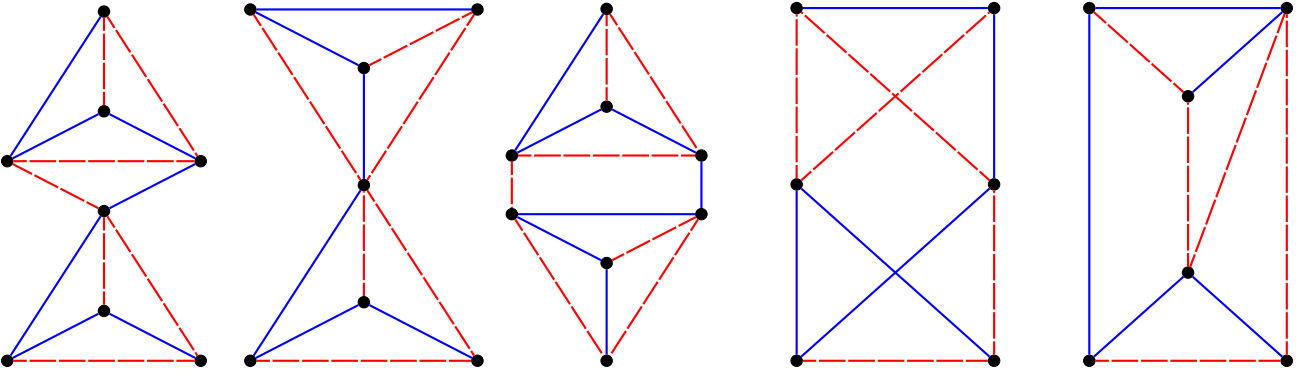
\centering
  \tikzset{every picture/.style={scale=1.5, yscale=0.9, graphfinal}}



  \caption{One example of a bispanning graph per connectivity class $(\operatorname{vconn}(G), \operatorname{econn}(G))$.}\label{fig:vconn-econn}
\end{figure}

With the last theorem, we can classify all bispanning graphs by connectivity.
\begin{theorem}[{\textls[-5]{vertex- and edge-connectivity of bispanning graphs}} \R{\cite[sect.~2.2]{baumgart2009ranking}}]\label{thm:bispanning-conn}
  If $G$ is a bispanning graph with $|V| \geq 2$, then
  $$(\operatorname{vconn}(G),\operatorname{econn}(G)) \in \{ (1,2), (1,3), (2,2), (2,3), (3,3) \} \,.$$
\end{theorem}
\begin{proof}
  Every bispanning graph is by definition $1$-vertex-connected and $2$-edge-connected, since there are at least two paths between any pair of vertices in the disjoint spanning trees. Moreover, no bispanning graph can be $4$-vertex or $4$-edge-connected due to existence of a vertex of degree two or three. Moreover, due to theorem~\ref{thm:vconn<=econn}, $\operatorname{vconn}(G) \leq \operatorname{econn}(G)$, only the stated list of combinations remains. Figure~\ref{fig:vconn-econn} shows examples for all listed combinations.
\end{proof}

% ------------------------------------------------------------------------------

\subsection{Atomic and Composite Bispanning Graphs}\label{sec:decomposition composite}

In the following, we will focus on subgraphs of bispanning graphs that are themselves bispanning. In this context, a ``bispanning subgraph'' is not an ``twofold'' modification of a spanning subgraph as defined in \ref{def:spanning tree}, but merely a shortening of ``subgraph, which itself is bispanning''.

\begin{definition}[composite and atomic bispanning graph \R{\cite{baumgart2009ranking}}]\label{def:atomic bispanning}
  A bispanning graph $G$ is called \emph<composite!bispanning graph>{composite} if it contains a bispanning subgraph other than itself and $K_1$. Otherwise $G$ is called \emph<atomic!bispanning graph>{atomic}.
  The two bispanning subgraphs $G$ and $K_1$ of any bispanning graph $G$ are called \emph<trivial!bispanning graph>{trivial}.
\end{definition}

The remarkable property of bispanning subgraphs is that if they are contracted, then the resulting graph remains bispanning. Obviously, the number of edges connected to the rest of the graph has the ``right'' balance to those inside the bispanning subgraph.

\begin{lemma}[edge balance during contraction of bispanning subgraphs]\label{lem:edge balance contraction}
  If $G$ is a bispanning graph, $G' = (V',E',\delta') \subseteq G$ a bispanning subgraph of $G$, and $\overline{G} = G \contr G' =: (\overline{V},\overline{E},\overline{\delta})$, then
  $|\overline{E}| = 2 (|V| - |V'|) = 2 |\overline{V}| - 2$\,.
\end{lemma}
\begin{proof}
  The definition of contraction~\ref{def:contraction} implies $|\overline{V}| = |V| - |V'| + 1$, and $|\overline{E}| = |E| - |E'|$, thus we have $|\overline{E}| = (2 |V| - 2) - (2 |V'| - 2) = 2 (|V| - |V'|) = 2 |\overline{V}| - 2$.
\end{proof}

But the previous lemma is only necessary for $\overline{G}$ to be bispanning, we have to use Nash-Williams' theorem~\ref{thm:nash-williams} to show that the inner structure of $G \contr G'$ remains bispanning. While the trivial bispanning subgraphs $K_1$ and $G$ are excluded from the definition of atomic above, for many of the following theorems they can be included as pathological cases. When this is possible, we simply omit the requirement of $G$ being composite, such as in the following:

\begin{theorem}[contracting bispanning subgraphs]\label{thm:contract bispanning subgraph}
  If $G$ is a bispanning graph and $G' \subseteq G$ is a bispanning subgraph of $G$, then $G \contr G'$ is bispanning.
\end{theorem}
\begin{proof}
  Let $G = (V,E,\delta)$ be a bispanning graph, $G' = (V',E',\delta') \subseteq G$ a bispanning subgraph, and $\overline{G} = (\overline{V}, \overline{E}, \overline{\delta}) := G \contr G'$ the graph within which $G'$ was contracted into $x \in \overline{V}$. We will apply theorem~\ref{thm:nash-williams} to $\overline{G}$ and consider a partition $\overline{P} = \{ \overline{V}_1, \ldots, \overline{V}_k \}$ of $\overline{V}$. The vertex $x$ is in exactly one member of the partition, say $\overline{V}_i$. Since all other vertices of $\overline{G}$ are also vertices of $G$, replacing only the member $\overline{V}_i$ with $V_i := (\overline{V}_i - x) \cup V'$ yields a partition $P$ of $V$ with the same number of members. Furthermore, $\overline{E}_{\overline{P}} = E_P$, as contraction (definition~\ref{def:contraction}) changes all edge ends in $V_i$ to $x$ and $E_P \cap E' = \emptyset$. As $G$ is bispanning, theorem~\ref{thm:nash-williams} implies $|E_P| \geq 2 (|P| - 1)$, and since $|\overline{P}| = |P|$, $|\overline{E}_{\overline{P}}| = |E_P|$, and lemma~\ref{lem:edge balance contraction}, the same theorem guarantees that $\overline{G}$ is bispanning.
\end{proof}

To show that a bispanning graph is atomic, the following modification of Nash-Williams' or Tutte's theorems are most useful:

\begin{theorem}[atomic bispanning graphs \R{\cite[thm.\,2.5]{baumgart2009ranking}}]\label{thm:atomic nash-williams}
  A bispanning graph $G=(V,E,\delta)$ is atomic, if and only if
  $$|E_P| > 2 (|P| - 1) \quad\text{for every non-trivial partition $P$ of $V$,}$$
  where $E_P$ is the set of edges with ends in different members of the partition $P$ and $|P|$ is the number of members. Any partition $P$ of $V$ with $|P| \neq 1$ and $|P| \neq |V|$ is called non-trivial.
\end{theorem}
\begin{proof}
  To show that the condition is sufficient, suppose there is a non-trivial partition $P = \{ V_1,\ldots,V_k \}$ of $V$ with $|E_P| = 2 (k - 1)$. Let $E_i$ be the edge set of $G[V_i]$ for $i = 1,\ldots,k$. Applying theorem~\ref{thm:nash-williams} to each $G[V_i]$ with the finer partition $P_i := (P - V_i) \cup \{ \{v\} \mid v \in V_i \}$ (isolating each vertex of $V_i$), we have $|E_{P_i}| = |E_P| + |E_i| = 2(k-1) + |E_i| \geq 2 (k - 1 + |V_i| - 1)$, from which follows $|E_i| \geq 2 (|V_i| - 1)$.
  So, in total, there are $\sum_{i=1}^k |E_i| \geq 2 (|V| - k)$ edges inside the induced subgraphs $G[V_i]$. On the other hand, since $|E_P| = 2 (k - 1)$ edges are outside the induced subgraphs, only exactly $2 (|V| - 1) - 2 (k - 1) = 2 (|V| - k)$ remain to be inside. Thus each subgraph $G[V_i]$ has exactly $2 (|V_i| - 1)$ edges. Since $P$ is non-trivial, there is a $V_i$ with $1 < |V_i| < |V|$ and hence $G[V_i]$ is a non-trivial bispanning subgraph. Thus $G$ is composite if such a partition $P$ exists.

  To show that the condition is necessary, assume $G$ is composite and $G[V']$ is a non-trivial bispanning subgraph with $V' \subseteq V$. As $G[V']$ is bispanning, the edge set $E'$ of $G[V']$ contains $2 |V'| - 2$ edges (theorem~\ref{thm:bispanning edge count}). Consider the partition $P = \{ V' \} \cup \{ \{v\} \mid v \in V \setminus V' \}$, which has $1 + |V| - |V'|$ members. $E_P$ are all edges other than $E'$, thus $|E_P| = |E| - |E'| = (2 |V| - 2) - (2 |V'| - 2) = 2 (|V| - |V'| + 1 - 1) = 2 (|P| - 1)$. Thus for every composite bispanning graph a partition with equality exists, matching the vertex set of the non-trivial bispanning subgraph.
\end{proof}

\begin{corollary}[atomic bispanning graphs]\label{thm:atomic tutte}
  A bispanning graph $G=(V,E,\delta)$ is atomic, if and only if
  $$|X| > 2 ( \operatorname{comp}(G \setminus X) - 1)$$
  for all edge subsets $X \subseteq E$ except $X = \emptyset$ and $X = E$.
\end{corollary}
\begin{proof}
  The corollary follows from theorem~\ref{thm:atomic nash-williams} and the equivalence of theorems~\ref{thm:nash-williams} and \ref{thm:tutte bispanning}.
\end{proof}

\begin{figure}
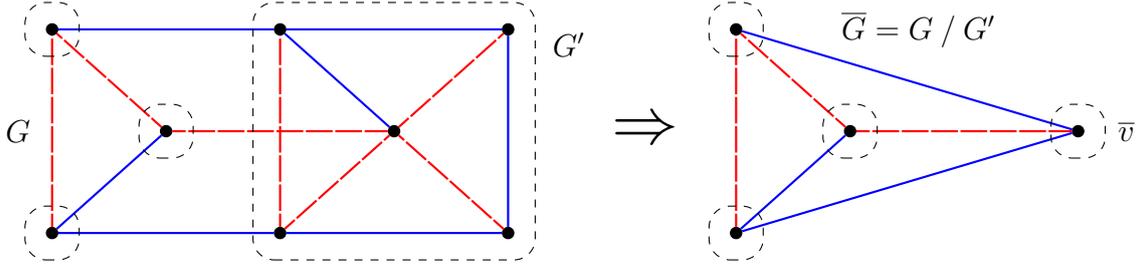
\centering
  \tikzset{every picture/.style={scale=3, graphfinal}}

  \caption{A composite bispanning graph $G$ with subgraph $G' \cong W_5$, contraction $\overline{G} = G \contr G' \cong K_4$, and a partition of $V$ with $|E_P| = 2(|P|-1)$.}\label{fig:composite-decompose}
\end{figure}

Consider the example bispanning graph $G$ in figure~\ref{fig:composite-decompose}. It contains $W_5$ as a bispanning subgraph and if we regard the partition $P$ containing the $W_5$ subgraph alongside all other vertices as singletons (see circled vertices), then there are exactly six edges in $E_P$: those outside of the $W_5$. These six edges are necessarily balanced among the disjoint trees $S$ and $T$ (see lemma~\ref{lem:bispanning subtrees}). Hence when contracting the $W_5$ subgraph into a single vertex $\overline{v}$, the resulting graph is bispanning.

In the following theorem we summarize many simple properties, which immediately make a bispanning graph composite. These properties are so common, that about 97\% of all simple bispanning graphs with $|V| \leq 12$, about 95\% with $|V| \leq 10$, and 88\% with $|V| \leq 8$ are composite (see table~\ref{tab:num bispanning}, page~\pageref{tab:num bispanning}).

\begin{theorem}[sufficient conditions for composite bispanning graphs]
  Let $G = (V,E,\delta)$ be a bispanning graph with $|V| \geq 3$.
  \begin{enumerate}
  \item If $G$ contains a pair of parallel edges, then $G$ is composite.\label{thm:parallel edge composite}

  \item All atomic bispanning graphs are simple.\label{thm:atomics are simple}

  \item If $G$ contains a cut-vertex ($\operatorname{vconn}(G) = 1$), then $G$ is composite.\label{thm:bispanning composite cut-vertex}

  \item If $\operatorname{econn}(G) = 2$, then $G$ is composite.\label{thm:bispanning composite econn2}

  \item If $G$ contains a vertex of degree $2$, then $G$ is composite.\label{thm:bispanning composite degree2}

  \end{enumerate}
\end{theorem}
\begin{proof}
  \begin{enumerate}
  \item Let $e_1,e_2 \in E$ be two parallel edges with $\delta(e_1) = \delta(e_2) = \{ v_1,v_2 \}$, then $G[\{v_1,v_2\}]$ is a non-trivial bispanning subgraph with two vertices and two edges. Thus $G$ is composite.

  \item This follows immediately from theorem~\ref{thm:parallel edge composite} and definition~\ref{def:simple graph}.

  \item Let $\tilde{v}$ be a cut-vertex and $V_1 \dotcup \cdots \dotcup V_k = V \setminus \{\tilde{v}\}$ the vertex sets of the $k > 1$ components of $G - v$. If we take $V' := V_i \cup \{\tilde{v}\}$ for an arbitrary $V_i$, then the subgraph $G[V']$ is by definition~\ref{def:bispanning} bispanning since any pair of disjoint trees of $G$ can be restricted to $G[V']$, yielding two disjoint trees.

  \item Let $\{ e_1, e_2 \}$ be an edge cut of $G$ and $V_1 \dotcup V_2 = V$ the vertex sets of the two components of $G \setminus \{ e_1, e_2 \}$. There cannot be more components, since $G - e_1$ and $G - e_2$ are connected. Consider the partition $P = \{ V_1, V_2 \}$. Since $|E_P| = 2 = 2 (|P| - 1)$, theorem~\ref{thm:atomic nash-williams} implies that $G$ is composite.

  \item Existence of a vertex of degree 2 implies $\operatorname{econn}(G) = 2$, so theorem~\ref{thm:bispanning composite econn2} applies.

  \end{enumerate}
\end{proof}

Due to the properties listed in the previous theorem, only a few connectivity classes remain for atomic bispanning graphs. In the next subsections, we will investigate them individually in detail.
\begin{corollary}[vertex- and edge-connectivity of atomic bispanning graphs]\label{cor:atomic-vconn-econn}
  Every atomic bispanning graph has $(\operatorname{vconn}(G),\operatorname{econn}(G)) \in \{ (2,3), (3,3) \}$.
\end{corollary}
\begin{proof}
  Of the possible combinations listed in theorem~\ref{thm:bispanning-conn}, $(1,2)$ and $(1,3)$ are ruled out by theorem~\ref{thm:bispanning composite cut-vertex}, and $(2,2)$ is ruled out by \ref{thm:bispanning composite econn2}.
\end{proof}

Intuitively, atomic bispanning graphs have a higher ``sparse internal connectivity'' than composite graphs, though of course they have the same vertex/edge ratio. This connectivity is strong enough that one can select an arbitrary pair of edges in an atomic bispanning graph, and simply contract one and delete the other. The resulting graph remains bispanning, though often it is no longer atomic. This arbitrary selection is a stronger reduction property than reversal of an ``edge-split-attach'' operation (see theorem~\ref{thm:inductive operations}), as any pair of edge can be picked.
\begin{lemma}[contract-deleting an edge pair in an atomic bispanning graph]\label{lem:contract-delete atomic bispanning}
  If $G = (V,E,\delta)$ is an atomic bispanning graph and $e,f \in E$ are any two edges with $e \neq f$, then $G \contr e - f$ is a bispanning graph.
\end{lemma}
\begin{proof}
  Let $\overline{G} = (\overline{V},\overline{E},\overline{\delta}) := G \contr e - f$ be the resulting smaller graph, within which $e$ was contracted into a new vertex $v_e$. We will apply theorem~\ref{thm:nash-williams} to $\overline{G}$, which obviously has the right vertex/edge balance, and consider a partition $\overline{P} = \{ \overline{V}_1, \ldots, \overline{V}_k \}$ of $\overline{V}$. The vertex $v_e$ is in exactly one member of the partition, say $\overline{V}_i$. Since all other vertices of $\overline{G}$ are also vertices of $G$ we can construct a partition $P$ of $G$ by taking $\overline{P}$, and replacing only $\overline{V}_i$ with $\overline{V}_i - v_e + x_e + y_e$, where $\delta(e) = \{ x_e, y_e \}$ are the ends of $e$. Intuitively, we undo the contraction of $e$ within the partition of $\overline{V}$, gaining a partition of $V$ of the same size. If $f \notin E_P$, then $|E_P| = |\overline{E}_{\overline{P}}|$, since the number of edges crossing the partitions is preserved during contraction; though some change their ends from $x_e$ or $y_e$ to $v_e$. If $f \in E_P$, then simply $|E_P| = |\overline{E}_{\overline{P}}| + 1$, since we have to account for the deleted edge. Due to theorem~\ref{thm:atomic nash-williams} we known $|E_P| \geq 2 (|P|-1)+1$ since $G$ is atomic. Thus we have $|\overline{E}_{\overline{P}}| \geq |E_P| - 1 \geq 2 (|P|-1) + 1 - 1 = 2 (|P|-1)$ for all partitions of $\overline{V}$ and hence theorem~\ref{thm:nash-williams} guarantees that $\overline{G}$ is bispanning.
\end{proof}

While theorem~\ref{thm:contract bispanning subgraph} states that one can contract a bispanning subgraph into a representative vertex, it does not immediately imply that if one has two specific disjoint spanning trees of $G$, that these remain valid after contraction (without the removed edges of course). Likewise, if we choose to expand a vertex into a bispanning subgraph, it is not obvious that \emph{any} pair of disjoint spanning trees of the expanded subgraph fit together with \emph{any} pair of spanning trees of the outside graph.  While the validity of these lemmata is intuitive, the proof is rather technical as one has to reach back to the spanning tree equivalences.

\begin{lemma}[projection and expansion of bispanning subgraphs]\label{lem:bispanning subtrees}
  Let $G = (V,E,\delta)$ be a bispanning graph, $G' = (V',E',\delta') \subseteq G$ a bispanning subgraph of $G$, and $\overline{G} = (\overline{V},\overline{E},\overline{\delta}) = G \contr G'$.
  \begin{enumerate}
  \item If $E = S \dotcup T$ are two disjoint spanning trees of $G$, then $S \cap E'$ and $T \cap E'$ are disjoint spanning trees of $G'$, and $S \cap \overline{E}$ and $T \cap \overline{E}$ are disjoint spanning trees of $\overline{G}$.\label{lem:bispanning subtrees split}

  \item If $S'$ and $T'$ are two disjoint spanning trees of $G'$ and $\overline{S}$ and $\overline{T}$ are two disjoint spanning trees of $\overline{G}$, then $S' \dotcup \overline{S}$ and $T' \dotcup \overline{T}$ are two disjoint spanning trees of $G$.\label{lem:bispanning subtrees combine}
  \end{enumerate}
\end{lemma}
\begin{proof}
  \begin{enumerate}
  \item Assume first that $S \cap E'$ is not a spanning tree of $G'$, then due to theorem~\ref{thm:tree acyclic}, $S \cap E'$ is either not acyclic or $|S \cap E'| \neq |V'| - 1$. However, $S \cap E'$ is certainly acyclic in $G'$, because $S$ is acyclic in $G$ and $G' \subseteq G$. So assume $|S \cap E'| = |V'| - 1 - d$ for some integer $d > 0$, then, however, $|T \cap E'| = |E'| - |V'| + 1 + d = (2 |V'| - 2) - |V'| + 1 + d = |V'| - 1 + d$, which requires that the spanning tree $T$ contains a cycle within the subgraph $G[T \cap E'] \subseteq G[T]$. As this is impossible, $S \cap E'$ is a spanning tree of $G'$, and the argument with $S$ and $T$ exchanged shows the same for $T \cap E'$.

    Assume second that $S \cap \overline{E}$ is not a spanning tree of $\overline{G}$, then due to theorem~\ref{thm:tree acyclic}, $S \cap \overline{E}$ is either acyclic or $|S \cap \overline{E}| \neq |\overline{V}| - 1$. As $S$ is acyclic in $G$, $G[S] \contr V'$ is also acyclic. Since $G[S] \contr V'$ has the edge set $S \cap \overline{E}$, $S \cap \overline{E}$ is acyclic in $\overline{G}$ and we have to assume $|S \cap \overline{E}| \neq |\overline{V}| - 1$. As in the previous paragraph one can use the edge balance arguments of the disjoint tree decompositions to show that if $S \cap \overline{E}$ is disconnected in $\overline{G}$, then $T \cap \overline{E}$ contains a cycle, and vice versa. As the same arguments apply to $T \cap \overline{E}$ as well, $S \cap \overline{E}$ and $T \cap \overline{E}$ are spanning trees of $\overline{G}$ that are obviously disjoint.

  \item We will first consider $S'$ and $\overline{S}$, and apply theorem~\ref{thm:tree connected} to $S' \dotcup \overline{S}$. Due to the definition of contraction (\ref{def:contraction}) and theorem~\ref{thm:tree connected}, we have $|\overline{V}| = |V| - |V'| + 1$ and $|S' \dotcup \overline{S}| = |S'| + |\overline{S}| = |V'| - 1 + |\overline{V}| - 1 = |V| - 1$, since both are spanning trees of disjoint edge sets. Next we show that $G[S' \dotcup \overline{S}]$ is connected: take any $v,w \in V$, then $v$ and $w$ are both either in $V'$ or $\overline{V} - \overline{v}$, so there are four cases. If $v,w \in V'$, then they are connected by a path in $G'[S'] \subseteq G[S' \dotcup \overline{S}]$. If $v \in V'$ and $w \in \overline{V} - \overline{v}$, then there exists a path $\overline{v},e_1,v_1,\ldots,e_k,w$ in $\overline{G}[\overline{S}]$ from $\overline{v}$ to $w$, which can be expanded to a path in $G[S' \dotcup \overline{S}]$ using a path in $G'[S']$ from $v$ to the other end $\delta(e_1) - v_1 \in V'$ of $e_1$. The case $v \in \overline{V} - \overline{v}$ and $w \in V'$ is handled symmetrically. If both $v,w \in \overline{V} - \overline{v}$, then there exists a path $v,e_1,v_1,\ldots,e_k,w$ in $\overline{G}[\overline{S}]$, which must be expanded to a path in $G[S' \dotcup \overline{S}]$ only if it contains $\overline{v}$: replace $\overline{v}$ with a path in $G'[V']$ from $\delta(e_i) - \overline{v} \in V'$ to $\delta(e_{i+1}) - \overline{v} \in V'$. So $G[S' \dotcup \overline{S}]$ is connected and $|S' \dotcup \overline{S}| = |V| - 1$, thus $S' \dotcup \overline{S}$ is a spanning tree of $G$ due to theorem~\ref{thm:tree connected}.

    The same argument analogously applies to $T'$ and $\overline{T}$, so $T' \dotcup \overline{T}$ is a spanning tree of $G$, and the disjoint union of a pair of disjoint trees obviously yields a pair of disjoint trees.
  \end{enumerate}
\end{proof}

The structure of composite bispanning graphs severely constrains where cycles and cuts of an edge can run. The following lemma makes these constraints clear, and, as a corollary, unique edge exchanges are restricted in composite graphs: they either stay within a bispanning subgraph or outside of it.
\begin{lemma}[containment of cycles and cuts in composite bispanning graphs]\label{lem:bispanning subcyclecut}
  Let $G = (V,E,\delta)$ be a bispanning graph with $E = S \dotcup T$, $G' = (V',E',\delta') \subseteq G$ a bispanning subgraph of $G$, and $\overline{E} = E \setminus E'$ the edges outside of $G'$.
  \begin{enumerate}
    \item If $e \in S \cap E'$ is an edge in the bispanning subgraph, then $C_G(T,e) \subseteq (T \cap E') + e$.
    \item If $e \in S \cap \overline{E}$ is an edge outside of the bispanning subgraph, then $D_G(S,e) \subseteq (T \cap \overline{E}) + e$.
  \end{enumerate}
  The two cases apply to $e \in T \cap E'$ and $e \in T \cap \overline{E}$ analogously.
\end{lemma}
\begin{proof}
  \begin{enumerate}
  \item If $e \in S \cap E'$, then the cycle $C_G(T,e)$ is fully contained in the bispanning subgraph $G'$, since $T \cap E'$ is a spanning tree within it, due to lemma~\ref{lem:bispanning subtrees}. Thus $C_G(T,e) \subseteq (T \cap E') + e$.

  \item If $e \in S \cap \overline{E}$, then the cut $D_G(S,e)$ contains no edge of the bispanning subgraph $G'$, since $S \cap E'$ remains connected in $S - e$ due to lemma~\ref{lem:bispanning subtrees}.
  \end{enumerate}
\end{proof}

\begin{corollary}[containment of unique edge exchanges]\label{lem:composite ue containment}
  Let $G = (V,E,\delta)$ be a bispanning graph with $E = S \dotcup T$, $G' = (V',E',\delta') \subseteq G$ a bispanning subgraph of $G$, $\overline{G} = (\overline{V},\overline{E},\overline{\delta}) = G \contr G'$, and $(e,f) \in S \times T$ a unique edge exchange.
  \begin{enumerate}
  \item If $e \in S \cap E'$, then $f \in T \cap E'$.
  \item If $e \in S \cap \overline{E}$, then $f \in T \cap \overline{E}$.
  \end{enumerate}
  As the theorem is stated without ordering $S \dotcup T$, it also applies if $(e,f) \in T \times S$.
\end{corollary}
\begin{proof}
  As $(e,f)$ is a unique edge exchange, $D_G(S,e) \cap C_G(T,e) = \{ e,f \}$ with $e \neq f$. Due to lemma~\ref{lem:bispanning subcyclecut}, in case (i), $C_G(T,e) \subseteq (T \cap E') + e$ and hence $f \in T \cap E'$, and in case (ii), $D_G(S,e) \subseteq (T \cap \overline{E}) + e$ and hence $f \in T \cap \overline{E}$.
\end{proof}

As preparation for our structure theorem of unique exchange graphs for composite bispanning graphs, we need to define the Cartesian graph product. As we consider both undirected and directed exchange graphs, we require both definitions.
In intuitive words, the Cartesian product contains a copy of the graph $G_1$ for every vertex of $G_2$, including all edges, and vice versa. See figure~\ref{fig:cartesian graph product} for an example.
\begin{definition}[Cartesian graph product $G_1 \times G_2$ \R{\cites[\S\,2.5]{halin1989graphentheorie}{harary1969graph}}]\label{def:cartesian graph product}
  The \emph{Cartesian graph product}\index{Cartesian graph product}\index{graph!Cartesian product} $G_1 \times G_2$\symbol{graph product}{$G_1 \times G_2$}{Cartesian graph product} of two graphs $G_1 = (V_1,E_1,\delta_1)$ and $G_2 = (V_2,E_2,\delta_2)$ is the graph $G = (V,E,\delta)$ such that both vertex set $V = V_1 \times V_2$ and edge set $E = (E_1 \times V_2) \cup (V_1 \times E_2)$ are Cartesian products, and $\delta( (e_1,v_2) ) = \delta(e_1) \times \{ v_2 \}$ and $\delta( (v_1,e_2) ) = \{ v_1 \} \times \delta(e_2)$ for all $v_1 \in V_1$, $v_2 \in V_2$, $e_1 \in E_1$, and $e_2 \in E_2$ appropriately.

  Likewise, the \emph{Cartesian graph product} $G_1 \times G_2$ of two directed graphs $G_1 = (V_1,E_1,\delta_1)$ and $G_2 = (V_2,E_2,\delta_2)$ is the directed graph $G = (V,E,\delta)$ such that both vertex set $V = V_1 \times V_2$ and edge set $E = (E_1 \times V_2) \cup (V_1 \times E_2)$ are Cartesian products, and $\delta( (e_1,v_2) ) = ((x_{e_1}, v_2), (y_{e_1},v_2))$ and $\delta( (v_1,e_2) ) = ( (v_1,x_{e_2}), (v_1,y_{e_2}) )$ for all $v_1 \in V_1$, $v_2 \in V_2$, $e_1 \in E_1$, $e_2 \in E_2$, and $\delta_1(e_1) = (x_{e_1},y_{e_1})$, $\delta_2(e_2) = (x_{e_2},y_{e_2})$ appropriately.
\end{definition}

Formally, the Cartesian graph product forms an Abelian half-group with neutral element $K_1$, which means that $(A \times B) \times C \cong A \times (B \times C)$, $A \times B = B \times A$, and $A \times K_1 = A$ for all graphs $A$, $B$, and $C$.
For our application it is more important, that paths ``\emph{multiply}'': given a path $P_1$ in $G_1$ and a path $P_2$ in $G_2$, starting at $v_1$ and $v_2$ and ending in $w_1$ and $w_2$, respectively, then these two paths can be used to identify paths from $(v_1,w_1)$ to $(v_2,w_2)$ in $G_1 \times G_2$ by taking one step at a time either in $P_1$ or in $P_2$ in \emph{any order}.

\begin{figure}
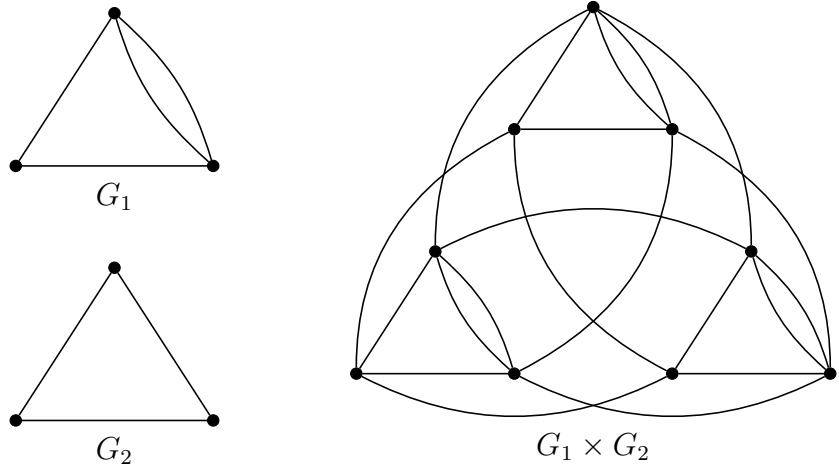
\centering
  \tikzset{every picture/.style={scale=1.5, graphfinal}}

  \caption{Example of a Cartesian graph product.}\label{fig:cartesian graph product}
\end{figure}

With the Cartesian graph product we can explicitly describe the structure of $\vec{\tau}_3$ and $\tau_3$ for all bispanning subgraphs in the next theorem. This is most useful for composite bispanning graphs, since it allows a reduction to smaller graphs. The trivial cases $K_1$ and $G$ also work, since $\tau_3(K_1) = (\{ (\emptyset, \emptyset) \}, \emptyset) \cong K_1$.
\begin{theorem}[decomposing $\tau_3$ of composite bispanning graphs]\label{thm:composite tau decompose}
  If $G$ is a bispanning graph and $G' \subseteq G$ a bispanning subgraph, then
  $$\vec{\tau}_3(G) \cong \vec{\tau}_3(G') \times \vec{\tau}_3(G \contr G') \,,$$
  and
  $$\tau_3(G) \cong \tau_3(G') \times \tau_3(G \contr G') \,.$$
\end{theorem}
\begin{proof}
  Let $G = (V,E,\delta)$ and $G' = (V',E',\delta') \subseteq G$, and $\overline{G} = (\overline{V},\overline{E},\overline{\delta}) := G \contr G'$ the graph $G$ with $G'$ contracted into $\overline{v} \in \overline{V}$ (regard figure~\ref{fig:composite-decompose} again). Due to theorem~\ref{thm:contract bispanning subgraph}, $\overline{G}$ is bispanning, so $\tau_3(\overline{G})$ is well defined.  To show that $\vec{\tau}_3(G) \cong \vec{\tau}_3(G') \times \vec{\tau}_3(G \contr G')$, we have to provide two mappings on the vertex and arc sets, $\varphi_v$ and $\varphi_e$, that are bijections and together form an isomorphism (definition~\ref{def:isomorphism directed}).

  We first consider $\varphi_v : V_{\tau(G)} \rightarrow V_{\tau(G')} \times V_{\tau(\overline{G})}$, $(S,T) \mapsto ( (S \cap E', T \cap E') ,\, (S \cap \overline{E}, T \cap \overline{E}) )$, which is a bijection on the vertex set with inverse mapping $\psi_v : ( (S',T'), (\overline{S},\overline{T}) ) \mapsto (S' \dotcup \overline{S}, T' \dotcup \overline{T})$. The mapping $\varphi_v$ takes a pair of spanning trees of $G$ and projects them onto pairs of spanning trees of $G'$ and $\overline{G}$, while $\psi_v$ simply composes the two disjoint spanning trees to accommodate the expansion of $\overline{v}$.  Correctness of the forward mapping $\varphi_v$ was shown as lemma~\ref{lem:bispanning subtrees split}, and correctness of the reverse mapping $\psi_v$ as lemma~\ref{lem:bispanning subtrees combine}.

  Due to $E' \dotcup \overline{E} = E$, we have $(\psi_v \circ \varphi_v)(\, (S,T) \,)) = \psi_v(\, (S \cap E', T \cap E') ,\, (S \cap \overline{E}, T \cap \overline{E}) \,) = ( (S \cap E') \dotcup (S \cap \overline{E}), (T \cap E') \dotcup (T \cap \overline{E}) ) = (S,T)$ and $(\varphi_v \circ \psi_v)(\,(S',T'), (\overline{S},\overline{T}) \,) ) = \varphi_v(\, (S' \dotcup \overline{S}, T' \dotcup \overline{T}) \,) = ( ((S' \dotcup \overline{S}) \cap E', (T' \dotcup \overline{T}) \cap E' ), ((S' \dotcup \overline{S}) \cap \overline{E}, (T' \dotcup \overline{T}) \cap \overline{E}) ) = (S',T'), (\overline{S},\overline{T})$. So $\varphi_v$ is a bijection on the vertex sets of $\vec{\tau}_3(G)$ and $\vec{\tau}_3(G') \times \vec{\tau}_3(G \contr G')$, and $\tau_3(G)$ and $\tau_3(G') \times \tau_3(G \contr G')$.

  The second mapping, $\varphi_e : E_{\vec{\tau}_3(G)} \rightarrow (E_{\vec{\tau}_3(G')} \times V_{\tau(\overline{G})}) \cup (V_{\tau(G')} \times E_{\vec{\tau}_3(\overline{G})})$, is formally more complicated, but mostly already handled in lemma~\ref{lem:composite ue containment}.

  Let $(e,f,S,T) \in E_{\vec{\tau}_3(G)}$ be a $(e,f) \in S \times T$ unique $S$ edge exchange for $(S,T)$ to $(S - e + f, T + e - f)$, then there are two cases for $e$: it is either inside the bispanning subgraph, with $e \in S \cap E'$, or outside of it, with $e \in S \cap \overline{E}$. Due to lemma~\ref{lem:composite ue containment}, we known $f \in S \cap E'$ in the first case and $f \in S \cap \overline{E}$ in the second, hence, the unique exchange stays either within $G'$ or within $\overline{G}$.

  The two cases above also apply analogously to a $(e,f) \in T \times S$ unique $T$ edge exchange. We can thus formally define the mapping of unique exchange arcs as
  \[
  \varphi_e( (e,f,S,T) ) \!=\!
  \left\{
    \begin{array}{@{}l@{}}
      ( (e,f,S \cap E', T \cap E'), (S \cap \overline{E}, T \cap \overline{E}) ) \in E_{\vec{\tau}_3(G')} \times V_{\tau(\overline{G})} \\
      \qquad\text{if $(e,f) \in S \times T$ is a unique $S$ edge exchange in $(S,T)$ and $e \in S \cap E'$,} \\
      ( (S \cap E', T \cap E'), (e,f,S \cap \overline{E}, T \cap \overline{E}) ) \in V_{\tau(G')} \times E_{\vec{\tau}_3(\overline{G})} \\
      \qquad\text{if $(e,f) \in S \times T$ is a unique $S$ edge exchange in $(S,T)$ and $e \in S \cap \overline{E}$,} \\
      ( (e,f,S \cap E', T \cap E'), (S \cap \overline{E}, T \cap \overline{E}) ) \in E_{\vec{\tau}_3(G')} \times V_{\tau(\overline{G})} \\
      \qquad\text{if $(e,f) \in T \times S$ is a unique $T$ edge exchange in $(S,T)$ and $e \in T \cap E'$,} \\
      ( (S \cap E', T \cap E'), (e,f,S \cap \overline{E}, T \cap \overline{E}) ) \in V_{\tau(G')} \times E_{\vec{\tau}_3(\overline{G})} \\
      \qquad\text{if $(e,f) \in T \times S$ is a unique $T$ edge exchange in $(S,T)$ and $e \in T \cap \overline{E}$,}
    \end{array}
  \right.\!
  \]
  which maps every unique exchange edge from $\vec{\tau}_3(G)$ to $\vec{\tau}_3(G') \times \vec{\tau}_3(\overline{G})$.

  The reverse arc mapping \( \psi_e : (E_{\vec{\tau}_3(G')} \times V_{\tau(\overline{G})}) \cup (V_{\tau(G')} \times E_{\vec{\tau}_3(\overline{G})}) \rightarrow E_{\vec{\tau}_3}(G) \) is straight-forward, since the decision which subtree the unique exchange applies to can be implicitly encoded in $e \in S'$, $e \in T'$, $e \in \overline{S}$, or $e \in \overline{T}$. So we just map \( \psi_e ( (e,f,S',T'), (\overline{S},\overline{T}) ) = (e,f,S' \dotcup \overline{S}, T' \dotcup \overline{T}) \), and \( \psi_e ( (S',T'), (e,f,\overline{S},\overline{T}) ) = (e,f,S' \dotcup \overline{S}, T' \dotcup \overline{T}) \).

  To complete the isomorphism we have to show \( (\varphi_v \times \varphi_v)(\delta_{\vec{\tau}_3(G)}(e)) = \delta_{\times}(\varphi_e(e)) \) for all $e \in E_{\vec{\tau}_3(G)}$, where $\delta_\times$ is the incidence function of $\vec{\tau}_3(G') \times \vec{\tau}_3(\overline{G})$ as given by definition~\ref{def:cartesian graph product}.  So let $(e,f,S,T) \in E_{\vec{\tau}_3(G)}$ be an arc in $\vec{\tau}_3(G)$, then either $(e,f) \in (S,T)$ or $(e,f) \in (T,S)$.

  If $(e,f) \in (S,T)$ then we have a $S$ unique edge exchange for $(S,T)$, so immediately \( (\varphi_v \times \varphi_v)(\delta_{\vec{\tau}_3(G)}( (e,f,S,T) )) = (\varphi_v \times \varphi_v)( [ (S,T), (S - e + f, T + e - f) ] ) = ( [ (S \cap E', T \cap E'), (S \cap \overline{E}, T \cap \overline{E}) ], [ ((S - e + f) \cap E', (T + e - f) \cap E'), ((S - e + f) \cap \overline{E}, (T + e - f) \cap \overline{E}) ] ) \). This can simplified to $(\varphi_v \times \varphi_v)(\delta_{\vec{\tau}_3(G)}( (e,f,S,T) ))$
  \[
  = \begin{cases}
    \begin{aligned}[t]
      ( & [ (S \cap E', T \cap E'), (S \cap \overline{E}, T \cap \overline{E}) ], \\
      & [ ((S - e + f) \cap E', (T + e - f) \cap E'), (S \cap \overline{E}, T \cap \overline{E}) ] )
    \end{aligned}
    & \text{if } (e,f) \in (S \cap E',T \cap E') \,, \\
    \begin{aligned}[t]
      ( & [ (S \cap E', T \cap E'), (S \cap \overline{E}, T \cap \overline{E}) ], \\
      & [ (S \cap E', T \cap E'), ((S - e + f) \cap \overline{E}, (T + e - f) \cap \overline{E}) ] )
    \end{aligned}
    & \text{if } (e,f) \in (S \cap \overline{E},T \cap \overline{E}) \,.
  \end{cases}
  \]
  Similarly if $(e,f) \in (T,S)$, we have a $T$ unique edge exchange for $(S,T)$, and it follows $(\varphi_v \times \varphi_v)(\delta_{\vec{\tau}_3(G)}( (e,f,S,T) ))$
  \[
  = \begin{cases}
    \begin{aligned}[t]
      ( & [ (S \cap E', T \cap E'), (S \cap \overline{E}, T \cap \overline{E}) ], \\
        & [ ((S + e - f) \cap E', (T - e + f) \cap E'), (S \cap \overline{E}, T \cap \overline{E}) ] )
    \end{aligned}
    & \text{if } (e,f) \in (T \cap E', S \cap E') \,, \\
    \begin{aligned}[t]
      ( & [ (S \cap E', T \cap E'), (S \cap \overline{E}, T \cap \overline{E}) ], \\
        & [ (S \cap E', T \cap E'), ((S + e - f) \cap \overline{E}, (T - e + f) \cap \overline{E}) ] )
    \end{aligned}
    & \text{if } (e,f) \in (T \cap \overline{E}, S \cap \overline{E}) \,.
  \end{cases}
  \]
  To show the equality, we apply $\delta_\times$ as defined by \ref{def:cartesian graph product} to the four cases of $\varphi_e((e,f,S,T))$. As this is rather tedious and the cases highly symmetrical, we describe only the first case, verbosely: consider $\delta_\times ( (e,f,S \cap E', T \cap E'), (S \cap \overline{E}, T \cap \overline{E}) )$ where $(e,f) \in S \times T$ is a unique $S$ edge exchange and $e \in S \cap E'$. This is the case ``$\delta((e_1,v_2)) = ([x_{e_1},v_2],[y_{e_1},v_2])$'' from the definition, which expands in this case to \(( [ (S \cap E', T \cap E'), (S \cap \overline{E}, T \cap \overline{E}) ], [ ((S \cap E') - e + f, (T \cap E') + e - f), (S \cap \overline{E}, T \cap \overline{E}) ] )\), because $\delta'( (S \cap E', T \cap E') ) = [ (S \cap E', T \cap E'), ((S \cap E') - e + f, (T \cap E') + e - f) ]$. As $e,f \in E'$, we have $(S \cap E') - e + f = (S - e + f) \cap E'$ and similar equations, and thus the result matches the first case from $(\varphi_v \times \varphi_v)(\delta_{\vec{\tau}_3(G)}( (e,f,S,T) ))$.

  Above, we showed the isomorphism for $\vec{\tau}_3$, but due to the simple reduction to bidirectional unique edge exchanges, it is clear that the same applies for $\tau_3$.
\end{proof}

The last theorem can be used to decompose $\tau_3$ of composite bispanning graphs. As an example, consider $\tau_3(B_{4,3})$ as shown in figure~\ref{fig:tau-x4}, page~\pageref{fig:tau-x4}. It is composite with three non-trivial bispanning subgraphs: the parallel edge pair $B_2$, and two subgraphs $B_{3,2}$, which are the complements of the two degree two vertices. The $\tau_3$ graph of these three bispanning subgraphs has exactly two vertices and two parallel edges between them. The remaining graph of $B_{4,3}$ after contracting $B_2$ or $B_{3,2}$ has an exchange graph with four vertices and eight edges (see figure~\ref{fig:tau-x3}). Using theorem~\ref{thm:composite tau decompose}, we can deduce why the exchange graph in figure~\ref{fig:tau-x4} has such a regular structure:
\begin{align*}
  \tau_3(B_{4,3}) &\cong \tau_3(B_2) \times \tau_3(B_{3,1}) && \text{contract parallel edge pair $B_2$,} \\
                  &\cong \tau_3(B_{3,2}) \times \tau_3(B_2) && \text{or contract a subgraph $B_{3,2}$, and } \\
                  &\cong \tau_3(B_2) \times \tau_3(B_2) \times \tau_3(B_2) && \text{reapply \ref{thm:composite tau decompose} to the other $B_{3,2}$.}
\end{align*}

We close this section with a simple corollary to theorem~\ref{thm:composite tau decompose}:
\begin{corollary}[connectivity of $\tau_3$ of composite bispanning graphs]
  If $G$ is a bispanning graph, $G' \subseteq G$ is a bispanning subgraph, and both $\vec{\tau}_3(G')$ and $\vec{\tau}_3(G \contr G')$ are connected, then $\vec{\tau}_3(G)$ is connected. The same is true for $\tau_3$.
\end{corollary}

% ------------------------------------------------------------------------------

\subsection{Decomposing Bispanning Graphs with Vertex-Connectivity Two}\label{sec:decompose 2vconn}

With theorem~\ref{thm:composite tau decompose} of the last section, one can decompose the unique exchange graph of all composite bispanning graphs. Hence, it remains to find a way to decompose atomic bispanning graphs, which turns out to be much more difficult. In this section we focus on the case $(\operatorname{vconn}(G),\operatorname{econn}(G)) = (2,3)$ of corollary~\ref{cor:atomic-vconn-econn}, and describe a method to decompose a bispanning graph at a vertex cut of size 2. Small cuts are a natural place for such separations, however, we need to elaborate in detail how to compose both unique exchange graphs at such a cut.

For this purpose we consider a less known method to combine two graphs: the $k$-clique sum operation, which joins two graphs at two $k$-cliques. We define this operation using a graph union followed by contraction. See figure~\ref{fig:example clique sums} for an example of the $1$-, $2$-, and $3$-clique sums of two graphs, and figure~\ref{fig:2-clique composition} for a $2$-clique sum composition of two bispanning graphs.

\begin{definition}[graph union $G_1 \cup G_2$]
  The \emph<union!graph>{graph union}\symbol{graph union}{$G_1 \cup G_2$}{graph union} $G_1 \cup G_2$ of two graphs $G_1 = (V_1,E_1,\delta_1)$ and $G_2 = (V_2,E_2,\delta_2)$ with disjoint vertex and edge sets is the graph $G = (V_1 \cup V_2, E_1 \cup E_2, \delta)$ such that $\delta|_{E_1} = \delta_1$ and $\delta|_{E_2} = \delta_2$. It contains independent copies of $G_1$ and $G_2$.
\end{definition}

\begin{figure}
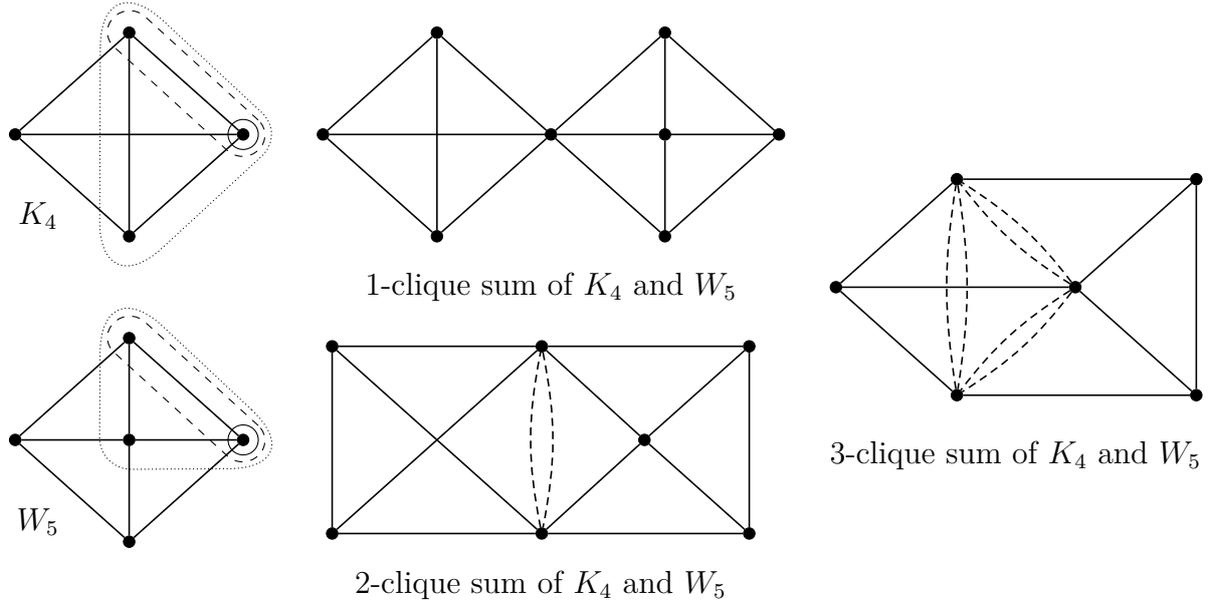
\centering
  \tikzset{every picture/.style={scale=1.5, graphfinal}}

  \caption[Examples of $1$-, $2$-, and $3$-clique sums of the graphs $K_4$ and $W_5$.]
  {Examples of $1$-, $2$-, and $3$-clique sums of the graphs $K_4$ and $W_5$ with vertices of the clique marked with solid, dashed, and dotted lines, where $E'$ can be selected from the dashed edges in the resulting clique sum.}\label{fig:example clique sums}
\end{figure}

\begin{definition}[$k$-clique sum $G_1 \oplus_k G_2$]\label{def:clique-sum}
  Given two graphs $G_1 = (V_1,E_1,\delta_1)$ and $G_2 = (V_2,E_2,\delta_2)$, each containing a $k$-clique $(V'_1,E'_1) \subseteq G_1$ and $(V'_2,E'_2) \subseteq G_2$, together with a bijection $\sigma : V'_1 \rightarrow V'_2$ and remaining edge subset $E' \subseteq E'_1 \cup E'_2$, then we call the graph
$$
G = ((G_1 \cup G_2) \contr X) \setminus \overline{E}
\quad\text{with}\quad X = \{ \{ v', \sigma(v') \} \mid v' \in V'_1 \}
\quad\text{and}\quad \overline{E} = (E'_1 \cup E'_2) \setminus E'
$$ the \emph{$k$-clique sum}\index{clique sum@$k$-clique sum}\symbol{graph sum}{$G_1 \oplus_k G_2$}{$k$-clique sum} of $G_1$ and $G_2$ with $\sigma$ and $E'$. In this operation the two $k$-cliques are contracted in the graph union, after which all edges of the joined cliques are deleted, except the $E'$ remaining edges.

  The $k$-clique sum is also simply called the \emph{$k$-sum}\index{graph!sum@$k$-sum}\index{sum!graph@graph ($k$-sum)} and written as $G_1 \oplus_k G_2$, where $\sigma$ and $E'$ are given by the context.
\end{definition}

The $k$-clique sum operation is commonly used in graph minor theory~\cite{demaine2005algorithmic,lovasz2006graph} to describe inductive graph constructions from basic building blocks. In these constructions, $E'$ is often left open and can be chosen arbitrarily, as little is said about the edges of the clique afterwards.

However, in our bispanning graph scenario the challenge when using clique sums is to keep the \emph{right edge balance}. As $1$-clique~sums are already handled by theorem~\ref{thm:bispanning composite cut-vertex}, we are mainly interested in $2$-clique sums in this section, and later in $3$-clique sum, due to their importance in matroid construction. In the first theorem, we combine bispanning graphs using a $2$-clique sum.

\begin{theorem}[$2$-clique sums of bispanning graphs are bispanning]\label{thm:2-sum bispanning}
  Any $2$-clique sum of two bispanning graphs with remaining edges $E' = \emptyset$ is bispanning.
\end{theorem}
\begin{proof}
  Let $G_1 = (V_1,E_1,\delta_1)$ and $G_2 = (V_2,E_2,\delta_2)$ be two bispanning graphs, $d_1 \in E_1$ and $d_2 \in E_2$ a $2$-clique (an edge) in $G_1$ and $G_2$ with the ends $\delta(d_1) = \{ x_1,y_1 \}$ and $\delta(d_2) = \{ x_2,y_2 \}$, $\sigma : \{ x_1, y_1 \} \rightarrow \{ x_2, y_2 \}$ one of the two possible bijections, and $S_1 \dotcup T_1$ and $S_2 \dotcup T_2$ two disjoint spanning trees of $G_1$ and $G_2$, respectively.

  Consider the graph $G = (V,E,\delta) = G_1 \oplus_2 G_2 = ((G_1 \cup G_2) \contr \{ \{ x_1, \sigma(x_1) \}, \{ x_2, \sigma(x_2) \} \} - d_1 - d_2$, which is the $2$-clique sum of $G_1$ and $G_2$ with $\sigma$ and $E' = \emptyset$, and within which $\{ x_1,\sigma(x_1) \}$ and $\{ y_1,\sigma(y_1) \}$ were contracted into $\overline{x}$ and $\overline{y}$ (see figure~\ref{fig:2-clique composition}). Obviously, $|V| = |V_1| + |V_2| - 2$ and $|E| = |E_1| + |E_2| - 2$, which yields the edge balance required by theorem~\ref{thm:bispanning edge count}, as $|E_i| = |V_i| - 2$. Let $\overline{V}_1 = V_1 - x_1 - y_1 + \overline{x} + \overline{y}$ and $\overline{V}_2 = V_2 - x_2 - y_2 + \overline{x} + \overline{y}$ be the remaining vertex sets of $G_1$ and $G_2$ in $G$ plus the contracted vertices.

Using the existing spanning trees we can directly construct two disjoint spanning trees of $G$: let without loss of generality $S_1$ contain $d_1$, and $T_2$ contain $d_2$, otherwise switch labels. We then claim that $S := (S_1 - d_1) \dotcup S_2$ and $T := T_1 \dotcup (T_2 - d_2)$ are two disjoint spanning trees of $G$. Due to $d_1$ and $d_2$ being removed, $S \cup T = E$  is clear, and $S \cap T = \emptyset$, since the spanning trees were initially disjoint. We also have $|S| = |S_1| - 1 + |S_2| = |V_1| - 2 + |V_2| - 1 = |V| - 1$ and $|T| = |T_1| + |T_2| - 1 = |V_1| - 1 + |V_2| - 2 = |V| - 1$, so due to theorem~\ref{thm:tree connected} it remains to show that the two edge sets are connected. Since $d_1$ is a bridge in $G_1[S_1]$ (theorem~\ref{thm:tree bridge}), either $\overline{x}$ or $\overline{y}$ cannot be connected to any other vertex in $G[\overline{V}_1]$ using edges only from $S_1 - d_1$. However, since $S_2$ connects any pair of vertices in $G[\overline{V}_2]$, including $\overline{x}$ and $\overline{y}$, any pair of vertices in $G$ can be connected using $S$. The argument analogously applies to $d_2 \in T_2$, $G_2[T_2]$, and $T_1$, and shows that $T$ spans $G$. Thus $S$ and $T$ are disjoint spanning trees of $G$.
\end{proof}

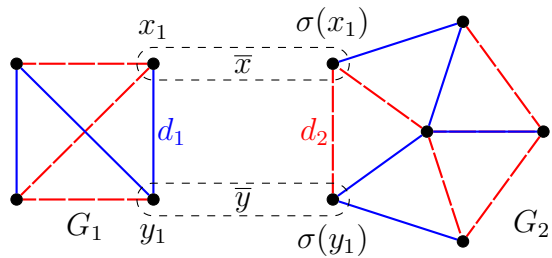
\begin{figure}\centering
  \begin{tikzpicture}[graphfinal,scale=0.9]

    \begin{scope}
      \node (x1) [odot] at (-1,-1) {1};
      \node (x2) [odot] at (-1,+1) {2};
      \node (x3) [odot] at (+1,-1) {3};
      \node (x4) [odot] at (+1,+1) {4};

      \draw[B] (x1) -- (x2);
      \draw[R] (x1) -- (x3);
      \draw[R] (x1) -- (x4);
      \draw[B] (x2) -- (x3);
      \draw[R] (x2) -- (x4);
      \draw[B] (x3) -- node[right, inner sep=1pt] {$d_1$} (x4);

      \node at (0,-1.4) {$G_1$};
    \end{scope}

    \begin{scope}[xshift=5cm,scale={1/sin(180/5)}]
      \node (y0) [odot] at (0,0) {0};
      \node (y1) [odot] at (0:1) {1};
      \node (y2) [odot] at (1*72:1) {2};
      \node (y3) [odot] at (2*72:1) {3};
      \node (y4) [odot] at (3*72:1) {4};
      \node (y5) [odot] at (4*72:1) {5};

      \draw[R] (y0) -- (y1);
      \draw[B] (y0) -- (y2);
      \draw[R] (y0) -- (y3);
      \draw[B] (y0) -- (y4);
      \draw[R] (y0) -- (y5);
      \draw[B] (y0) -- (y1);
      \draw[R] (y1) -- (y2);
      \draw[B] (y2) -- (y3);
      \draw[R] (y3) -- node[left, inner sep=1pt] {$d_2$} (y4);
      \draw[B] (y4) -- (y5);
      \draw[R] (y5) -- (y1);

      \node at (0.9,-0.8) {$G_2$};
    \end{scope}

    \foreach \i in {1,...,4} {
      \node (nx\i) at (x\i) [circle,inner sep=6pt] {};
    }
    \foreach \i in {0,...,5} {
      \node (ny\i) at (y\i) [circle,inner sep=6pt] {};
    }

    \draw [dashed,rounded corners=6pt] (nx3.135) rectangle (ny4.315);
    \draw [dashed,rounded corners=6pt] (nx4.135) rectangle (ny3.315);

    \node [above=1mm of x4] {$x_1$};
    \node [above=1mm of y3] {$\sigma(x_1)$};

    \node [below=1mm of x3] {$y_1$};
    \node [below=1mm of y4] {$\sigma(y_1)$};

    \node at ($(nx4)!0.5!(ny3)$) {$\overline{x}$};
    \node at ($(nx3)!0.5!(ny4)$) {$\overline{y}$};

  \end{tikzpicture}
  \caption{Example composition scheme of a $2$-clique sum of $G_1 = K_4$ and $G_2 = W_6$.}\label{fig:2-clique composition}
\end{figure}

After the previous theorem it is clear that one can combine two bispanning graphs at any edges using a $2$-clique sum operation. However, can two known disjoint spanning tree pairs be kept during the operation? The previous proof answers this affirmatively, if $d_1 \in S_1$ and $d_2 \in T_2$. It uses the trick ``otherwise switch labels'' to gain generality. Hence, if one wants to combine two given bispanning graphs, the joined edges have to be of \emph{different colors}: otherwise one has to invert one of the tree's colors.

After composing two bispanning graphs in the previous theorem, we now focus on \emph{decomposing} bispanning graphs with vertex-connectivity two. Obviously, any general graph with $\operatorname{vconn}(G) = 2$ can be represented as a $2$-clique sum of two subgraphs. However, for bispanning graphs we have to decide how to handle the edges at the \emph{join seam} (see figure~\ref{fig:2-vconn decomposition}) to retrain the property of being bispanning.

This question is surprisingly difficult, because in general it depends on the separated components which the $k \in \{ 0,1,2 \}$ edges at the seam belong to, since the components have to maintain edge balance. If there are $k = 2$ edges at the seam, then the situation is clear: the graph can be composed via $2$-clique sum from two graphs at the two parallel edge pairs. The case $k = 1$ is most difficult, and we will only give a short proof later. Due to the difficulty of $k = 1$, it is curiously coincidental that only $k = 0$ is possible for \emph{atomic} bispanning graphs with vertex-connectivity two, which is the class we are most interested in.

\begin{figure}
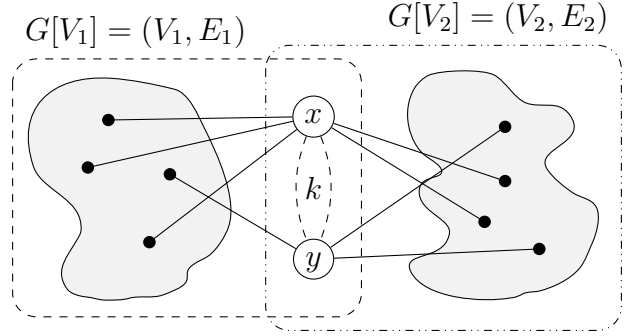
\centering

  \caption[Decomposition scheme of a bispanning graph with vertex-connectivity two.]{Decomposition scheme of a bispanning graph with vertex-connectivity two at the vertex-cut $\{x,y\}$.}\label{fig:2-vconn decomposition}
\end{figure}

\begin{theorem}[{\textls*[-15]{decomposing atomic bispanning graphs with vertex-connectivity two}}]\label{thm:2-vconn decompose}
  If $G = (V,E)$ is an atomic bispanning graph $G$ with $\operatorname{vconn}(G) = 2$ and vertex-cut $\{x,y\} \subseteq V$, then $G$ contains no edge $\{x,y\}$ and $G$ is the 2-clique sum of two simple bispanning graphs.
\end{theorem}
\begin{proof}
  Let $V_1,V_2 \subseteq V$ be two subsets with $V_1 \cup V_2 = V$, $V_1 \cap V_2 = \{ x,y \}$, and $V_1 \neq V$ or $V_2 \neq V$, such that $G[V_1] = (V_1,E_1)$ and $G[V_2] = (V_2,E_2)$ are connected (see figure~\ref{fig:2-vconn decomposition}, which shows the structure of this decomposition). These two subgraphs contain the two or more components $G \setminus \{ x,y \}$ decomposes into, together with $x$, $y$, and the connecting edges.

  We now have to consider the edges between $x$ and $y$: since $G$ is simple, there can be zero or one edges with incidence $\{x,y\}$, let $k \in \{ 0,1 \}$ be this number. As we include $x$, $y$ and the $k$ edges between them in both subgraphs, we have $|V_1| + |V_2| = |V| + 2$ and $|E_1| + |E_2| = |E| + k$, from which easily $|E| = 2 |V| - 2 = 2 (|V_1| + |V_2| - 2) - 2$ follows.

  As $G$ is atomic, theorem~\ref{thm:atomic nash-williams} can be applied to the partitions $P_1 = \{ \{v\} \mid v \in V_1 \setminus \{ x,y \} \} \cup \{ V_2 \}$ and $P_2 = \{ V_1 \} \cup \{ \{v\} \mid v \in V_2 \setminus \{ x,y \} \}$, which isolate each vertex of the corresponding subgraph except $x$ and $y$, and group the other subgraph with the vertex-cut. The theorem yields $|E_1| - k \geq 2 (|V_1| - 2) + 1$ and $|E_2| - k \geq 2 (|V_2| - 2) + 1$ for the two subgraphs, since all edges except the $k$ edges $\{x,y\}$ are counted. From $|E_2| - k \geq 2 (|V_2| - 2) + 1$, and substitutions with the equations between vertex and edge sets above, follows $|E_1| \leq 2 (|V_1| - 2) + 1$, and analogously $|E_2| \leq 2 (|V_2| - 2) + 1$ (see the proof of theorem~\ref{thm:2-vconn decompose more} for details). Combined, we thus have $|E_i| \leq 2 (|V_i| - 2) + 1$ and $|E_i| \geq 2 (|V_i| - 2) + 1 + k$ for both $i = 1,2$.

  Obviously, $k = 1$ is impossible, thus $k = 0$ and $G$ contains no edge $\{x,y\}$. Hence, we get the equations $|E_1| = 2 |V_1| - 3$ and $|E_2| = 2 |V_2| - 3$. It remains to show that $G_1 := G[V_1] + \{x,y\}$ and $G_2 := G[V_2] + \{x,y\}$ are bispanning graphs. Let $P_1$ be a partition of $V_1$, then we can extend $P_1$ to the whole of $V$ as $P = P_1 \cup \{ \{v\} \mid v \in V \setminus V_1 \}$, which has $|P| = |P_1| + |V_2| - 2$ members. Combining theorem~\ref{thm:atomic nash-williams} for $G$ with $|E_P| = |E_{P_1}| + |E_2|$, we get $|E_{P_1}| = |E_P| - |E_2| \geq 2 (|P| - 1) + 1 - |E_2| = 2 (|P_1| - 1) + 2 (|V_2| - 2) + 1 - |E_2| = 2 (|P_1| - 1)$, as $|E_2| = 2 |V_2| - 3$, and thus theorem~\ref{thm:nash-williams} guarantees that $G[V_1] + \{x,y\}$ is a bispanning subgraph. The same applies to $G[V_1] + \{x,y\}$. Combining $G_1$ and $G_2$ at their $2$-cliques $\{x,y\}$ the $2$-clique sum with $\sigma = \operatorname{id}_{\{x,y\}}$ and $E' = 0$ yields exactly $G$.
\end{proof}

Decomposing \emph{composite} bispanning graphs with vertex-connectivity two is much more difficult. Even worse: it is not always possible. The smallest counter example is $B_{4,4}$ from figure~\ref{fig:small-bispanning}, as a plain triangle is obviously not a bispanning graph. $B_{4,4}$ contains parallel edges, but counter examples without parallel edges exist as well. Remarkably, one \emph{can} always decompose a composite bispanning graph with vertex-connectivity two, if at least one edge is in the vertex cut. As this result is not as important as the same one for atomic graphs, we give only a short proof.
\begin{theorem}[{\textls[-5]{decomposing more bispanning graphs with vertex-connectivity two}}]\label{thm:2-vconn decompose more}
  Given a bispanning graph $G$ with $\operatorname{vconn}(G) = 2$ where $\{x,y\}$ is a vertex-cut and $G$ contains at least one edge incident to $x$ and $y$, then $G$ is the 2-clique sum of two bispanning graphs.
\end{theorem}
\begin{proof}
  Let $G = (V,E,\delta)$ be a bispanning graph with $\operatorname{vconn}(G) = 2$, and $\{ x,y \} \subseteq V$ a vertex set for which $G - x - y$ is not connected (this requires $x \neq y$). Let $V_1,V_2 \subseteq V$ be two subsets with $V_1 \cup V_2 = V$, $V_1 \cap V_2 = \{ x,y \}$, and $V_1 \neq V$ or $V_2 \neq V$, such that $G[V_1] = (V_1,E_1,\delta_1)$ and $G[V_2] = (V_2,E_2,\delta_2)$ are connected (see figure~\ref{fig:2-vconn decomposition}).  Let $k$ be the number of edges $e$ with $\delta(e) = \{x,y\}$. As we include $x$, $y$ and the $k$ edges between them in both subgraphs, we have $|V_1| + |V_2| = |V| + 2$ and $|E_1| + |E_2| = |E| + k$, from which easily $|E| = 2 |V| - 2 = 2 (|V_1| + |V_2| - 2) - 2$ follows.

  As $G$ is a bispanning graph, theorem~\ref{thm:nash-williams} can be applied to the partitions $P_1 = \{ \{x\} \mid x \in V_1 \setminus \{ v,w \} \} \cup \{ V_2 \}$ and $P_2 = \{ V_1 \} \cup \{ \{v\} \mid v \in V_2 \setminus \{ v,w \} \}$, which isolate each vertex of a subgraph except $v$ and $w$, and clump the other subgraph with the vertex-cut. The theorem this time yields only $|E_1| - k \geq 2 (|V_1| - 2)$ and $|E_2| - k \geq 2 (|V_2| - 2)$ for the two subgraphs, as they are not necessarily atomic. Using the equations for vertex and edge balance we get
  \begin{align*}
    &|E_2| - k \geq 2 (|V_2| - 2) \quad\Rightarrow\quad |E| + k - |E_1| - k \geq 2 (|V_2| - 2) \\
    &\Rightarrow\quad 2 (|V_1| + |V_2| - 2) - 2 - |E_1| \geq 2 (|V_2| - 2) \quad \Rightarrow\quad 2 |V_1| - 2 - |E_1| \geq 0 \\
    &\Rightarrow\quad |E_1| \leq 2|V_1| - 2 \,, \qquad\text{so we have}\qquad 2|V_1| - 4 + k \leq |E_1| \leq 2|V_1| - 2 \,.
  \end{align*}
  Analogously, we get the same result for $|E_2|$ and $|V_2|$. Next we need to find out which cases provide valid edge balances for $G[V_1] + \{x,y\}$ and $G[V_2] + \{x,y\}$, which are the prospective components for the $2$-clique sum. Due to $E_1$ and $E_2$ both including the $k$ edges, we have to remove $k$ edges and reach $|E_i| = 2|V_i| - 3$ prior to adding the two edges which are used to join the $2$-clique sum.

  One can now see that for $k = 0$, we have three cases $(|E_1|,|E_2|) \in \{ (2|V_1| - 4, 2|V_2| - 2),\; (2|V_1| - 3, 2|V_2| - 3),\; (2|V_1| - 2, 2|V_2| - 4) \}$, though the first and last cases can be handled symmetrically. While the case $(|E_1|,|E_2|) = (2|V_1| - 3, 2|V_2| - 3)$ leads to a valid edge balance for $G[V_1] + \{x,y\}$ and $G[V_2] + \{x,y\}$, the other cases requires $(|E_1|,|E_2|) = (2|V_1| - 4, 2|V_2| - 2)$, which prohibits that $G[V_1] + \{x,y\}$ and $G[V_2] + \{x,y\}$ are bispanning graphs. These two subcases for $k = 0$ are the reason why only some composite bispanning graphs with vertex-connectivity two and no edge between $x$ and $y$ can be decomposed.

  For $k = 1$, we have two symmetric cases $(|E_1|,|E_2|) \in \{ (2|V_1| - 3, 2|V_2| - 2),\; (2|V_1| - 2, 2|V_2| - 3)$. The one extra edge, hence has to be removed from the larger edge set to correct the edge balance in both graphs $G[V_1] + \{x,y\}$ and $G[V_2] + \{x,y\}$. We omit the proof that both graphs are bispanning, as it is very similar to the proof of theorem~\ref{thm:2-vconn decompose}.

  For $k = 2$, there is only the case $(|E_1|,|E_2|) = (2|V_1| - 2, 2|V_2| - 2)$, which requires one of the two additional edges to be removed from each of the graphs $G[V_1] + \{x,y\}$ and $G[V_2] + \{x,y\}$.
\end{proof}

\begin{figure}
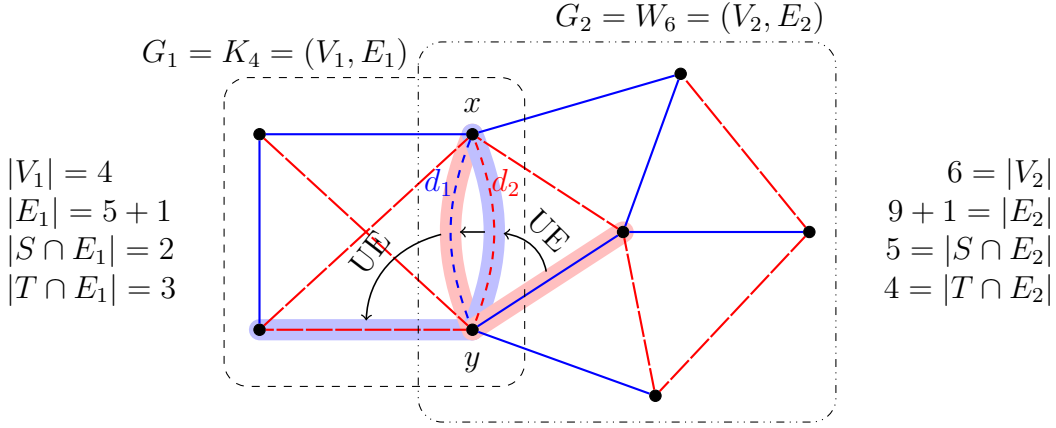
\centering
  \tikzset{every picture/.style={scale=3.5}}

  \caption[Composition of a unique exchange in $G = (V,E) = K_4 \oplus_2 W_6$.]
  {Composition of a unique exchange in $G = (V,E) = K_4 \oplus_2 W_6$ ($|V| = 8$, $|E| = 14$).}\label{fig:2-clique UE-decompose}
\end{figure}

Consider the atomic bispanning graph $G$ with $\operatorname{vconn}(G) = 2$ in figure~\ref{fig:2-clique UE-decompose}, which is a $2$-clique sum of $K_4$ and $W_6$. Of the $6$ edges in $K_4$ only $5$ remain in $G$, and of the $10$ edges in $W_6$ only $9$ remain. Due to theorem~\ref{thm:2-sum bispanning}, the $2$-clique sum composition of the two graphs is a bispanning graph. On the other hand, since $G$ is atomic, there is no edge at a vertex cut of size two, and $G$ can be decomposed at the cut, as proved in theorem~\ref{thm:2-vconn decompose}.

All previous theorems only considered whether the composition or decomposition remains bispanning. However, the goal of this section is to consider how to construct the unique exchange graph $\vec{\tau}_3(G)$ from a $2$-clique sum decomposition $G = G_1 \oplus_2 G_2$. Hence, we have to map pairs of disjoint bispanning trees $(S_1,T_1)$ from $G_1$ and $(S_2,T_2)$ from $G_2$ to $G$, and vice versa. In the proof of theorem~\ref{thm:2-vconn decompose}, we already showed that two pairs of disjoint trees can be combined, if and only if the edges in the join seam are in different sets. In the following theorem we extend this to show a full bijection from $V_{\tau(G)}$ to this restricted set of pairs of $(S_1,T_1)$ and $(S_2,T_2)$, which will be called $V_{\eta_{d_1,d_2}(G_1,G_2)}$.

Consider figure~\ref{fig:2-clique UE-decompose} again for the intuition behind the mapping. In a sense, the edge in the join seam $\{x,y\}$ is expanded by the $2$-clique sum into the adjoined graph. For example the red edge in $W_6$ in the figure is expanded into the red tree of $K_4$. This tree must connect $x$ and $y$ inside $K_4$, otherwise the join result is no longer a tree. Correspondingly, the blue tree must not connect the vertices $x$ and $y$ in $K_4$, otherwise it creates a cycle. The difficulty in the following theorem is to show that this also works when \emph{decomposing} any pair of disjoint spanning trees of $G$ into disjoint spanning trees of the subgraphs.

\begin{lemma}[bijection of disjoint spanning trees in $2$-clique sum decomposition]\label{lem:2-vconn tree mapping}
  If $G = (V,E)$ is an atomic bispanning graph with $\operatorname{vconn}(G) = 2$, which due to theorem~\ref{thm:2-vconn decompose} can be decomposed into the $2$-clique sum $G = G_1 \oplus_2 G_2$ of two simple bispanning graphs $G_1 = (V_1,E_1)$ and $G_2 = (V_2,E_2)$ at the edges $d_1 \in E_1$ and $d_2 \in E_2$ with $\sigma = \operatorname{id}$ and $E' = \emptyset$, then there is a bijection for the the following mapping:
  $$
  \varphi_v : V_{\tau(G)} \rightarrow V_{\eta_{d_1,d_2}(G_1,G_2)} :=
  \begin{aligned}[t]
    \{ & ((S_1,T_1),(S_2,T_2)) \in V_{\tau(G_1)} \times V_{\tau(G_2)} \mid \\
       & \text{not } ( (d_1 \in S_1 \text{ and } d_2 \in S_2) \text{ or } (d_1 \in T_1 \text{ and } d_2 \in T_2) ) \} \,.
  \end{aligned}
  $$
\end{lemma}
\begin{proof}
  We regard how to map vertices $\varphi_v : V_{\tau(G)} \rightarrow V_\eta \subseteq V_{\tau(G_1)} \times V_{\tau(G_2)}$, where $V_\eta$ is short for $V_{\eta_{d_1,d_2}(G_1,G_2)}$. To map a pair of disjoint spanning trees $(S,T)$ from $G$ to $G_1$ and $G_2$, we obviously have to restrict them to the corresponding subgraph and add $d_1$ and $d_2$ to the appropriate tree, as already discussed in the proof of theorem~\ref{thm:2-sum bispanning}.  However, as before in the proof to~\ref{thm:composite tau decompose}, it is not immediately clear that for \emph{every pair} $(S,T) \in V_{\tau(G)}$ the restrictions of $S$ and $T$ to $G_1$ and $G_2$ are spanning trees and how $d_1$ and $d_2$ are added appropriately.

  Let $(S,T) \in V_{\tau(G)}$ be two disjoint spanning trees. Due to the $2$-clique sum, $E = (E_1 - d_1) \dotcup (E_2 - d_2)$, and we have $|E_1 - d_1| = 2 |V_1| - 3 = |S \cap E_1| + |T \cap E_1|$, which is an uneven integer. Thus either $|S \cap E_1|$ or $|T \cap E_1|$ is smaller, and we add $d_1$ to this smaller intersection: let $(S',T') \in \{ ((S \cap E_1) + d_1, T \cap E_1), (S \cap E_1, (T \cap E_1) + d_1) \}$ as appropriate. As $S$ and $T$ are disjoint spanning trees, the sum $|S'| + |T'| = |E_1| = 2 |V_1| - 2$, since they contain all edges of $G_1$. With the same edge balance argument as in the proof of $\varphi_v$ in theorem~\ref{thm:composite tau decompose}, we gain $|S'| = |V_1| - 1 = |T'|$, since the lack of an edge in one supposed spanning tree set implies a cycle in the other edge set in the original graph, where they are disjoint spanning trees. Thus $(S',T')$ are disjoint spanning trees of $G_1$. Analogously, the same argument applies to $E_2$ and $G_2$.

  As $|S \cap E_1| + |S \cap E_2| = |V| - 1 = |T \cap E_1| + |T \cap E_2|$, it is impossible that both intersections on $S$, $|S \cap E_1|$ and $|S \cap E_2|$, are smaller than the corresponding ones on $T$, $|T \cap E_1|$ and $|T \cap E_2|$, and vice versa. So only one of the two restrictions of $S$ is smaller, never both, and we can formally define the vertex mapping
  $$
  \varphi_v( (S,T) ) =
  \begin{cases}
    \begin{aligned} ( & [(S \cap E_1) + d_1, (T \cap E_1)], \\[-2pt] & [(S \cap E_2), (T \cap E_2) + d_2] ) \end{aligned} &
    \begin{aligned} \text{if } ( & |S \cap E_1| = |V_1| - 2 \text{\,, } |T \cap E_1| = |V_1| - 1 \text{\,,} \\[-2pt]
                                 & |S \cap E_2| = |V_2| - 1 \text{\,, } |T \cap E_2| = |V_2| - 2 ) \text{\,, or} \end{aligned} \\[3\medskipamount]
    \begin{aligned} ( & [(S \cap E_1), (T \cap E_1) + d_1], \\[-2pt] & [(S \cap E_2) + d_2, (T \cap E_2)] ) \end{aligned} &
    \begin{aligned} \text{if } ( & |S \cap E_1| = |V_1| - 1 \text{\,, } |T \cap E_1| = |V_1| - 2 \text{\,,} \\[-2pt]
                                 & |S \cap E_2| = |V_2| - 2 \text{\,, } |T \cap E_2| = |V_2| - 1) \,. \end{aligned}
  \end{cases}
  $$
  The reverse vertex mapping $\psi_v : V_\eta \rightarrow V_{\tau(G)}$ is easier, since the definition of $V_\eta$ forbids $(d_1 \in S_1 \text{ and } d_2 \in S_1)$ or $(d_1 \in T_2 \text{ and } d_2 \in T_2)$, so either ($d_1 \in S_1$ and $d_2 \in T_2$) or ($d_1 \in T_2$ and $d_2 \in S_1$). In both cases the mapping
  $$\psi_v( [(S_1,T_1), (S_2,T_2)] ) = ((S_1 \dotcup S_2) \setminus \{ d_1,d_2 \}, (T_1 \dotcup T_2) \setminus \{ d_1,d_2 \})$$
  yields a pair of disjoint spanning trees of $G$, since even though $d_1$ or $d_2$ is a bridge in the corresponding tree-graph, deletion of $d_1$ or $d_2$ does not disconnect the union, because the other union-joined tree cannot contain $d_1$ or $d_2$, and thus connects the ends of $d_1$ or $d_2$.  Clearly, $\varphi_v$ and $\psi_v$ are inverse to each other.
\end{proof}

After having shown a bijection of disjoint spanning trees of $G$ to a subset of trees of $G_1 \oplus_2 G_2$, we can consider in detail what happens with fundamental cycles and cuts during this composition. In the next theorem we will show how cycles/cuts of the composed graph $G$ can be restricted (or projected) down to cycles/cuts of $G_1$ or $G_2$. This is easier than showing how cycles/cuts can be expanded at the join seam, and sufficient for the following theorems.

\begin{lemma}[projection of cut and cycle in $2$-clique sum decomposition]\label{lem:2-vconn cutcycle mapping}
  Let $G = (V,E)$ be an atomic bispanning graph with $\operatorname{vconn}(G) = 2$, which due to theorem~\ref{thm:2-vconn decompose} can be decomposed into the $2$-clique sum $G = G_1 \oplus_2 G_2$ of two simple bispanning graphs $G_1 = (V_1,E_1)$ and $G_2 = (V_2,E_2)$ at the edges $d_1 \in E_1$ and $d_2 \in E_2$ with $\sigma = \operatorname{id}$ and $E' = \emptyset$. If $S \dotcup T = E$ is a pair of disjoint spanning trees, and $e \in S \cap E_1$ an edge, then
  \begin{enumerate}
  \item $D_{G_1}(S,e) = (D_G(S,e) \cap E_1) \;\,\phantom{+ d_1}$ if $D_G(S,e) \cap E_2 = \emptyset$, and
  \item $D_{G_1}(S,e) = (D_G(S,e) \cap E_1) + d_1$ if $D_G(S,e) \cap E_2 \neq \emptyset$; likewise
  \item $C_{G_1}(T,e) = (C_G(T,e) \cap E_1) \;\,\phantom{+ d_1}$ if $C_G(T,e) \cap E_2 = \emptyset$, and
  \item $C_{G_1}(T,e) = (C_G(T,e) \cap E_1) + d_1$ if $C_G(T,e) \cap E_2 \neq \emptyset$.
  \end{enumerate}
  The same holds analogously for $e \in T \cap E_1$, and for $e \in S \cap E_2$ or $e \in T \cap E_2$ with $G_1$ and $G_2$, $E_1$ and $E_2$, and $d_1$ and $d_2$ interchanged.
\end{lemma}
\begin{proof}
  Clearly, $E_1 \cap E_2 = \emptyset$, and $E_1 \cup E_2 = E \dotcup \{ d_1, d_2 \}$. Let $d_1 = d_2 = \{ x,y \}$ be the two vertices of the $2$-vertex cut.
  Cases (i) and (iii) are trivial: if the cut $D_G(S,e)$ or cycle $C_G(T,e)$ contains no edges from $E_2$, then it is fully contained in the edge set $E_1 \setminus \{ d_1 \}$. Hence, the cut or cycle remains fully in $G_1$ and restriction to $E_1$ removes no edges. In cases (ii) and (iv), intuitively, one has to replace the cycle and cut edges outside of $G_1$ by the edge $d_1$ as representative: in case (ii) it represents all cut edges in $E_2$ and in (iv) it represents the remaining cycle edges in $E_2$.

  As $S$ is a spanning tree, $G[S] - e$ contains the edge cut $D_G(S,e)$. If there exists an edge $c \in D_G(S,e) \cap E_2 \neq \emptyset$, then the two vertices of the vertex cut $x$ and $y$ are disconnected in $G[S] - e$, since one end of $c$ (and likewise $e$) is connected to $x$ and the other to $y$, but neither to both. Consequentially, $d_1 \in T$, since if $d_1 = \{x,y\} \in S$ were in the other tree, then there would be a cycle in $G[S \cap E_1] + d_1$ containing $d_1$ and $e$. As $x$ and $y$ are disconnected in $G[S] - e$, they also are disconnected in $G[S \cap E_1] - e$. Hence, $d_1 \in D_{G_1}(S,e)$, and in total $D_{G_1}(S,e) = (D_G(S,e) \cap E_1) + d_1$.

  Likewise, if $c \in C_G(T,e) \cap E_2 \neq \emptyset$, then the cycle in $G[T] + e$ has to contain both $x$ and $y$ to reach $c$, as it can pass every vertex at most once. Consequentially, $d_2 \in S$, as otherwise there would be a cycle in $G[T \cap E_2] + d_2$ which runs though $d_2$ and $e$. As $d_2$ and $d_1$ must be from different trees, $d_1 \in T$. Hence, $d_1 \in C_{G_1}(T,e)$, as $d_1 \in T$ and both ends $x$ and $y$ are connected in $G[T \cap E_1] + e$, and in total $C_{G_1}(T,e) = (C_G(T,e) \cap E_1) + d_1$.
\end{proof}

Goal of this section is to compose the unique exchange graph $\tau_3(G)$ for $G = G_1 \oplus_2 G_2$ from the unique exchange graphs $\tau_3(G_1)$ and $\tau_3(G_2)$. This will be shown in the following main theorem. Intuitively, we can start constructing $\tau_3(G)$ from $\tau_3(G_1) \times \tau_3(G_2)$, since unique exchanges in either subgraph $G_1$ and $G_2$ of $G$ are independent (as shown by theorem~\ref{lem:2-vconn cutcycle mapping}), \emph{unless} they involve the edges $d_1$ and $d_2$ at the join seam. The independent unique exchanges can be taken in any order, which is expressed in the Cartesian graph product $\times$ by the way it multiplies possible paths in the exchange graphs. It only remains to determine what to do with unique exchanges involving $d_1$ and $d_2$. Consider again figure~\ref{fig:2-clique UE-decompose}: the idea to handle unique exchanges involving the join seam is to \emph{link} these together whenever possible: if a unique edge exchange $(e,d_2)$ forces $d_2$ to swap, then $d_1$ must be swapped and this $(d_1,f)$ must itself be a unique exchange. One can regard $(d_2,d_1)$ as an implicit third edge exchange inside the join seam, which is always forcible. It is important to note, that only when both $(e,d_2)$ and $(d_1,f)$ with $e \in E_2$ and $f \in E_1$ (or analogously $(e,d_1)$ and $(d_2,f)$ with $e \in E_1$ and $f \in E_2$) are unique exchanges then they can be linked to a unique exchange in $G$. If either is not, then the link must not be made. This removes many possible transitions, e.g., where $d_1$ is not a unique exchange candidate.

\begin{theorem}[composing $\vec{\tau}_3(G)$ of atomic bispanning graphs with $\operatorname{vconn}(G) = 2$]\label{thm:2-vconn tau decompose}
  If $G$ is an atomic bispanning graph with $\operatorname{vconn}(G) = 2$, which due to theorem~\ref{thm:2-vconn decompose} can be decomposed into the $2$-clique sum $G = G_1 \oplus_2 G_2$ of the two simple bispanning graphs $G_1 = (V_1,E_1)$ and $G_2 = (V_2,E_2)$ at the edges $d_1 \in E_1$ and $d_2 \in E_2$ with $\sigma = \operatorname{id}$ and $E' = \emptyset$, then
  $$\vec{\tau}_3(G) \cong \eta_{d_1,d_2}(\, \vec{\tau}_3(G_1), \vec{\tau}_3(G_2) \,) \,,$$
  where $\eta_{d_1,d_2}$ joins the two exchange graphs $\vec{\tau}_3(G_1) = (V_{\tau(G_1)}, E_{\vec{\tau}_3(G_1)}, \delta_{\vec{\tau}_3(G_1)})$ and $\vec{\tau}_3(G_2) = (V_{\tau(G_2)}, E_{\vec{\tau}_3(G_2)}, \delta_{\vec{\tau}_3(G_2)})$ into the exchange graph $(V_\eta,E_\eta,\delta_\eta)$ with
  \begin{enumerate}
  \item vertex set
    \(
    \begin{aligned}[t]
      V_\eta = \{ & ((S_1,T_1),(S_2,T_2)) \in V_{\tau(G_1)} \times V_{\tau(G_2)} \mid \\
      & \text{not } ( (d_1 \in S_1 \text{ and } d_2 \in S_2) \text{ or } (d_1 \in T_1 \text{ and } d_2 \in T_2) ) \} \,,
    \end{aligned}
    \)
  \item
    \(
    \begin{aligned}[t]
      \text{arc set }
      E_\eta = \{ &(e,f,(S_1,T_1),(S_2,T_2)) \mid \\
      \text{[case (a)]\qquad} & (e,f,S_1,T_1) \in E_{\vec{\tau}_3(G_1)} \,, (S_2,T_2) \in V_{\tau(G_2)} \text{ and not } (e = d_1 \text{ or } f = d_1) \,, \\
      \text{[case (b)]\qquad} & (e,f,S_2,T_2) \in E_{\vec{\tau}_3(G_2)} \,, (S_1,T_1) \in V_{\tau(G_1)} \text{ and not } (e = d_2 \text{ or } f = d_2) \,, \\
      \text{[case (c)]\qquad} & (\, (e,d_1,S_1,T_1) \in E_{\vec{\tau}_3(G_1)} \text{ and } (d_2,f,S_2,T_2) \in E_{\vec{\tau}_3(G_2)} \,) \,, \text{ or} \\
      \text{[case (d)]\qquad} & (\, (e,d_2,S_2,T_2) \in E_{\vec{\tau}_3(G_2)} \text{ and } (d_1,f,S_1,T_1) \in E_{\vec{\tau}_3(G_1)} \,) \, \} \,, \\
    \end{aligned}
    \)
  \item and incidence $\delta_\eta(\, (e,f,(S_1,T_1),(S_2,T_2)) \,)$ \\
    \( \null\hspace{-6mm}
    = \begin{cases}
      (\, [(S_1,T_1),(S_2,T_2)] ,\, [(S_1 - e + f,T_1 + e - f),(S_2,T_2)] \,) & \text{if } (e,f) \in (S_1,T_1) \text{ [case (a)]} \,, \\
      (\, [(S_1,T_1),(S_2,T_2)] ,\, [(S_1 + e - f,T_1 - e + f),(S_2,T_2)] \,) & \text{if } (e,f) \in (T_1,S_1) \text{ [case (a)]} \,, \\
      (\, [(S_1,T_1),(S_2,T_2)] ,\, [(S_1,T_1),(S_2 - e + f,T_2 + e - f)] \,) & \text{if } (e,f) \in (S_2,T_2) \text{ [case (b)]} \,, \\
      (\, [(S_1,T_1),(S_2,T_2)] ,\, [(S_1,T_1),(S_2 + e - f,T_2 - e + f)] \,) & \text{if } (e,f) \in (T_2,S_2) \text{ [case (b)]} \,, \\
      (\, [(S_1,T_1),(S_2,T_2)] ,\, [(S_1 - e,T_1 + e),(S_2 + f,T_2 - f)] \,) & \text{if } (e,f) \in (S_1,T_2) \text{ [case (c)]} \,, \\
      (\, [(S_1,T_1),(S_2,T_2)] ,\, [(S_1 + e,T_1 - e),(S_2 - f,T_2 + f)] \,) & \text{if } (e,f) \in (T_1,S_2) \text{ [case (c)]} \,, \\
      (\, [(S_1,T_1),(S_2,T_2)] ,\, [(S_1 + f,T_1 - f),(S_2 - e,T_2 + e)] \,) & \text{if } (e,f) \in (S_2,T_1) \text{ [case (d)]} \,, \\
      (\, [(S_1,T_1),(S_2,T_2)] ,\, [(S_1 - f,T_1 + f),(S_2 + e,T_2 - e)] \,) & \text{if } (e,f) \in (T_2,S_1) \text{ [case (d)]} \,. \\
    \end{cases}
    \)
  \end{enumerate}
\end{theorem}
\begin{proof}
  To show the isomorphism, we need to show that bijections exist for the two mappings $\varphi_v : V_{\tau(G)} \rightarrow V_\eta$ and $\varphi_e : E_{\vec{\tau}(G)} \rightarrow E_\eta$, and that they are compatible to each other as required by definition~\ref{def:isomorphism directed}.
  Let $\{x,y\} \subseteq V$ be the vertex-cut, hence $d_1 = d_2 = \{x,y\}$ in the graph $G_1$ and $G_2$.
  The previous lemma~\ref{lem:2-vconn tree mapping} already showed the necessary bijection for $\varphi_v$, which splits or joins disjoint spanning trees at $d_1$ and $d_2$, provided they are in appropriately different spanning trees of $V_{\tau(G_1)} \times V_{\tau(G_2)}$.

  We thus regard the forward edge mapping $\varphi_e : E_{\vec{\tau}(G)} \rightarrow E_\eta$. Let $(e,f,S,T) \in E_{\vec{\tau}(G)}$ be an edge in the unique exchange graph of $G$, then we can consider which of the two subgraphs $G_1$ and $G_2$ the exchanged edges $e$ and $f$ belong to once mapped to $\varphi_v( (S,T) )$.

  There are two easy cases (a) and (b), where both exchanged edges are exclusively in either $G_1$ or $G_2$, so either $(e,f) \in \{ (S \cap E_1, T \cap E_1), (T \cap E_1, S \cap E_1) \}$ or $(e,f) \in \{ (S \cap E_2, T \cap E_2), (T \cap E_2, S \cap E_2) \}$. We consider prototypically $(e,f) \in (S \cap E_1, T \cap E_1)$, as the other three cases are symmetrical. As $(e,f)$ is a unique $S$ edge exchange in $G$, $D_G(S,e) \cap C_G(T,e) = \{e,f\}$. If both cut and cycle are fully in $E_1$, $D_G(S,e) \cap C_G(T,e) \subseteq E_1$, then obviously also $D_{G_1}(S,e) \cap C_{G_1}(T,e) = \{e,f\}$ (cases (i) and (iii) of lemma~\ref{lem:2-vconn cutcycle mapping}), and we can simply map the unique exchange to the corresponding one in $G_1$. If either cut $D_G(S,e)$ or cycle $C_G(T,e)$ is not fully contained in $E_1$, but not both, then case (ii) or (iv) of lemma~\ref{lem:2-vconn cutcycle mapping} occurs, but not both. If just one case occurs, we still have $D_{G_1}(S,e) \cap C_{G_1}(T,e) = \{e,f\} \subseteq E_1$, and we can still map the unique exchange to the corresponding on in $G_1$. The fourth case, where both cut $D_G(S,e)$ and cycle $C_G(T,e)$ are not fully contained in $E_1$, is impossible: as both cut and cycle contain an edge of $G_2$, at least one edge of $G_2$ exists in their intersection. This can be verified by considering that the two vertices $x$ and $y$ are in different components of $G[S] - e$, while the cycle $C_{G}(T,e)$ contains a path from $x$ to $y$ in $G_2[T]$. The edge of $G_2$ in the intersection contradicts the presumption $D_G(S,e) \cap C_G(T,e) = \{e,f\}$, namely that the exchanged edges are in the same subgraph. Hence, we have mapped all unique edge exchanges of cases (a) and (b) to the corresponding subgraph and can now consider unique edge exchanges of $G$ which cross the vertex cut boundary.

  If the exchanged edges are in different subgraphs, say prototypically $(e,f) \in (S \cap E_1, T \cap E_2)$ is a unique $S$ edge exchange of $(S,T)$ in $G$, then we have to show that this unique exchange can be decomposed into a unique exchange in $G_1$ and one in $G_2$. By definition we have $D_G(S,e) \cap C_G(T,e) = \{e,f\}$. As $e$ and $f$ are in different subgraphs, the cycle $C_G(T,e)$ passes both $x$ and $y$ (the ends of $d_1$ and $d_2$), and the cut $D_G(S,e)$ disconnects $x$ and $y$ in $G[S] - e$. Due to these facts we can decompose cycle and cut into the parts contained in $G_1$ and $G_2$. We have $D_{G_1}(S \cap E_1,e) = (D_G(S,e) \cap E_1) + d_1$ and $C_{G_1}(T \cap E_1,e) = (C_G(T,e) \cap E_1) + d_1$ (see lemma~\ref{lem:2-vconn cutcycle mapping}), in which $d_1$ basically replaces all edges in the subgraph $G_2$. For $G_2$, however, we get $D_{G_2}(S \cap E_2, d_2) = (D_G(S,e) \cap E_2) + d_2$ and $C_{G_2}(T \cap E_2, d_2) = (C_G(T,e) \cap E_2) + d_2$, since in $G_2$ we have to add $d_2$ to gain the same (remainder) cycle and cut, again basically $d_2$ replaces all edges in $G_1$.  With these equations, we have
  \begin{align*}
    & D_{G_1}(S \cap E_1,e) \cap C_{G_1}(T \cap E_1,e) = ([D_G(S,e) \cap E_1] + d_1) \cap ([C_G(T,e) \cap E_1] + d_1) \\
    & \quad = [D_G(S,e) \cap E_1] \cap [C_G(T,e) \cap E_1] + d_1 = [ D_G(S,e) \cap C_G(T,e) ] \cap E_1 + d_1 = \{ e,d_1 \} \,,
  \end{align*}
  which guarantees that $(e,d_1) \in (S \cap E_1, (T \cap E_1) + d_1)$ is a $(S \cap E_1)$ unique exchange for $(S \cap E_1, (T \cap E_1) + d_1)$ in $G_1$. And in $G_2$,
  \begin{align*}
    & D_{G_2}(S \cap E_2,d_2) \cap C_{G_2}(T \cap E_2,d_2) = ([D_G(S,e) \cap E_2] + d_2) \cap ([C_G(T,e) \cap E_2] + d_2) \\
    & \quad = [D_G(S,e) \cap E_2] \cap [C_G(T,e) \cap E_2] + d_2 = [ D_G(S,e) \cap C_G(T,e) ] \cap E_2 + d_2 = \{ d_2,f \} \,,
  \end{align*}
  which guarantees that $(d_2,f) \in ((S \cap E_2) + d_2, T \cap E_2)$ is a $((S \cap E_2) + d_2)$ unique exchange for $((S \cap E_2) + d_2, T \cap E_2)$ in $G_2$.
  The argument can be applied symmetrically to a unique $T$ edge exchange with $(e,f) \in (T \cap E_1, S \cap E_2)$, and analogously to unique $S$ and $T$ edge exchanges with $(e,f) \in (S \cap E_2, T \cap E_1)$ and $(e,f) \in (T \cap E_2, S \cap E_1)$. Formally, these four combinations are mapped to edges of $E_\eta$ given by the cases (c) and (d).

  In summary, we can define the following forward edge mapping $\varphi_e$, wherein the eight different unique edge exchange types are implicitly encoded in the way $e$ and $f$ are drawn from the restricted spanning trees:
  \[
  \varphi_e( (e,f,S,T) ) =
  \begin{cases}
    (e,f, \varphi_v(S,T))
    & \hspace{-27mm} \begin{aligned}[t] \text{if } &(e,f) \in (S \cap E_1, T \cap E_1) \,, &&(e,f) \in (T \cap E_1, S \cap E_1) \,, \\
                                                   &(e,f) \in (S \cap E_2, T \cap E_2) \,\text{, or}\!\!\!&&(e,f) \in (T \cap E_2, S \cap E_2) \,, \end{aligned} \\
    \begin{aligned}[t] (e,f, &(S \cap E_1, (T \cap E_1) + d_1), \\ & ((S \cap E_2) + d_2, T \cap E_2)) \end{aligned}
    & \begin{aligned}[t] \text{if } &(e,f) \in (S \cap E_1, T \cap E_2) \text{ or} \\ &(e,f) \in (T \cap E_2, S \cap E_1) \,, \end{aligned} \\
    \begin{aligned}[t] (e,f, &((S \cap E_1) + d_1, T \cap E_1), \\ & (S \cap E_2, (T \cap E_2) + d_2)) \end{aligned}
    & \begin{aligned}[t] \text{if } &(e,f) \in (T \cap E_1, S \cap E_2) \text{ or} \\ &(e,f) \in (S \cap E_2, T \cap E_1) \,. \end{aligned} \\
  \end{cases}
  \]
  As before, the inverse mapping $\psi_e : E_\eta \rightarrow E_{\vec{\tau}(G)}$ is easier, let $(e,f,(S_1,T_1),(S_2,T_2)) \in E_\eta$ be of one of the four cases (a)--(d).

  Unique exchange edges of type (a), with $(e,f,S_1,T_1) \in E_{\vec{\tau}_3(G_1)}$ and $(S_2,T_2) \in V_{\tau(G_2)}$, as well as type (b), with $(e,f,S_2,T_2) \in E_{\vec{\tau}_3(G_2)}$ and $(S_1,T_1) \in V_{\tau(G_1)}$, can directly be applied to $G$, since $\{ e,f \} \cap \{ d_1,d_2 \} = \emptyset$ and $d_1$ or $d_2$ in the cut and cycle can be replaced with the expanded cut or cycle in the other subgraph without possibly effecting the intersection.

  An edge of type (c), which exists if $(e,d_1,S_1,T_1) \in E_{\vec{\tau}_3(G_1)}$ and $(d_2,f,S_2,T_2) \in E_{\vec{\tau}_3(G_2)}$, can be mapped directly to the combined unique exchange  $(e,f,\psi_v[(S_1,T_1),(S_2,T_2)])$ in $G$. The chaining on swaps on $d_1$ and $d_2$, as illustrated in figure~\ref{fig:2-clique UE-decompose}, ensures that the combined trees $S' := S_1 \dotcup S_2$ and $T' := T_1 \dotcup T_2$ of $G_1 \oplus G_2$ do not contain both $d_1$ and $d_2$ (as required by $V_\eta$). The same applies to edges of type (d), where $e$ is in $G_2$ and $f$ in $G_1$. In the end, the inverse mapping is as easy as:
  \begin{align*}
    \psi_e( (e,f,S_1,T_1,S_2,T_2) ) &= (e,f,\psi_v[(S_1,T_1),(S_2,T_2)]) \\
                                   &= (e,f, (S_1 \dotcup S_2) \setminus \{d_1,d_2\}, (T_1 \dotcup T_2) \setminus \{d_1,d_2\}) \,.
  \end{align*}
  Again, one can clearly see that $\varphi_e$ and $\psi_e$ are inverse to each other, since they mainly compose or decompose into the subgraphs and add or remove $d_1$ and $d_2$.

  To complete the graph isomorphism, it remains to show $(\varphi_v \times \varphi_v)( \delta_{\vec{\tau}_3(G)}((e,f,S,T)) ) = \delta_\eta( \varphi_e((e,f,S,T)) )$ for all edges $(e,f,S,T) \in E_{\vec{\tau}(G)}$. This is intuitively obvious, but rather tedious to spell out in detail due to the eight cases, but we expanded some below for reference.

  In the case $(e,f) \in (S,T)$, we have $(\varphi_v \times \varphi_v)( \delta_{\vec{\tau}_3(G)}((e,f,S,T)) ) = (\varphi_v \times \varphi_v)( (S,T), (S - e + f, T + e - f) ) = (\varphi_v(S,T), \varphi_v(S - e + f, T + e - f))$, as of definition~\ref{def:directed tau}. This result will match the first of each of the four cases in $\delta_\eta$, as we now have to distinguish which subgraphs $e$ and $f$ belong to. If $(e,f) \in (S \cap E_1, T \cap E_1)$, we have case (a) and $\delta_\eta( \varphi_e((e,f,S,T)) ) = \delta_\eta( (e,f,\varphi_v(S,T)) ) = ( \varphi_v(S,T), (S_1 - e + f,T_1 + e - f), (S_2,T_2) ) = ( \varphi_v(S,T), \varphi(S - e + f, T + e - f) )$, due to where $e$ and $f$ are in $G$. If $(e,f) \in (S \cap E_1, T \cap E_2)$, we have case (c) and $\delta_\eta( \varphi_e((e,f,S,T)) ) = \delta_\eta( (e,f,\varphi_v(S,T)) ) = ( \varphi_v(S,T), (S_1 - e,T_1 + e), (S_2 + f,T_2 - f) ) = ( \varphi_v(S,T), \varphi(S - e + f, T + e - f) )$, due to where $e$ and $f$ are in $G$. If $(e,f) \in (S \cap E_2, T \cap E_1)$, we have case (d) and $\delta_\eta( \varphi_e((e,f,S,T)) ) = \delta_\eta( (e,f,\varphi_v(S,T)) ) = ( \varphi_v(S,T), (S_1 + f,T_1 - f), (S_2 - e,T_2 + e) ) = ( \varphi_v(S,T), \varphi(S - e + f, T + e - f) )$, due to where $e$ and $f$ are in $G$.  If $(e,f) \in (S \cap E_2, T \cap E_2)$, we have case (b) and $\delta_\eta( \varphi_e((e,f,S,T)) ) = \delta_\eta( (e,f,\varphi_v(S,T)) ) = ( \varphi_v(S,T), (S_1,T_1), (S_2 - e + f,T_2 + e - f) ) = ( \varphi_v(S,T), \varphi(S - e + f, T + e - f) )$, due to where $e$ and $f$ are in $G$.

  In the case $(e,f) \in (T,S)$, we have $(\varphi_v \times \varphi_v)( \delta_{\vec{\tau}_3(G)}((e,f,S,T)) ) = (\varphi_v \times \varphi_v)( (S,T), (S + e - f, T - e + f) ) = (\varphi_v(S,T), \varphi_v(S + e - f, T - e + f))$, as of definition~\ref{def:directed tau}. This result matches the second of each of the four cases in $\delta_\eta$, if one distinguishes which subgraphs $e$ and $f$ belong to, as above. We omit these fours cases, as they are symmetrical to the four cases in the last paragraph.
\end{proof}

\begin{figure}[p]
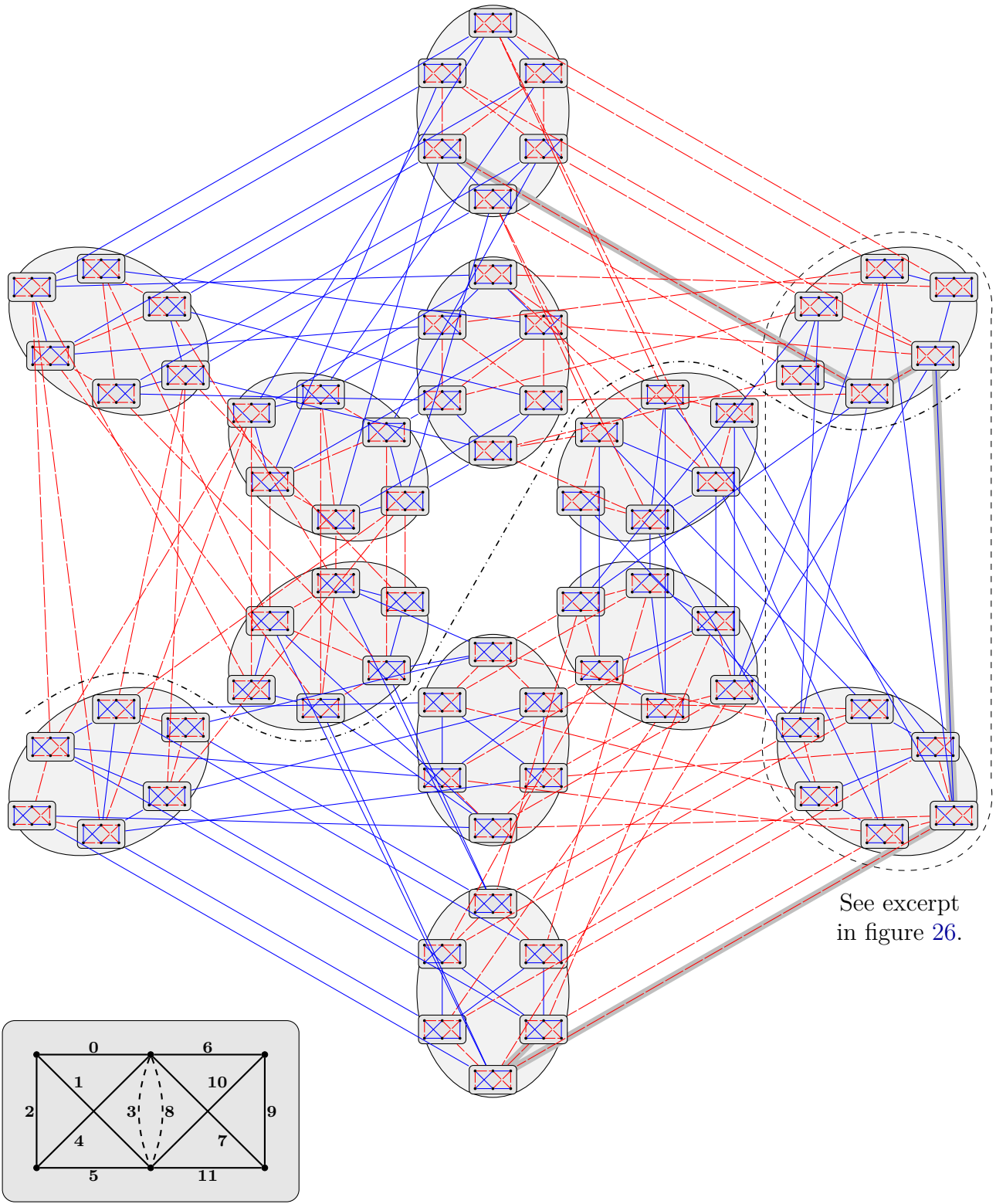
\centering
  \begin{leftfullpage}

  \caption{The unique exchange graph $\tau_3(K_4 \oplus_2 K_4)$.}\label{fig:tau-k4k4}
  \end{leftfullpage}
\end{figure}

% \CompleteKFourKFour{node}{R,B,B,B,R,B}
\def\CompleteKFourKFour#1#2{
  \begin{scope}[
    shift={#1},
    scale=1.05,
    SmallNode/.append style={inner sep=0.8pt},
    LineStyle0/.style={R},
    LineStyle1/.style={R},
    LineStyle2/.style={R},
    LineStyle3/.style={R},
    LineStyle4/.style={R},
    LineStyle5/.style={R},
    LineStyle6/.style={R},
    LineStyle7/.style={R},
    LineStyle8/.style={R},
    LineStyle9/.style={R},
    LineStyle10/.style={R},
    LineStyle11/.style={R},
    SmallLabel0/.style={SmallLabel},
    SmallLabel1/.style={SmallLabel, pos=0.3},
    SmallLabel2/.style={SmallLabel, swap},
    SmallLabel4/.style={SmallLabel, pos=0.7},
    SmallLabel5/.style={SmallLabel, swap},
    SmallLabel6/.style={SmallLabel},
    SmallLabel7/.style={SmallLabel, pos=0.7, swap},
    SmallLabel9/.style={SmallLabel},
    SmallLabel10/.style={SmallLabel, pos=0.7},
    SmallLabel11/.style={SmallLabel, swap},
    SmallLabel3/.style={SmallLabel, swap},
    SmallLabel8/.style={SmallLabel},
    BendStyle0/.style={},
    BendStyle1/.style={},
    BendStyle2/.style={},
    BendStyle3/.style={bend left=-15, dashed},
    BendStyle4/.style={},
    BendStyle5/.style={},
    BendStyle6/.style={},
    BendStyle7/.style={},
    BendStyle8/.style={bend left=15, dashed},
    BendStyle9/.style={},
    BendStyle10/.style={},
    BendStyle11/.style={},
    ]

    \pgfmathtruncatemacro{\x}{{#2}[0]}
    \tikzset{LineStyle\x/.append style={B}}
    \pgfmathtruncatemacro{\x}{{#2}[1]}
    \tikzset{LineStyle\x/.append style={B}}
    \pgfmathtruncatemacro{\x}{{#2}[2]}
    \tikzset{LineStyle\x/.append style={B}}
    \pgfmathtruncatemacro{\x}{{#2}[3]}
    \tikzset{LineStyle\x/.append style={B}}
    \pgfmathtruncatemacro{\x}{{#2}[4]}
    \tikzset{LineStyle\x/.append style={B}}
    \pgfmathtruncatemacro{\x}{{#2}[5]}
    \tikzset{LineStyle\x/.append style={B}}

    \tikzset{LineStyle3/.append style={dashed}}
    \tikzset{LineStyle8/.append style={dashed}}

    \node (0) [at={(-1,0.5)}, SmallNode] {0};
    \node (1) [at={(0,0.5)}, SmallNode] {1};
    \node (2) [at={(-1,-0.5)}, SmallNode] {2};
    \node (3) [at={(0,-0.5)}, SmallNode] {3};
    \node (4) [at={(1,0.5)}, SmallNode] {4};
    \node (5) [at={(1,-0.5)}, SmallNode] {5};

    \def\sA{{0,0,0,1,1,2,1,1,1,4,3,3}}
    \def\tA{{1,3,2,3,2,3,4,5,3,5,4,5}}

    \foreach \x [count=\i from 0] in {0,1,2,3,4,5,6,7,8,9,10,11} {
      \pgfmathtruncatemacro{\s}{\sA[\i]}
      \pgfmathtruncatemacro{\t}{\tA[\i]}
      \draw[LineStyle\x] (\s) to[BendStyle\x] node[SmallLabel\x] {\i} (\t);
    }

    % \node at (0,2) {\bf #1};
  \end{scope}
}

\begin{figure}
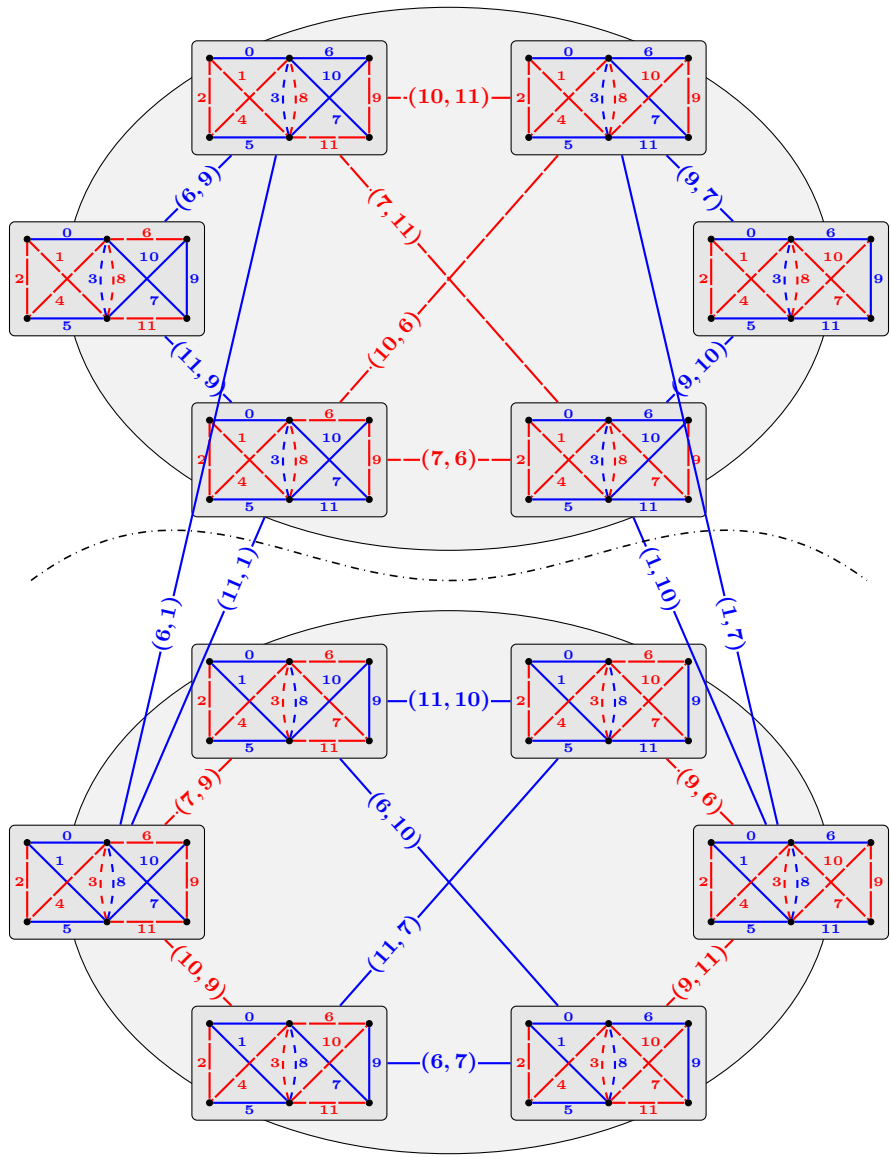
\centering
  \begin{fullpage}

  \caption{Excerpt of two groups of the unique exchange graph $\tau_3(K_4 \oplus_2 K_4)$.}\label{fig:tau-k4k4 zoom}
  \end{fullpage}
\end{figure}

The smallest atomic bispanning graph with vertex-connectivity two is $B_{6,3} \cong K_4 \oplus_2 K_4$. The unique exchange graph $\tau_3(K_4 \oplus_2 K_4)$ has 72 vertices and 224 undirected exchange edges, and is shown in figure~\ref{fig:tau-k4k4} with a layout that corresponds to the view taken by theorem~\ref{thm:2-vconn tau decompose}. The vertex set $V_\eta \subseteq V_{\tau(G_1)} \times V_{\tau(G_2)}$ is drawn such that the left part of the graph $K_4 \oplus_2 K_4$ is fixed within each of the circled groups, and these groups are arranged such that the left part corresponds to the vertices in $\tau_3(K_4)$ in figure~\ref{fig:tau-k4}. Two of the groups are enlarged in figure~\ref{fig:tau-k4k4 zoom} for more details.

Within each group, the left subgraph determines the colors of the two virtual edges $3$ and $8$ in the join seam: for example, in the enlarged figure~\ref{fig:tau-k4k4 zoom}, the top graphs' implicit join edge $3$ is blue, and the right one's $8$ red. By fixing this edge's color, only $6$ of the $12$ possible vertices in $V_{\tau(G_2)}$ remain valid: those shown in the circled groups. In figure~\ref{fig:tau-k4}, these halves of $V_{\tau(K_4)}$ are separated by the dash-dotted line which splits the graph. Within each circled group the vertices are connected by unique exchanges as in $G_2 = K_4$, since these edges do not include any unique exchanges involving the join seam's edges.

Unique exchange edges between the groups correspond to exchanges in $G_1$ (also $= K_4$), where \emph{all} edges \emph{not} involving the join seam edges $3$ and $8$ are retained. In $\tau_3(K_4)$ of figure~\ref{fig:tau-k4} the dash-dotted line crosses exactly all unique edge exchanges involving edge $3$. The dash-dotted line is mapped to the same division in $\tau_3(K_4 \oplus_2 K_4)$ of figure~\ref{fig:tau-k4k4}: it crosses exactly all unique exchange edges involving the join seam edges $3$ and $8$. All unique exchange edge outside the groups and not crossed by the dash-dotted line are retained (which can be seen by the parallel edges). Those crossed are retained if the three-hop unique edge exchanges involving $3$ and $8$ are valid in both $G_1$ and $G_2$. The two groups shown in the enlarged figure~\ref{fig:tau-k4k4 zoom} are selected such that they are connected only by three-hop unique edge exchanges which swap the two virtual edges in the seam.

Regarding the question of whether the graph $\tau_3(G_1 \oplus_2 G_2)$ remains connected if $\tau_3(G_1)$ and $\tau_3(G_2)$ are connected, these three-hop edges are the most important ones.

\begin{theorem}[joining unique exchange cyclic base orderings at $2$-clique sums]\label{thm:join-2sum}
  Let $G_1 = (V_1,E_1,\delta_1)$ and $G_2 = (V_2,E_2,\delta_2)$ be two bispanning graphs, $G = G_1 \oplus_2 G_2$ a $2$-clique sum with join edges $d_1 \in E_1$ and $d_2 \in E_2$, and $(S,T)$ a pair of disjoint spanning trees of $G$.

  If $\varphi_v((S,T)) = ((S_1,T_1),(S_2,T_2))$ are $(S,T)$ restricted to pairs of disjoint trees in $G_1$ and $G_2$, and if unique exchange cyclic base orderings are given from $(S_1,T_1)$ to $(T_1,S_1)$ in $G_1$ and from $(S_2,T_2)$ to $(T_2,S_2)$ in $G_2$, then one can construct a unique exchange cyclic base ordering from $(S,T)$ to $(T,S)$ in $G$.
\end{theorem}
\begin{proof}
  Let $\arr{ (e_1,f_1), \ldots, (e_m,f_m) }$ be a unique exchange cyclic base ordering of $G_1$ from $(S_1,T_1)$ to $(T_1,S_1)$, and $\arr{ (g_1,h_1), \ldots, (g_n,h_n) }$ of $G_2$ from $(S_2,T_2)$ to $(T_2,S_2)$ where $2 m = |E_1|$ and $2 n = |E_2|$.

  The join edge $d_1$ is some $e_i$ or $f_i$; let $i$ be the correct index. Likewise, $d_2$ is either a $g_j$ or a $h_j$, where $j$ is the corresponding index. To join the two edge swap sequences at the seam, we must have $d_1$ and $d_2$ on opposite ``sides'' of a unique edge exchange. Hence, if $d_1 = e_i$ and $d_2 = g_j$, or some $d_1 = f_i$ and $d_2 = h_j$, then we have to take the reverse of one of the unique exchange cyclic base orderings as described in theorem~\ref{thm:uecbo-reversibility}. We will chose to reverse the second, $\arr{ (h_n,g_n), \ldots, (h_1,g_1) }$,  This reversed UECBO still applies to the path $(S_2,T_2)$ to $(T_2,S_2)$.

  Having two UECBOs wherein the edges $d_1$ and $d_2$ appear on opposite sides of a unique edge exchange, one can merge the two edge swap sequences by selecting edge swaps from the sequences in any order, provided the swaps at index $i$ (involving $d_1$) and $j$ (involving $d_2)$ are joined into one unique exchange, and that all swaps with index $< i$ and $< j$ are done before the join and all $> i$ and $> j$ are done afterwards.

  The resulting CBO is a UECBO, since all swaps $< i$ and $< j$ do not contain $d_1$ or $d_2$ and hence are applicable only to the corresponding subgraph \emph{before} $d_1$ and $d_2$ are swapped, and all swaps $> i$ and $> j$ are applicable only \emph{after} the join seam edges have been swapped.
\end{proof}

\begin{corollary}[connectivity of $\tau_3$ of $2$-clique sums]
  If $G = G_1 \oplus_2 G_2$ and $\tau_3(G_1)$ and $\tau_3(G_2)$ are connected, then $\tau_3(G)$ is connected.
\end{corollary}

To illustrate theorem~\ref{thm:join-2sum} by example, we now describe how to join two UECBOs of $K_4$ (figure~\ref{fig:tau-k4}, page~\pageref{fig:tau-k4}) to calculate UECBOs for $K_4 \oplus_2 K_4$ (figure~\ref{fig:tau-k4k4}). For the example, we first arbitrarily select as source tree pair the one at the top of figure~\ref{fig:tau-k4}, to which the inverted target tree pair is located at the bottom of the figure. Figure~\ref{fig:example join-2sum} shows the selected tree pair and how it will be combined. Next, we arbitrarily select two UECBOs from the top vertex to the bottom: the first corresponds to the right-most path along the ``outer rim'' and the second the path meeting the ``left inner rim'' of the graph. These are $A := \arr{ (2,5), (1,3), (0,4) } = \carr{ 2,3,0 }{ 5,1,4 }$ (right outer rim) and $B := \arr{ (1,0), (2,4), (5,3) } = \carr{ 0,2,3 }{ 1,4,5 }$ (left inner rim).

To join the two instances of $K_4$, we have to renumber of the second instances (by adding $6$ to all edges), and invert the colors such that the join seam $3$ and $8$ are unequal. The translation yields $\arr{ (7,6), (8,10), (11,9) } = \carr{ 6,8,9 }{ 7,10,11 }$, and the inversion $\arr{ (9,11), (10,8), (6,7) } = \carr{ 11,10,7 }{ 9,8,6 } := B'$.

The join edges $3$ and $8$ are locates on the same side of their unique exchanges in the two UECBOs, $A$ and $B'$, hence, we have to reverse one of them as described in theorem~\ref{thm:uecbo-reversibility}. We reverse the second one to $\arr{ (7,6), (8,10), (11,9) } = \carr{ 6,8,9 }{ 7,10,11 } := B^r$, which still applies to the same tree pair.

We can now join the UECBOs $A = \arr{ (2,5), (1,3), (0,4) }$ and $B^r = \arr{ (7,6), (8,10), (11,9) }$ as described in theorem~\ref{thm:join-2sum}: the UEs involving the join seam edges must be matched, and all UEs before and after can be ordered arbitrarily. This delivers the following four UECBOs for $K_4 \oplus_2 K_4$:
\begin{align*}
  \arr{ (2,5), (7,6), (1,10), (0,4), (11,9) }\,, &&
  \arr{ (7,6), (2,5), (1,10), (0,4), (11,9) }\,, \\
  \arr{ (2,5), (7,6), (1,10), (11,9), (0,4) }\,, &&
  \arr{ (7,6), (2,5), (1,10), (11,9), (0,4) }\,.
\end{align*}
By reversing the first UECBOs instead of the second, we could calculate four more UECBOs with $(10,1)$, which are reversed versions of the four above.  The first of the four UECBOs above is marked in figure~\ref{fig:tau-k4k4} as the light gray path from the top to the bottom vertex.

\begin{figure}
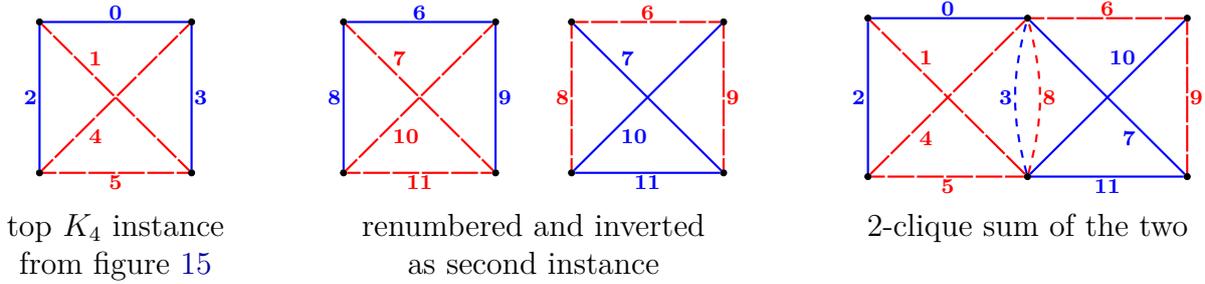
\centering

  \caption{The two $K_4$ tree pair instances and their $2$-clique sum used in the example to theorem~\ref{thm:join-2sum}.}\label{fig:example join-2sum}
\end{figure}

% ------------------------------------------------------------------------------

\subsection{Reducing Atomic Bispanning Graphs at a Vertex of Degree Three}\label{sec:reduce vdeg3}

Due to the previous two sections, the only class of bispanning graphs for which $\tau_3$ remains to be composed are atomic, and have vertex- and edge-connectivity three. Obviously, such graphs have no degree two vertex, hence there are at least four vertices with degree three (theorem~\ref{thm:bispanning deg23}).  In this section we commonly refer to such a degree three vertex as $v$ and to its adjacent vertices with $x$, $y$, and $z$ (as illustrated, e.g., in figure~\ref{fig:reducing w5}). The goal of this section is to describe how $\vec{\tau}_3(G)$ of an atomic bispanning graph can be derived from $\vec{\tau}_3(G')$ of smaller bispanning graphs.

There are three ways such a vertex $v$ could have resulted from an edge-split-attach operation during an inductive construction: the split edge must have originally been between two of the three vertices $\{x,y,z\}$, and the attachment must have been done at the remaining vertex. In the following theorem, we name the three possible graphs from which the operation could have originated the \emph{reduction} graphs $G_{x,y}$, $G_{x,z}$, and $G_{y,z}$. Note that this is not the only way a degree three vertex could be been constructed, because the ``attach'' of the edge-split-attach operation can freely select a vertex and increase its adjacency. For the reduction to work, however, it suffices that these are \emph{possible} constructions.
\begin{theorem}[{\textls[-8]{reduction of an atomic bispanning graph at a degree three vertex}}]\label{thm:reducing degree three}
  If $G = (V,E)$ is an atomic bispanning graph and $v \in V$ a vertex with $\deg(v) = 3$ and the three adjacent vertices $x,y,z \in V$, then the three graphs
  \begin{align*}
    \begin{aligned}
      G_{x,y} :=&\; G \contr \{v,x\} - \{x,z\} \\
              =&\; G \contr \{v,y\} - \{y,z\} \,,
    \end{aligned} &&
    \begin{aligned}
      G_{x,z} :=&\; G \contr \{v,x\} - \{x,y\} \\
              =&\; G \contr \{v,z\} - \{z,y\} \,, \\
    \end{aligned} &&
    \begin{aligned}
      G_{y,z} :=&\; G \contr \{v,y\} - \{y,x\} \\
              =&\; G \contr \{v,z\} - \{z,x\} \\
    \end{aligned}
  \end{align*}
  are bispanning. We call them the three \emph{reduction} graphs of $G$ at $v$, or say that $G$ is \emph{reduced} at $v$. For $(a,b) \in \{ (x,y), (x,z), (y,z) \}$, one edge with ends $a$ and $b$ remains in $G_{a,b}$ after the contraction, and we name this edge $e_{a,b}$ with $\delta_{G_{a,b}}(e_{a,b}) = \{a,b\}$.
\end{theorem}
\begin{proof}
  Due to theorem~\ref{thm:atomics are simple} considering simple graphs is sufficient, though the resulting $G_{x,y}$, $G_{x,z}$, and $G_{y,z}$ may not be simple anymore. Since all three graphs are constructed using a contraction and a deletion of distinct edges in an atomic bispanning graph, lemma~\ref{lem:contract-delete atomic bispanning} guarantees that the resulting graphs are bispanning.
\end{proof}

Due to the definition by contraction, the resulting graphs $G_{x,y}$, $G_{x,z}$, and $G_{y,z}$ have similar vertex and edges sets: the vertex set of $G_{a,b}$ is clearly $V - v$ and the edge set $E - e_x - e_y - e_z + e_{a,b}$. Furthermore, we want to point out that the premise that the bispanning graph must be atomic is necessary. For example $B_{6,6}$ (figure~\ref{fig:small simple bispanning}, page~\pageref{fig:small simple bispanning}) is a simple counterexample: reduction of the top degree three vertex results in two bispanning graphs and one which is not bispanning.

In figure~\ref{fig:reducing w5}, the three reduction graphs of a degree three vertex from theorem~\ref{thm:reducing degree three} in $W_5$ are shown. In a sense, this example can seen as prototypical, since only the labeled vertices matter for the reduction. Though, in general, the three vertices $x$, $y$, and $z$ may or may not be adjacent in $G$ (in $W_5$ two of the three combinations are adjacent).

In the example in the figure, we already show the reduction for specific pairs of disjoint trees, even though theorem~\ref{thm:reducing degree three} does not consider the disjoint trees. This is the topic of the next theorem, but is better first explained with an example. Consider how many possible colorings of the three edges $0$, $1$, $7$, or in general $e_x$, $e_y$, and $e_z$ exist. One edge must have a color different from the other two, while the other two must have the same color. Combinatorially, we can freely pick one edge and its color, thereby determining the others. This yields in total six possible combinations. Three of these are shown in figure~\ref{fig:reducing w5}, the other three have inverted colors.

During the reduction, the two edges of the same color are replaced by (or contracted to) the former split edge $e_{a,b}$. For example, in the figure's center column, $e_{x,z} = \{1,7\}$, since $e_x = 1$ and $e_z = 7$ are in $S$ (blue). As this undoes the edge-split-attach, the reduced pair of trees are a valid pair of disjoint spanning trees of $G_{a,b}$. In the figure $G_{x,z} = G_{1,7} \cong K_4$, where $(S,T)$ from $G = W_5$ is reduced to $(S - 1 - 7 + \{1,7\}, T - 0)$. This mapping of pairs of disjoint spanning trees will be called $\rho_{e_{a,b},c}$ and $\rho^{-1}_{e_{a,b},c}$ in the following theorem. The mapping $\rho^{-1}_{e_{a,b},c}$ takes $(S,T)$ from $G$ to a pair of disjoint spanning trees of \emph{one} of the reduction graphs $G_{x,y}$, $G_{x,z}$, or $G_{x,y}$, where the destination depends on the three edges adjacent to $v$. For example, if $0$ and $1$ are blue, and $7$ red, as in the left column in the figure~\ref{fig:reducing w5}, then the tree pairs are reduced as shown in $G_{x,y} = G_{1,0}$. Likewise, if $0$ and $7$ are blue, and $1$ red (right column), then they are reduced to $G_{y,z} = G_{0,7}$. While $G_{1,0}$ and $G_{0,7}$ are isomorphic in this example, they need not be in general, and in the context of our recombination they are considered different graphs, since the have different edge sets.

Hence, one can map any pair of disjoint trees to a pair of disjoint trees in exactly one of the reduction graphs, and vice versa. We prove this bijection in the following theorem, written as a bijection of the vertex sets of exchange graphs. In this context $\rho_{e_{a,b},c}$ maps a pair of disjoint spanning trees from the reduction graph $G_{a,b}$ to $G$, and $\rho^{-1}_{e_{a,b},c}$ takes a pair of spanning trees from $G$ and maps it to one of the reduction graphs. It is helpful to read $\rho_{e_{a,b},c}$ as ``add $v$ by splitting $e_{a,b} = \{ a,b \}$ and attaching to $c$'', and $\rho^{-1}_{e_{a,b},c}$ as ``remove $v$, which is attached to $a$ and $b$ with the same color and to $c$ by the other, and replace $v$ with $e_{a,b} = \{ a,b \}$''.

\begin{figure}
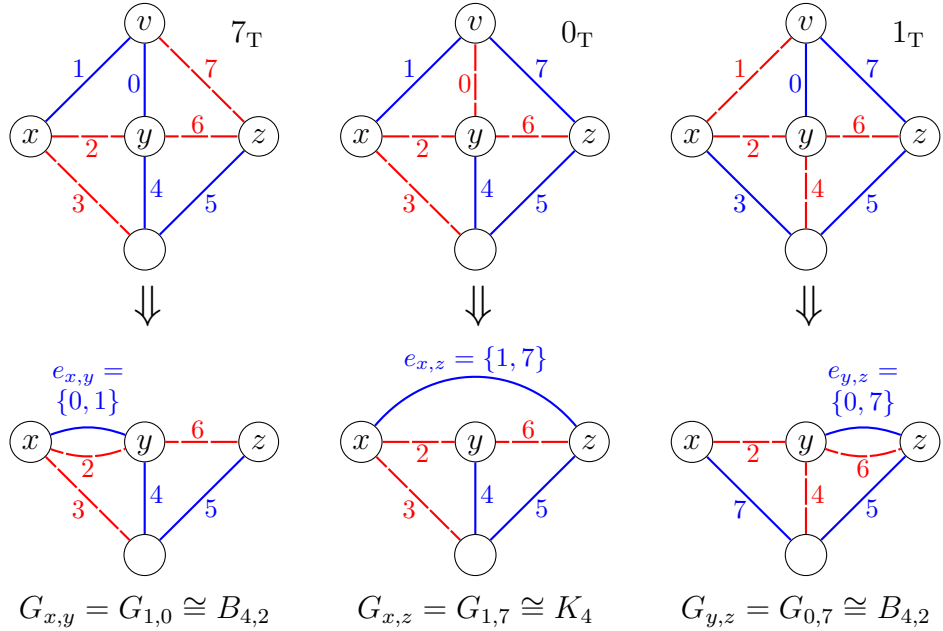
\centering

  \caption{Reducing $W_5$ at the vertex $v$ of degree three into the graph $G_{x,y}$, $G_{x,z}$, and $G_{y,z}$.}\label{fig:reducing w5}
\end{figure}

\begin{theorem}[mapping vertices of the exchange graph at a vertex of degree three]\label{thm:decompose rho3}
  Let $G = (V,E)$ be an atomic bispanning graph, $v \in V$ a vertex with $\deg(v) = 3$ and the three adjacent vertices $x,y,z \in V$, and $G_{x,y}$, $G_{x,z}$, $G_{y,z}$ the three reduction graphs as defined in theorem~\ref{thm:reducing degree three}. If $V_{\tau(G)}$, $V_{\tau(G_{x,y})}$, $V_{\tau(G_{x,z})}$, and $V_{\tau(G_{y,z})}$ are the vertex sets of an exchange graph on $G$, $G_{x,y}$, $G_{x,z}$, and $G_{y,z}$, then
  \[
  V_{\tau(G)} = \rho_{e_{x,y},z}(V_{\tau(G_{x,y})}) \dotcup \rho_{e_{x,z},y}(V_{\tau(G_{x,z})}) \dotcup \rho_{e_{y,z},x}(V_{\tau(G_{y,z})}) \,,
  \]
  where $\rho_{e_{ab},c} : V_{\tau(G_{a,b})} \rightarrow V_{\tau(G)}$\symbol{rho}{$\rho_{e_{ab},c}(S,T)$}{edge-split-attach operation} with
  \[
  (S,T) \mapsto
  \begin{cases}
    (S - e_{a,b} + \{v,a\} + \{v,b\}, T + \{v,c\}) & \text{if } e_{a,b} \in S \,, \\
    (S + \{v,c\}, T - e_{a,b} + \{v,a\} + \{v,b\}) & \text{if } e_{a,b} \in T \,, \\
  \end{cases}
  \]
  for any $(a,b) \in \{ (x,y), (x,z), (y,z) \}$, which basically performs the ``edge-split-attach'' operation on $(S,T)$ (see theorem~\ref{thm:edge-split-attach}).
\end{theorem}
\begin{proof}
  To show that $V_{\tau(G)} \subseteq \rho_{e_{x,y},z}(V_{\tau(G_{x,y})}) \dotcup \rho_{e_{x,z},y}(V_{\tau(G_{x,z})}) \dotcup \rho_{e_{y,z},x}(V_{\tau(G_{y,z})})$, let $(S,T)$ be a pair of disjoint spanning trees of $G$. Then of the three edges $\{v,x\}$, $\{v,y\}$, and $\{v,z\}$ one edge is in one tree, say $S'$, and the other two in the other tree, say $T'$, with $\{ S', T' \} = \{ S, T \}$. Thus there are six configurations the three edges can have in a vertex of $V_{\tau(G)}$.

  If $\{v,x\}$ is a leaf edge in $S$, then we map $(S,T)$ to the vertex $(S - \{v,x\}, T - \{v,y\} - \{v,z\} + e_{y,z})$ in $V_{\tau(G_{y,z})}$, since clearly this is a pair of disjoint spanning trees of $G_{y,z}$. Likewise, if $\{v,x\}$ is a leaf edge in $T$, we can map symmetrically to $(S - \{v,y\} - \{v,z\} + e_{y,z}, T - \{v,x\})$, also in $V_{\tau(G_{y,z})}$. Likewise, we can map the pair of trees to $G_{x,z}$ if $\{v,y\}$ is a leaf edge and to $G_{x,y}$ if $\{v,z\}$ is a leaf edge. These individual mappings are in fact the inverse of the operation $\rho_{e_{a,b},c}$:
  \begin{align*}
  \rho^{-1}_{e_{a,b},c} &: \{ (S,T) \in V_{\tau(G)} \mid \{v,a\}, \{v,b\} \in S \text{ or } \{v,a\}, \{v,b\} \in T \} \rightarrow V_{\tau(G_{a,b})}, \\
    (S,T) &\mapsto \begin{cases}
      (S - \{v,c\}, T - \{v,a\} - \{v,b\} + e_{a,b}) & \text{if } \{v,a\}, \{v,b\} \in T \text{ and } \{v,c\} \in S \,, \\
      (S - \{v,a\} - \{v,b\} + e_{a,b}, T - \{v,c\}) & \text{if } \{v,a\}, \{v,b\} \in S \text{ and } \{v,c\} \in T \,,
    \end{cases}
  \end{align*}
  So each vertex of $V_{\tau(G)}$ has a specific configuration of the three edges, which corresponds to a vertex in exactly one of the three reduction graphs. Regarding the opposite inclusion, $\rho_{e_{x,y},z}(V_{\tau(G_{x,y})}) \dotcup \rho_{e_{x,z},y}(V_{\tau(G_{x,z})}) \dotcup \rho_{e_{y,z},x}(V_{\tau(G_{y,z})}) \subseteq V_{\tau(G)}$, one can take any pair of trees from the three reduction graphs, say from $G_{a,b}$, and apply $\rho_{e_{a,b},c}$ to map the trees to a pair of disjoint spanning trees in $G$, as already stated in the theorem. Thus both vertex sets are equal.
\end{proof}

The previous theorem yields an elegant way to calculate the number of vertices in an exchange graph of an atomic bispanning graph.

\begin{corollary}[number of vertices in exchange graph of atomic bispanning graph]
  Let $G = (V,E)$ be an atomic bispanning graph, $v \in V$ a vertex with $\deg(v) = 3$ and the three adjacent vertices $x,y,z \in V$, and $G_{x,y}$, $G_{x,z}$, $G_{y,z}$ as defined in theorem~\ref{thm:reducing degree three}. If $V_{\tau(G)}$, $V_{\tau(G_{x,y})}$, $V_{\tau(G_{x,z})}$, and $V_{\tau(G_{y,z})}$ are the vertex sets of an exchange graph on $G$, $G_{x,y}$, $G_{x,z}$, and $G_{y,z}$, then
  \[
  |V_{\tau(G)}| = |V_{\tau(G_{x,y})}| + |V_{\tau(G_{x,z})}| + |V_{\tau(G_{y,z})}| \,.
  \]
\end{corollary}

To test the theorems, we can apply the reduction to $K_4$, since it is atomic. The three reduction graphs are all isomorphic to $B_{3,2}$ (see figure~\ref{fig:small-bispanning}, page~\pageref{fig:small-bispanning}), regardless of the vertex at which one reduces. While they are isomorphic, during the reduction the labeling must be retained. Consider in figure~\ref{fig:tau-k4} (page~\pageref{fig:tau-k4}) the vertex $v$ of $K_4$ in the top left corner, which is incident to $x$, $y$, and $z$ by the edges $0$, $1$, and $2$. If we apply the reduction to this vertex, we get $V_{\tau(G_{x,y})}$, $V_{\tau(G_{x,z})}$, and $V_{\tau(G_{y,z})}$, each containing four distinct spanning trees of $B_{3,2}$. These three subsets are marked in figure~\ref{fig:tau-k4} with $0_{\text{S}}$, $0_{\text{T}}$, $1_{\text{S}}$, $1_{\text{T}}$, $2_{\text{S}}$, $2_{\text{T}}$, where $\rho_{e_{0,1},2}(V_{\tau(G_{0,1})})$ is marked with $2_{\text{S}}$ and $2_{\text{T}}$, $\rho_{e_{0,2},1}(V_{\tau(G_{0,2})})$ is marked with $1_{\text{S}}$ and $1_{\text{T}}$, and $\rho_{e_{1,2},0}(V_{\tau(G_{1,2})})$ is marked with $0_{\text{S}}$ and $0_{\text{T}}$. The number in the markers identifies the edge in the reduction, whose color is different from the other two (the single color).

As a less pathological example, consider the reduction of $W_5$, of which the exchange graph is illustrated in figure~\ref{fig:tau-w5} (page~\pageref{fig:tau-w5}). We can calculate $|V_{\tau(W_5)}|$ by taking $|V_{\tau(B_{4,2})}| + |V_{\tau(K_4)}| + |V_{\tau(B_{4,2})}|$. From figure~\ref{fig:tau-x4} (page~\pageref{fig:tau-x4}) we get $|V_{\tau(B_{4,2})}| = 8$ and from figure~\ref{fig:tau-k4} $|V_{\tau(K_4)}| = 12$, so in total $|V_{\tau(W_5)}| = 28$. These $28$ possible disjoint spanning trees are shown in figure~\ref{fig:tau-w5}. As with $K_4$, we labeled and grouped the vertices of $\tau_3(W_5)$ according to the color combination of the three edges, as shown in the reduction in figure~\ref{fig:reducing w5}. The combinations are marked with $0_{\text{S}}$, $0_{\text{T}}$, $1_{\text{S}}$, $1_{\text{T}}$, $7_{\text{S}}$, and $7_{\text{T}}$, and suggestively grouped. Note that this grouping is somewhat arbitrary, since if we choose to reduce $W_5$ at a different vertex of degree three, then the grouping is different, even though the exchange graph itself is isomorphic. Another observation from figure~\ref{fig:tau-w5} is that vertices from ``opposite'' sets $a_{\text{S}}$ and $a_{\text{T}}$, like $0_{\text{S}}$ and $0_{\text{T}}$, are never adjacent.

\begin{figure}\centering
  \begin{tikzpicture}[graphfinal,
    scale=2.5,
    elabel/.style={sloped, inner sep=1pt},
    ]

    \draw[fill=black!5,rounded corners=2mm] (-1.2,0.2) rectangle (1.2,-0.95);
    \node at (-0.85,-0.75) {$G_{a,b}$};

    \coordinate (x) at (-1,0);
    \coordinate (y) at (0,0);
    \coordinate (z) at (1,0);
    \coordinate (v) at (0,1.2);

    \coordinate (w) at (-0.6,-0.8);
    \coordinate (e1) at (0.4,-0.6);
    \coordinate (e2) at (0.5,-0.3);
    \coordinate (e3) at (0.1,-0.8);

    \draw[B] (e3) .. controls +(0.2,0.1) and +(0.2,-0.2) .. (y);

    \draw[Rthick, postaction={draw, R}] (z)
    .. controls +(-0.1,-0.3) and +(0.3,0.0)
    .. (e2)
    .. controls +(-0.2, 0.0) and +(-0.3,-0.8)
    .. (y) -- (v) -- (z);

    \node (x) [vdot] at (x) {$a$};
    \node (y) [vdot] at (y) {$b$};
    \node (z) [vdot] at (z) {$c$};
    \node (v) [vdot] at (v) {$v$};

    % \node (e1) [odotsmall] at (e1) {};
    % \node (e2) [odotsmall] at (e2) {};

    \draw[R,dashed] (x) -- node [vlabel,below,inner sep=3pt] {$e_{a,b}$} (y);
    \draw[R] (x) --
        node [vlabel,pos=0.3,inner sep=3pt,swap] {$e_a$}
        node [elabel,above] {non-cycle}
        (v);
    \path (y) --
        node [red,vlabel,pos=0.2,inner sep=4pt,swap] {$e_b$}
        node [pos=0.5,inner sep=3pt] (j2) {}
        node [red,elabel,above,inner sep=3pt] {cycle}
        (v);
    \draw[B] (z) --
        node [pos=0.3, vlabel] {$e_c$}
        node [pos=0.6,inner sep=3pt] (j1) {}
        node [elabel,above,inner sep=5pt] {attachment}
        (v);

    \draw[R] (x) .. controls +(-0.1,-0.7) and +(-0.1,0.4)
    .. (-0.6,-0.6) .. controls +(0.0,-0.1) and +(0.0,0.1)
    .. (w);

    \draw[R] (e2) .. controls +(0.2,0.0) and +(0.2,0.0) .. (e1);

    \draw[B] (z)
    .. controls +(0.2,-0.6) and +(0.2,0.0) ..
    (e3)
    .. controls +(-0.8,0.0) and +(1.1,-0.6) ..
    (x);

    \draw[->] (j1) to [bend left=12] node[below,xshift=2pt] {UE} (j2);

  \end{tikzpicture}
  \caption{Sketch of attachment, cycle, and non-cycle edges at a vertex of degree three.}\label{fig:vdeg3 labeling}

\end{figure}
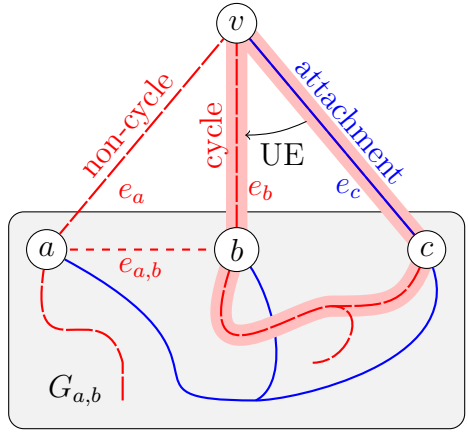

Before considering unique exchanges, let us summarize the previous definitions and theorems: given a graph $G$ with a vertex $v$ of degree three, we labeled the adjacent vertices $x$, $y$, and $z$, and the corresponding edges $e_x$, $e_y$, and $e_z$. We then showed that the vertices of $\vec{\tau}_3(G)$ can be mapped bijectively to one of the vertices in either $\vec{\tau}_3(G_{x,y})$, $\vec{\tau}_3(G_{x,z})$, or $\vec{\tau}_3(G_{y,z})$. These graphs are called the three \emph{reduction graphs} of $G$, since they are defined by reducing $G$ at $v$ in one of the three possible ways to reverse an edge-split-attach operation.

In the following, we will consider a \emph{particular} pair of disjoint spanning trees $(S,T)$ of $G$. As there are six possible ways the edges $\{ e_x, e_y, e_z \}$ are contained in $S$ and $T$, we define $(a,b) \in \{ (x,y), (x,z), (y,z) \}$ in the following definition to correspond to the reduction graph $G_{a,b}$ from which $(S,T)$ can be obtained via an edge-split-attach operation.

\begin{definition}[attachment, cycle, and non-cycle edges at a vertex of degree three]\label{def:vdeg3 edge names}
  Let $G = (V,E)$ be an atomic bispanning graph, $v \in V$ a vertex with $\deg(v) = 3$, the three adjacent vertices $x,y,z \in V$, the incident edges $e_x, e_y, e_z \in E$, and the three reduction graphs $G_{x,y}$, $G_{x,z}$, $G_{y,z}$ as defined in theorem~\ref{thm:reducing degree three}.

  If $(S,T) \in V_{\tau(G)}$ are two disjoint spanning trees of $G$, then one can select $(a,b) \in \{ (x,y), (x,z), (y,z) \}$ such that $e_a$ and $e_b$ are the pair of equally colored edges in either $S$ or $T$ ($\{e_a,e_b\} \subseteq S$ or $\{e_a,e_b\} \subseteq T$). Hence, $G_{a,b}$ is the one reduction graph where the edge-split-attach operation, which splits the edge $e_{a,b}$ into $e_a$ and $e_b$, can also be applied to the ``lift'' a pair of disjoint spanning trees $(S',T')$ from $G_{a,b}$ to $(S,T)$ in $G$, as described by $\rho_{e_{a,b},c}(S,T)$ in theorem~\ref{thm:decompose rho3}.

  We refer to the remaining edge $e_c$ with $c \in \{x,y,z\} \setminus \{a,b\}$, as the \emph{attachment} edge of $v$ and $(S,T)$, or sometimes as the \emph{single colored} edge.

  The edges $e_a$ and $e_b$ are hence the \emph{double colored} edges. Furthermore, exactly one of $e_a$ and $e_b$ is contained in $C_G(\cdot,e_c)$, the cycle closed by $e_c$ in the other tree. We call this edge the \emph{cycle} edge of $v$ and $(S,T)$, and the other the \emph{non-cycle} edge of $v$ and $(S,T)$ (see figure~\ref{fig:vdeg3 labeling}).
\end{definition}

As we already have a bijective decomposition of the vertices in $\vec{\tau}_3(G)$ to the exchange graphs $\vec{\tau}_3(G_{x,y})$, $\vec{\tau}_3(G_{x,z})$, and $\vec{\tau}_3(G_{y,z})$, we are now interested in identifying all unique exchanges of $\vec{\tau}_3(G)$ provided those of the reduction graphs. For this we will show that all unique exchanges of $\vec{\tau}_3(G)$ fall into one of the following four categories:
\begin{enumerate}
\item Unique exchanges \emph{lifted} from a reduction graph $G_{x,y}$, $G_{x,z}$, or $G_{y,z}$ to $G$. \\ We will show that most, but not all, unique exchanges can be lifted from $G_{x,y}$, $G_{x,z}$, and $G_{y,z}$, and that these deliver \emph{all} unique exchange of $G$ except those involving the three edges $e_x$, $e_y$, and $e_z$.

\item The leaf unique exchanges guaranteed by the single colored (attachment) edge among $\{ e_x, e_y, e_z \}$ for each tree pair $(S,T)$.

\item Unique exchanges \emph{forwarded} for the split edge $e_{a,b}$ in $G_{a,b}$ to either $e_a$ or $e_b$ in $G$ for each tree pair $(S,T)$. Both unique exchanges from and to the split edge $e_{a,b}$ are retained.

\item An additional unique exchange from the cycle edge to the attachment edge among $\{ e_x, e_y, e_z \}$, which only occurs under specific circumstances.

\end{enumerate}

To prove this classification as theorem~\ref{thm:vdeg3 classify}, we first have to establish conditions for the four classes of unique exchanges in $\vec{\tau}_3(G)$. We begin with the most broad of these four types: which of the unique exchanges from a reduction graph $\vec{\tau}_3(G_{a,b})$ can be lifted to $\vec{\tau}_3(G)$?
\begin{lemma}[lifting of unique exchanges from reduced graphs]\label{lem:vdeg3 lift ue}
  Let $G = (V,E)$ be an atomic bispanning graph, $v \in V$ a vertex with $\deg(v) = 3$, the three adjacent vertices $x,y,z \in V$, the incident edges $e_x, e_y, e_z \in E$, and the three reduction graphs $G_{x,y}$, $G_{x,z}$, $G_{y,z}$ as defined in theorem~\ref{thm:reducing degree three}. Furthermore, let $(a,b) \in \{ (x,y), (x,z), (y,z) \}$ select one of the reduction graph $G_{a,b}$.

  If $(e,f,S,T) \in E_{\vec{\tau}_3(G_{a,b})}$ is a unique edge exchange for a pair of disjoint spanning trees $(S,T) \in V_{\tau(G_{a,b})}$ in $G_{a,b}$ that does not involve the split edge $e_{a,b}$ ($e,f \neq e_{a,b}$), then
the unique edge exchange can be \emph{lifted} to $(e,f,\rho_{e_{a,b},c}(S,T)) \in E_{\vec{\tau}_3(G)}$, if
  \begin{equation}\label{eq:brokenUE S}
    e_{a,b} \in S \text{ and } e \notin D_{G_{a,b}}(S,e_{a,b}) \cap C_G(T + e_c, e_a) \cap C_G(T + e_c, e_b) \,,
  \end{equation}
  or
  \begin{equation}\label{eq:brokenUE T}
    e_{a,b} \in T \text{ and } e \notin D_{G_{a,b}}(T,e_{a,b}) \cap C_G(S + e_c, e_a) \cap C_G(S + e_c, e_b) \,.
  \end{equation}

  We call all unique exchange excluded by conditions~\eqref{eq:brokenUE S} and \eqref{eq:brokenUE T} \emph<broken!unique exchange>{broken} by $\rho_{e_{a,b},c}(S,T)$.

  Furthermore, for all spanning trees $(S,T) \in V_{\tau(G_{a,b})}$ and edge pairs $e,f \neq e_{a,b}$, the inverse is true as well: if $(e,f,S,T) \notin E_{\vec{\tau}_3(G_{a,b})}$ or if $(e,f,S,T) \in E_{\vec{\tau}_3(G_{a,b})}$ but conditions \eqref{eq:brokenUE S} and \eqref{eq:brokenUE T} are false, then $(e,f,\rho_{e_{a,b},c}(S,T)) \notin E_{\vec{\tau}_3(G)}$.
\end{lemma}
\begin{proof}
  Given $(e,f,S,T) \in E_{\vec{\tau}_3(G_{a,b})}$, we first assume that $(e,f) \in S \times T$ is a unique $S$ edge exchange in $G_{a,b}$, so we have $D_{G_{a,b}}(S,e) \cap C_{G_{a,b}}(T,e) = \{e,f\}$. To determine whether $(e,f)$ remains a unique exchange in $G$, we have to consider how the cut $D_{G_{a,b}}(S,e)$, the cycle $C_{G_{a,b}}(T,e)$, and their intersection change when performing the edge-split-attach operation $\rho_{e_{a,b},c}(S,T)$ to $(\overline{S},\overline{T})$ in $G$. The edge-split-attach operation goes from $G_{a,b}$ to $G$ by splitting $e_{a,b}$, adding $v$, and attaching $v$ to $c$ (see figure~\ref{fig:vdeg3 cut expansion}, page~\pageref{fig:vdeg3 cut expansion}). We consider all edges other than $e_{a,b}$ to remain identical under $\rho_{e_{a,b},c}(S,T)$.

  Both $D_{G_{a,b}}(S,e) \subseteq T + e$ and $C_{G_{a,b}}(T,e) \subseteq T + e$ (see remark~\ref{rem:cycle cut sets}, page~\pageref{rem:cycle cut sets}), hence if $e_{a,b} \in S$, then neither cut nor cycle can change, as $e_c \in \overline{T}$ is a leaf edge. Thus, we have $D_{G_{a,b}}(S,e) = D_G(\overline{S},e)$ and $C_{G_{a,b}}(T,e) = C_G(\overline{T},e)$. Hence, if $e_{a,b} \in S$, all unique $S$ edge exchanges (other than $e = e_{a,b}$) are retrained. These are included by the right hand side of condition~\eqref{eq:brokenUE S}, because $D_{G_{a,b}}(S,e_{a,b}) \cap C_G(T + e_c, e_a) \cap C_G(T + e_c, e_b)$ is a subset of $T + e_c$.

  So consider what happens when $e_{a,b} \in T$ is split into $e_a, e_b \in \overline{T}$. Any cycle $C_{G_{a,b}}(T,e) \subseteq T + e$ of $G_{a,b}$ containing $e_{a,b}$ will be expanded to a cycle of $G$ containing $e_a$ and $e_b$. All other cycles remain unchanged. Due to theorem~\ref{thm:duality cycle cut}, the split edge $e_{a,b} \in C_{G_{a,b}}(T,e)$, if and only if $e \in D_{G_{a,b}}(T, e_{a,b})$. Hence, $e \in D_{G_{a,b}}(T, e_{a,b})$ is an equivalent criterion for the cases where the cycle $C_{G_{a,b}}(T,e)$ expands to $C_G(\overline{T},e) = C_{G_{a,b}}(T,e) - e_{a,b} + e_a + e_b$. For the unique exchange to break, however, the two edges $e_a,e_b$ must \emph{additionally} be in the cut $D_G(\overline{S},e)$.

  Next consider how cuts $D_{G_{a,b}}(S,e)$ change due to the edge-split-attach operation. Again from theorem~\ref{thm:duality cycle cut}, we have $e_a \in D_G(\overline{S},e)$, if and only if $e \in C_G(\overline{S}, e_a)$, and likewise for $e_b$. Hence, we have three conditions, which need to occur simultaneously such that $D_G(\overline{S},e) \cap C_G(\overline{T},e) = \{e_a,e_b,f,f'\}$, which are $e \in D_{G_{a,b}}(T, e_{a,b})$, $e \in C_G(\overline{S}, e_a) = C_G(S + e_c, e_a)$, and $e \in C_G(\overline{S}, e_b) = C_G(S + e_c, e_b)$. If any condition is not met, then $D_G(\overline{S},e) \cap C_G(\overline{T},e)$ remains $\{e,f\}$ and the unique exchange can be lifted from $G_{a,b}$ to $G$. Condition \eqref{eq:brokenUE T} presents the three conditions as an edge intersection. The inverse is true as well: if $e$ is excluded by the conditions, then the intersection of cut and cycle is expanded and the unique exchange is broken.

  The case if $(e,f) \in T \times S$ is a unique $T$ edge exchange in $G_{a,b}$ is symmetrically. We have $D_{G_{a,b}}(T,e) \cap C_{G_{a,b}}(S,e) = \{e,f\}$. If $e_{a,b} \in T$, then the unique edge exchange can always be lifted to $G$, as the edge split operation cannot effect the intersection. If $e_{a,b} \in S$, then condition~\eqref{eq:brokenUE S} assures that the intersection remains size two, and the unique $T$ edge exchange be lifted to $G$.

  Finally, considering the inverse: if $(e,f,S,T) \notin E_{\vec{\tau}_3(G_{a,b})}$, then $|D_{G_{a,b}}(S,e) \cap C_{G_{a,b}}(T,e)| > 2$. As the edge split operation can only expand cycles and cuts by replacing $e_{a,b}$ with $e_a$ and $e_b$ in cycles, and $e_{a,b}$ with $e_a$, $e_b$, or both in cuts, the intersections remains $> 2$. Hence, $(e,f)$ cannot become a unique exchange in $G$.
\end{proof}

\begin{figure}
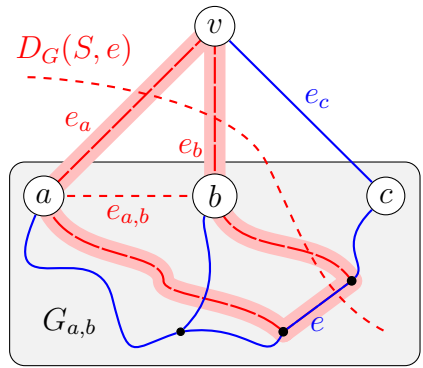
\centering
  \tikzset{every picture/.style={scale=2.5, graphfinal}}


  \caption{Sketch of a edge-split-attach operation, which breaks a unique exchange from $e$.}\label{fig:vdeg3 cut expansion}

\end{figure}

\begin{figure}[p]\centering
  %
  \captionsetup{justification=centering}
  \caption[The unique exchange graph $\tau_3(W_5)$.]{%
    The unique exchange graph $\tau_3(W_5)$.\\%
    Parallel $S$ and $T$ edge exchanges $(e,f)$ and $(f,e)$ are marked as $\{e,f\}$.}\label{fig:tau-w5}
\end{figure}

For an example of unique exchanges that can be lifted from reduction graphs consider $\tau_3(W_4)$ in figure~\ref{fig:tau-w5}.  The vertices are visually group to correspond to the reduction graph depicted in figure~\ref{fig:reducing w5}. First focus on the eight vertices marked with $7_\text{S}$ and $7_\text{T}$.  These correspond to the left-most reduction graph $B_{4,2}$ in figure~\ref{fig:reducing w5} with the split edge $\{0,1\}$. The full exchange graph $\tau_3(B_{4,2})$ is shown in figure~\ref{fig:tau-x4} (page~\pageref{fig:tau-x4}). Without the conditions \eqref{eq:brokenUE S} and \eqref{eq:brokenUE T} in theorem~\ref{lem:vdeg3 lift ue}, all unique exchanges would be lifted. But this is clearly not the case.

The split edge $\{0,1\}$ is explicitly excluded in theorem~\ref{lem:vdeg3 lift ue}. This exclusion breaks $\tau_3(B_{4,2})$ into two components: those where the split edge $\{0,1\}$ is blue and those where is it is red. The components contain four vertices and eight edges each, and (bijectively) map into $\tau_3(W_4)$ as the eight vertices marked with $7_\text{S}$ ($0$ and $1$ are red, $7$ is blue) and $7_\text{T}$ ($0$ and $1$ are blue, $7$ is red). Notice that no unique exchange edges go from $7_\text{S}$ to $7_\text{T}$, because these would require four edges to change color.

Of the eight edges in the two split components of $\tau_3(B_{4,2})$, only six are lifted into $\tau_3(W_5)$, the remaining two are excluded by \eqref{eq:brokenUE S} or \eqref{eq:brokenUE T}. While these conditions are stated rather technically, they have a surprisingly nice visual correspondence. This correspondence is better explained using the graph in figure~\ref{fig:vdeg3 large}, as $W_4$ only yields pathological examples. In this example, $e_{a,b} \in S$ is blue, hence only \eqref{eq:brokenUE S} applies. The cycle ``$C_G(T + e_c, e_a)$'' from \eqref{eq:brokenUE S} corresponds to the red cycle closed by coloring $e_a$ red. Hence, the term ``$C_G(T + e_c, e_a) \cap C_G(T + e_c, e_b)$'' is the intersection of the red cycle closed by $e_a$ and the red cycle closed by $e_b$. This intersection can be seen as the red path starting in $v$, going to $c$ via $e_c$, and onward up to the vertex where both cycles depart from another (which is the vertex $b$ in the example). For \eqref{eq:brokenUE S} to be fulfilled, the edges in this red path must \emph{also} be in the cut $D_{G_{a,b}}(S,e_{a,b})$, which can be calculated in the reduction graph $G_{a,b}$ on the left (or in $G$ on the right and excluding $e_c$). In the example, this leaves only one edge $e_1$ which is excluded by \eqref{eq:brokenUE S}, and hence the unique exchange $(e_1,e_2)$ is broken.

\begin{figure}
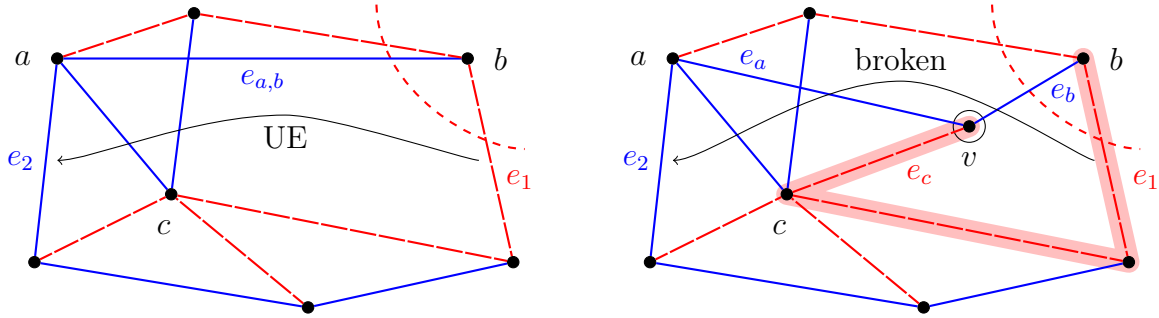
\centering
  \def\mygraph{

    \coordinate (c0) at (0,0);
    \coordinate (c1) at (-1.3,0.3);
    \coordinate (c2) at (-0.7,0.5);
    \coordinate (c3) at (0.5,0.3);
    \coordinate (c4) at (0.7,-0.6);
    \coordinate (c5) at (-0.2,-0.8);
    \coordinate (c6) at (-1.4,-0.6);
    \coordinate (c7) at (-0.8,-0.3);
  }


  \caption{The conditions of lemma~\ref{lem:vdeg3 lift ue} visualized and a broken unique exchange that needs five steps to mend.}\label{fig:vdeg3 large}

\end{figure}

The unique exchanges lifted from reduction graphs encompass all exchanges, except those involving the edges $e_x$, $e_y$, and $e_z$ at the vertex $v$ of degree three. Among these edges there are three classes of unique exchanges, which will be considered in the next three lemmata. The first is straight-forward: the single colored edge always delivers a leaf unique exchange.

\begin{lemma}[the attachment leaf unique exchanges of a degree three vertex]\label{lem:vdeg3 singleUE}
  Let $G = (V,E)$ be an atomic bispanning graph, $v \in V$ a vertex with $\deg(v) = 3$, the three adjacent vertices $x,y,z \in V$, and $(a,b) \in \{ (x,y), (x,z), (y,z) \}$ such that $G_{a,b}$ is one of the reduction graphs as defined in theorem~\ref{thm:reducing degree three}.

  If $(S,T) \in V_{\tau(G)}$ and $e_c$ is the attachment edge, $e_a$ the cycle edge, and $e_b$ the non-cycle edge of $v$ and $(S,T)$, then $(e_c,e_a,S,T) \in E_{\vec{\tau}_3(G)}$ is a leaf unique edge exchange.
\end{lemma}
\begin{proof}
  Without loss of generality, let $e_c \in S$ and $e_a \in T$. The edge $e_c$ is a leaf edge in $S$, hence $D_G(S,e_c) = \{ e_a, e_b, e_c \}$. As $e_c, e_a \in C_G(T,e_c)$, we have $D_G(S,e_c) \cap C_G(T,e_c) = \{ e_c, e_a \}$, and $(e_c,e_a) \in S \times T$ is a leaf unique $S$ edge exchange.
\end{proof}

Next consider unique edge exchanges involving the split edge $e_{a,b}$ of a reduction graph $G_{a,b}$ for a pair of trees $(S,T)$. We have two directions: unique exchanges $(e_{a,b},f)$ and unique exchanges $(f,e_{a,b})$, where $f$ is some other edge in the reduction graph (see figure~\ref{fig:vdeg3 forward ue}). Any unique edge exchange $(f,e_{a,b})$ \emph{can be forwarded} from $G_{a,b}$, but it remains unclear whether it becomes $(f,e_a)$ or $(f,e_b)$ in $G$. It turns out, that both possibilities occur, and to which unique exchange $(f,e_{a,b})$ is \emph{forwarded} in $G$ depends completely on how the attachment edge $e_c$ is connected to $f$, which determines how the cut $D_{G_{a,b}}(\cdot,f)$ changes to $D_G(\cdot,f)$; hence on the overall structure of the graph.

Any unique edge exchange $(e_{a,b},f)$ is \emph{also forwarded} from the reduction graph $G_{a,b}$, even though one can only color either $e_a$ or $e_b$ in one step. Instead of directly checking whether $(e_a,f)$ or $(e_b,f)$ are valid, one can determine these unique exchanges using theorem~\ref{thm:reversibility unique exchange}: for every unique exchange $(f,e_a)$ or $(f,e_b)$ from $(S,T)$ to $(S',T')$, we automatically know that $(e_a,f)$ or $(e_b,f)$ exists for $(S',T')$ to $(S,T)$, such that determining only one direction suffices.

\begin{figure}
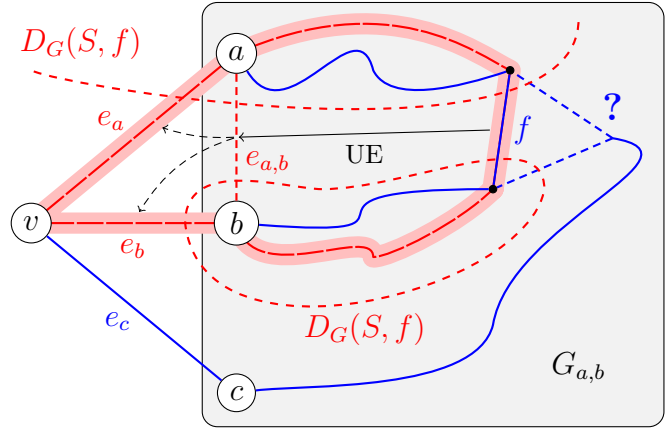
\centering
  \tikzset{every picture/.style={scale=2.5, graphfinal}}


  \caption{Sketch of how a unique exchange targeting the split edge $e_{a,b}$ in a reduction graph is forwarded to either $e_a$ or $e_b$ in $G$.}\label{fig:vdeg3 forward ue}

\end{figure}

\begin{lemma}[unique exchanges forwarded for the split edge of reduction graphs]\label{lem:vdeg3 forward ue}
  Let $G = (V,E)$ be an atomic bispanning graph, $v \in V$ a vertex with $\deg(v) = 3$, the three adjacent vertices $x,y,z \in V$, and $(a,b) \in \{ (x,y), (x,z), (y,z) \}$ such that $G_{a,b}$ is one of the reduction graphs as defined in theorem~\ref{thm:reducing degree three}.

  If $(f,e_{a,b},S,T) \in E_{\vec{\tau}_3(G_{a,b})}$ is a unique exchange in $G_{a,b}$ which targets the split edge $e_{a,b}$, then
  \begin{enumerate}
    \item $(f,e_2,\rho_{e_{a,b},c}(S,T)) \in E_{\vec{\tau}_3(G)}$ is a unique exchange, and
    \item $(e_2,f,S',T') \in E_{\vec{\tau}_3(G)}$ is a unique exchange in $G$,
  \end{enumerate}
  where $e_2 \in \{ e_a, e_b \}$ depends on how the attachment edge changes the cut $D_{G_{a,b}}(\cdot,e_{a,b})$, and $(S',T') \in V_{\tau(G)}$ is the resulting pair of tree in (i).
\end{lemma}
\begin{proof}
  Without loss of generality, let $e_c \in S$ and $e_a, e_b \in T$ in $G_{a,b}$, as also illustrated in figure~\ref{fig:vdeg3 forward ue}. Since $(f,e_{a,b})$ is a unique edge exchange in $G_{a,b}$, we have $D_{G_{a,b}}(S,f) \cap C_{G_{a,b}}(T,f) = \{ f, e_{a,b} \}$. The edge-split-attach takes $(S,T)$ to $(\overline{S},\overline{T}) := (S + e_c, T - e_{a,b} + e_a + e_b)$. It is clear that this extends the cycle such that $C_G(\overline{T}, f) = C_{G_{a,b}}(T,f) - e_{a,b} + e_a + e_b$. The cut is also extended, however, $D_{G}(\overline{S},f)$ cannot contain both $e_a$ and $e_b$, since $e_{a,b} \in D_{G_{a,b}}(S,f)$, which implies that $a$ and $b$ are in different components of $G_{a,b}[S] - f$. The attachment edge $e_c$ connects $v$ in $G[\overline{S}] - f$ to exactly one of the ends of $f$. This attachment enlarges $D_G(\overline{S},f)$ to include either $e_a$ or $e_b$, and possibly many additional edges branching off the path from $v$ to the end of $f$. However, since the cycle only increases by $e_a$ and $e_b$, $D_G(\overline{S},f) \cap C_G(\overline{S},f) = \{ f, e_a \}$ or $= \{ f, e_b \}$, hence, either $(f,e_a)$ or $(f,e_b)$ is a unique $S'$ edge exchange in $G$. Figure~\ref{fig:vdeg3 forward ue} further illustrates the proof, and shows both possible cuts $D_G(\overline{S},f)$. The second type of unique exchanges in (ii) immediately results from theorem~\ref{thm:reversibility unique exchange}.
\end{proof}

We have now established three classes of unique exchanges in $\vec{\tau}_3(G)$. The remaining class goes exclusively from the \emph{cycle edge} to the \emph{attachment edge} (in the reverse direction of the leaf unique exchange in lemma~\ref{lem:vdeg3 singleUE}). This unique exchange occurs only for some tree pairs $(S,T)$ of $G$, and we found no correspondence with the edge-split-attachment operation. It apparently appears solely due to the original unique exchange definition begin fulfilled for the two edges.

Figure~\ref{fig:vdeg3 cycle-attach} shows two example graphs and tree pairs, which show this extra unique exchange. If the two graphs are reduced at the circled vertex for this particular tree pair, then the reduction results in the indicated cycle and non-cycle edges, we left the attachment edge unlabeled. Since the ``cycle'' edge closes a cycle containing the attachment edge, for the extra unique exchange to occur, its intersection with the corresponding cut is decisive. This cut is related to the cut in $G_{a,b}$ in the same way as described in the proof of lemma~\ref{lem:vdeg3 forward ue}. However, both cycle and cut are selected directly by the parameters of the edge-split-attach operation, and apparently cannot be derived from the reduction graph.

\begin{figure}
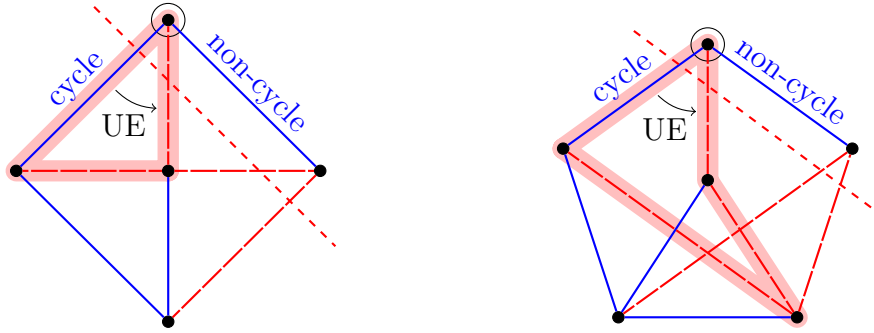


  \hfill%

  \hfill\null%

  \caption{Two examples of unique edge exchanges from the cycle edge to the attachment edge.}\label{fig:vdeg3 cycle-attach}
\end{figure}

\begin{lemma}[extra unique exchange from cycle edge to attachment edge]\label{lem:vdeg3 extra ue}
  Let $G = (V,E)$ be an atomic bispanning graph, $v \in V$ a vertex with $\deg(v) = 3$, the three adjacent vertices $x,y,z \in V$, and $(a,b) \in \{ (x,y), (x,z), (y,z) \}$ such that $G_{a,b}$ is one of the reduction graphs as defined in theorem~\ref{thm:reducing degree three}.

  If $(S,T) \in V_{\tau(G)}$ and $e_c$ is the attachment edge, $e_a$ the cycle edge, and $e_b$ the non-cycle edge of $v$ and $(S,T)$, then $(e_a,e_c,S,T) \in E_{\vec{\tau}_3(G)}$ is a unique edge exchange if $D_G(\cdot, e_a) \cap C_G(\cdot, e_a) = \{ e_a, e_c \}$.
\end{lemma}
\begin{proof}
  The unique exchange $(e_a,e_c,S,T) \in E_{\vec{\tau}_3(G)}$ is due immediately to definition~\ref{def:unique edge exchange}.
\end{proof}

We are now prepared to prove theorem~\ref{thm:vdeg3 classify}, since all possible unique exchanges in $\vec{\tau}_3(G)$ have been classified. We verified this classification using our computer program by composing the unique exchange graph $\vec{\tau}_3(G)$ from those of the three reduction graphs for all atomic bispanning graphs with at most ten vertices. This classification can also be seen as a \emph{composition} method to construct $\vec{\tau}_3(G)$ from the $\vec{\tau}_3$ of the reduction graphs and a few additional unique exchange checks at the vertex of degree three.

\begin{theorem}[classification of unique exchanges by reduction]\label{thm:vdeg3 classify}
  Let $G = (V,E)$ be an atomic bispanning graph, $v \in V$ a vertex with $\deg(v) = 3$, the three adjacent vertices $x,y,z \in V$, and $G_{x,y}$, $G_{x,z}$, $G_{y,z}$ the reduction graphs as defined in theorem~\ref{thm:reducing degree three}.

All unique exchanges in $\vec{\tau}_3(G)$ can be classified into exactly one of the four categories defined by lemmata~\ref{lem:vdeg3 lift ue}--\ref{lem:vdeg3 extra ue}.

\begin{enumerate}
\item Unique exchanges \emph{lifted} from a reduction graph $G_{x,y}$, $G_{x,z}$, or $G_{y,z}$ to $G$.

\item The leaf unique exchanges due to the single colored (attachment) edge among $\{ e_x, e_y, e_z \}$ for each tree pair $(S,T)$.

\item Unique exchanges \emph{forwarded} for the split edge $e_{a,b}$ in $G_{a,b}$ to either $e_a$ or $e_b$ in $G$ for each tree pair $(S,T)$.

\item An additional unique exchange from the cycle edge to the attachment edge among $\{ e_x, e_y, e_z \}$, which only occurs under specific circumstances.

\end{enumerate}
\end{theorem}
\begin{proof}
  To prove that all unique exchanges in $\vec{\tau}_3(G)$ can be classified using lemmata~\ref{lem:vdeg3 lift ue}--\ref{lem:vdeg3 extra ue}, consider any pair of disjoint spanning trees $(S,T) \in V_{\tau(G)}$. Let $e_a$ be the cycle edge, $e_b$ the non-cycle edge, $e_c$ the attachment edge, and $G_{a,b}$ the appropriate reduction graph as defined in definition~\ref{def:vdeg3 edge names}. In the following, we will regard all edges $e$ of $G$, and check that a unique exchange $(e,f)$ exists, if and only if it is classified according to the list in theorem~\ref{thm:vdeg3 classify}.

  First consider all edges $e \in S \dotcup T$ with $e \notin \{ e_a, e_b, e_c \}$. If $(e,f)$ is \emph{not} a unique exchange in $G_{a,b}$ for any $f$, then $(e,f)$ is not a unique exchange in $G$ for any $f$, because during an split-edge-attach operation cycles and cuts can only grow. Hence, unique exchanges cannot be created from the reduction graph $G_{a,b}$ by chance. All edges $e,f \in E[G_{a,b}]$, provided $e,f \neq e_{a,b}$, are handled by lemma~\ref{lem:vdeg3 lift ue}: they can be \emph{lifted} from $G_{a,b}$ if and only if they adhere to the conditions of the lemma, otherwise we call the unique exchange \emph{broken} by the edge-split-attach operation. Unique exchanges $(e,f)$ of $G_{a,b}$ where $e = e_{a,b}$ or $f = e_{a,b}$, are handled by lemma~\ref{lem:vdeg3 forward ue} and forwarded to $e_a$ or $e_b$ (on both sides of the unique exchange).

  The attachment edge $e_c$ always yields a unique exchange $(e_c,e_a,S,T)$ (by lemma~\ref{lem:vdeg3 singleUE}).

  Thus remains to prove that we handled all unique exchanges from $e_a$ and $e_b$ correctly. If $(e_a,f)$ is a unique exchange for any $f$, then it is handled implicitly by lemma~\ref{lem:vdeg3 forward ue} if $f \neq e_c$ and explicitly by lemma~\ref{lem:vdeg3 extra ue} if $f = e_c$. Similarly, if $(e_b,f)$ is a unique exchange for any $f$, then it is handled by lemma~\ref{lem:vdeg3 forward ue}, since $(e_b, e_c)$ is impossible.

  Hence, all unique exchanges of $\vec{\tau}_3(G)$ are classified by lemmata~\ref{lem:vdeg3 lift ue}--\ref{lem:vdeg3 extra ue}.
\end{proof}

\begin{figure}
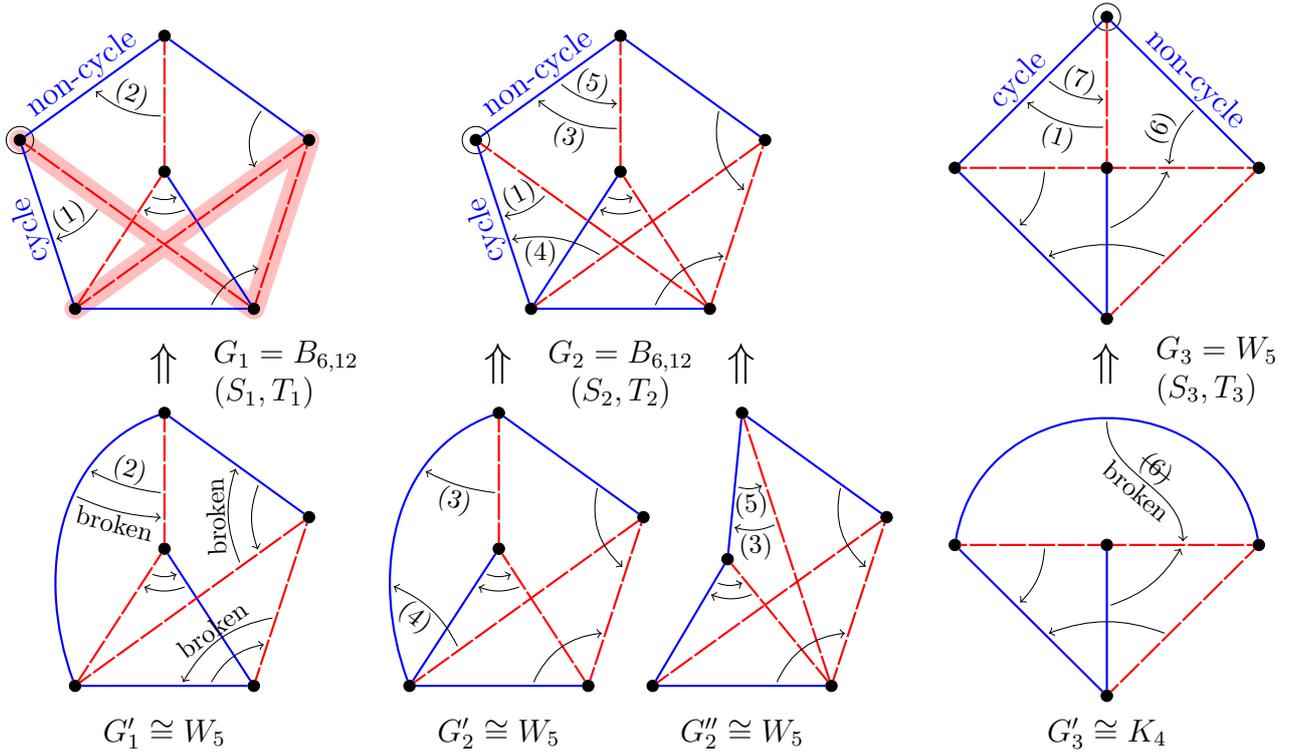


  \hfill%

  \hfill\null%

  \caption{Three examples of the classification of unique edge exchanges in $\vec{\tau}_3(G)$ graphs.}\label{fig:vdeg3 complete}
\end{figure}

To illustrate the classification method of unique exchanges in $\vec{\tau}_3(G)$ regard the three examples in figure~\ref{fig:vdeg3 complete}. The top row contains three atomic bispanning graphs $G_1$, $G_2$, and $G_3$, each shown with one particular pair of disjoint spanning trees. These are to be reduced at the circled degree three vertex, and the bottom row contains the three corresponding reduction graphs $G'_1$, $G'_2$, and $G'_3$, together with a fourth reduction graph $G''_2$ needed for an explanation later. Each small black arrow indicates a possible unique edge exchange for the tree pair and graph. We will discuss those exchanges marked with numbers in more detail.

All unique exchange arrows without number in the top row of graph are lifted directly from the reduction graphs $G'_1$, $G'_2$, and $G'_3$ in the bottom row according to lemma~\ref{lem:vdeg3 lift ue}. Reduction graphs $G'_2$ and $G'_3$ contain no broken unique exchanges, while $G'_1$ contains three broken unique exchanges which are not lifted to $G_1$. For two of these this is due to the long intersection of $C_{G_1}(S + e_c,e_a) \cap C_{G_1}(S + e_c,e_b)$. The third is due to exchanges $(e_{a,b},f)$ being broken, unless implied by lemma~\ref{lem:vdeg3 forward ue}.

The unique exchange marked with (1) in each graph is due to the leaf unique exchange from the attachment edge to the cycle edge, as described in lemma~\ref{lem:vdeg3 singleUE}.

Unique exchange (2) in $G'_1$ is an instance of the form $(f,e_{a,b})$ where $e_{a,b}$ is the split edge: it is forwarded into $G_1$ by determining the target, the non-cycle edge, as described in lemma~\ref{lem:vdeg3 forward ue}.

Unique exchanges (3) and (4) in $G'_2$ are also instances of the form $(f,e_{a,b})$ where $e_{a,b}$ is the split edge: these are forwarded into $G_2$, one targeting the cycle edge and the other the non-cycle edge due to the way the cut is changed (lemma~\ref{lem:vdeg3 forward ue}).

Unique exchange (5) in $G_2$ is of the form $(e_a,f)$ or $(e_b,f)$, which is an exchange that can only be implied by lemma~\ref{lem:vdeg3 singleUE} as the reverse of a unique exchange of the form $(f,e_{a,b})$ from a different reduction graph. This other reduction graph is $G''_2$, which shows the unique exchange (5) of the form $(f,e_{a,b})$ which implies the exchange (5) in $G_2$.

Unique exchange (6) in $G_3$ is another exchange implied by lemma~\ref{lem:vdeg3 singleUE} as the reverse of a unique exchange from a different reduction graph. Curiously, (6) does not correspond to the unique exchange marked with \sout{(6)} in $G'_3$. This type of unique exchange cannot be lifted, since it would require coloring two edges in $G_3$.

Unique exchange (7) is of the extra type from cycle edge to attachment edge as described in lemma~\ref{lem:vdeg3 extra ue}, which apparently can only be determined directly.

The examples in figure~\ref{fig:vdeg3 complete} were chosen to exhibit all classes of unique exchanges described by theorem~\ref{thm:vdeg3 classify}. As mentioned above, the decomposition was tested using a computer program for all atomic bispanning graphs with at most ten vertices.

Of the classes described in theorem~\ref{thm:vdeg3 classify}, those ``lifted'' from the reduction graphs are the most frequent as the regarded atomic bispanning graphs increase in size. Consequently, the conditions for broken unique exchanges, excluded in the lifting theorem~\ref{lem:vdeg3 lift ue} are probably most important when regarding larger graphs.

% ------------------------------------------------------------------------------

\subsection{Approaches to Proving Connectivity of \texorpdfstring{$\tau_3(G)$}{tau3(G)}}\label{sec:connectivity ideas}

In this section we summarize a number of attempted approaches to show that $\tau_3(G)$ is connected. None of them form a full proof, many are mere ideas that one could follow in future work. Alongside these attempts, we report computational evidence for the connectivity of $\tau_3(G)$, which provides hints about the discussed proof ideas.  We then close this section by presenting the most ``difficult'' small bispanning graphs.

% ------------------------------------------------------------------------------

\subsubsection{Peeling Vertices with Degree Three}

The first approach coming to mind is to recursively ``peel'' vertices of degree three by first taking the leaf unique exchange guaranteed by the single color attachment edge, and then reversing an edge-split-attach operation to gain a smaller graph. Due to the discussion in section~\ref{sec:reduce vdeg3}, the reader may already suspect that this approach is too naive to be correct. We discuss it here despite of this, since it is the most intuitive inductive approach.

First, let us clarify this naive \emph{peeling algorithm}. The following procedure calculates a sequence of unique exchanges:
\begin{enumerate}
\item Given an atomic bispanning graph $G$ and a pair of disjoint spanning trees $(S,T)$, select a vertex $v$ with degree three.

\item Determine the cycle edge $e_a$, the non-cycle edge $e_b$, and attachment edge $e_c$ at $v$ in $(S,T)$. We will assume, without loss of generality, $e_a, e_b \in S$ and $e_c \in T$.

\item Append the leaf unique exchange $(e_c,e_a)$ to the output sequence, and let $(S',T') := (S - e_a - e_b + \{ e_a,e_c \}, T - e_c)$ be the pair of disjoint spanning trees in $G'$, \emph{reduced} at the vertex $v$ \emph{after} the unique exchange.

\item If the reduced bispanning graph $G'$ is still atomic, repeat, otherwise perform a different decomposition.

\end{enumerate}
The procedure leaves open how the calculated unique exchange sequence in the reduced graphs is lifted back to the graph $G$. It also leaves open how to manage non-atomic graphs. Instead of discussing these open questions, we will consider an example, in which there is no correct solution to these questions.

In figure~\ref{fig:vdeg3 peel}, we performed the naive peeling algorithm on the triangle-free atomic bispanning graph with seven vertices. The algorithm delivers the unique exchange sequence $\arr{ (2,0), (4,5), (7,\{3,4\}), (\{6,7\},9) }$ until it reduced $G$ to the composite bispanning graph $B_{3,2}$, which contains four possible unique exchanges.

Now, we can try to reapply this sequence to the base graph in figure~\ref{fig:vdeg3 peel apply}: $(2,0)$, $(4,5)$ can be applied straight-forwardly. All edges already swapped are marked with check marks. The unique exchange $(7,\{3,4\})$ also maps to $G$ as desired: $(7,3)$ is valid. The first difficulty arises when trying to apply $(\{6,7\},9)$: is $6$ or $7$ the edge to swap? In our example, $7$ has already been swapped, so we'll assume that $6$ is the correct choice. This leads to the unique exchange $(6,9)$, which nicely matches $(\{6,7\},9)$.

But then two edge pairs remain unswapped: $(1,8)$ and $(11,9)$ are the only viable next unique exchanges (excluding steps to undo previous ones). But these do not match those delivered by the peeling algorithm: $(11,\{1,2\})$, $(9,\{8,\{6,7\}\})$ or one these reversed. They are explicitly misordered: no reordered or swapping can fix the sequence.

This example uncovers the problems of the peeling approach: the further a graph is reduced, the less the unique exchanges can be mapped back to the original graph.

One possible modification to the peeling algorithm is to first recurse into the reduced graph, and then append the unique exchange to the sequence. This reverses the order they are applied to the original graph: the deepest ones first, then the outer ones. But this approach is just as problematic, because it completely ignores the broken unique exchanges described in theorem~\ref{lem:vdeg3 lift ue}.

\begin{figure}
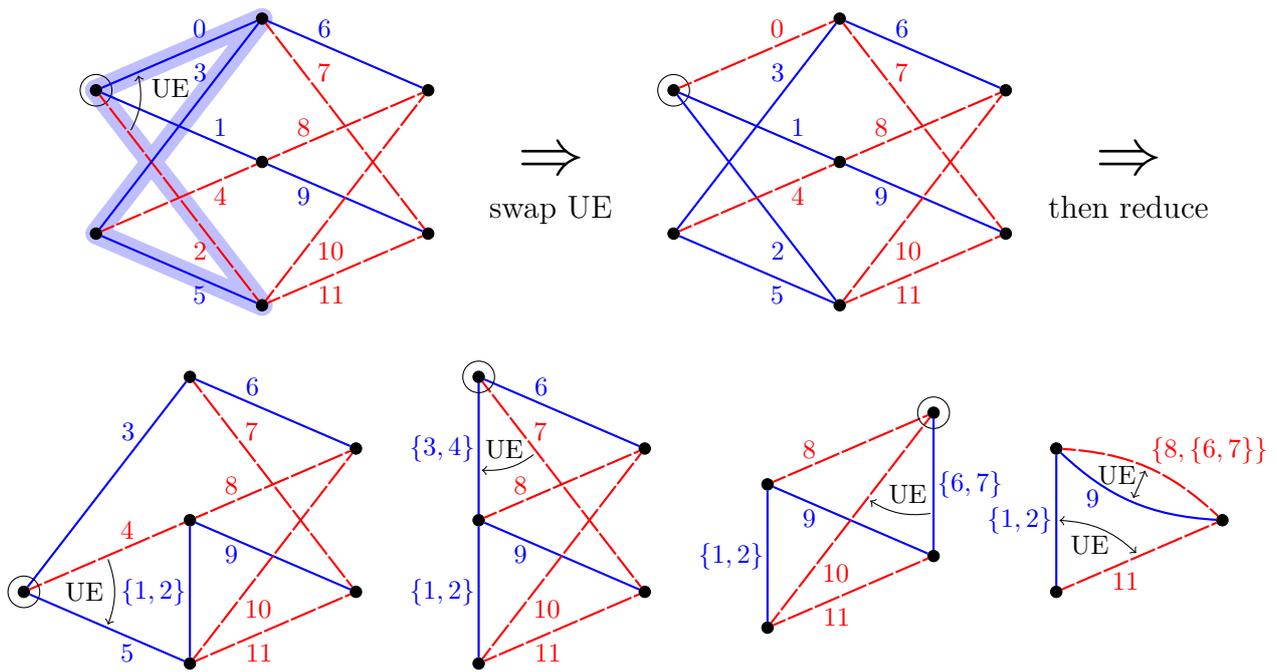
\centering


  \caption{``Peeling'' vertices of degree three from an atomic bispanning graph to find a sequence of unique exchanges.}\label{fig:vdeg3 peel}

\end{figure}

\begin{figure}
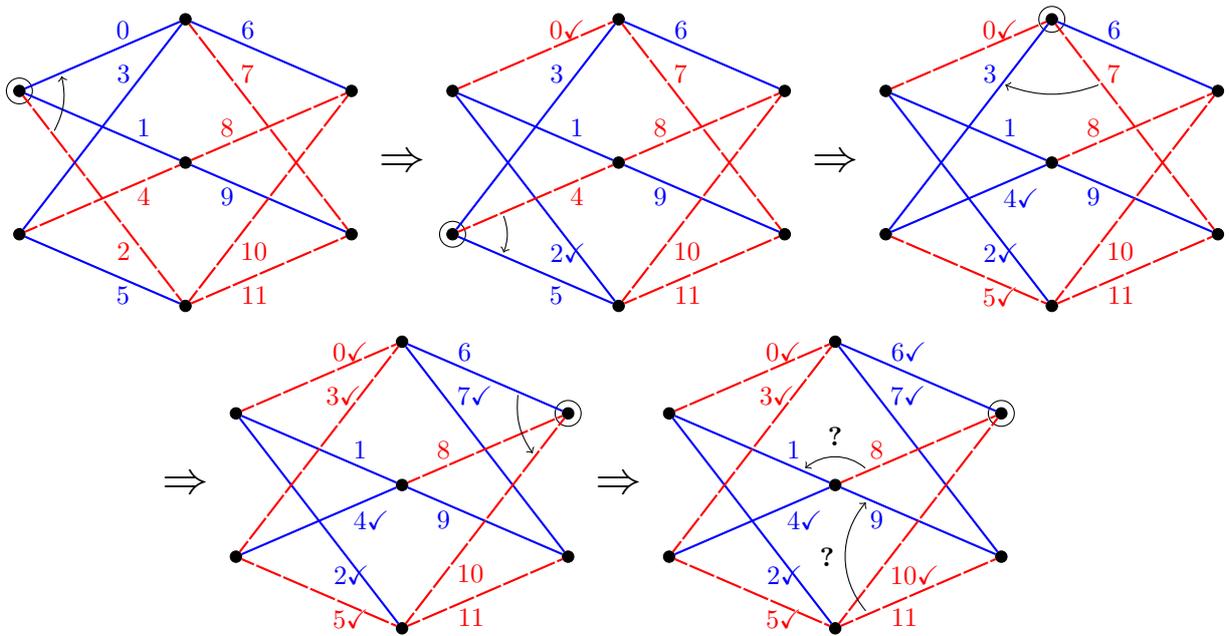
\centering

  \def\mygraph{
      \node [odot] (n0) at (-2.3,1) {0};
      \node [odot] (n1) at (-2.3,-1) {1};
      \node [odot] (n2) at (0,2) {2};
      \node [odot] (n3) at (0,0) {3};
      \node [odot] (n4) at (0,-2) {4};
      \node [odot] (n5) at (2.3,1) {5};
      \node [odot] (n6) at (2.3,-1) {6};
  }

  %
  \caption{Applying the unique exchanges from the ``peeling'' sequence in figure~\ref{fig:vdeg3 peel} leads to mismatching exchanges.}\label{fig:vdeg3 peel apply}

\end{figure}

\subsubsection{Attempting to Mend Broken Unique Exchanges}\label{sec:mend-brokenUEs}

Another approach picks up the ideas successfully applied to construct UECBOs for atomic bispanning graphs with vertex-connectivity two in theorem~\ref{thm:join-2sum} (page~\pageref{thm:join-2sum}).

Given an atomic bispanning graph $G$ and a pair of disjoint spanning trees $(S,T)$, one can assume to have a UECBO for the corresponding reduction graph(s) $G_{a,b}$. To lift this UECBO back to $G$ one has to insert the additional unique exchange involving the edges $\{ e_a, e_b, e_c \}$. Since we know the structure of these exchanges, this is probably easily doable by replacing the split edge $e_{a,b}$ by a sequence containing the attachment leaf unique exchange.

However, as described in theorem~\ref{lem:vdeg3 lift ue}, some of the unique exchanges in the UECBO sequence may be broken. Hence, this approach hinges on finding a method to \emph{mend} broken unique exchanges in the UECBO by finding an alternative swap sequence.

This appears simple, however, the number of steps in an alternative path through $\tau_3(G)$ is not necessarily small. The broken unique exchange in figure~\ref{fig:vdeg3 large} (page~\pageref{fig:vdeg3 large}) is an example which needs \emph{at least five} additional unique exchange steps to mend. This shows that mending broken unique edge exchanges is equivalent to finding a unique exchange swap sequence for any pair of non-uniquely exchangeable edges, which is probably an even more difficult problem than proving the connectivity of $\tau_3(G)$.

\subsubsection{Inverting Branches of a Tree using Leaf-UEs}

The third approach is based on the idea that every atomic bispanning graph has at least four leaf unique exchanges at the four vertices of degree three. Extending this idea to ``branches'' in the spanning tree seems natural: these branches yield sequences of leaf unique exchanges. And since the spanning tree can be decomposed into branches, this may lead to a UECBO without recursion.

Using our computer program, we experimentally tested that $\tau_3(G)$, if restricted to only leaf unique edge exchanges, remains connected for all atomic bispanning graphs with vertex- and edge-connectivity and at most 16 vertices. At the same time, examples of bispanning graphs exist such that the diameter of $\tau_3(G)$ restricted to leaf unique exchanges is larger than $\frac{|E|}{2}$ (tree pairs exist which cannot be inverted using only $\frac{|E|}{2}$ leaf unique exchanges). Hence, we cannot hope to find UECBOs consisting of only leaf unique exchanges in general.

The problem with following branches using leaf unique exchanges is that they tend to cancel each other out: if $\arr{ e_1,e_2,e_3 }$ is a branch in one of the trees, then $(e_1,f_1)$ is a unique exchange which makes $e_2$ a leaf edge. One would then continue with $(e_2,f_2)$, etc., but very often $f_2 = e_1$, hence undoing progress. If one ignores this, and continues finishing this branch, and thereafter using another branch, then this algorithm can go into loops.

Connected with the idea of recursively taking UECBOs in the previous section, and also with looking at branches in spanning trees, is the following general theorem by Dirac, which applies well to atomic bispanning graphs:
\begin{theorem}[atomic bispanning graphs have a minor isomorphic to $K_4$ \R{\cite{dirac1952property}}]\label{thm:dirac}
  A $2$-vertex-connected simple graph in which the degree of every vertex is at least three has a minor isomorphic to $K_4$.
\end{theorem}
Using this theorem it might be possible to partition an atomic bispanning graph into four ``regions'', which are separated by the paths forming the ``skeleton'' of $K_4$ between the four vertices of degree three. One could then find unique exchange sequences recursively for the four piece, and combine then while following the unique exchange rules given by $\tau_3(K_4)$.

\subsubsection{3-Clique Sum Decomposition}

In section~\ref{sec:decompose 2vconn} we considered $2$-clique sums of bispanning graphs, and decompositions at vertex cuts of size two. The success of that section suggests looking into $3$-clique sums, with which graphs can be joined at triangles.

Figure~\ref{fig:example clique sums} (page~\pageref{fig:example clique sums}) already showed a $3$-clique sum of $K_4$ and $W_5$. When considering the resulting sum, we have to immediately note that of the six edges in the two $3$-cliques, \emph{exactly two} have to be retained for the edge balance $|V| = 2 |E| - 2$ of bispanning graphs to hold.

This brings up the first problem of joining two arbitrary bispanning graphs at a triangle using a $3$-clique sum: $|E'| = 2$ is required in definition~\ref{def:clique-sum}, but how should the two edges be selected? It is easy to construct examples where selecting the edges arbitrarily leads to invalid bispanning graphs.

\begin{remark}[validity of $3$-clique sum of two small bispanning graphs]
  Using our computer program, we verified that for any pair of simple bispanning graphs with at most five vertices, and any configuration of triangles in the two graphs, at least \emph{two} different selections of two edges $E'$ exist such that the resulting $3$-clique sum is a valid bispanning graph.
\end{remark}

Hence, bispanning graphs can apparently always be joined, but it is completely unclear how to determine the two edges without checking each possible choice. The same is not true if \emph{two pairs of spanning trees} are given for the two bispanning graphs: it is not always possible to keep the spanning trees (excluding the four removed edges) during the $3$-clique sum. This situation is similar to the join conditions for $2$-clique sums, where the two edges need to have unequal colors, but in the $3$-clique sum case it is unknown what these conditions are.

Regarding the opposite, decomposition of atomic bispanning graphs with vertex-connectivity three, we also found the following:
\begin{remark}[number of edges inside a vertex cut of size three]
  No vertex cut of size three in an atomic bispanning graphs with at most eleven vertices has more than two edges incident solely to vertices in the cut.
\end{remark}
Hence, there are always at most two edge ``inside'' the vertex cut. The remark can probably be proven using Nash-Williams' theorem~\ref{thm:atomic nash-williams} on atomic bispanning graphs.  This low number of edges makes it possible to decompose all atomic bispanning graphs with vertex-connectivity three using a $3$-clique sum.

However, it is unclear how to join the two $\tau_3(G)$ graphs at the $3$-clique sum, since one would have to consider how the cuts and cycles change. This then depends on the two edges in the vertex cut, but also on those incident to the vertices. It is unclear how to prove something for this large number of combinations.

%%%%%%%%%%%%%%%%%%%%%%%%%%%%%%%%%%%%%%%%%%%%%%%%%%%%%%%%%%%%%%%%%%%%%%%%%%%%%%%%
\clearpage

\section{Conclusion: Empirical Evidence for Connectivity of \texorpdfstring{$\tau_3(G)$}{tau3(G)}}

As stated above, none of the approaches in this section could be developed to a full proof of the connectivity of $\tau_3(G)$ for all bispanning graphs in this thesis. Using our computer program, we verified that $\tau_3(G)$ is connected for all simple bispanning graphs with up to 20 vertices.

The approach to mend broken unique exchanges discussed in subsection~\ref{sec:mend-brokenUEs} hinges on finding an alternative path. We verified empirically that this alternative path can be found \emph{even inside} the subcomponent of $\tau_3(G)$ lifted from $\tau_3(G_{a,b})$ within which the broken unique exchange occurs: these six subcomponents of $\tau_3(G)$ themselves are already connected for all atomic bispanning graphs with at most twelve vertices.  One can see this for example in $\tau_3(W_5)$ in figure~\ref{fig:tau-w5} (page~\pageref{fig:tau-w5}): the six marked groups are themselves connected.

However, we did not find a way to prove this. In future, one could try to use the conditions in lemma~\ref{lem:vdeg3 lift ue} to prove the connectivity of these subcomponents. The trick may be that broken unique exchanges are very rare, they only have one color, and are located along a path starting at the attachment edge. Hence, there cannot be ``too many'' of them. We empirically verified this: the larger the bispanning graphs grow, the less broken unique exchanges occur. Together with the theorems from section~\ref{sec:reduce vdeg3}, it is then easy to connect the subcomponents with the attachment edges.

Another untouched aspect of the composition at vertices of degree three may further help: that the resulting exchange graph of the composition must be equal \emph{regardless} at which of the four degree three vertices one decomposed the atomic bispanning graph. We did not find any way to use this fact, but believe that it may be crucial.

Since we could not find any counter examples for connectivity of $\tau_3(G)$, we calculated the most ``difficult'' instances of the problem for small numbers of vertices. For this we determined for each graph $G$ the \emph{minimum number of different paths} from $(S,T)$ to $(T,S)$ over all pairs of spanning trees $(S,T)$ in $\tau_3(G)$. We call this number $\nu(G)$ and assume that the less paths there are, the more ``difficult'' the instance is.

Figure~\ref{fig:bispanning difficult} in the appendix shows the most difficult pairs of spanning trees for bispanning graphs with at most twelve vertices. All the difficult graphs have only the minimum of four unique exchanges in their initial configuration. We invite the reader to try to solve the more difficult graphs in the minimum number of necessary steps using the Java applet at \url{http://panthema.net/2016/uegame/}, they are available under the menu ``named graphs''.

With respect to the connectivity of $\tau_3(G)$, one would like to see that $\nu$ grows with the number of vertices. However, figure~\ref{fig:difficult B12,1} has only 24 possible exchange paths of length $\frac{|E|}{2}$ for this particular pair of spanning trees.  Due to these empirical results, we believe that $\tau_3(G)$ is always connected, but that bispanning graphs exist for which $\tau_3(G)$ has a diameter larger than $\frac{|E|}{2}$.

%%%%%%%%%%%%%%%%%%%%%%%%%%%%%%%%%%%%%%%%%%%%%%%%%%%%%%%%%%%%%%%%%%%%%%%%%%%%%%%%
\clearpage
\appendix

\section{Bispanning Graphs Collection}\label{sec:collection}

\begin{figure}[b!]
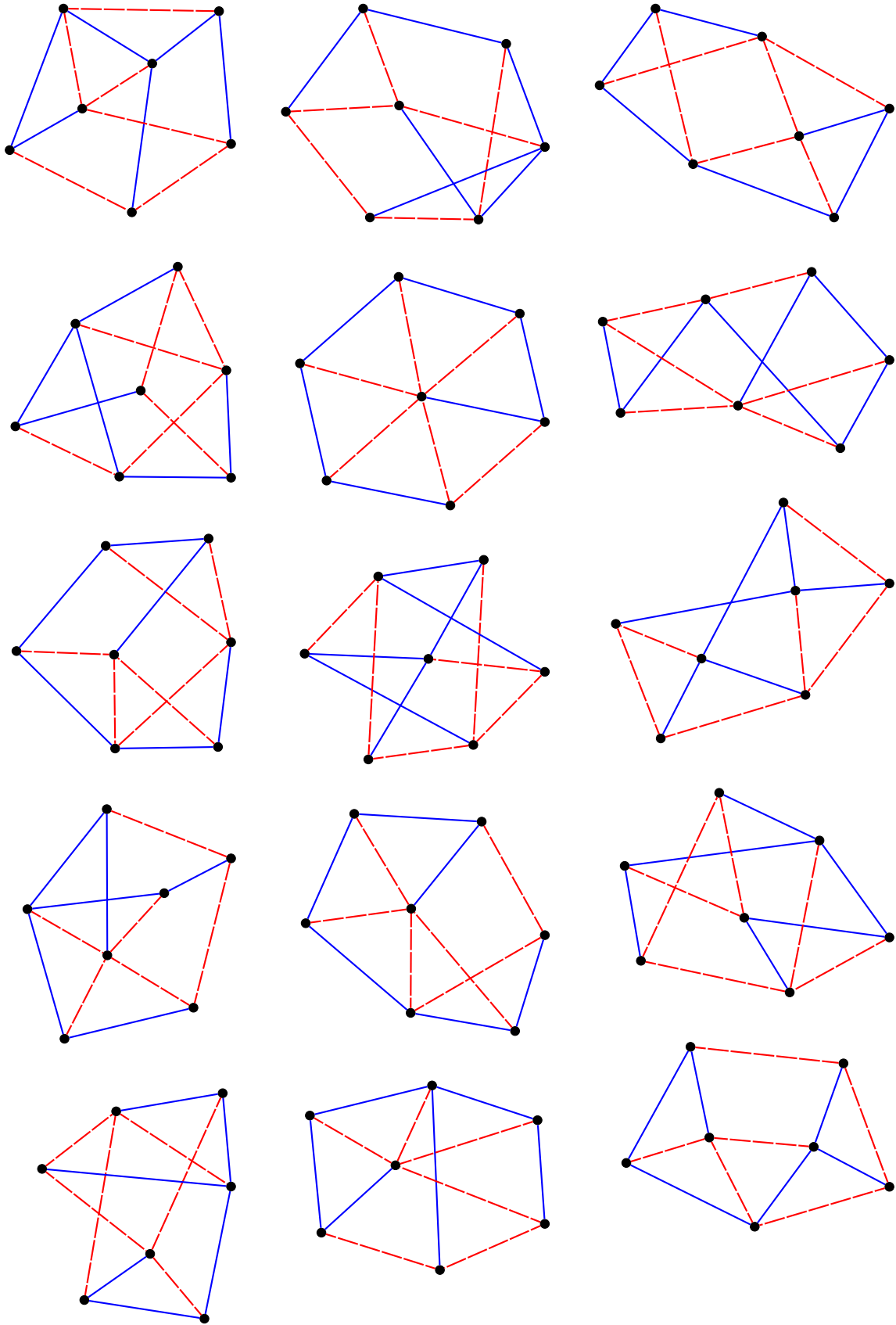
\centering
  \tikzset{every picture/.style={scale=0.25, graphfinal}}

  \caption{All non-isomorphic atomic bispanning graphs with seven vertices.}\label{fig:bispanning7}
\end{figure}

\begin{figure}[p]\centering
  \tikzset{every picture/.style={scale=0.14, graphfinal}}
  %
  \caption{All non-isomorphic atomic bispanning graphs with eight vertices.}\label{fig:bispanning8-1}
\end{figure}

\begin{figure}[p]\centering
  \ContinuedFloat
  \tikzset{every picture/.style={scale=0.14, graphfinal}}
  %
  \caption[]{All non-isomorphic atomic bispanning graphs with eight vertices (continued).}\label{fig:bispanning8-2}
\end{figure}

\begin{figure}[p]\centering

  \hfill%
  \subcaptionbox{$K_4$, $\nu = 8$}{%
    %
  }
  \hfill%
  \subcaptionbox{$W_5$, $\nu = 24$}{%
    %
  }
  \hfill%
% V6:i0x0y13/i1x20y13/i2x40y13/i3x10y30/i4x20y0/i5x30y30/;E10:i0t0h3c2/i1t1h3c1/i2t2h3c1/i3t0h4c1/i4t1h4c2/i5t2h4c2/i6t0h5c1/i7t1h5c2/i8t2h5c1/i9t3h5c2/;
  \subcaptionbox{$B_{6,12}$, $\nu = 8$}{%
    %
  }
  \hfill\null
  \bigskip

  \hfill%
  % V7:i0x855y425/i1x903y702/i2x305y693/i3x288y413/i4x577y561/i5x596y360/i6x569y756/;E12:i0t0h4c2/i1t1h4c1/i2t2h4c2/i3t3h4c1/i4t0h5c1/i5t1h5c1/i6t2h5c2/i7t3h5c2/i8t0h6c2/i9t1h6c2/i10t2h6c1/i11t3h6c1/;
  \subcaptionbox{$B_{7,1}$, only triangle-free, $\nu = 84$}{%
    %
  }
  \hfill%
  % V8:i0x10y20/i1x20y10/i2x40y35/i3x0y35/i4x40y15/i5x0y15/i6x30y30/i7x20y40/;E14:i0t0h4c1/i1t1h4c2/i2t2h4c1/i3t0h5c1/i4t1h5c2/i5t3h5c2/i6t0h6c2/i7t1h6c1/i8t2h6c2/i9t3h6c1/i10t0h7c2/i11t1h7c1/i12t2h7c1/i13t3h7c2/;
  \subcaptionbox{$B_{8,1}$, only triangle-free, $\nu = 178$}{%
    %
  }
  \hfill\null%
  \bigskip

  \hfill%
  % V9:i0x1y0/i1x3y2/i2x1y2/i3x1y1/i4x3y1/i5x0y1/i6x0y2/i7x2y2/i8x2y1/;E16:i0t0h4c1/i1t1h4c1/i2t0h5c2/i3t2h5c2/i4t3h5c1/i5t0h6c2/i6t2h6c1/i7t3h6c2/i8t1h7c1/i9t2h7c1/i10t3h7c2/i11t4h7c2/i12t0h8c1/i13t1h8c2/i14t2h8c2/i15t3h8c1/;
  \subcaptionbox{$B_{9,1}$, $\nu = 284$}{%
    %
  }
  \hfill%
  % V9:i0x1y0/i1x3y2/i2x1y2/i3x1y1/i4x3y1/i5x0y1/i6x0y2/i7x2y2/i8x2y1/;E16:i0t0h4c1/i1t1h4c1/i2t0h5c2/i3t2h5c2/i4t3h5c1/i5t0h6c2/i6t2h6c1/i7t3h6c2/i8t1h7c1/i9t2h7c1/i10t3h7c2/i11t4h7c2/i12t0h8c1/i13t1h8c2/i14t2h8c2/i15t3h8c1/;
  \subcaptionbox{$B_{9,2}$, $2 {\times} 2$-$K_4$ grid, $\nu = 288$}{%
    %
  }
  \hfill\null%
  \bigskip

  \hfill%
  % V10:i0x20y6/i1x0y-25/i2x0y25/i3x-20y6/i4x30y-12/i5x-30y13/i6x30y13/i7x-30y-12/i8x-10y-10/i9x10y-10/;E18:i0t0h4c1/i1t1h4c2/i2t0h5c2/i3t2h5c1/i4t2h6c2/i5t3h6c1/i6t4h6c2/i7t1h7c1/i8t3h7c2/i9t5h7c1/i10t0h8c1/i11t1h8c2/i12t2h8c1/i13t3h8c2/i14t0h9c2/i15t1h9c1/i16t2h9c2/i17t3h9c1/;
  \subcaptionbox{$B_{10,1}$, triangle-free, $\nu = 122$}{%
    %
  }
  \hfill%
  % V10:i0x10y10/i1x10y20/i2x30y20/i3x30y10/i4x0y25/i5x40y5/i6x40y25/i7x0y5/i8x20y10/i9x20y20/;E18:i0t0h4c2/i1t1h4c1/i2t2h5c2/i3t3h5c1/i4t2h6c1/i5t3h6c2/i6t4h6c1/i7t0h7c1/i8t1h7c2/i9t5h7c2/i10t0h8c1/i11t1h8c1/i12t2h8c2/i13t3h8c2/i14t0h9c2/i15t1h9c2/i16t2h9c1/i17t3h9c1/;
  \subcaptionbox{$B_{10,2}$, triangle-free, $\nu = 124$}{%
    %
  }
  \hfill\null%

  \caption[The most ``difficult'' bispanning graphs and pairs of spanning trees.]{The most ``difficult'' bispanning graphs and pairs of spanning trees, where $\nu$ is the number of different paths from $(S,T)$ to $(T,S)$ through $\tau_3(G)$.}\label{fig:bispanning difficult}
\end{figure}

\begin{figure}\centering
  \ContinuedFloat

  \hfill%
  % V11:i0x30y24/i1x0y24/i2x15y0/i3x25y17/i4x5y17/i5x15y20/i6x30y8/i7x0y8/i8x20y8/i9x15y30/i10x10y8/;E20:i0t0h5c1/i1t1h5c2/i2t0h6c2/i3t2h6c2/i4t3h6c1/i5t1h7c1/i6t2h7c1/i7t4h7c2/i8t2h8c2/i9t3h8c1/i10t4h8c1/i11t5h8c2/i12t0h9c2/i13t1h9c1/i14t3h9c2/i15t4h9c1/i16t2h10c1/i17t3h10c2/i18t4h10c2/i19t5h10c1/;
  \subcaptionbox{$B_{11,1}$, triangle-free, $\nu = 224$}{%

  }
  \hfill%
  % V11:i0x3y2/i1x1y2/i2x3y1/i3x1y1/i4x4y1/i5x0y1/i6x4y2/i7x0y2/i8x2y2/i9x2y0/i10x2y1/;E20:i0t0h4c1/i1t1h5c2/i2t0h6c2/i3t2h6c1/i4t4h6c2/i5t1h7c1/i6t3h7c2/i7t5h7c1/i8t0h8c2/i9t1h8c1/i10t2h8c1/i11t3h8c2/i12t2h9c2/i13t3h9c1/i14t4h9c2/i15t5h9c1/i16t0h10c1/i17t1h10c2/i18t2h10c2/i19t3h10c1/;
  \subcaptionbox{$B_{11,2}$, $\nu = 496$}{%
    %
  }
  \hfill\null%
  \bigskip

  \hfill%
  % V12:i0x30y30/i1x10y30/i2x30y0/i3x10y0/i4x25y10/i5x20y20/i6x40y20/i7x0y20/i8x30y20/i9x10y20/i10x15y10/i11x20y0/;E22:i0t0h5c1/i1t1h5c2/i2t0h6c2/i3t2h6c1/i4t1h7c1/i5t3h7c2/i6t1h8c1/i7t2h8c2/i8t4h8c1/i9t6h8c2/i10t0h9c2/i11t3h9c1/i12t4h9c2/i13t7h9c1/i14t2h10c1/i15t3h10c1/i16t4h10c2/i17t5h10c2/i18t2h11c2/i19t3h11c2/i20t4h11c1/i21t5h11c1/;
  \subcaptionbox{$B_{12,1}$, triangle-free, $\nu = 24$\label{fig:difficult B12,1}}{%
    %
  }
  \hfill%
  % V12:i0x10y0/i1x30y0/i2x30y20/i3x15y30/i4x12y40/i5x40y30/i6x0y10/i7x40y10/i8x25y30/i9x10y20/i10x0y30/i11x28y40/;E22:i0t0h5c1/i1t1h5c2/i2t0h6c1/i3t2h6c1/i4t1h7c2/i5t2h7c2/i6t2h8c1/i7t3h8c1/i8t4h8c2/i9t5h8c2/i10t1h9c1/i11t3h9c1/i12t4h9c2/i13t6h9c2/i14t0h10c2/i15t3h10c2/i16t4h10c1/i17t7h10c1/i18t2h11c2/i19t3h11c2/i20t4h11c1/i21t5h11c1/;
  \subcaptionbox{$B_{12,2}$, $\nu = 80$}{%
    %
  }
  \hfill\null%
  \bigskip

  \hfill%
  % V12:i0x10y5/i1x10y25/i2x30y10/i3x30y20/i4x10y10/i5x10y20/i6x30y5/i7x30y25/i8x20y10/i9x3y15/i10x37y15/i11x20y20/;E22:i0t0h4c2/i1t1h5c1/i2t0h6c1/i3t2h6c2/i4t1h7c2/i5t3h7c1/i6t2h8c1/i7t3h8c2/i8t4h8c2/i9t5h8c1/i10t0h9c1/i11t1h9c2/i12t4h9c1/i13t5h9c2/i14t2h10c1/i15t3h10c2/i16t6h10c1/i17t7h10c2/i18t2h11c2/i19t3h11c1/i20t4h11c1/i21t5h11c2/;
  \subcaptionbox{$B_{12,3}$, $\nu = 160$}{%
    %
  }
  \hfill%
  % V12:i0x10y10/i1x25y20/i2x40y10/i3x40y0/i4x10y0/i5x10y20/i6x50y10/i7x0y10/i8x40y20/i9x20y10/i10x30y10/i11x25y0/;E22:i0t0h5c1/i1t1h5c2/i2t2h6c1/i3t3h6c2/i4t0h7c2/i5t4h7c1/i6t5h7c2/i7t1h8c1/i8t2h8c2/i9t6h8c1/i10t0h9c1/i11t1h9c2/i12t3h9c1/i13t4h9c2/i14t1h10c1/i15t2h10c2/i16t3h10c1/i17t4h10c2/i18t0h11c2/i19t2h11c1/i20t3h11c2/i21t4h11c1/;
  \subcaptionbox{$B_{12,4}$, $\nu = 168$}{%
    %
  }
  \hfill\null%

  \caption[]{The most ``difficult'' bispanning graphs and pairs of spanning trees, where $\nu$ is the number of different paths from $(S,T)$ to $(T,S)$ through $\tau_3(G)$ (continued).}
\end{figure}

\begin{figure}
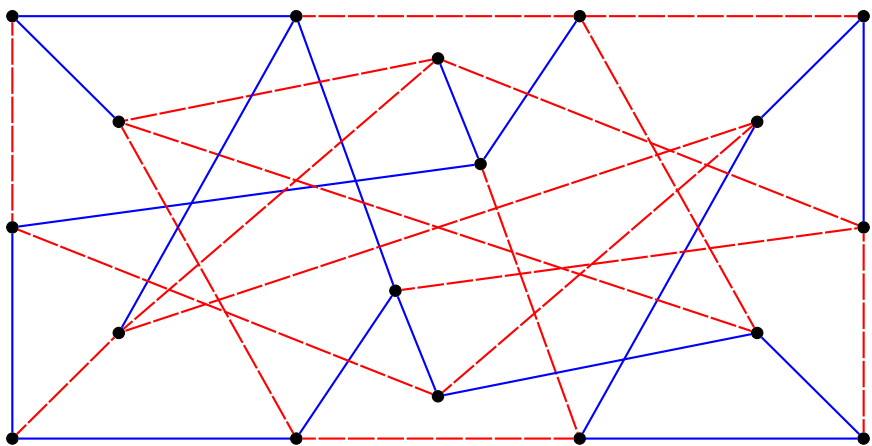
\centering
  %
  \caption{$B_{18,1}$, one of eight square-free bispanning graphs with $|V| = 18$.}\label{fig:bispanning-squarefree}
\end{figure}

%%%%%%%%%%%%%%%%%%%%%%%%%%%%%%%%%%%%%%%%%%%%%%%%%%%%%%%%%%%%%%%%%%%%%%%%%%%%%%%%

\clearpage
\phantomsection
\def\indexname{Index and Symbols}
\addcontentsline{toc}{section}{\indexname}
\def\indexspace{\par\medskip}
% prepend symbol glossary to index
\renewenvironment{theindex}{%
  % two columns with less space inbetween
  \columnsep=15pt
  \columnseprule=0pt
  \twocolumn[\section*{\indexname}\label{sec:index}]%
  \markboth{\MakeMarkcase{\indexname}}{\MakeMarkcase{\indexname}}%
  \parindent=0pt
  \setlength{\parskip}{0pt plus .3 pt}%
  \setlength{\parfillskip}{0pt plus 1fil}%
  % pdf bookmarks for each letter
  \def\textbf##1{{\pdfbookmark[2]{##1}{indexletter.##1}\bfseries{##1}}}

  \printglossary[type=symbols]%
  \indexspace

  \printglossary[type=exgraphs]%
  \indexspace

  \renewcommand*\item{\par\hangindent 40pt}
}{
  \onecolumn
}
\printindex

%\DeclareFieldFormat{isbn}{\mkbibacro{ISBN}\addcolon\space \textl{#1}}

\clearpage
\phantomsection
\labelalphawidth=10ex
\def\refname{\bibname}
\addcontentsline{toc}{section}{\refname}
\printbibliography

\end{document}